\documentclass[11pt,british,twoside]{report}

\usepackage{amsmath, amsfonts, amsthm, amssymb, multicol}
\usepackage{graphicx}
\usepackage{float}
\usepackage{verbatim}
\usepackage{epigraph}

\addtocontents{toc}{\setcounter{tocdepth}{1}}
\numberwithin{equation}{section}

\usepackage{stmaryrd}
\SetSymbolFont{stmry}{bold}{U}{stmry}{m}{n} %

\usepackage{IEEEtrantools}
\usepackage{bm}

\usepackage{etoolbox}
\usepackage{hyperref}
\usepackage[T1]{fontenc}
\usepackage[utf8]{inputenc}
\usepackage{babel}
\usepackage{csquotes}
\usepackage{mathtools}
\usepackage[protrusion=true,expansion=true]{microtype}
\usepackage{mathrsfs}
\usepackage{enumerate}
\usepackage{latexsym, euscript, epic, eepic, color}
\usepackage{blindtext}
\usepackage{subfigure}
\usepackage{enumitem}
\usepackage{amscd}

\usepackage{comment}
\usepackage{bbm}
\usepackage{dsfont}
\usepackage{booktabs}
\usepackage{wrapfig}
\usepackage{caption}

\usepackage{longtable}
\usepackage{listings}
\usepackage[scaled]{beramono}
\usepackage{pgffor}
\usepackage[all]{xy}

\usepackage{bookmark}

\usepackage{tikz, pgfplots}
\pgfplotsset{compat=1.15}
\usetikzlibrary{intersections,positioning,cd,calc,bending}
\usetikzlibrary{decorations.markings,shapes}
\usetikzlibrary{arrows,arrows.meta}
\usetikzlibrary{patterns}
\tikzcdset{arrow style=tikz,diagrams={>=stealth}}
\tikzset{%
    lvtx/.style={circle,inner sep=1pt,draw=black},
    vtx/.style={circle,inner sep=1.2pt,draw=black,fill=black},
    elbl/.style={fill=white,circle,inner sep=1pt},
    edge/.style={thick,->,>=stealth,shorten <=1pt,shorten >=1pt}
}

\PassOptionsToPackage{hyphens}{url}

\allowdisplaybreaks

\usepackage[a4paper,left=30mm,right=30mm,top=30mm,bottom=30mm]{geometry}

\newtheorem{thm}{Theorem}[section]
\newtheorem*{theorem*}{Theorem}
\newtheorem{lemma}[thm]{Lemma}
\newtheorem{lma}[thm]{Lemma}
\newtheorem{theorem}[thm]{Theorem}
\newtheorem{definition}[thm]{Definition}
\newtheorem{defn}[thm]{Definition}
\newtheorem{prop}[thm]{Proposition}
\newtheorem{example}[thm]{Example}
\newtheorem{corollary}[thm]{Corollary}
\newtheorem{cor}[thm]{Corollary}

\newtheorem{rem}[thm]{Remark}
\newtheorem{question}[thm]{Question}

\newtheorem{proofpart}{Part}
\AtBeginEnvironment{proof}{\setcounter{proofpart}{0}} %

\definecolor{trp}{rgb}{1,1,1}

\newcommand\setItemnumber[1]{\setcounter{enumi}{\numexpr#1-1\relax}}
\definecolor{red}{rgb}{1,0,.2}

\newcommand{\N}{\mathbb{N}}
\newcommand{\Rd}{\mathbb{R}^d}
\newcommand{\ubd}{\overline{\dim}_{\mathrm{B}}}
\newcommand{\bd}{\dim_{\mathrm{B}}}
\newcommand{\lbd}{\underline{\dim}_{\mathrm{B}}}
\newcommand{\uid}{\overline{\dim}_{\theta}}
\newcommand{\lid}{\underline{\dim}_{\theta}}
\newcommand{\upd}{\overline{\dim}^{\Phi}}
\newcommand{\lpd}{\underline{\dim}^{\Phi}}
\newcommand{\asd}{\dim_{\mathrm{A}}}
\newcommand{\asp}{\dim_{\mathrm{A}}^\theta}

\newcommand{\uasp}{\overline{\dim}_\mathrm{A}^\theta}
\newcommand{\qad}{\dim_\mathrm{qA}}
\newcommand{\fix}{\mathrm{fix}}
\newcommand{\R}{\mathbb{R}}

\newcommand{\btau}{\boldsymbol\tau}

\newcommand{\ii}{\mathbf{i}}
\newcommand{\jj}{\mathbf{j}}

\newcommand{\iih}{\boldsymbol{\hat\imath}}
\newcommand{\jjh}{\boldsymbol{\hat\jmath}}
\newcommand{\kkh}{\mathbf{\hat k}}
\newcommand{\ih}{\hat\imath}
\newcommand{\jh}{\hat\jmath}
\newcommand{\kh}{\hat k}
\newcommand{\iiv}{\overline{\imath}}
\newcommand{\jjv}{\overline{\jmath}}

\newcommand{\supp}{\mathrm{supp}}

\newcommand{\diam}{\operatorname{diam}}

\newcommand{\diniu}[1]{D^{#1}\!}
\newcommand{\dinil}[1]{D_{#1}}

\renewcommand{\epsilon}{\varepsilon}

\renewcommand{\geq}{\geqslant}
\renewcommand{\leq}{\leqslant}

\DeclareMathOperator{\dimH}{\operatorname{dim}_{\mathrm{H}}}

\DeclareMathOperator{\dimlB}{\underline{\operatorname{dim}}_{\mathrm{B}}}
\DeclareMathOperator{\dimuB}{\overline{\operatorname{dim}}_{\mathrm{B}}}
\DeclareMathOperator{\dimA}{\operatorname{dim}_{\mathrm{A}}}
\DeclareMathOperator{\dimL}{\operatorname{dim}_{\mathrm{L}}}

\DeclareMathOperator{\Q}{{\mathbb{Q}}}
\DeclareMathOperator{\Z}{{\mathbb{Z}}}

\usepackage[backend=biber,mincrossrefs=9,doi=false,
url=false,isbn=false,
date=year,style=alphabetic,
giveninits=true,maxnames=99,maxalphanames=99,
sortcites=true,sorting=nyt,
hyperref=true,backref=true,clearlang=true]{biblatex}
\renewbibmacro{in:}{%
  \ifentrytype{article}{}{\printtext{\bibstring{in}\intitlepunct}}}
\DeclareFieldFormat[article,periodical]{volume}{\mkbibbold{#1}}%
\DeclareFieldFormat[article,periodical,unpublished]{citetitle}{#1}
\DeclareFieldFormat[article,periodical,unpublished]{title}{#1}
\DeclareFieldFormat{pages}{#1}
\DeclareFieldFormat[inproceedings]{title}{#1\isdot}
\DeclareFieldFormat[misc]{title}{#1\isdot}
\DeclareFieldFormat[misc,article]{note}{#1\addcomma}

\DeclareLabelalphaTemplate{%
  \labelelement{
    \field[final]{shorthand}
    \field{label}
    \field[strwidth=3,strside=left,ifnames=1]{labelname}
    \field[strwidth=1,strside=left]{labelname}
  }
}
\DeclareFieldFormat{extraalpha}{#1}%

\makeatletter
\def\verbatim@font{\normalfont\rmfamily}
\makeatother
\makeatletter
\DeclareFieldFormat{eprint:arxiv}{%
  arXiv\addcolon\space
  \ifhyperref
    {\href{http://arxiv.org/\abx@arxivpath/#1}{%
       \nolinkurl{#1}%
       \iffieldundef{eprintclass}
     {}
     {\addspace\mkbibbrackets{\thefield{eprintclass}}}}}
    {\nolinkurl{#1}
     \iffieldundef{eprintclass}
       {}
       {\addspace\mkbibbrackets{\thefield{eprintclass}}}}}
\makeatother

\AtEveryBibitem{%
  \clearlist{language}%
}
\AtEveryBibitem{
  \clearfield{number}}
\begin{document}

\begin{titlepage}
    \centering
    \vfill
    {\bfseries\Huge
        Interpolating between Hausdorff and box dimension\\
        \vskip2cm
        \huge Amlan Banaji\\
    }    
    \vfill
    \includegraphics[width=5cm]{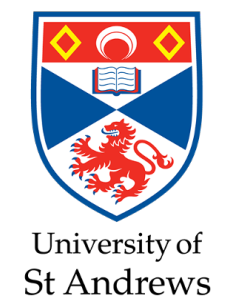} %
    \vfill
    {\large
        A Thesis Submitted for the Degree of PhD \\
at the\\
 University of St Andrews\\
 \vskip1cm
     Submitted 2023
}
    \vfill
\end{titlepage}

\pagenumbering{roman}

\makeatletter
\renewenvironment{abstract}{%
    \if@twocolumn
      \section*{\abstractname}%
    \else %
      \begin{center}%
        {\bfseries \LARGE\abstractname\vspace{\z@}}%
      \end{center}%
      \quotation
    \fi}
    {\if@twocolumn\else\endquotation\fi}
\makeatother
\begin{abstract}
\addcontentsline{toc}{chapter}{Abstract}
\noindent Hausdorff and box dimension are two familiar notions of fractal dimension. 
Box dimension can be larger than Hausdorff dimension, because in the definition of box dimension, all sets in the cover have the same diameter, but for Hausdorff dimension there is no such restriction. 
This thesis focuses on a family of dimensions parameterised by ${\theta \in (0,1)}$, called the \emph{intermediate dimensions}, which are defined by requiring that $\diam(U) \leq (\diam(V))^{\theta}$ for all sets $U,V$ in the cover. \smallbreak
\noindent We begin by generalising the intermediate dimensions to allow for greater refinement in how the relative sizes of the covering sets are restricted. 
These new dimensions can recover the interpolation between Hausdorff and box dimension for compact sets whose intermediate dimensions do not tend to the Hausdorff dimension as ${\theta \to 0}$. 
We also use a Moran set construction to prove a necessary and sufficient condition, in terms of Dini derivatives, for a given function to be realised as the intermediate dimensions of a set. \smallbreak
\noindent We proceed to prove that the intermediate dimensions of limit sets of infinite conformal iterated function systems are given by the maximum of the Hausdorff dimension of the limit set and the intermediate dimensions of the set of fixed points of the contractions. 
This applies to sets defined using continued fraction expansions, and has applications to dimensions of projections, fractional Brownian images, and general H\"older images. \smallbreak
\noindent Finally, we determine a formula for the intermediate dimensions of all self-affine Bedford--McMullen carpets. The functions display features not witnessed in previous examples, such as having countably many phase transitions. 
We deduce that two carpets have equal intermediate dimensions if and only if the multifractal spectra of the corresponding uniform Bernoulli measures coincide. This shows that if two carpets are bi-Lipschitz equivalent then the multifractal spectra are equal. 
\end{abstract}

\chapter*{Acknowledgements}
\addcontentsline{toc}{chapter}{Acknowledgements}
\section*{General acknowledgements}

First and foremost, it is a pleasure to thank my PhD supervisor, Jonathan Fraser. He has always been there to provide excellent advice when I have needed it, and has been patiently willing to help whenever I have encountered any issues. 
I am also very grateful to my second supervisor, Kenneth Falconer, for proofreading draft versions of this thesis and my papers really thoroughly, and for many invaluable discussions. The depth and breadth of my supervisors' knowledge of fractal geometry is outstanding, and I am appreciative of the way they guided me through the first years of my PhD, which were affected by the Covid pandemic. At the same time, it has been great to have the freedom to explore my own mathematical interests. 
I am extremely grateful to Lars Olsen and P\'eter Varj\'u for examining this thesis and making very many helpful comments and suggestions. 
\smallbreak
I am grateful to everyone in the School of Mathematics and Statistics at the University of St Andrews who helped make my time here so rewarding. 
The history of the analysis group is inextricably linked to that of my office, which we call the `bunker.' 
I have had the pleasure of sharing the space with Raquel Couto, Douglas Howroyd, Aleksi Py\"or\"al\"a, Liam Stuart, Shuqin Zhang and Boyuan Zhao. 
Particular thanks go to Stuart Burrell and Lawrence Lee for being so helpful while I was settling in. 
My time at St Andrews would not have been the same without other members of the analysis group including Natalia Jurga, Istv\'an Kolossv\'ary, Lars Olsen and Mike Todd. I will miss the weekly trips to the Whey Pat pub. 
Special thanks go to my friend Alex Rutar - 
I have learned a great deal from our long discussions, about not just mathematics but so many different aspects of life! 
I would like to thank all the friends I have made from across the School and wider St Andrews community (too many to list individually). I have enjoyed the annual trips to the Burn and many lunches in the common room! 
\smallbreak
During my PhD, it has been helpful in many different ways to have got to know mathematicians from across the UK and around the world. 
I have received lots of welcome comments about my work, many of which have resulted in improvements to content that is included in this thesis. In particular, I thank Simon Baker, Haipeng Chen (Clarence), De-Jun Feng, Kathryn Hare, Antti K\"aenm\"aki, Henna Koivusalo, Thomas Jordan, Hui Rao, Pablo Shmerkin, Ville Suomala, Justin Tan, Sascha Troscheit, Mariusz Urba\'nski, Yimin Xiao, Alexia Yavicoli, Han Yu, Yuchen Zhu, and several anonymous referees. \smallbreak
Finally, I would like to thank my family. I thank my father, Murad Banaji, for proofreading a draft of this thesis and for many enjoyable discussions about mathematics and academia. I also thank him, and my mother, Kirsten Shirke, for all their support and encouragement throughout these years.

\section*{Publications and collaboration}

Whilst at St Andrews, I have worked on many projects including \cite{Banaji2023gen,Banaji2022moran,
Banaji2021infinite,Banaji2021bedford,
Banaji2022assouad,
Banaji2022popcorn,Banaji2022geo,
BanajiPreprintphiassouad,BanajiPreprintgl}, all of which are freely available on arXiv (see the links below and in the bibliography). 
Several of these papers have been collaborations, and I have learned a great deal from each of these. 
I hope that I can continue to work with my co-authors many times more in the future. 
Chapters~\ref{s:generalised}, \ref{s:attainable}, \ref{s:infinite} and \ref{s:bm} are broadly based on the following four papers respectively. 

\begin{itemize}

\item A.~Banaji. Generalised intermediate dimensions. \\
\emph{Monatsh. Math.}, \textbf{202} (2023), 465--506. \href{https://arxiv.org/abs/2011.08613}{https://arxiv.org/abs/2011.08613}

\item A.~Banaji and A.~Rutar. Attainable forms of intermediate dimensions. \\
\textit{Ann. Fenn. Math.} \textbf{47} (2022), 939--960. \href{https://arxiv.org/abs/2111.14678}{https://arxiv.org/abs/2111.14678}

\item A.~Banaji and J.~M.~Fraser. Intermediate dimensions of infinitely generated attractors. \\
\textit{Trans. Amer. Math. Soc.} \textbf{376} (2023), 2449--2479. \href{https://arxiv.org/abs/2104.15133}{https://arxiv.org/abs/2104.15133}  

\item A.~Banaji and I.~Kolossv\'ary. Intermediate dimensions of Bedford--McMullen carpets with applications to Lipschitz equivalence. \\
\textit{Adv. Math.} \textbf{449} (2024), 109735. \href{https://arxiv.org/abs/2111.05625}{https://arxiv.org/abs/2111.05625}

\end{itemize}

In addition, there is a small amount of material from the following two papers. 

\begin{itemize}

\item A.~Banaji and H.~Chen. Dimensions of popcorn-like pyramid sets. \\
\emph{J. Fractal Geom.} \textbf{10} (2023), 151--168. \href{https://arxiv.org/abs/2212.06961}{https://arxiv.org/abs/2212.06961}

\item A.~Banaji and J.~M.~Fraser. Assouad type dimensions of infinitely generated self-conformal sets. \\
\emph{Nonlinearity} \textbf{37} (2024), 045004. \href{https://arxiv.org/abs/2207.11611}{https://arxiv.org/abs/2207.11611}

\end{itemize}
I thank Alex Rutar, Jonathan Fraser, Istv\'an Kolossv\'ary, Haipeng Chen and Sascha Troscheit for our collaborations, and for agreeing for me to use material from these papers in this thesis. 
I will explain in more detail in Section~\ref{s:structure} where in the thesis each of these papers is used. Each of the papers constitutes substantial original research and has been published in a high quality peer-reviewed mathematics journal.

\section*{Funding}
This work was supported by a Leverhulme Trust Research Project Grant (RPG-2019-034).

\section*{Copyright and further comments}

The copyright for this thesis `Interpolating between Hausdorff and box dimension' rests with the author, Amlan Banaji. Any information derived from it should be acknowledged. 
This version of the thesis is licenced under the Creative Commons Attribution~4.0 International (CC~BY~4.0) license: \href{https://creativecommons.org/licenses/by/4.0/}{https://creativecommons.org/licenses/by/4.0/}

This arXiv version of the thesis is very similar to the version which was uploaded to the St Andrews Research Repository on 28/11/2023 (after viva and revisions) at 

\href{https://research-repository.st-andrews.ac.uk/handle/10023/28591}{https://research-repository.st-andrews.ac.uk/handle/10023/28591} 

The DOI \href{https://doi.org/10.17630/sta/642}{https://doi.org/10.17630/sta/642} can be used to cite or link to the research repository version. 
Since then I have updated the bibliography and made some typo corrections. 
Comments, questions, and mistake/typo corrections are very welcome; after leaving St Andrews I began a postdoc with Simon Baker at Loughborough University, and my up-to-date email address can be found on my personal website \href{https://amlan-banaji.github.io/}{https://amlan-banaji.github.io/} 

Based on work in this thesis, I have written a two-page `microthesis' which was published in the May 2024 edition of the Newsletter of the London Mathematical Society. It can be found at \href{https://amlan-banaji.github.io/files/MicrothesisLMS.pdf}{https://amlan-banaji.github.io/files/MicrothesisLMS.pdf}   

Some key words and phrases associated with this thesis are: fractal geometry, Hausdorff dimension, box dimension, intermediate dimensions, dimension interpolation, Moran set, iterated function system, self-conformal, self-affine, Bedford--McMullen carpet.

\chapter*{Declarations}
\addcontentsline{toc}{chapter}{Declarations}
\noindent {\large \textbf{Candidate's declaration}}\medbreak
\noindent I, Amlan Faiyaz Banaji, do hereby certify that this thesis, submitted for the degree of PhD, which is approximately 39,000 words in length, has been written by me, and that it is the record of work carried out by me, or principally by myself in collaboration with others as acknowledged, and that it has not been submitted in any previous application for any degree. I confirm that any appendices included in my thesis contain only material permitted by the `Assessment of Postgraduate Research Students' policy.\medbreak
\noindent I was admitted as a research student at the University of St Andrews in September 2019. \medbreak
\noindent I received funding from an organisation or institution and have acknowledged the funder(s) in the full text of my thesis. \medbreak
\noindent {\bf Date:} 23/02/2023 \hspace{0.5cm} {\bf Signature of Candidate:}  
  \bigbreak
\noindent {\large \textbf{Supervisor's declaration}} \medbreak
\noindent I hereby certify that the candidate has fulfilled the conditions of the Resolution and Regulations appropriate for the degree of PhD in the University of St Andrews and that the candidate is qualified to submit this thesis in application for that degree. I confirm that any appendices included in the thesis contain only material permitted by the `Assessment of Postgraduate Research Students' policy. \medbreak
\noindent {\bf Date:} 23/02/2023 \hspace{0.5cm} {\bf Signature of Supervisor 1:} 
 \smallbreak
\noindent {\bf Date:} 23/02/2023 \hspace{0.5cm} {\bf Signature of Supervisor 2:} 
\clearpage

\noindent {\large \textbf{Permission for publication}} \medbreak
\noindent In submitting this thesis to the University of St Andrews we understand that we are giving permission for it to be made available for use in accordance with the regulations of the University Library for the time being in force, subject to any copyright vested in the work not being affected thereby. We also understand, unless exempt by an award of an embargo as requested below, that the title and the abstract will be published, and that a copy of the work may be made and supplied to any bona fide library or research worker, that this thesis will be electronically accessible for personal or research use and that the library has the right to migrate this thesis into new electronic forms as required to ensure continued access to the thesis. \medbreak
\noindent I, Amlan Faiyaz Banaji, confirm that my thesis does not contain any third-party material that requires copyright clearance.\medbreak
\noindent The following is an agreed request by candidate and supervisor regarding the publication of this thesis:\smallbreak
\noindent \textbf{Printed copy}
No embargo on print copy.\\
\textbf{Electronic copy}
No embargo on electronic copy. \medbreak
\noindent {\bf Date:} 23/02/2023 \hspace{0.5cm} {\bf Signature of Candidate:}  
 \smallbreak
\noindent {\bf Date:} 23/02/2023 \hspace{0.5cm} {\bf Signature of Supervisor 1:} 
 \smallbreak
\noindent {\bf Date:} 23/02/2023 \hspace{0.5cm} {\bf Signature of Supervisor 2:} 

\clearpage 
\noindent
{\large \textbf{Underpinning research data or digital outputs}} \medbreak
\noindent \noindent {\large \textbf{Candidate's declaration}} \smallbreak
\noindent I, Amlan Faiyaz Banaji, hereby certify that no requirements to deposit original research data or digital outputs apply to this thesis and that, where appropriate, secondary data used have been referenced in the full text of my thesis. \medbreak
\noindent {\bf Date:} 23/02/2023 \hspace{0.5cm} {\bf Signature of Candidate:}  

\newgeometry{top=1.4cm,bottom=2.2cm,left=3cm,right=3cm}
\tableofcontents
\restoregeometry

\chapter{Introduction}\label{s:intro}

\pagenumbering{arabic}

\section{Fractal geometry}\label{s:fractalsintro}

\epigraph{Beautiful, damn hard, increasingly useful. That's fractals.}{Benoît Mandelbrot}

Fractals are geometric objects which display intricate structure at arbitrarily small scales. 
There is no precise definition of a fractal, but many exhibit some form of self-similarity, in the sense that the fractal is made up of several copies of itself which are scaled down and possibly distorted. For instance, the Sierpi\'nski carpet in Figure~\ref{f:intro} is comprised of eight scaled copies of itself. 
Fractal features are ubiquitous in nature; in one striking example, the buds of a Romanesco broccoli resemble scaled copies of the entire flower bud. 
\begin{figure}[ht]
\centering
\subfigure{\includegraphics[width=.3\textwidth]{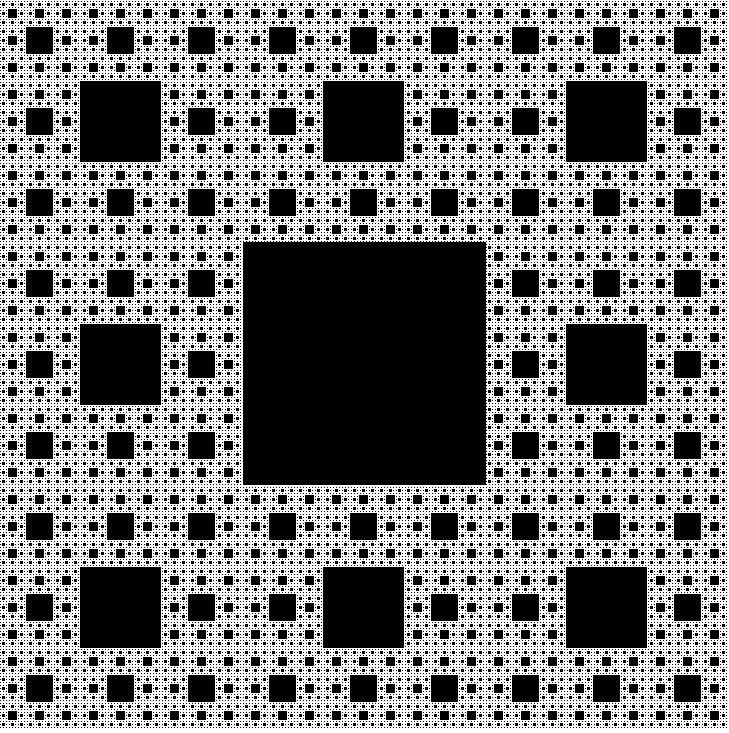}}
\qquad \qquad 
\subfigure{\includegraphics[width=.4\textwidth]{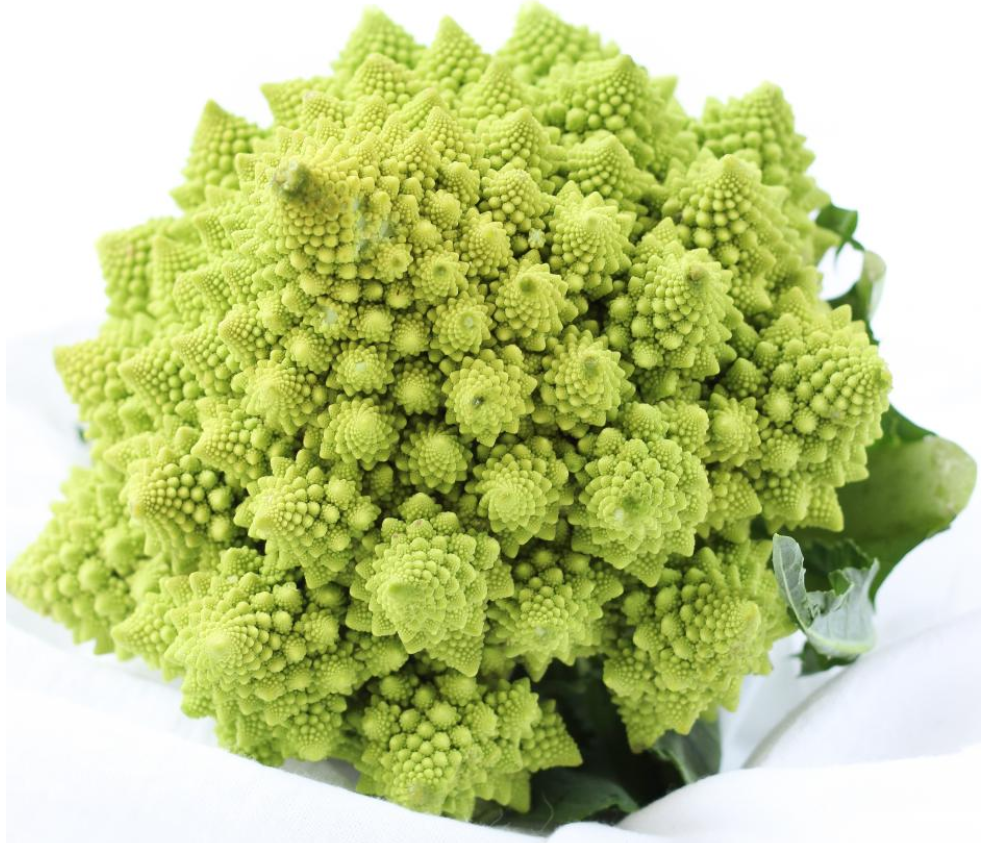}}
\caption{Left: the Sierpi\'nski carpet is a self-similar fractal. Right: Romanesco broccoli exhibits fractal-like features. Both pictures are \href{https://creativecommons.org/publicdomain/zero/1.0/}{CC0}, from references~\cite{sierpinski2023url},~\cite{romanesco2023url} respectively.}\label{f:intro}
\end{figure}
Fractal patterns have been used for centuries by many different cultures in creative work such as art and architecture. 
For example, clusters of houses in Benin city and its surrounding villages (in present-day Nigeria) are laid out in fractal patterns~\cite{Eglash1999african}. 
However, it was not until the 1970s that Mandelbrot~\cite{Mandelbrot1977first,
Mandelbrot1982nature} coined the term `fractal' (from the Latin word `fractus,' meaning `fractured' or `broken'), and widely popularised the concept. 
Today, fractal geometry is a flourishing branch of mathematics which puts fractals into a rigorous framework~\cite{Barnsley2012everywhere,
Barnsley2006superfractals}, and has particular relevance to chaotic dynamical systems~\cite{Peitgen2004chaos}.

An especially important notion in fractal geometry is that of dimension. 
The question of how best to define dimension for fractal sets is a challenging problem because, unlike for smooth manifolds, there is no obvious way to associate `tangent' vector spaces to points on fractals. 
Therefore, it is natural to define notions of dimension that make sense even for fractal sets using covers the set (or a part of it). 
Most notions of dimension satisfy standard properties such as $\dim M = d$ for a smooth $d$-manifold, and $\dim E \leq \dim F$ whenever $E \subseteq F$. %
For fractal sets, notions of dimension very often take non-integer values. For example, if $S$ is the Sierpi\'nski carpet from Figure~\ref{f:intro} and $\dim$ is Hausdorff or box dimension (described below), then $\dim S = \log 8 / \log 3 \approx 1.89$. This makes intuitive sense, since $S$ appears to fill up more space than a one-dimensional curve, but less than a two-dimensional filled square. 
For general background on fractal geometry and dimension theory, we refer the reader to Falconer's seminal books~\cite{Falconer2014main,Falconer1997techniques}.

The Hausdorff dimension is perhaps the most widely used notion of fractal dimension in mathematics; we give the precise definition in Section~\ref{s:notation}. 
One first defines $s$-dimensional Hausdorff measure $\mathcal{H}^s$ on $\Rd$ for each $s \geq 0$. If a set $F$ happens to satisfy $0 < \mathcal{H}^s(F) <\infty$ for some $s$ (for example if $F$ is a disc and $s = 2$), then for all $t \neq s$, the value of $\mathcal{H}^t(F)$ can be shown to be 0 or $\infty$. This suggests that $s$ is in some sense the correct value at which to measure $F$, and the Hausdorff dimension of $F$ is $s$. Hausdorff dimension can also be defined for sets whose Hausdorff measure is never positive and finite.

Box dimension is another familiar notion of fractal dimension. It has some properties that may be considered mathematically undesirable; for example countable compact sets can have positive box dimension (but always have 0 Hausdorff dimension). However, in many situations, box dimension is easier to calculate or estimate numerically than Hausdorff dimension. Box dimension of real-world fractals such as coastlines can be estimated over a range of scales. 
If the box dimension of a set $F$ exists and equals $s$ then this says that the number of balls of size $\delta$ needed to cover $F$ scales approximately like $\delta^{-s}$ as $\delta \to 0^+$; if box dimension does not exist then one can define upper and lower box dimension by considering the scales at which the set looks largest or smallest respectively. 
There are many longstanding open problems about Hausdorff and box dimension, such as the Kakeya conjecture~\cite{Katz2002kakeya}. 
The Hausdorff dimension of any set cannot exceed its lower (or upper) box dimension. 
For many nice sets, box dimension exists and coincides with Hausdorff dimension, but this is not always the case. 
If the dimensions do indeed coincide, this indicates that the set has a large amount of spatial regularity. For example, the dimension version of Falconer's distance set conjecture is known for this class of sets~\cite{ShmerkinPreprintdistancehigherdim}, but is wide open in general.

The main topic of this thesis are the intermediate dimensions, which are an example of dimension interpolation. 
This area has gathered significant interest since around~2018; for a survey of this topic we refer the reader to~\cite{Fraser2021interpolating}. The idea is to consider two different notions of dimension and find a geometrically meaningful family of dimensions which lie between them. This family should share some characteristics of both dimensions, but provide more information about sets than either does in isolation. The hope is that, as well as being interesting in its own right and leading to a rich theory, dimension interpolation can help illuminate why for some sets the two endpoint dimensions give different values. 
A different example of dimension interpolation is the Assouad spectrum, which lies between the upper box and Assouad dimensions, giving information about the `thickest' part of the set. 
We will encounter many applications of dimension interpolation in this thesis, for example to distinguish when sets are not bi-Lipschitz equivalent or to bound the H\"older distortion between two sets. 
As discussed below, the intermediate dimensions can give information about dimensions of images of sets under projections or stochastic processes, and the Assouad spectrum has been used in functional analysis and conformal geometry.

\section{Structure of thesis}\label{s:structure}

The introduction (Chapter~\ref{s:intro}) describes background material, mostly from~\cite{Falconer2020firstintermediate,Falconer2014main}; we provide references where appropriate. 

Chapter~\ref{s:generalised} introduces a family of dimensions, which we call the $\Phi$-intermediate dimensions, and is based on our paper~\cite{Banaji2023gen} which has been published in \emph{Monatshefte f\"ur Mathematik}. 
These dimensions also lie between Hausdorff and box dimension, and put the intermediate dimensions into a more general framework by restricting the sizes of allowable covers in ways that allow for greater refinement than in the definition of the intermediate dimensions. 
We show that for any compact subset of an appropriate space, these dimensions can be used to `recover the interpolation' between the Hausdorff and box dimensions of sets for which the intermediate dimensions are discontinuous at $\theta=0$, thus providing more refined geometric information about such sets. 
We also study many analytic and geometric properties of the $\Phi$-intermediate dimensions, and investigate their relationships with several other notions of dimension. Moreover, we prove Hölder distortion estimates which imply bi-Lipschitz stability for the $\Phi$-intermediate dimensions. We prove a mass distribution principle and Frostman type lemma, and use these to study dimensions of product sets, and to show that the lower versions of the dimensions, unlike the upper versions, are not finitely stable. 
Furthermore, we extend the theory from Euclidean space to a wider class of metric spaces, namely those that are uniformly perfect and doubling with more than one point.

Chapter~\ref{s:attainable} describes the general behaviour of the intermediate dimensions, and is based on the paper~\cite{Banaji2022moran} (joint with A.~Rutar, published in \emph{Annales Fennici Mathematici}). 
In Section~\ref{s:genbounds} (which also includes a little material from~\cite{Banaji2023gen}), we provide general bounds for the intermediate dimensions, some of which can be proved using the continuity bounds for the $\Phi$-intermediate dimensions in Chapter~\ref{s:generalised}. 
These results improve existing bounds in the literature, and in Section~\ref{s:lattice}, based on joint work with J.~M.~Fraser from~\cite{Banaji2021infinite}, we use them to calculate the intermediate dimensions of the inversion of the lattice $\{ 1^p,2^p,3^p,\dotsc \}^d$ in the unit $d$-sphere in $\Rd$. 
In Section~\ref{s:popcorn}, based on joint work with H. Chen from the paper~\cite{Banaji2022popcorn} published in the \emph{Journal of Fractal Geometry}, we use the bounds to calculate the intermediate dimensions of the graph of the well-known `popcorn function' and its pyramid-like higher-dimensional analogues. 
One bound that we prove is that any function that can be realised as the intermediate dimensions of a subset of Euclidean space must satisfy a straightforward constraint in terms of its Dini derivatives. In fact a converse result also holds: \emph{any} function that satisfies this constraint (together with some mild continuity and monotonicity assumptions) can be realised as the intermediate dimension function of some set, showing that a wide variety of behaviour is possible. We prove this using a Moran set construction that is homogeneous at each fixed scale, but has inhomogeneity between different scales. 
We also show that the lower and upper intermediate dimensions can be prescribed simultaneously, using a set which behaves like a union of homogeneous Moran sets at each fixed scale. 

Chapter~\ref{s:infinite} relates to limit sets of iterated function systems consisting of a countably \emph{infinite} number of contractions, and is based on a joint paper~\cite{Banaji2021infinite} with J.~M.~Fraser which has been published in \emph{Transactions of the American Mathematical Society}. 
Our main results are in the case when all the contractions are conformal. Under a natural separation condition we prove that the intermediate dimensions of the limit set are given by the maximum of the Hausdorff dimension of the limit set and the intermediate dimensions of the set of fixed points of the contractions. This builds on work of Mauldin and Urbański concerning the Hausdorff and upper box dimension. 
Our results apply to well-studied examples such as sets of numbers which have real or complex continued fraction expansions with restricted entries. 
We prove general upper bounds for the Hausdorff, box and intermediate dimensions of infinitely generated attractors in terms of a topological pressure function, without assuming conformality or separation conditions. We also make a few remarks from the preprint~\cite{Banaji2022assouad} (also joint with J.~M.~Fraser), in particular that our results can be applied to infinite parabolic IFSs. 
Moreover, we show that the limit set of a `generic' infinite IFS has box and intermediate dimensions equal to the ambient spatial dimension, where `generic' can refer to either full measure or comeagre. 

Chapter~\ref{s:bm} relates to self-affine Bedford--McMullen carpets, and is based on the joint paper~\cite{Banaji2021bedford} with I.~Kolossv\'ary which has been published in \emph{Advances in Mathematics}. 
In Theorem~\ref{thm:main}, which we consider to be one of the best results in this thesis, we calculate a precise formula for the intermediate dimensions of any Bedford--McMullen carpet for the whole range of $\theta \in (0,1)$, in terms of a certain rate function from large deviations theory. 
The intermediate dimensions exist and are strictly increasing in~$\theta$, and the function $\theta\mapsto \dim_{\theta}\Lambda$ exhibits interesting features not witnessed on any previous example, such as having countably many phase transitions, between which it is analytic and strictly concave. 
We make an unexpected connection to multifractal analysis by showing that two carpets have equal intermediate dimensions if and only if the Hausdorff multifractal spectra of the uniform Bernoulli measures supported on the two carpets are equal. Since intermediate dimensions are bi-Lipschitz invariant, this shows that the equality of these multifractal spectra is a necessary condition for two such carpets to be bi-Lipschitz equivalent.

\section{Notation and preliminaries}\label{s:notation}
 Throughout this thesis, we denote the natural logarithm by $\log$, and the cardinality of a set $S$ by $\# S$. 
We write $a \lesssim b$ to mean $a \leq c b$ for some constant $c$ which may depend on parameters in the subscript of $\lesssim$ but is independent of other parameters unless stated otherwise. For each $x > 0$, we write $ \lfloor x \rfloor \coloneqq \max \{ \, n \in \mathbb{N}^+ : n \leq x \, \}$. 
 All the sets $F$ we consider will be assumed to be non-empty and totally bounded subsets of an underlying metric space. We usually take the underlying space to be $\Rd$ with the Euclidean metric, but the theory in Chapter~\ref{s:generalised} works in more general spaces. 
We denote the (Euclidean) diameter of a subset of $\Rd$ by $|\cdot|$, and $d$-dimensional Lebesgue measure on $\Rd$ by $\mathcal{L}_d$. The symbol $\N$ will denote $\{1,2,3,\dotsc\}$, and $||\cdot||$ will denote either the Euclidean norm on $\Rd$ or the supremum norm of a continuous function, depending on context. 

We write 
\begin{equation}\label{e:defineball} 
B(x,\delta) \coloneqq \{ \, y \in \Rd : ||y-x|| < \delta \, \} 
\end{equation}
for the open ball of radius $\delta>0$ centred at $x \in \Rd$, and $B^F(x,r) \coloneqq B(x,r) \cap F$. 
We denote by $N_\delta(F)$ the smallest integer such that there exist $x_1,\dotsc,x_{N_\delta(F)} \in F$ such that 
 \begin{equation}\label{e:ndeltadefn}
 F \subseteq \bigcup_{i = 1}^{N_\delta(F)} B(x_i,\delta/2).
 \end{equation}
  For subsets of $\Rd$, by calculations similar to \cite[Equivalent definitions~2.1]{Falconer2014main}, there are several other definitions for $N_\delta(F)$ which would work equally well when calculating dimensions. %
For $U \subseteq \Rd$ and $\delta >0$ let 
\[ \mathcal{S}_\delta (U) \coloneqq \{ \, x \in \Rd : \mbox{ there exists } y \in U \mbox{ such that } ||x-y|| \leq \delta \, \}\] be the closed $\delta$-neighbourhood of $U$. 
For $d \in \N$ and $r \geq 1$, we denote by $A_{d,r} \in \N$ the smallest integer such that for all $U \subset \Rd$ there exist $U_1,\dotsc,U_{A_{d,r}} \subseteq \Rd$, each of diameter $|U|/r$, which \emph{cover} $U$, meaning that 
 \begin{equation}\label{doublingconst}
 U \subseteq \bigcup_{k=1}^{A_{d,r}} U_k.
 \end{equation}
 Given $x\in\Rd$, we denote the $j$\textsuperscript{th} coordinate of $x$ by $x^{(j)}$. 
We write $\overline{F}$ to denote the topological closure of $F$.

Hausdorff measure and dimension were first introduced in the early 20th century in~\cite{Caratheodory1914hausdorff,Hausdorff1919main}. 
Given $s \geq 0$ and a finite or countable set $\mathcal{U} = \{U_1,U_2,\dotsc \}$ of non-empty subsets of $\Rd$, we call the quantity $\sum_i |U_i|^s$ the \emph{$s$-cost} of $\mathcal{U}$. 
As explained in \cite[Chapter~3]{Falconer2014main}, the Hausdorff content of a subset $F$ of $\Rd$ can be defined, for $s \geq 0$ and $\delta > 0$, by 
  \[\mathcal{H}_\delta^s(F) = \inf{\left\{ \sum_{i=1}^\infty |U_i|^s \middle| F \subseteq \bigcup_{i=1}^{\infty} U_i, \mbox{diam}(U_i) \leq \delta \right\} }. \]
  As $\delta$ decreases, the class of covers is reduced so the infimum increases, and therefore converges to a limit 
  \[\mathcal{H}_\delta^s(F) \to \mathcal{H}^s(F) \in [0,\infty] \mbox{ as } \delta \to 0^+,\]
  called the \emph{$s$-dimensional Hausdorff measure} of $F$. 
   It can be shown that this is an outer measure on $\Rd$, and so its restriction to the $\mathcal{H}^s$-measurable sets is a measure, as is its further restriction to the Borel sets. 
 It is straightforward to see that for each $F$ there is a unique $s \geq 0$, called the \emph{Hausdorff dimension} of $F$, such that if $0 \leq t < s$ then $\mathcal{H}^t(F) = \infty$ and if $t > s$ then $\mathcal{H}^t(F) = 0$, as illustrated in Figure~\ref{f:HausdorffMeas}. 
 The $s$-dimensional Hausdorff measure of $F$ may be any value in $[0,\infty]$; if it is positive and finite then $F$ is called an $s$\emph{-set}. 
 Sets with Hausdorff dimension less than 1 are necessarily totally disconnected, see \cite[Proposition~3.5]{Falconer2014main}. 
 Hausdorff measures and dimension have been studied in detail in~\cite{Mattila1995book,Federer1996hausdorff,
 Rogers1998hausdorff,Falconer1985otherbook}. 
  \begin{figure}[ht]
\center{\includegraphics[width=.5\textwidth]
        {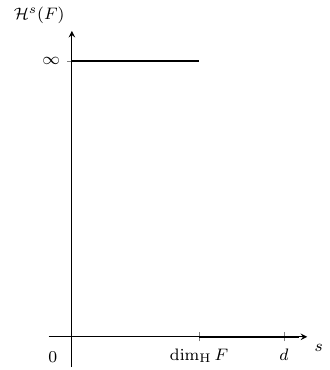}}
        \caption{\label{f:HausdorffMeas}
        Graph of the $s$-dimensional Hausdorff measure of a subset of $\Rd$ against $s$.}
\end{figure}

The upper and lower box dimensions, also called box-counting, Minkowski--Bouligand or Minkowski dimensions, originated in~\cite{Bouligand1928box,Pontrjagin1932box} and are respectively defined by 
\begin{equation}\label{e:boxdimdef}
 \ubd F \coloneqq \limsup_{\delta \to 0^+} \frac{\log N_\delta(F)}{-\log \delta}; \qquad \lbd F \coloneqq \liminf_{\delta \to 0^+} \frac{\log N_\delta(F)}{-\log \delta}.
 \end{equation}
 If the two coincide, it is called simply the \emph{box dimension}, denoted $\dim_{\mathrm B} F$. 
 Box dimension can also be described in terms of the Lebesgue measure of the $\delta$-neighbourhood of the set: 
 \[ \ubd F = d - \liminf_{\delta \to 0^+} \frac{\log \mathcal{L}_d(\mathcal{S}_\delta (F))}{\log \delta}    ; \qquad \lbd = d - \limsup_{\delta \to 0^+} \frac{\log \mathcal{L}_d(\mathcal{S}_\delta (F))}{\log \delta}    \]
 for $F \subset \Rd$, see \cite[Proposition~2.4]{Falconer2014main}. 
The packing dimension, introduced in~\cite{Tricot1982packing} and studied in~\cite{Mattila1995book}, can be defined using packing measure, or equivalently as the modified upper box dimension: 
\begin{equation}\label{e:definepackingdim}
 \dim_\mathrm{P} F \coloneqq \inf\left\{\, \sup_{i \in \N} \ubd F_i : F \subseteq \bigcup_{i=1}^\infty F_i, \mbox{ each } F_i \mbox{ non-empty and bounded}\, \right\}. 
 \end{equation}

 Since the important work of Assouad~\cite{Assouad1977thesis,Assouad1979dimension,
 Assouad1983dimension} and Larman~\cite{Larman1967lower}, other notions of dimension which describe the local scaling behaviour of sets have been studied. The \emph{Assouad} and \emph{lower} dimensions, studied in detail in~\cite{Fraser2020book}, give information about the `thickest' and `thinnest' part of a set respectively. 
The Assouad dimension of a subset $F$ of a metric space with more than one point is defined by  
\begin{align}\label{e:assouaddef}
\begin{split}
      \dim_\mathrm{A} F = \inf\{ \, \alpha : \mbox{there exists } &C>0\mbox{ such that } N_r(B(x,R)\cap F) \leq C(R/r)^{\alpha}  \\*
      &\mbox{ for all } x \in F \mbox{ and } 0<r<R \, \}. 
      \end{split}
    \end{align}
    Dually, the \emph{lower dimension} of $F$ is defined by
\begin{align*}
      \dim_\mathrm{L} F = \sup\{ \,
      \lambda : \mbox{there exists } &C>0\mbox{ such that } N_r(B(x,R)\cap F) \geq C(R/r)^\lambda \\*
      &\mbox{ for all } x \in F \mbox{ and } 0<r<R\leq |F| \, \}. 
    \end{align*}
      For $\theta \in (0,1)$, the \emph{Assouad spectrum} of $F$ at $\theta$ is defined by fixing the scales $R = r^\theta$ in the definition of Assouad dimension: 
   \begin{multline*}
      \asp F = \inf\left\{
      s : \mbox{ there exists }C>0\mbox{ such that for all } x \in F \mbox{ and } 
    \right. \\*
     \qquad \left. 0< R \leq 1, \mbox{ we have } N_{R^{1/\theta}}(B(x,R)\cap F) \leq C R^{s(1 - 1/\theta)}\right\}. 
    \end{multline*}%
    Clearly $\ubd F \leq \asp F \leq \asd F$ for all $\theta \in (0,1)$. 
    The \emph{lower spectrum} is defined by 
    \begin{multline*}
      \dim_{\mathrm L}^{\theta} F = \sup\left\{
      s : \mbox{ there exists }C>0\mbox{ such that for all } x \in F \mbox{ and } 
    \right. \\*
     \qquad \left. 0< R \leq 1, \mbox{ we have } N_{R^{1/\theta}}(B(x,R)\cap F) \geq C R^{s(1 - 1/\theta)}\right\}. 
    \end{multline*}
    The Assouad spectrum is not always monotonic in $\theta$ (see~\cite{Fraser2018firstassspec}). The \emph{upper Assouad spectrum} at $\theta$, however, \emph{is} monotonic, and is defined by 
      \begin{multline*}
      \uasp F = \inf\left\{
      s : \mbox{ there exists }C>0\mbox{ such that for all } x \in F \mbox{ and } 
    \right. \\*
     \qquad \left. 0<r\leq R^{1/\theta} \leq R \leq 1, \mbox{ we have } N_r(B(x,R)\cap F) \leq C(R/r)^s\right\}. 
    \end{multline*}
    
    The Assouad spectrum was introduced in~\cite{Fraser2018firstassspec} and has been calculated for various families of fractals in~\cite{Banaji2022assouad,
    Fraser2018secondassouad,
    FraserPreprintsullivan,Burrell2022spirals} and other works. 
    Rutar~\cite{Rutar2022assouad} has given a complete description of the attainable forms of Assouad spectra of sets, showing that a wide variety of behaviour is possible in general. 
    The Assouad spectrum can be used to give information about dimensions of orthogonal projections of sets~\cite{Falconer2021projections}, and has applications related to spherical maximal functions~\cite{Anderson2021assouadspec,
    Roos2023assouadspectrum} and conformal geometry~\cite{GaritsisPreprintconformal}. 
The \emph{quasi-Assouad dimension}, introduced in~\cite{Lu2016qa}, can be defined by  
\begin{equation}
 \qad F \coloneqq \lim_{\theta \to 1^-} \asp F,
 \end{equation}
or equivalently $\qad F \coloneqq \lim_{\theta \to 1^-} \uasp F$ (see~\cite[Corollary~3.3.7]{Fraser2020book}). 
 We always have 
 \[ \dim_{\mathrm H} F \leq \ubd F \leq \asp F \leq \uasp F \leq \qad F \leq \asd F,\] 
  and all inequalities can be strict. 
  We sometimes write $\dim_{\mathrm A}^1$ or $\overline{\dim}_{\mathrm A}^1$ to mean the quasi-Assouad dimension, and since $\uasp F \to \ubd F$ as $\theta \to 0^+$, we sometimes write $\dim_{\mathrm A}^{0} F$ or $\overline{\dim}_{\mathrm A}^0 F$ to mean the upper box dimension of $F$. 
  The Assouad spectrum and upper Assouad spectrum are continuous in $\theta \in (0,1)$, see~\cite{Fraser2018firstassspec,Fraser2019twospectra}. In~\cite{Fraser2019twospectra}, Fraser et al. show that we always have $\overline{\dim}_\mathrm{A}^\theta F = \sup_{\theta' \in (0,\theta]} \dim_\mathrm{A}^{\theta'} F$.

    There are also various different notions of fractal dimension of a measure. The \emph{Assouad dimension} of a Borel probability measure $\mu$ is 
    \begin{align*}
 \dim_\mathrm{A} \mu \coloneqq \inf \left\{ \, s \geq 0 : \right. &\left.\mbox{there exists } A > 0 \mbox{ such that if } 0 < r < R \leq |\mbox{supp}(\mu)| \right. \\*
&\left. \mbox{  and } x \in \mbox{supp}(\mu) \mbox{ then } \frac{\mu(B(x,R))}{\mu(B(x,r))} \leq A \left(\frac{R}{r}\right)^s \right\}. 
\end{align*}
    To obtain bounds involving the lower dimension in Chapters~\ref{s:generalised} and~\ref{s:attainable}, we will use the dual notion of \emph{lower dimension} of a measure: 
\begin{align}\label{e:lowerdimmeas}
\begin{split}
 \dim_\mathrm{L} \mu \coloneqq \sup \left\{ \, \lambda \geq 0 : \right. &\left.\mbox{there exists } A > 0 \mbox{ such that if } 0 < r < R \leq |\mbox{supp}(\mu)| \right. \\*
&\left. \mbox{  and } x \in \mbox{supp}(\mu) \mbox{ then } \frac{\mu(B(x,R))}{\mu(B(x,r))} \geq A \left(\frac{R}{r}\right)^\lambda \right\}. 
\end{split}
\end{align}
A measure $\mu$ is said to be \emph{doubling} if there exists $M \geq 1$, called the \emph{doubling constant}, such that $\mu(B(x,2r)) \leq M \mu(B(x,r))$ for all $x \in \mbox{supp}(\mu)$ and $r > 0$. 
For further details we refer the reader to~\cite[Section~4.1]{Fraser2020book}.

  Given a closed set $D \subset \Rd$, a \emph{contraction} on $D$ is a map $S \colon D \to D$ for which there exists $r<1$ with $||S(x) - S(y)|| \leq r||x-y||$ for all $x,y \in D$. 
 An \emph{iterated function system (IFS)} on $D$ is a finite set $\Phi \coloneqq \{F_1,\dotsc,F_m\}$ of contractions on $D$, with $m \geq 2$. 
 Hutchinson~\cite{Hutchinson1981attractor} showed that given such an IFS, there is a unique non-empty compact set $K \subseteq D$ called the \emph{attractor} or \emph{limit set} of the IFS, such that 
 \[K = \bigcup_{i=1}^m F_i(K).\] 
 This can be proved either directly or using Banach's contraction mapping theorem. 
  The attractor is very often fractal in nature; examples include the Sierpi\'nski carpet in Figure~\ref{f:intro} (page~\pageref{f:intro}) and the fractal in~\ref{f:carpetIFS} (page~\pageref{f:carpetIFS}). 
  Define $F(E) \coloneqq \cup_{i=1}^m F_i(E)$ for non-empty, compact sets $E$, and let $F^0(E) = E$ and $F^k(E) = F(F^{k-1}(E))$ for each $k \in \N$. 
  Then if $E$ is any non-empty compact subset of $D$ such that $F_i(E) \subseteq E$ for all $i$, then \[K = \bigcap_{k=0}^\infty F^k(E). \]
  Furthermore, $K$ is the closure of the set of fixed points of finite compositions $F_{i_1} \circ \dotsb \circ F_{i_p}$ of the $F_i$. 
  
  Given positive numbers $p_1,\dotsc,p_m$ summing to 1, one can define a measure by repeated subdivision according to the probabilities $p_i$. The resulting probability measure $\mu$ will have $\supp(\mu) = K$ and satisfy 
  \[ 
  \mu(A) = \sum_{i=1}^m p_i \mu(F_i^{-1}(A))
  \]
  for all Borel sets $A$. 
If all of the contractions are similarity maps, which means that there exists $r_i$ depending only on the map for which 
$||F_i(x)-F_i(y)|| = r_i||x-y||$ for all $x,y \in D$, then the attractor is said to be a \emph{self-similar set} and $\mu$ a \emph{self-similar measure}. 
In~\cite{Varju2021selfsimsurvey}, Varj\'u has surveyed results relating to the dimension theory of self-similar sets and measures in $\mathbb{R}$.

To obtain dimension results, one often assumes the \emph{open set condition (OSC)}, which means that there is a non-empty, bounded, open set $V \subset \Rd$ such that 
\[ V \supseteq \bigcup_{i=1}^m F_i(V)\]
 with the union disjoint. Intuitively, this says that the components $F_i(K)$ do not overlap too much. 
If this is satisfied, then the Hausdorff, box, Assouad and lower dimensions of self-similar sets are all equal to the \emph{similarity dimension} $\dim_{\mathrm{sim}} \Phi$, which is the unique non-negative number satisfying \[\sum_{i=1}^m r_i^{\dim_{\mathrm{sim}} \Phi}=1, \] 
see~\cite{Hutchinson1981attractor,Falconer2014main,Fraser2020book}. 
This is known as the \emph{Hutchinson--Moran formula.} 
  Moreover, the attractor has positive and finite Hausdorff measure in its Hausdorff dimension, and this measure is \emph{Ahlfors regular}: there exists $C \geq 1$ such that $C^{-1} R^{\dim_{\mathrm{sim}} \Phi} \leq \mu(B_R) \leq C R^{\dim_{\mathrm{sim}} \Phi}$ for all closed balls $B_R$ of radius $0 < R < \mbox{diam}(\mbox{supp}(\mu))$. 
  Even if the OSC is not satisfied, the Hausdorff and box dimension of all self-similar sets still coincide. The famous \emph{exact overlaps conjecture} asks whether the only way the Hausdorff dimension of a self-similar set in the real line can differ from the general upper bound $\min\{\dim_{\mathrm{sim}} \Phi ,1\}$ is if there are exact overlaps, in other words if different finite compositions of the defining contractions can result in the same function. 
  Hochman~\cite{Hochman2014selfsim} has made important progress in this direction, showing that any potential counter-example to the conjecture would have to be given by an IFS with very rapidly accumulating cylinders (but without exact overlaps). The first examples of IFSs with this `super-exponential concentration' property were given in~\cite{Baker2021superexp,
  Barany2021superexp}, and there has been further work on the dimension theory of such IFSs~\cite{Rapaport2022exactalgebraic,
  Rapaport2024homogthreemaps}. 
  There are also longstanding open problems about dimensions and absolute continuity of overlapping self-similar measures such as Bernoulli convolutions~\cite{Breuillard2019bernoulli,
  Varju2019transcendental,Varju2019abscts,
  Varju2016bernoullisurvey,Erdos1939bernoulli,
  Erdos1940bernoulli}.

 We often require the metric spaces we work with to satisfy certain properties, especially in Chapter~\ref{s:generalised}. 
\begin{defn}\label{d:unifperf}%
For $c \in (0,1)$ we say a metric space $X$ is \emph{$c$-uniformly perfect} if for all $x \in X$ and $R \in \mathbb{R}$ such that $0 < R < |X|$ we have %
\[ B(x,R) \setminus B(x,cR) \neq \varnothing. \]
The space $X$ is \emph{uniformly perfect} if there exists $c \in (0,1)$ such that $X$ is $c$-uniformly perfect. 
\end{defn}
Intuitively, a metric space is uniformly perfect if it does not have islands which are very separated from the rest of the space. 
\begin{defn}
A metric space is said to be \emph{doubling} if there exists a constant $M \in \mathbb{N}$ (called the \emph{doubling constant}) such that for every $x \in X$ and $r>0$, there exist $x_1,\dotsc,x_M \in X$ such that $B(x,2r) \subseteq \bigcup_{i=1}^M B(x_i,r)$. 
\end{defn} 
    In~\cite[Section~13.1.1]{Fraser2020book} it is shown that a metric space $X$ with more than one point is uniformly perfect if and only if $0 < \dim_\mathrm{L} X$. Such a space cannot have any isolated points, so must be infinite. 
    It is also shown that a space $X$ is doubling if and only if $\dim_\mathrm{A} X < \infty$. %
    In this case we will see in Proposition~\ref{basicbounds} that all dimensions of all subsets $F$ will be finite, as we will need to assume for many of the results in this chapter. 
    A metric space is said to be \emph{Ahlfors regular} if there exists $s>0$, $C \geq 1$ and a Borel regular measure~$\mu$ supported on~$X$ such that $C^{-1} R^s \leq \mu(B_R) \leq C R^s$ for all closed balls $B_R$ of radius $0 < R < \mbox{diam}(X)$. 
    By \cite[Corollary~14.15]{Heinonen2001metric}, every Ahlfors regular space with more than one point is uniformly perfect and doubling. 
    An example of such a space which is not bi-Lipschitz equivalent to any subset of $\Rd$ is the Heisenberg group with its usual Carnot-Carath{\'e}odory metric, see~\cite{LeDonne2015,Pansu1989,Semmes1996}. 
    In Chapters~\ref{s:generalised} and~\ref{s:attainable}, we will use the fact that if $F$ is a complete, uniformly perfect, totally bounded, doubling metric space with more than one point, then 
    \begin{align}\label{e:lowerdimexistmeas}
    \begin{split}
     \dimA F &= \inf \{ \, \dimA \mu : \mu \in \mathcal{P}_F \, \}, \\*
     \dimL F &= \sup \{ \, \dimL \mu : \mu \in \mathcal{P}_F \, \},
     \end{split}
     \end{align}
     where $\mathcal{P}_F$ is the set of doubling Borel regular finite outer measures $\mu$ with $\supp \mu = F$. 
    For more on these results, we refer the reader to \cite{Bylund2000lower,Konyagin1988assouad,
    Luukkainen1998assouad,Kaenmaki2013assouad}, \cite[Theorem~3.2]{Kaenmaki2017}, and \cite[Section~4.1]{Fraser2020book}. 

\section{Intermediate dimensions}\label{s:intdimsintro}

\subsection{Definitions and general theory}

Since the Hausdorff dimension of a set $F$ is the infimum of values $s$ for which $\mathcal{H}^s(F) = 0$, an equivalent definition of Hausdorff dimension is 
\begin{equation}\label{hausdorffdef}
\begin{aligned} \dim_\mathrm{H} F = \inf \{ \, s \geq 0 : &\mbox{ for all } \epsilon >0 \mbox{ there exists a finite or countable cover } \\*
& \phantom{-}\{U_1,U_2,\dotsc\} \mbox{ of } F \mbox{ such that } \sum_i |U_i|^s \leq \epsilon \,\}. 
\end{aligned}
\end{equation}%
It is clear from~\eqref{e:boxdimdef} that lower box dimension can equivalently be defined as 
\begin{equation}
\begin{aligned}
\underline{\mbox{dim}}_{\mathrm B} F = \inf \{ \, s \geq 0 : &\mbox{ for all } \epsilon >0 \mbox{ there exists a finite cover } \{U_1,U_2,\dotsc \} \\*
&\mbox{ of } F \mbox{ such that } |U_i| = |U_j| \mbox{ for all } i,j, \mbox{ and } \sum_i |U_i|^s \leq \epsilon \, \}.
\end{aligned}
\end{equation}

These definitions look rather similar, and we see that in the definition of Hausdorff dimension there is no restriction on the size of the covering sets, whereas for the box dimension all the sets in the cover have the same size, as illustrated in~\cite[Figures~1.2 and~1.3]{Fraser2020book}. 
Note also that it is immediate from the definitions that $\dim_{\mathrm H} F \leq \underline{\dim}_{\mathrm B} F \leq \overline{\dim}_{\mathrm B} F$ always holds. 
This motivates the definition of the intermediate dimensions, which lie between the Hausdorff and box dimensions and require the sizes of the covering sets to be restricted in a way that depends on a parameter $\theta$. These dimensions form the basis around which this thesis is built. 
    Given $\theta\in[0,1]$, we say that a family of sets $\{U_i\}_{i}$ is a $(\delta,\theta)$-cover of $F$ if 
    \begin{equation}
     F\subseteq\bigcup_{i} U_i \quad \mbox{and} \quad \forall i: \, \delta^{1/\theta}\leq |U_i|\leq\delta 
     \end{equation}
     where for convenience we take $\delta^{1/0}=0$. 
\begin{defn}[Falconer--Fraser--Kempton~\cite{Falconer2020firstintermediate}]\label{d:intdimdef}
For $\theta \in [0,1]$, the \emph{upper $\theta$-intermediate dimension} of a non-empty, bounded subset $F \subset \Rd$ is given by 
\begin{align*} \uid F = \inf \{ \, s \geq 0 : &\mbox{ for all } \epsilon >0 \mbox{ there exists } \delta_0 \in (0,1] \mbox{ such that for all } \delta \in (0,\delta_0) \\*
&\mbox{ there exists a } (\delta,\theta)\mbox{-cover } \{U_i\}_{i}\mbox{ of } F \mbox{ such that }\sum_i |U_i|^s \leq \epsilon \, \}.
\end{align*}
Similarly the \emph{lower $\theta$-intermediate dimension} of $F$ is 
\begin{align*} \lid F = \inf \{ \, s \geq 0 : &\mbox{ for all } \epsilon >0 \mbox{ and } \delta_0 \in (0,1] \mbox{ there exists } \delta \in (0,\delta_0) \\* &\mbox{ and a }(\delta,\theta)\mbox{-cover } \{U_i\}_{i}\mbox{ of } F \mbox{ such that }\sum_i |U_i|^s \leq \epsilon \,\}.
\end{align*}
If these coincide, then we refer to the \emph{intermediate dimension} of $F$, denoted $\dim_{\theta} F$. Note that $\overline{\dim}_{1} = \ubd$ and $\underline{\dim}_{1} = \lbd$ and $\overline{\dim}_{0} = \underline{\dim}_{0} \coloneqq \dim_\mathrm{H}$. 
\end{defn}
For all non-empty, bounded $F \subset \Rd$, these satisfy the inequalities 
\begin{align}\label{e:dimrelations}
\begin{split}
 0 \leq \dim_\mathrm{H} F \leq &\lid F \leq \uid F \leq \ubd F \leq \asd F \leq d, \\*
  &\lid F \leq \lbd F \leq \ubd F. 
 \end{split}
 \end{align}
 The intermediate dimensions have been studied in \cite{Burrell2021thesis,
Tan2020,Daw2023,Falconer2021intdimsurvey,
Fraser2021interpolating,
Burrell2021projections,Burrell2022brownian,
Kolossvary2022bm,Tan2020,Burrell2022spirals,
Falconer2020firstintermediate,Falconer2021seq,
Banaji2023gen,Banaji2022moran,
Banaji2021infinite,Banaji2021bedford,
Banaji2022popcorn,Daw2022phd,Douzi2022hewitt,
Feng2024intermediate}. 
A specific variant was used in \cite{Kukavica2012} (before the paper~\cite{Falconer2020firstintermediate}) to study the singular sets of certain partial differential equations. 
The intermediate dimensions satisfy standard properties that most other dimensions satisfy; for example if $E \subseteq F$ then $\dim_{\theta} E \leq \dim_{\theta} F$. 
In several ways, the intermediate dimensions behave more like box than Hausdorff dimension. 
For example, it is straightforward to see that box and intermediate dimensions are unchanged under taking closure of the set, but Hausdorff dimension is not. 
A dimension $\dim$ is said to be \emph{countably stable} if for all countable sequences of sets $F_1,F_2,\dotsc$, it holds that 
\[ \dim\left( \bigcup_{n=1}^{\infty} F_n \right) = \sup\{ \, \dim F_n : n \in \N \, \}. \]
It is shown in \cite[Proposition~3.1]{Falconer2020firstintermediate} that for $p>0$ and $0 \leq \theta \leq 1$, 
\[ \dim_{\theta} (\{0\} \cup \{ \, n^{-p} : n \in \N \, \}) = \frac{\theta}{p+\theta}, \]
so although Hausdorff dimension is countably stable, box and intermediate dimensions are not.

Some examples of the possible forms of intermediate dimension functions are given in~\cite[Section~3.2]{Falconer2020firstintermediate}, and a full characterisation is obtained in Chapter~\ref{s:attainable}. 
The maps $\theta \mapsto \uid F$ and $\theta \mapsto \lid F$ are trivially increasing in $\theta \in [0,1]$. They were shown in~\cite[Section~2.1]{Falconer2020firstintermediate} to be continuous in $\theta \in (0,1]$. 
For many classes of sets for which the intermediate dimensions have been calculated, they are also continuous at $\theta = 0$, so fully interpolate between the Hausdorff and box dimensions. Such classes include elliptical polynomial spirals~\cite{Burrell2022spirals}, concentric spheres and attenuated topologist's sine curves~\cite{Tan2020}, polynomial sequences and lattice sets (see \cite[Proposition~3.1]{Falconer2020firstintermediate} and Section~\ref{s:lattice}), popcorn-like pyramid sets (see Section~\ref{s:popcorn}), and Bedford--McMullen carpets (see 
\cite[Section~4]{Falconer2020firstintermediate} and Section~\ref{s:bm}). 
In Section~\ref{s:holderintro}, we will see that continuity of the intermediate dimensions at $\theta = 0$ has powerful consequences. 

On the other hand, there are a plethora of compact subsets of $\R$, such as $\{0\} \cup \left\{\,\frac{1}{\log k} : k \in \mathbb{N},k \geq 3 \,\right\}$ (see \cite[Section~3.2]{Falconer2020firstintermediate}), for which the intermediate dimensions are constant at the value of the box dimension and discontinuous at $\theta=0$, thus providing very little information about the set. 
Note that every compact subset of $\mathbb{R}$ can be obtained by starting with a closed interval and removing a sequence of disjoint open intervals from it. 
Now fix any non-increasing, summable sequence of positive numbers $(a_k)_{k=1}^{\infty}$ such that $-\log a_k / \log k \to 1$ as $k \to \infty$ (for example $a_k \coloneqq k^{-1}(\log (2k))^{-2}$). %
By \cite[Theorem~1]{Besicovitch1954cutout}, for all $s \in [0,1]$, one can start with a closed interval of length $\sum_{k=1}^\infty a_k$ and recursively cut out open intervals of length $a_k$ in such a way that the resulting compact set $F$ satisfies $\dim_{\mathrm H} F = s$. 
But by \cite[Section~3.2]{Falconer1997techniques}, $\dim_{\mathrm B} F = 1$ (independent of precisely which intervals are removed). %
It was shown in \cite[Proposition~2.4]{Falconer2020firstintermediate} (see also Chapter~\ref{s:attainable}) that this implies that $\dim_{\theta} F = 1$ for all $\theta \in (0,1]$. Therefore if $\dim_{\mathrm H} F < 1$ then $\theta \mapsto \dim_{\theta} F$ is discontinuous at $\theta=0$.

Following~\cite{Burrell2021projections}, for a bounded and non-empty set $F\subset \Rd$, $\theta\in(0,1)$ and $s\in[0,d]$, we introduce
\begin{equation}\label{eq:41}
S_{\delta, \theta}^{s}(F)\coloneqq \inf \Big\{ \, \sum_{i}\left|U_{i}\right|^{s}:\left\{U_{i}\right\}_{i} \text { is a cover of } F \text { such that } \delta^{1/\theta} \leq\left|U_{i}\right| \leq \delta \text { for all } i \, \Big\}.
\end{equation}
The motivation for introducing $S_{\delta, \theta}^{s}(F)$ is that from~\cite[Lemma~2.1]{Burrell2021projections} and the definitions of $\underline{\dim}_{\theta}F$ and $\overline{\dim}_{\theta}F$ it follows that
\begin{equation*}
\underline{\dim}_{\theta} F=\text { the unique } s \in[0, d] \text { such that } \liminf_{\delta\searrow 0} \frac{\log S_{\delta, \theta}^{s}(F)}{-\log \delta}=0
\end{equation*}
and
\begin{equation}\label{eq:45}
\overline{\dim}_{\theta} F=\text { the unique } s \in[0, d] \text { such that } \limsup_{\delta\searrow 0} \frac{\log S_{\delta, \theta}^{s}(F)}{-\log \delta}=0.
\end{equation}
For each $\theta \in (0,1)$, $\liminf_{\delta\searrow 0} \frac{\log S_{\delta, \theta}^{s}(F)}{-\log \delta}$ and $\limsup_{\delta\searrow 0} \frac{\log S_{\delta, \theta}^{s}(F)}{-\log \delta}$ are strictly decreasing and continuous functions of $s$. 

Some methods for estimating Hausdorff and box dimension are given in~\cite[Chapter~4]{Falconer2014main}. In particular, to obtain an upper bound, one often uses covers which arise naturally in the construction of the fractal, and to obtain lower bounds one often puts an appropriate measure on the set and applies a mass distribution principle. 
Similar strategies often work for the intermediate dimensions as well. 
In particular, we will use the following version of the mass distribution principle for the intermediate dimensions of Falconer, Fraser and Kempton,~\cite[Proposition~2.2]{Falconer2020firstintermediate}. 

\begin{prop}\label{prop:mdp}
Let $F$ be a non-empty, bounded subset of $\Rd$, and let $\theta \in [0,1]$, $s \geq 0$, $\delta_0 \in (0,1)$. Suppose that for all $\delta \in (0,\delta_0)$ there exists a Borel measure $\mu_\delta$ with support $\supp(\mu_\delta) \subseteq F$ such that $\mu_\delta(U) \leq |U|^s$ for all Borel sets $U \subset \Rd$ with $\delta^{1/\theta} \leq |U| \leq \delta$. Then 
\[\liminf_{\delta\searrow 0} \frac{\log S_{\delta, \theta}^{s}(F)}{-\log \delta} \geq \liminf_{\delta\searrow 0} \frac{\log \mu_\delta(\supp(\mu_\delta))}{-\log \delta}.\]
The same holds if we replace $\liminf$ with $\limsup$.  
\end{prop}

\begin{proof}
If $\{U_i\}$ is a cover of $F$ with $\delta^{1/\theta} \leq |U_i| \leq \delta$ for all $i$, then $\mathrm{supp}(\mu_\delta) \subseteq F \subseteq \cup_i U_i$. Therefore 
\[ \mu_\delta (\mathrm{supp}(\mu_\delta)) \leq \sum_i \mu_\delta(U_i) \leq \sum_i |U_i|^s.\] 
Since the cover was arbitrary, also $\mu_\delta (\mathrm{supp}(\mu_\delta)) \leq S_{\delta,\theta}^s (F)$. 
\end{proof}
The intermediate dimensions also satisfy an appropriate analogue of Frostman's lemma (see \cite[Section~2.3]{Falconer2020firstintermediate} and Section~\ref{s:frostman} of this thesis). 
Moreover, bounds for dimensions of products have been obtained in \cite[Section~2.5]{Falconer2020firstintermediate} and Section~\ref{s:products}. 

There are several natural questions about the intermediate dimensions which no-one has yet investigated, and which we will not pursue in this thesis. We give three possible lines of enquiry here; others relevant to the different chapters, and some specific open questions, are given later in the thesis. 
\begin{itemize}
\item In this thesis we write $\overline{\dim}_{0} = \underline{\dim}_{0} = \dim_\mathrm{H}$ to keep notation consistent with the literature on intermediate dimensions (where `continuity at $\theta=0$' is frequently discussed) and with Definition~\ref{d:intdimdef} under the convention $\delta^{1/0} = 0$. 
However, it could be argued that it is mathematically more natural to define these quantities as $\overline{\dim}_{0} F = \lim_{\theta \to 0^+} \uid F$ and $\underline{\dim}_{0} F = \lim_{\theta \to 0^+} \lid F$. 
One could study these limits in their own right, for example by asking what geometric information they provide about $F$, characterising when they coincide with Hausdorff dimension, and calculating them for some families of sets where they do not. 
\item 
Douzi and Selmi~\cite{Douzi2022hewitt} have introduced a family of dimensions which they call the \emph{modified intermediate dimensions} by making the intermediate dimensions countably stable, similarly to~\eqref{e:definepackingdim}. They are larger than the Hausdorff dimension but smaller than the packing/Hewitt--Stromberg/modified box dimensions. 
It could be of interest to calculate the modified intermediate dimensions of families of sets such as those considered in Chapters~\ref{s:infinite} or~\ref{s:bm}. However, we suspect that for many dynamically defined fractals they will coincide with the intermediate dimensions, and for others, such as level sets of local dimensions of self-affine measures on Bedford--McMullen carpets, this is likely to be a hard problem because even packing dimension is not fully understood. 
\item
Many notions of fractal dimensions of sets (such as Hausdorff, packing, box, Assouad) have analogous notions for measures, and it would be natural to try to define an appropriate notion of intermediate dimensions of a measure. 
\end{itemize}

\subsection{Dimensions of images of sets}\label{s:holderintro}

Different notions of fractal dimension, including the intermediate dimensions, can be used to give information about the possible H{\"o}lder exponents of maps between different sets. For a more in-depth discussion of the H{\"o}lder mapping problem in the context of dimension theory we refer the reader to~\cite[Section~17.10]{Fraser2020book}. 
Let $(X,d_X)$ and $(Y,d_Y)$ be metric spaces. 
 We say that a map $f \colon X \to Y$ is \emph{H{\"o}lder}, \emph{$\alpha$-H{\"o}lder} or \emph{$C,\alpha$-H{\"o}lder} if 
\[ d_Y(f(x_1),f(x_2)) \leq C d_X(x_1,x_2)^\alpha \qquad \mbox{for all }x_1,x_2 \in X\] 
for constants $\alpha \in (0,1]$ and $C \in [0,\infty)$, and we call $\alpha$ the \emph{H{\"o}lder exponent}. 
It is a straightforward exercise to show that if $f \colon F \to \Rd$ is $\alpha$-H\"older and $\dim$ is any one of Hausdorff, upper box, lower box, or (for fixed $\theta \in [0,1]$) upper $\theta$-intermediate or lower $\theta$-intermediate dimensions, then 
\begin{align}\label{generalholderint} 
\dim f(F) \leq \alpha^{-1}\dim F.
\end{align}
 For further H{\"o}lder distortion estimates for the intermediate dimensions we refer the reader to~\cite[Theorem~3.1]{Burrell2022brownian}. 
 
 In Section~\ref{holdersection} we prove more such estimates for generalised intermediate dimensions which, interestingly, are different to~\eqref{generalholderint}. 
 In previous examples such as elliptical polynomial spirals~\cite{Burrell2022spirals}, the box dimension gives the best information about H{\"o}lder exponents. 
In this thesis, however, we will see that for several classes of sets, the intermediate dimensions for $\theta \in (0,1)$ can give better information than either the Hausdorff or box dimensions. 
In particular, this is the case for some popcorn-like pyramid graphs (see Corollary~\ref{c:holder} on page~\pageref{c:holder}), continued fraction sets (see Example~\ref{holderint} on page~\pageref{holderint}), and Bedford--McMullen carpets (see Proposition~\ref{p:biLip} on page~\pageref{p:biLip}). 
 Fraser~\cite{Fraser2019spiral} showed that for the spiral winding problem, another spectrum of dimensions (the Assouad spectrum) gives better information about H{\"o}lder exponents than has been obtained from either of the two dimensions (the upper box and Assouad dimensions) that it interpolates between.

 \emph{Lipschitz} maps are simply $1$-H\"older maps and we see from~\eqref{generalholderint} that the intermediate dimensions of sets cannot increase under Lipschitz maps. Two sets $F$ and $G$ are said to be \emph{bi-Lipschitz equivalent} if there exists a Lipschitz bijection $f \colon F \to G$ with a Lipschitz inverse. 
It is clear from the definitions that all of the notions of dimension considered in this thesis take the same values for two bi-Lipschitz equivalent sets. 
 Dimensions can therefore be used to give necessary conditions for two sets to be bi-Lipschitz equivalent. As we will see in Chapter~\ref{s:bm}, for the intermediate dimensions this is particularly relevant in the setting of Bedford--McMullen carpets.

 Potential-theoretic methods have been used to study the intermediate dimensions (and variants) of images of sets under various deterministic and random functions in~\cite{Burrell2021projections,Burrell2022brownian,
 Daw2023,Douzi2022hewitt,Feng2024intermediate}. 
Burrell, Falconer and Fraser~\cite{Burrell2021projections} have proved a Marstrand-type projection theorem for the intermediate dimensions, namely that the intermediate dimensions of orthogonal projections of a set are almost surely independent of the choice of subspace. 
 Continuity of the intermediate dimensions has powerful consequences, as illustrated by the following result. 
\begin{theorem}[Burrell--Falconer--Fraser]\label{t:burrellproj}
Let $1 \leq k < d$ be integers and let $F \subset \Rd$ be bounded with $\dim_{\mathrm H} F < k$ and $\uid F$ continuous at $\theta=0$. 
Then there exists $c<k$ such that $\ubd \pi(F) \leq c$ for every orthogonal projection $\pi \colon \Rd \to \mathbb{R}^k$, and $\ubd \pi(F) = c$ for almost every such orthogonal projection $\pi$ (with respect to the natural measure on the Grassmannian). 
The same holds with $\overline{\dim}$ replaced by $\underline{\dim}$ throughout. 
\end{theorem}
\begin{proof}
This follows by combining~\cite[Corollary~6.4]{Burrell2021projections} with~\cite[Theorem~1.8]{Falconer2020projections}. 
\end{proof} 
In Example~\ref{proj}, we observe that this can be applied to a particular class of dynamically generated sets whose intermediate dimensions we prove are continuous. 
The modified intermediate dimensions also satisfy a Marstrand-type projection result~\cite{Douzi2022hewitt}.

Fractional Brownian motion is an important stochastic process, introduced by Mandelbrot and Van Ness~\cite{Mandelbrot1968brownian} and studied by Kahane~\cite{Kahane1985fractbrown}; we refer the reader to those texts for the precise definition. 
In the cases of interest to us, it is a random function $B_\alpha \colon \Rd \to \Rd$, where $0 < \alpha < 1$ and $d \in \N$ are fixed. 
One can write $B_\alpha = (B_{\alpha,1},\dotsc,B_{\alpha,d})$ and show that each $B_{\alpha,i} \colon \Rd \to \R$ is almost surely locally $(\alpha - \epsilon)$-H\"older continuous for all $\epsilon>0$ but almost nowhere differentiable. Moreover, $B_{\alpha,i}(0) = 0$, and the increments $B_{\alpha,i}(x) - B_{\alpha,i}(y)$ are normally distributed with mean $0$ and variance $|x-y|^{2\alpha}$. For all $x,y \in \Rd$ and distinct $i,j \in \{1,\dotsc,n\}$, the processes $B_{\alpha,i}(x)$ and $B_{\alpha,j}(y)$ are independent. 
The case $\alpha = 1/2$ is usual Brownian motion, and in this case the increments $B_{\alpha,i}(x) - B_{\alpha,i}(y)$ are independent, but if $\alpha \neq 1/2$ then the increments are dependent. 

Falconer~\cite{Falconer2021seq} has explicitly computed the intermediate dimensions of fractional Brownian images of certain sequence sets. 
Burrell~\cite[Corollary~3.7]{Burrell2022brownian} has shown that for sets $F$ with $\uid F$ continuous at $\theta = 0$, almost surely 
\begin{equation}\label{e:burrellbrownian}
 \ubd F < d \qquad \mbox{ if } \alpha > \frac{1}{d}\dim_{\mathrm H} F,
 \end{equation}
but that almost surely $\ubd F = d$ if $\alpha \leq \frac{1}{d}\dim_{\mathrm H} F$. The analogous result holds for the lower versions of the dimensions. 
Note the interplay between Hausdorff and box dimension in~\eqref{e:burrellbrownian}. 
This result can in particular be applied for several classes of sets whose intermediate dimensions we prove are continuous in this thesis, such as lattices (see Section~\ref{s:lattice}), popcorn-like pyramid sets (see Section~\ref{s:popcorn}), and sets of numbers with real or complex continued fraction expansions with restricted entries (see Section~\ref{ctdfracsect}). 
Intermediate dimensions of more general Rosenblatt processes have been studied in~\cite{Daw2023}. 

After the paper on which Chapter~\ref{s:generalised} is based appeared on arXiv, Feng~\cite{Feng2024intermediate} showed that the potential-theoretic methods in \cite{Burrell2021projections,Burrell2022brownian} can be adapted to study the $\Phi$-intermediate dimensions. 
He has obtained information about $\Phi$-intermediate dimensions of images of sets under projections and fractional Brownian motion if the function $\Phi$ satisfies the property that for all $\epsilon > 0$, $\delta^{\epsilon} \log \Phi(\delta) \to 0$ as $\delta \to 0$. 
He has also shown that for every subset $E$ of the symbolic space, the intermediate and $\Phi$-intermediate dimensions of the projections of $E$ under typical self-affine coding maps are constant and given by formulas in terms of capacities.

\chapter{Generalised intermediate dimensions}\label{s:generalised}

\section{Introduction}

\subsection{Discussion of main results}\label{s:phidiscussion}

This chapter introduces a more general family of dimensions and is based on~\cite{Banaji2023gen}. 
Recall from Section~\ref{s:intdimsintro} that the intermediate dimensions are always continuous in $\theta \in (0,1]$ but are often discontinuous at $\theta = 0$. 
In this chapter, we introduce the $\Phi$-intermediate dimensions, by restricting the sizes of the covering sets to lie in a wider class of intervals of the form $[\Phi(\delta),\delta]$ for more general functions $\Phi$. These dimensions give even more refined geometric information than the intermediate dimensions about sets for which the intermediate dimensions are discontinuous at $\theta = 0$. %
While many results for the $\Phi$-intermediate dimensions are similar to results for the intermediate dimensions, others, such as the H{\"o}lder distortion estimates in Theorem~\ref{holder}, are rather different. 

A class of dimensions which generalise the Assouad spectrum were defined in \cite[Section~9]{Fraser2018firstassspec}, greatly developed by Garc{\'{i}}a, Hare and Mendivil in~\cite{Garcia2020}, and further studied in~\cite{Garcia2019-2,Troscheit2019,
Hare2022randomass,HarePreprintrandomass2,
Hare2020measures,BanajiPreprintphiassouad}. %
They are defined by fixing the relative scales in more general ways than for the Assouad spectrum, thus giving more refined geometric information about sets whose quasi-Assouad dimension is less that the Assouad dimension. These Assouad-like dimensions were part of our original motivation for considering the $\Phi$-intermediate dimensions, and there are parallels between the two settings in a philosophical sense. 

This chapter is structured as follows. 
In Section~\ref{prelimsection}, we define the notions of dimension that we will work with and make some standing assumptions to reduce repetition. 
In Section~\ref{ctysection}, we give relationships between the different notions of dimension (Propositions~\ref{basicbounds} and~\ref{compareintermediate}). In Theorem~\ref{maincty} and Proposition~\ref{zerocontinuityprop} we prove quantitative %
continuity-like properties for the $\Phi$-intermediate dimensions, which intuitively say that if two functions $\Phi$ and $\Phi_1$ are `close' to each other then the dimensions of subsets do not differ too much. Interestingly, the precise bounds depend on the Assouad and lower dimensions of the set, which give information about its extremal scaling properties. From this result we deduce a condition for the $\Phi$- and $\Phi_1$-intermediate dimensions to coincide for all subsets with finite Assouad dimension (Proposition~\ref{hardcomparisoncor}~(ii)). 
 
In Section~\ref{holdersection} we prove H{\"o}lder distortion estimates for the $\Phi$-intermediate dimensions (Theorem~\ref{holder}) which are different from the usual bound~\eqref{generalholderint} on page~\pageref{generalholderint}. The estimates imply bi-Lipschitz stability (Corollary~\ref{philipschitz}), which is an important property that most notions of dimension satisfy. This means that the $\Phi$-intermediate dimensions provide yet another invariant for the classification of subsets up to bi-Lipschitz image. 

In Section~\ref{masssection} we prove a mass distribution principle (Lemma~\ref{massdistprinc}) and a converse, a Frostman type lemma (Lemma~\ref{frostman}) for the $\Phi$-intermediate dimensions. The latter is an example of where the extension from Euclidean space to the more general metric spaces in which we work is non-trivial; we use an analogue of the dyadic cubes in general doubling metric spaces given in~\cite{Hytonen2010}. The mass distribution principle and Frostman type lemma combine to give Theorem~\ref{massfrostman}, a useful alternative definition of the $\Phi$-intermediate dimensions in terms of measures. We use this characterisation to prove Theorem~\ref{producttheorem} on the dimensions of product sets, giving new bounds in terms of the dimensions of the marginals, one of which we improve further in the case of self-products. In particular, ($\underline{\dim}^\Phi,\overline{\dim}_\mathrm{B}$) and ($\underline{\dim}_\theta,\overline{\dim}_\mathrm{B}$) satisfy the inequalities~\eqref{dimpairineq} that many `dimension pairs' satisfy, although our upper bound for $\overline{\dim}^\Phi (E \times F)$ is different to what might be expected. 
We also use the mass distribution principle to prove in Proposition~\ref{finitestability} that the lower versions of the intermediate and $\Phi$-intermediate dimensions are not finitely stable (in contrast to the upper versions).

Proposition~\ref{finitestability} also gives an example of a set to which Theorem~\ref{recoverinterpolation}, which is perhaps the most significant result of this chapter, can be applied. Theorem~\ref{recoverinterpolation} implies that for all compact subsets of an appropriate space there is a family of functions $\Phi$ which `recover the interpolation' between the Hausdorff and box dimensions, even if the intermediate dimensions are discontinuous at $\theta = 0$. %
In fact, there exists a single family of $\Phi$ which interpolate for both the upper and lower versions of the dimensions, and whose dimensions vary monotonically for all sets, but in Proposition~\ref{interpolatenotcts} we show that it might not be possible to ensure that the dimensions vary continuously for all other sets. 

\subsection{Definitions of dimensions}\label{prelimsection}%
   
For the purposes of this thesis, we make the following definition. 

 \begin{defn}\label{admissible}
 A function $\Phi \colon (0,\Delta) \to \mathbb{R}$ is \emph{admissible} if $\Phi$ is monotonic, $0<\Phi(\delta) \leq \delta$ for all $\delta \in (0,\Delta)$, and $\Phi(\delta)/\delta \to 0$ as $\delta \to 0^+$. 
 \end{defn}
 In some situations, in particular in Chapter~\ref{s:infinite}, it will be convenient to work with admissible functions that satisfy the additional mild condition that $\Phi(\delta)/\delta \to 0$ monotonically as $\delta \to 0^+$. This is satisfied by many reasonable functions such as $\delta^{1/\theta}$ and $e^{-\delta^{-0.5}}$), and we call such functions \emph{monotonically admissible}. 
  To minimise repetition, we make the following standing assumptions for the rest of this chapter: 
  \begin{itemize}
    \item The letter $\Phi$ will represent an arbitrary admissible function (except in Proposition~\ref{p:nonadmissible} where we explore the conditions on $\Phi$). 
  \item The underlying metric space is denoted by $X$ (or sometimes $Y$), and will be assumed to have more than one point and be uniformly perfect. %
  The letter $d$ will denote the the metric of the underlying space, and $c$ will usually denote the constant from Definition~\ref{d:unifperf}. %
  \item Subsets of $X$ are denoted by $F$ (or sometimes $E$ or $G$), and are assumed to be non-empty and totally bounded. 
  \end{itemize}
Using these conventions, and based on Definition~\ref{d:intdimdef}, we now make the main definition of this chapter. 
\begin{defn}\label{maindefinition}
We define the \emph{upper $\Phi$-intermediate dimension} of a subset $F$ by 
\begin{align*} \overline{\dim}^\Phi F = \inf \{ \, s \geq 0 : &\mbox{ for all } \epsilon >0 \mbox{ there exists } \delta_0 \in (0,1] \mbox{ such that for all } \delta \in (0,\delta_0) \\*
&\phantom{--}\mbox{ there exists} \mbox{ a cover } \{U_1,U_2,\dotsc\} \mbox{ of } F \mbox{ such that } \\*
&\phantom{--}\Phi(\delta)\leq |U_i| \leq \delta \mbox{ for all } i, \mbox{ and } \sum_i |U_i|^s \leq \epsilon \, \}.
\end{align*}
Similarly, we define the \emph{lower $\Phi$-intermediate dimension} of $F$ by 
\begin{align*} \underline{\dim}^\Phi F = \inf \{ \, s \geq 0 : &\mbox{ for all } \epsilon >0 \mbox{ and } \delta_0 \in (0,1] \mbox{ there exists } \delta \in (0,\delta_0) \mbox{ and a cover } \\*
&\phantom{--}\{U_1,U_2,\dotsc\} \mbox{ of } F \mbox{ such that } \\*
&\phantom{--}\Phi(\delta)\leq |U_i| \leq \delta \mbox{ for all } i, \mbox{ and } \sum_i |U_i|^s \leq \epsilon \,\}.
\end{align*}
If these two quantities coincide, we call the common value the $\Phi$\emph{-intermediate dimension} of $F$ and denote it by $\dim^\Phi F$. 
\end{defn}

In the above definition, the cover $\{U_i\}$ of $F$ is \emph{a priori} countable, but since it satisfies $0<\Phi(\delta)\leq |U_i|$ for all $i$, and $\sum_i |U_i|^s \leq \epsilon$, it must be finite. 
If $F$ were not totally bounded then the $\Phi$-intermediate dimensions of $F$ would be infinite according to Definition~\ref{maindefinition}. 
If $\theta \in (0,1)$ and $\Phi(\delta) = \delta^{1/\theta}$ for all $\delta \in [0,1]$, then it is straightforward to check that $\Phi$ is admissible, and $\overline{\dim}^\Phi F = \overline{\dim}_\theta F$ and $\underline{\dim}^\Phi F = \underline{\dim}_\theta F$ are the definitions of the upper and lower intermediate dimensions of $F$ at $\theta$, respectively.

\section{General bounds}\label{ctysection}

In this section we examine general bounds and continuity-like properties for the $\Phi$-intermediate dimensions. 
Recall that the letter $d$ is reserved for the metric in this chapter. 
Note that since we work in a more general space than $\R^n$, balls of radius $\delta$ (as defined in~\eqref{e:defineball}) can have diameter less than $2\delta$, which makes several of the proofs in this chapter more technical. 
In the special case $X = \R^n$, if $\dim_{\mathrm L} F$ is replaced by $0$, and $\dim_{\mathrm A} F$ is replaced by $n$, then the bounds in this section will hold for all subsets. 
The dimensions satisfy the inequalities in Proposition~\ref{basicbounds} below, as with the intermediate dimensions. 
We assume that the ambient metric space $X$ is $c$-uniformly perfect with more than one point, and that $\Phi(\delta)/\delta \to 0$ as $\delta \to 0^+$, to ensure that Proposition~\ref{basicbounds} will hold and to avoid cases like the two-point metric space, which would have infinite intermediate and $\Phi$-intermediate dimensions according to Definition~\ref{maindefinition}. 

\begin{prop}\label{basicbounds}
For a subset $F$, 
\begin{align*}
 0 \leq \dim_\mathrm{H} F \leq \underline{\dim}^\Phi F &\leq \overline{\dim}^\Phi F \leq \overline{\dim}_\mathrm{B} F \leq \dim_\mathrm{A} F \leq \dim_\mathrm{A} X, \mbox{ and} \\*
\underline{\dim}^\Phi F &\leq \underline{\dim}_\mathrm{B} F \leq \overline{\dim}_\mathrm{B} F. 
\end{align*}
\end{prop}

\begin{proof}
We first prove $\overline{\dim}^\Phi F \leq \overline{\dim}_\mathrm{B} F$. Recalling that $|X|$ is the diameter of $X$, since $\Phi(\delta)/\delta \to 0$, there exists $\Delta \in (0,\min\{|X|,1\})$ such that $\Phi(\delta)/\delta < c/2$ for all $\delta \in (0,\Delta)$. 
Let $s > \overline{\dim}_\mathrm{B} F$ and $\epsilon > 0$. Let $t \in (\overline{\dim}_\mathrm{B} F, s)$, so we can reduce $\Delta$ further to assume that $\Delta < \epsilon^{\frac{1}{s-t}}$ %
and that for all $\delta \in (0,\Delta)$ there exists a cover of $F$ by $\delta^{-t}$ or fewer sets $\{U_i\}$, each having diameter at most $\delta$. We may assume without loss of generality that each $U_i$ intersects $F$. If $|U_i| \geq \delta/2$ then leave $U_i$ in the cover unchanged. If $|U_i| < \delta/2$, then fix $x_i \in U_i$ and $y_i \in B(x_i,\delta/2)\setminus B(x_i,c\delta/2)$; add the point $y_i$ to $U_i$, and call the resulting cover $\{V_i\}$. For each $i$, 
\[\Phi(\delta) \leq c\delta/2 \leq |V_i| \leq  \delta \]
by the triangle inequality. Moreover, 
\[\sum_i |V_i|^s \leq \delta^{-t}\delta^s < \delta_0^{s-t} < \epsilon. \]
Thus $\overline{\dim}^\Phi F \leq s$ by Definition~\ref{maindefinition}, so $\overline{\dim}^\Phi F \leq \overline{\dim}_\mathrm{B} F$, as required. 

The proof that $\underline{\dim}^\Phi F \leq \underline{\dim}_\mathrm{B} F$ is similar. %
Indeed, let $s' > \underline{\dim}_\mathrm{B} F$ and $\epsilon' > 0$. Let $t' \in (\underline{\dim}_\mathrm{B} F, s')$, so for all $\Delta' \in (0,\min\{(\epsilon')^{\frac{1}{s'-t'}},|X|,1\})$ there exists $\delta' \in (0,\Delta')$ and a cover of $F$ by $(\delta')^{-t'}$ or fewer sets, each having diameter at most $\delta'$. As above, we can use this cover to form a cover $\{V'_j\}$ which satisfies $\Phi(\delta') \leq |V'_j| \leq \delta'$ for all $j$ and $\sum_j |V'_j|^s < \epsilon'$. Therefore $\underline{\dim}^\Phi F \leq s'$, so $\underline{\dim}^\Phi F \leq \underline{\dim}_\mathrm{B} F$. 

 The inequalities $\dim_\mathrm{H} F \leq \underline{\dim}^\Phi F$, $\underline{\dim}^\Phi F \leq \overline{\dim}^\Phi F$ and $\underline{\dim}_\mathrm{B} F \leq \overline{\dim}_\mathrm{B} F$ follow directly from the definitions. The inequality $\overline{\dim}_\mathrm{B} F \leq \dim_\mathrm{A} F$ holds by fixing $R = |F|$ in the definition~\ref{e:assouaddef}. The inequality $\dim_\mathrm{A} F \leq \dim_\mathrm{A} X$ follows from~\ref{e:assouaddef} since $F \subseteq X$. 
\end{proof}

It follows from Proposition~\ref{basicbounds} that if $F \subset \R^n$ is non-empty and bounded then $\underline{\dim}^\Phi F \leq \overline{\dim}^\Phi F \leq n$, and if in addition $F$ is open with respect to the Euclidean metric then $\underline{\dim}^\Phi F = \overline{\dim}^\Phi F = n$, as one would expect. 
There is no general relationship between the lower box dimension and the upper intermediate dimensions. 
Indeed, it is a straightforward exercise to construct a non-empty bounded $F \subset \R$ such that $\lbd F = 0$ and $\ubd F = 1$, in which case $\overline{\dim}_{\theta} F = 1$ for all $\theta \in (0,1]$ by \cite[Proposition~2.4]{Falconer2020firstintermediate} (see also Corollary~\ref{c:int-dim-bound}). 
But if $G = \{ \, 1/n : n \in \N \, \}$, then $\overline{\dim}_{\theta} G = \frac{\theta}{1+\theta} < 1/2 = \lbd G$ for all $\theta \in (0,1)$ by \cite[Proposition~3.1]{Falconer2020firstintermediate}. 

The dimensions satisfy the following basic properties. 
\begin{prop}\label{unprovedprop}
  \hfill \begin{enumerate}[label=(\roman*)]
\item Both $\overline{\dim}^\Phi$ and $\underline{\dim}^\Phi$ are \emph{increasing for sets}: if $E \subseteq F$ then $\overline{\dim}^\Phi E \leq \overline{\dim}^\Phi F$ and $\underline{\dim}^\Phi E \leq \underline{\dim}^\Phi F$. 
\item Both $\overline{\dim}^\Phi$ and $\underline{\dim}^\Phi$ are \emph{stable under closure}: $\overline{\dim}^\Phi F = \overline{\dim}^\Phi \overline{F}$ and $\underline{\dim}^\Phi F = \underline{\dim}^\Phi \overline{F}$. 
\end{enumerate}
\end{prop}
\begin{proof}
This is straightforward from the definition. 
\end{proof}

\begin{example}\label{dirichlet}
The set $F \coloneqq \mathbb{Q} \cap [0,1] \subset \mathbb{R}$ is countable, so $\dim_\mathrm{H} F = 0$, but $\underline{\dim}^\Phi F = \overline{\dim}^\Phi F = 1$ for all admissible $\Phi$, directly from Definition~\ref{maindefinition}. This demonstrates that:
\begin{itemize}
\item The dimensions $\underline{\dim}^\Phi$ and $\overline{\dim}^\Phi$ are different from $\dim_\mathrm{H}$. 
\item There are subsets of $\mathbb{R}$, such as $F$, for which there does not exist a family of admissible functions for which the $\Phi$-intermediate dimensions interpolate between the Hausdorff and box dimensions of the set. This means that the assumption of compactness in Theorem~\ref{recoverinterpolation} on page~\pageref{recoverinterpolation} cannot be removed in general. 
\item The dimensions $\underline{\dim}^\Phi$ and $\overline{\dim}^\Phi$ are not countably stable. 
\end{itemize}
\end{example}

In Proposition~\ref{whenequalsbox} we give a sufficient condition for the $\Phi$-intermediate dimension always to equal the box dimension. 
As an example, the function $\Phi(\delta) \coloneqq \frac{\delta}{-\log \delta}$ satisfies the assumptions of Proposition~\ref{whenequalsbox}. Recall that in this chapter $N_\delta(F)$ is defined as in~\eqref{e:ndeltadefn}. 
\begin{prop}\label{whenequalsbox}
Let $\Phi$ be an admissible function such that $\frac{\log \delta}{\log \Phi(\delta)} \to 1$ as $\delta \to 0^+$. Then for any subset $F$, $\overline{\dim}^\Phi F = \overline{\dim}_\mathrm{B} F$ and $\underline{\dim}^\Phi F = \underline{\dim}_\mathrm{B} F$. 
\end{prop}

\begin{proof}
We prove that $\overline{\dim}^\Phi F = \overline{\dim}_\mathrm{B} F$; the proof of $\underline{\dim}^\Phi F = \underline{\dim}_\mathrm{B} F$ is similar. %
Assume (for the purpose of obtaining a contradiction) that $\overline{\dim}^\Phi F < \overline{\dim}_\mathrm{B} F$, and let $s,t \in \mathbb{R}$ be such that $\overline{\dim}^\Phi F < s < t < \overline{\dim}_\mathrm{B} F$. Then for all sufficiently small $\delta$ there exists a cover $\{U_i\}$ of $F$ such that $\Phi(\delta) \leq |U_i| \leq \delta$ for all $i$, and $\sum_i |U_i|^s \leq 1$. Therefore 
\begin{equation*}
N_\delta(F) \delta^t \leq \sum_i \delta^t \frac{|U_i|^s|U_i|^{t-s}}{|U_i|^t} \leq \sum_i \delta^t \frac{|U_i|^s \delta^{t-s}}{(\Phi(\delta))^t}  \leq \left(\frac{\delta^{1+(t-s)/t}}{\Phi(\delta)}\right)^t, 
 \end{equation*}
 which converges to 0 as $\delta \to 0^+$. %
 This contradicts $t < \overline{\dim}_\mathrm{B} F$ and completes the proof. 
\end{proof}

We now consider continuity-like results for the $\Phi$-intermediate dimensions. 
The main such result is Theorem~\ref{maincty}, which roughly implies that if two admissible functions $\Phi$ and $\Phi_1$ are in a quantitative sense `close' to each other, then the $\Phi$ and $\Phi_1$-intermediate dimensions of sets whose Assouad dimension is not too large do not differ greatly. In a similar spirit, quantitative continuity results have been proven for the intermediate dimensions in $\R^n$, for example \cite[Proposition~2.1]{Falconer2020firstintermediate} and \cite[(14.2.2)]{Falconer2021intdimsurvey}. 
In Chapter~\ref{s:attainable}, we will see what Theorem~\ref{maincty} says about the $\theta$-intermediate dimensions of sets and deduce a complete characterisation of attainable forms of intermediate dimensions (we will also give a self-contained proof for completeness). 
\begin{thm}\label{maincty}
Let $\Phi$ and $\Phi_1$ be admissible functions. Let $F$ be a subset satisfying $0 < \dim_\mathrm{A} F < \infty$,  and assume that $\overline{F}$ is complete. %
Suppose that $0 < \overline{\dim}^\Phi F < \dim_\mathrm{A} F$, and let $\eta \in [0,\dim_\mathrm{A} F - \overline{\dim}^\Phi F)$. %
Define 
\begin{equation}\label{e:definelambdaalpha}
\gamma \coloneqq \frac{\overline{\dim}^\Phi F - \dim_\mathrm{L} F}{\overline{\dim}^\Phi F + \eta - \dim_\mathrm{L} F}; \qquad \alpha \coloneqq \frac{\dim_\mathrm{A} F-\overline{\dim}^\Phi F}{\dim_\mathrm{A} F-\overline{\dim}^\Phi F - \eta}. 
\end{equation}
If 
\begin{equation}\label{continuityassumption}
 \Phi_1\left(\delta\right) \leq (\Phi(\delta^{1/\alpha}))^{\gamma}
 \end{equation}
 for all sufficiently small $\delta > 0$, then $\overline{\dim}^{\Phi_1} F \leq \overline{\dim}^\Phi F + \eta$. 
The same holds with $\overline{\dim}$ replaced by $\underline{\dim}$ throughout. 
\end{thm}
By a similar argument, if we only assume that $\Phi_1\left(\delta\right) \leq (\Phi(\delta^{1/\alpha}))^{\gamma}$ (with $\gamma$ and $\alpha$ as in~\eqref{e:definelambdaalpha}) holds only for a sequence of $\delta \to 0^+$, then we can only conclude $\underline{\dim}^{\Phi_1} F \leq \overline{\dim}^\Phi F + \eta$. 

\begin{proof}
Without loss of generality assume $\eta > 0$, so $\gamma < 1 < \alpha$. The idea of the proof is to convert a cover for the interval $[\Phi(\delta),\delta]$ into a cover for $[\Phi_1(\delta^\alpha),\delta^\alpha]$. %
We do this by using the Assouad dimension to replace sets which are too large with sets of size $\delta^\alpha$ (corresponding to indices $I_1$). We use the lower dimension to replace sets which are too small with sets of size $(\Phi(\delta))^{\gamma}$ (corresponding to indices $I_3$). We have chosen the parameters $\gamma$ and $\alpha$ so that the `cost' of each of these actions in terms of how much the dimension can increase is the same, namely $\eta$. 

Without loss of generality we assume that $F$ is closed. %
Now for $s \in (\overline{\dim}^\Phi F,\dim_\mathrm{A} F - \eta)$ let $s' \in (\overline{\dim}^\Phi F,s)$, $a > \dim_\mathrm{A} F$ and $\lambda < \dim_\mathrm{L} F$ satisfy 
\begin{equation}\label{ctyparameters}
\gamma (s + \eta - \lambda) - (s'-\lambda) > 0  \qquad \mbox{and} \qquad  a - s' - \alpha (a-s-\eta) > 0.
\end{equation}
Let $c \in (0,1/2)$ be such that $X$ is $c$-uniformly perfect. 
Fix $C\in (0,\infty)$ such that $N_r(B(x,R)\cap F) \leq C(R/r)^a$ for all $x \in F$ and $0<r<R$. 
Since $F$ is assumed to be complete, by~\eqref{e:lowerdimexistmeas} there exists a doubling Borel probability measure $\mu$ with $\mbox{supp}(\mu) = F$ and $\dim_{\mathrm L} \mu \in (\lambda,\dim_{\mathrm L} F]$. %
In particular, there exists $A \in (0,1)$ such that if $0 < r < R \leq |F|$ and $x \in X$ then 
\[ \frac{\mu(B(x,R))}{\mu(B(x,r))} \geq A \left( \frac{R}{r}\right)^{\lambda}.\]
Fix $M>1$ such that $\mu$ is $M$-doubling. 

Let $\epsilon > 0$. Choose $\Delta >0$ such that for all $\delta \in (0,\Delta)$ there exists a cover $\{U_i\}_{i \in I}$ of $F$ such that $\Phi(\delta) \leq |U_i| \leq \delta$ for all $i$, and 
\[ \sum_i |U_i|^{s'} \leq  (c^{-(s+\eta)} M^2 A^{-1} 10^{s+\eta} + 3^{s+\eta} + 2^a C)^{-1} \epsilon. \]
We may reduce $\Delta$ to assume that~\eqref{continuityassumption} and $\delta/\Phi_1(\delta) \geq 5/c$ hold for all $\delta \in (0,\Delta)$, and $\Delta < 1$, $\Delta<|X|$. 
Write $I$ as a disjoint union $I = I_1 \cup I_2 \cup I_3$ where 
\begin{align*} 
I_1 &\coloneqq \{ \, i \in I : \Phi(\delta) \leq |U_i| < \Phi_1(\delta^\alpha) \, \} \\*
I_2 &\coloneqq \{ \, i \in I : \Phi_1(\delta^\alpha) \leq |U_i| \leq \delta^\alpha/2 \, \} \\*
I_3 &\coloneqq \{ \, i \in I : \delta^\alpha/2 < |U_i| \leq \delta \, \},
\end{align*}
noting that some of these sets may be empty. %
Let $z_1,\dotsc,z_K$ be a maximal $4\Phi_1(\delta^\alpha)$-separated subset of 
\[ F \setminus \left( \bigcup_{i \in I_2 \cup I_3} \mathcal{S}_{\Phi_1(\delta^\alpha)}(U_i)\right),\]
recalling that $\mathcal{S}_r(U)$ is the $r$-neighbourhood of $U$. 

For each $k \in I_3$ pick 
\[ x_{k,1},\dotsc,x_{k,\lfloor C(2|U_k|/\delta^\alpha)^a \rfloor} \in F \] 
such that 
\[ \mathcal{S}_{\Phi_1(\delta^\alpha)}(U_k) \cap F \subseteq \bigcup_{l=1}^{\lfloor 2^a C|U_k|^a\delta^{-a\alpha} \rfloor} B(x_{k,l},\delta^\alpha/2).\]
Define 
\begin{align*}
\mathcal{U}_1 &\coloneqq \{ \, B(z_m,5\Phi_1(\delta^\alpha)/c) : 1 \leq m \leq K \, \}, \\*
\mathcal{U}_2 &\coloneqq \{ \, \mathcal{S}_{\Phi_1(\delta^\alpha)} (U_j) : j \in I_2 \, \}, \\*
\mathcal{U}_3 &\coloneqq \bigcup_{k \in I_3} \{ \,  B(x_{k,l},\delta^\alpha/2) : 1 \leq l \leq \lfloor 2^a C|U_l|^a\delta^{-a\alpha} \rfloor \, \}. 
\end{align*}
Then $\mathcal{U}_1 \cup \mathcal{U}_2 \cup \mathcal{U}_3$ is a cover of $F$, and for sufficiently small $\delta$ the diameter of each covering set lies in the interval $[\Phi_1(\delta^\alpha),\delta^\alpha]$. 

We bound the $(s+\eta)$-powers of the diameters of each part of the cover separately. First consider $\mathcal{U}_1$. 
For $m \in \{1,\dotsc,K\}$ let $J_m \coloneqq \{\, i \in I_1 : U_i \cap B(z_m,\Phi_1(\delta^\alpha)) \neq \varnothing \, \}$. 
If $i \in J_m$, let $u_{i,m} \in U_i \cap B(z_m,\Phi_1(\delta^\alpha))$. Then 
\begin{align*}
 \mu(U_i) &\leq \mu(B(u_{i,m},2|U_i|)) \\
 &\leq A^{-1} \mu(B(u_{i,m},2\Phi_1(\delta^\alpha))) \left(\frac{\Phi_1(\delta^\alpha)}{|U_i|}\right)^{-\lambda} \\
&\leq M^2 A^{-1} \mu(B(z_m,\Phi_1(\delta^\alpha))) \left(\frac{\Phi_1(\delta^\alpha)}{|U_i|}\right)^{-\lambda}. 
\end{align*} 
Therefore 
\[ \mu(B(z_m,\Phi_1(\delta^\alpha))) \leq \sum_{i \in J_m} \mu(U_i) \leq M^2 A^{-1} \mu(B(z_m,\Phi_1(\delta^\alpha))) \cdot (\Phi_1(\delta^\alpha))^{-\lambda} \cdot \sum_{i \in J_m} |U_i|^{\lambda}. \]
Since $\mbox{supp}(\mu) = F$, we can cancel through by the positive number $\mu(B(z_m,\Phi_1(\delta^\alpha)))$. 
Note also that if $i \in I_1$ then there is at most one $m$ for which $U_i \cap B(z_m,\Phi_1(\delta^\alpha)) \neq \varnothing$. Therefore 
\begin{align*}
\sum_{U \in \mathcal{U}_1} |U|^{s+\eta} &\leq K (10c^{-1} \Phi_1(\delta^\alpha))^{s+\eta} \\
&\leq c^{-(s+\eta)} M^2 A^{-1} 10^{s + \eta} (\Phi_1(\delta^\alpha))^{s + \eta - \lambda} \sum_{i \in I} |U_i|^{\lambda}
\\&\leq c^{-(s+\eta)} M^2 A^{-1} 10^{s + \eta} (\Phi_1(\delta^\alpha))^{s + \eta - \lambda} (\Phi(\delta))^{-(s' - \lambda)} \sum_{i \in I} |U_i|^{s'} \\
&\leq c^{-(s+\eta)} M^2 A^{-1} 10^{s + \eta} (\Phi(\delta))^{\gamma(s + \eta - \lambda) - (s' - \lambda)} \sum_{i \in I} |U_i|^{s'} \\
&< c^{-(s+\eta)} M^2 A^{-1} 10^{s + \eta} \sum_{i \in I} |U_i|^{s'}, 
\end{align*}
where we used~\eqref{ctyparameters} in the last step. 

For $\mathcal{U}_2$, 
\[ \sum_{U \in \mathcal{U}_2} |U|^{s + \eta} \leq \sum_{j \in I_2} (3 |U_j|)^{s + \eta} \leq 3^{s+\eta} \sum_{j \in I} |U_j|^{s'}. \] 
Finally, consider $\mathcal{U}_3$. Since $|U_k| \leq \delta$ for $k \in I_3$, 
\begin{align*}
\sum_{k \in I_3} \sum_{l = 1}^{\lfloor 2^a C|U_k|^a\delta^{-a\alpha} \rfloor} |B(x_{k,l},\delta^\alpha/2)|^{s + \eta} &\leq \sum_{k \in I_3} 2^a C|U_k|^a \delta^{-a\alpha} \delta^{\alpha(s + \eta)} \\
&\leq 2^a C \delta^{-a\alpha + \alpha(s + \eta) + a-s'} \sum_{k \in I_3} |U_k|^{s'} \\
&\leq 2^a C \sum_{k \in I} |U_k|^{s'}. 
\end{align*}
Bringing the above bounds together, for all $\delta \in (0, \Delta)$, 
\[ \sum_{U \in \mathcal{U}_1 \cup \mathcal{U}_2 \cup \mathcal{U}_3} |U|^{s + \eta} \leq (c^{-(s+\eta)} M^2 A^{-1} 10^{s+\eta} + 3^{s+\eta} + 2^a C) \sum_{i \in I} |U_i|^{s'} \leq \epsilon. \]
It follows that $\overline{\dim}^{\Phi_1} F \leq s + \eta$, as required. 
The proof for when $\overline{\dim}$ is replaced by $\underline{\dim}$ is similar. 
\end{proof} 

We have a similar result for the case when the $\Phi$-intermediate dimension of $F$ is $0$. 
\begin{prop}\label{zerocontinuityprop}
Let $\Phi,\Phi_1$ be admissible functions, assume $0 < \dim_\mathrm{A} F < \infty$, let $\eta \in (0,\dim_\mathrm{A} F)$, and let $b>0$. 
If for all sufficiently small $\delta$, 
\begin{equation}\label{zerocontinuityassumption} 
\Phi_1(\delta) \leq \left(\Phi(\delta^{1/\alpha})\right)^b 
\qquad \mbox{where} \qquad \alpha = \alpha(\eta) \coloneqq \frac{\dim_\mathrm{A} F}{\dim_\mathrm{A} F - \eta}
\end{equation}
holds, then if $\underline{\dim}^\Phi F = 0$ then $\underline{\dim}^{\Phi_1} F \leq \eta$, and if $\overline{\dim}^\Phi F = 0$ then $\overline{\dim}^{\Phi_1} F \leq \eta$. 
If we assume only that~\eqref{zerocontinuityassumption} holds for a subsequence of $\delta \to 0^+$, then if $\overline{\dim}^\Phi F = 0$ then $\underline{\dim}^{\Phi_1} F \leq \eta$. 
\end{prop}

\begin{proof} 
This is a straightforward modification of the proof of Theorem~\ref{maincty}. A cover for $[\Phi(\delta),\delta]$ is converted into a cover for $[\Phi_1(\delta^\alpha),\delta^\alpha]$ by breaking up the largest sets using the Assouad dimension of $F$, and fattening the smallest sets. 
The details are left to the reader. 
\end{proof} 
In particular, if $\Phi_1(\delta) \leq (\Phi(\delta))^b$ holds for some $b>0$ and all sufficiently small $\delta$, then $\overline{\dim}^{\Phi} F = 0$ implies $\overline{\dim}^{\Phi_1} F = 0$. 
The following corollary of Theorem~\ref{maincty} and Proposition~\ref{zerocontinuityprop} says that if the underlying metric space is doubling, then if $\Phi$ and $\Phi_1$ are `close' in a way that depends only on $X$, then the difference between the $\Phi$- and $\Phi_1$-intermediate dimensions of subsets will be small, independently of the particular subset. %
\begin{cor}\label{indepctycor}
Let $X$ be a doubling metric space and suppose $F \subseteq X$ is bounded. If 
\begin{equation}\label{indepctyassumption}
\Phi_1\left(\delta^{\frac{\dim_\mathrm{A} X}{\dim_\mathrm{A} X - \eta}}\right) \leq (\Phi(\delta))^{\frac{\dim_\mathrm{A} X}{\dim_\mathrm{A} X + \eta}}
\end{equation}
holds for all sufficiently small $\delta$, then if $\overline{\dim}^\Phi F < \dim_\mathrm{A} F$ and $\eta \in [0,\dim_\mathrm{A} F - \overline{\dim}^\Phi F)$ then $\overline{\dim}^{\Phi_1} F \leq \overline{\dim}^\Phi F + \eta$, and the same holds with $\overline{\dim}$ replaced by $\underline{\dim}$ throughout. 
If we only assume that~\eqref{indepctyassumption} holds for a subsequence of $\delta \to 0^+$, and if $\overline{\dim}^\Phi F < \dim_\mathrm{A} F$ and $\eta \in [0,\dim_\mathrm{A} F - \overline{\dim}^\Phi F)$, then $\underline{\dim}^{\Phi_1} F \leq \overline{\dim}^\Phi F + \eta$. 
\end{cor} 

\begin{proof} 
Using notation from~\eqref{e:definelambdaalpha}, by Proposition~\ref{basicbounds}, 
\[ \gamma \leq \frac{\dim_\mathrm{A} X}{\dim_\mathrm{A} X + \eta}   \leq 1 \leq \frac{\dim_\mathrm{A} X}{\dim_\mathrm{A} X - \eta} \leq \alpha ,\]
so the result follows from Theorem~\ref{maincty} in the cases $\overline{\dim}^\Phi F > 0$ and $\underline{\dim}^\Phi F > 0$, and from Proposition~\ref{zerocontinuityprop} in the cases $\overline{\dim}^\Phi F = 0$ and $\underline{\dim}^\Phi F = 0$. 
\end{proof}

We now define an equivalence relation $\equiv$ on the set of admissible functions by setting $\Phi_1 \equiv \Phi_2$ if for every subset $F$ with $\dim_{\mathrm{A}} F < \infty$ of every underlying space $X$ we have $\overline{\dim}^{\Phi_1} F = \overline{\dim}^{\Phi_2} F$ and $\underline{\dim}^{\Phi_1} F = \underline{\dim}^{\Phi_2} F$. 
We define a non-strict partial order $\preceq$ on the equivalence classes of $\equiv$ by setting $[\Phi_1] \preceq [\Phi_2]$ if $\overline{\dim}^{\Phi_1} F \leq \overline{\dim}^{\Phi_2} F$ and $\underline{\dim}^{\Phi_1} F \leq \underline{\dim}^{\Phi_2} F$ for all such $F$. We abuse notation by writing $\Phi_1 \preceq \Phi_2$ to mean $[\Phi_1] \preceq [\Phi_2]$. 

Consider the topology on the set of equivalence classes generated by the basis of open sets 
\[ \{ \, N_{\Phi,\alpha} : \Phi \mbox{ an admissible function, } \alpha \in (1,\infty) \, \}, \]
where
\begin{align*}
 N_{\Phi,\alpha} \coloneqq \{ \, C : &\mbox{ there exists } \Phi_1 \in C \mbox{ such that } (\Phi(\delta^a))^a \leq \Phi_1(\delta) \leq (\Phi(\delta^{1/a}))^{1/a} \\*
 &\mbox{ for all } a \in (1,\alpha) \mbox{ and all sufficiently small } \delta \, \}.
 \end{align*}
Then for any given subset $F$ of finite Assouad dimension, the maps $[\Phi] \mapsto \overline{\dim}^{\Phi} F$ and $[\Phi] \mapsto \underline{\dim}^{\Phi} F$ are well-defined and, by Theorem~\ref{maincty} and Proposition~\ref{zerocontinuityprop}, continuous with respect to this topology. This provides motivation for calling these `continuity-like' results. 
It is natural to ask about abstract properties of the topological space, such as connectivity and separability (see \cite[Corollary~B]{BanajiPreprintphiassouad} and the remarks after it relating to the generalised Assouad dimensions), but we will not pursue this. 

Corollary~\ref{hardcomparisoncor} gives a condition for the dimensions to coincide for all sets. 

\begin{cor}\label{hardcomparisoncor}
Let $\Phi,\Phi_1$ be admissible functions. 
\begin{enumerate}[label=(\roman*)]
\item If for all $\alpha \in (1,\infty)$ there exists $\Delta > 0$ such that for all $\delta \in (0,\Delta)$ we have %
\begin{equation}\label{hardcomparisoneqn}
 \Phi_1(\delta) \leq (\Phi(\delta^{1/\alpha}))^{1/\alpha}
 \end{equation} 
 (noting that this will be the case if, for example, there exists $C \in (0,\infty)$ such that ${\limsup_{\delta \to 0^+} \frac{\Phi_1(C \delta)}{\Phi(\delta)} < \infty}$), then $\Phi_1 \preceq \Phi$. 
 If we only assume that for all $\alpha \in (1,\infty)$ and $\delta_0 > 0$ there exists $\delta \in (0,\delta_0)$ such that~\eqref{hardcomparisoneqn} holds, then we can only conclude that $\underline{\dim}^{\Phi_1} F \leq \overline{\dim}^\Phi F$ for all subsets $F$ with finite Assouad dimension. 
 
 \item If for all $\alpha \in (1,\infty)$ there exists $\Delta > 0$ such that for all $\delta \in (0,\Delta)$,  
 \begin{equation}\label{hardequality}
  (\Phi(\delta^\alpha))^\alpha \leq \Phi_1(\delta) \leq (\Phi(\delta^{1/\alpha}))^{1/\alpha}
  \end{equation}
 holds, then $\Phi \equiv \Phi_1$. 
 \end{enumerate}
 \end{cor}
 
 \begin{proof} 
In the cases $\overline{\dim}^\Phi F = 0$ and $\overline{\dim}^\Phi F = \dim_\mathrm{A} F$, (i) follows from Propositions~\ref{zerocontinuityprop} and~\ref{basicbounds}. If $0 < \overline{\dim}^\Phi F < \dim_\mathrm{A} F$ then for all $\eta \in [0,\dim_\mathrm{A} F - \overline{\dim}^\Phi F)$, by the case of~\eqref{hardcomparisoneqn} with 
 \[ \alpha \coloneqq \min\left\{\frac{\dim_\mathrm{A} F-\overline{\dim}^\Phi F}{\dim_\mathrm{A} F-\overline{\dim}^\Phi F - \eta},\frac{\overline{\dim}^\Phi F + \eta}{\overline{\dim}^\Phi F}\right\}, \]%
 it follows that $\overline{\dim}^{\Phi_1} F \leq \overline{\dim}^\Phi F + \eta$ by Theorem~\ref{maincty}. Since $\eta$ was arbitrary, $\overline{\dim}^{\Phi_1} F \leq \overline{\dim}^\Phi F$. 
 Similarly, in all cases $\underline{\dim}^{\Phi_1} F \leq \underline{\dim}^\Phi F$, so $\Phi_1 \preceq \Phi$. 
 The case when we only assume~\eqref{hardequality} along a subsequence is proved similarly, and~(ii) follows from~(i). 
 \end{proof}

We now use Corollary~\ref{hardcomparisoncor} to explore the conditions that can be imposed on the function $\Phi$, and show that nothing is really lost by only considering functions which are strictly increasing, invertible and continuous. %

\begin{prop}
For every admissible function $\Phi$ there exists an admissible function $\Phi_1 \colon (0,1) \to (0,1)$ that is a strictly increasing, $C^\infty$ diffeomorphism, such that $\Phi \equiv \Phi_1$. 

\end{prop}

\begin{proof}
Fix $N \in \mathbb{N}$ such that $\Phi$ is positive and increasing on $(0,2^{-N}]$ with $\Phi(2^{-N}) < 1$. We construct a strictly increasing function $\Phi_2 \colon (0,1] \to (0,1]$ by defining $\Phi_2$ to be linear on $[2^{-N},1]$ with $\Phi_2(2^{-N}) = \Phi(2^{-N})$ and $\Phi_2(1)=1$ and defining $\Phi_2$ inductively on $(0,2^{-N})$ as follows. Suppose we have defined $\Phi_2$ on $[2^{-n},1]$ for some $n \geq N$. If $\Phi(2^{-n-1}) < \Phi_2(2^{-n})$ then define $\Phi_2(2^{-n-1}) = \Phi(2^{-n-1})$ and $\Phi_2$ linear on $[2^{-n-1},2^{-n}]$. If, on the other hand, $\Phi(2^{-n-1}) = \Phi_2(2^{-n})$, then let $m > n$ be the smallest integer such that $\Phi(2^{-m}) < \Phi(2^{-n})$, define $\Phi_2(2^{-m}) \coloneqq \max\{\Phi_2(2^{-n})/2,\Phi(2^{-m})\}$, and define $\Phi_2$ to be linear on $[2^{-m},2^{-n}]$. Then by construction $\Phi_2$ is strictly increasing on $(0,1]$ with $\Phi_2(\delta/4) \leq \Phi(\delta)$ and $2\Phi_2(2\delta) \geq \Phi(\delta)$ for all $\delta \in (0,2^{-N-1})$. 
Each of the countably many points of non-differentiability of $\Phi_2$ can be locally made smooth to give an admissible function $\Phi_1 \colon (0,1) \to (0,1)$ that is $C^\infty$ on $(0,1)$, still strictly increasing, and such that $\Phi_2(\delta)/2 \leq \Phi_1(\delta) \leq 2\Phi_2(\delta)$ for all $\delta \in (0,2^{-N})$. Then 
\[ \Phi_1(\delta)/\delta \leq 2\Phi_2(\delta)/\delta \leq 2\Phi(4\delta)/\delta = 8\Phi(4\delta)/(4\delta) \xrightarrow[\delta \to 0^+]{} 0, \]
so $\Phi_1$ is admissible. Moreover, 
\[ \Phi(\delta/2)/4 \leq \Phi_2(\delta)/2 \leq \Phi_1(\delta) \leq 2\Phi_2(\delta) \leq 2\Phi(4\delta) \]
for all $\delta \in (0,2^{-N-3})$, so $\Phi_1 \equiv \Phi$ by Corollary~\ref{hardcomparisoncor}~(ii). By the smooth inverse function theorem, $\Phi_1$ has a $C^\infty$ inverse, as required. 
\end{proof}

In the proof of Theorem~\ref{inttypeub} on page~\pageref{inttypeub} we will use the following technical lemma. 
 \begin{lma}\label{phiinvertible}
 If $\Phi \colon (0,\Delta] \to \mathbb{R}$ is monotonically admissible then there exists an \emph{invertible} monotonically admissible function $\Phi_1 \colon (0,\Delta] \to (0,\Phi(\Delta)]$ such that $\phi \equiv \phi_1$. 
 \end{lma}
 
 \begin{proof}
 Define $\Phi_1$ by $\Phi_1(\Delta/2^n) \coloneqq \Phi(\Delta/2^n)$ and $\Phi_1$ linear on $[\Delta/2^{n+1},\Delta/2^n]$ for $n=0,1,2,\dotsc$. Then clearly $\Phi_1$ is invertible and satisfies $\Phi_1(\delta) \leq \delta$ for all $\delta \in (0,\Delta]$ and $\Phi_1(\delta)/\delta \searrow 0$ monotonically as $\delta \to 0^+$. Moreover 
 $ \Phi(\delta/2) \leq \Phi_1(\delta) \leq \Phi(2\delta) $
 for all $\delta \in (0,\Delta/2]$, so by Corollary~\ref{hardcomparisoncor}, $\phi \equiv \phi_1$.
 \end{proof}
 
The following proposition shows that the assumption that $\Phi$ is monotonic and strictly positive does not really lose anything. 
\begin{prop}\label{p:nonadmissible}
Let $\Phi \colon (0,\Delta] \to [0,\infty)$ be \emph{any} function (not necessarily monotonic) such that $\Phi(\delta)/\delta \to 0$ as $\delta \to 0^+$, and define the $\Phi$-intermediate dimensions as in Definition~\ref{maindefinition}. Let $F$ be a subset. 
\begin{enumerate}[label=(\roman*)]
\item 
If $\Phi_1$ is defined by $\Phi_1(\delta) \coloneqq \sup\{\,\Phi(\delta') : \delta' \in [0,\delta] \, \}$ then $\overline{\dim}^\Phi F = \overline{\dim}^{\Phi_1} F$. 
\item \begin{enumerate}[label=(\arabic*)]
\item If there is a sequence of $\delta \to 0^+$ for which $\Phi(\delta) = 0$ then $\underline{\dim}^\Phi F = \dim_\mathrm{H} F$. 

\item Suppose $\Phi(\delta)>0$ for all $\delta \in (0,\Delta)$ but for all $\delta_2 \in (0,\Delta)$ there exists $\delta_3 \in (0,\delta_2)$ such that $\inf\{\,\Phi(\delta) : \delta \in [\delta_3,\delta_2] \, \} = 0$. Then if $F$ is compact then $\underline{\dim}^\Phi F = \dim_\mathrm{H} F$. In particular, if $F$ is a non-empty, bounded subset of $X=\R^n$ then $\underline{\dim}^\Phi F = \dim_\mathrm{H} \overline{F}$. %

\item If $\Phi_2 \colon (0,\Delta) \to \mathbb{R}$ defined by $\Phi_2(\delta) \coloneqq \inf\{\,\Phi(\delta') : \delta' \in [\delta,\Delta] \, \}$ is positive for all $\delta \in (0,\Delta)$, then $\underline{\dim}^\Phi F = \underline{\dim}^{\Phi_2} F$. 
\end{enumerate}
\end{enumerate}

\end{prop}

\begin{proof}
We may assume that $\Delta < \min\{1,|X|\}$ and that $\Phi(\delta) \leq (1+2/c)^{-1} \delta$ for all $\delta \in (0,\Delta)$. In the proofs of the different parts of the proposition, the same symbols may take different values. 

(i) For all $\delta \in (0,\Delta)$, 
\[ \Phi_1(\delta)/\delta = \sup\{ \, \Phi(\delta')/\delta : \delta' \in (0,\delta] \, \} \leq \sup\{ \, \Phi(\delta')/\delta' : \delta' \in (0,\delta] \, \} \xrightarrow[\delta \to 0^+]{} 0, \]
and $\Phi_1(\delta)$ is monotonic, so $\Phi_1$ is admissible. 
Also, $\Phi(\delta)\leq \Phi_1(\delta)$, so $\overline{\dim}^\Phi F \leq \overline{\dim}^{\Phi_1} F$. 

It remains to prove the reverse inequality. %
Let $s>\overline{\dim}^\Phi F$ and $\epsilon > 0$. 
Then there exists $\delta_0>0$ such that for all $\delta \in (0,\min\{\delta_0,\Delta\})$ there exists a cover $\{U_i\}$ of $F$ such that $\Phi(\delta) \leq |U_i| \leq \delta$ for all $i$, and 
\begin{equation}\label{firstphicompareeps}
\sum_i |U_i|^s \leq 2^{-s}(1+1/c)^{-s} \epsilon.
\end{equation}
 Then if $\delta' \in (0,\delta_0)$ then there exists $\delta \in (0,\delta']$ such that $\Phi(\delta)\geq \Phi_1(\delta')/2$. Let $\{U_i\}$ be the cover corresponding to $\delta$ as above. For each $i$, if $|U_i| \geq \Phi_1(\delta')$ then leave $|U_i|$ in the cover unchanged, noting that $\Phi_1(\delta') \leq |U_i| \leq \delta \leq \delta'$. If $\Phi_1(\delta') > |U_i|$, on the other hand, then fix $p_i \in U_i$, and $q_i \in X$ such that $\Phi_1(\delta') \leq d(p_i,q_i) \leq \Phi_1(\delta')/c$. Replace $U_i$ in the cover by $U_i \cup \{q_i\}$, and denote the new cover of $F$ by $\{V_i\}_i$. Then 
\[ \Phi_1(\delta') \leq d(p_i,q_i) \leq |U_i \cup \{q_i\}| < (1+1/c)\Phi_1(\delta') \leq \delta'. \] 
Also,
 \[ |U_i \cup \{q_i\}| \leq 2(1+1/c)\Phi(\delta) \leq 2(1+1/c)|U_i|.\]
Therefore 
\[ \sum_i |V_i|^s \leq \sum_i (2(1+1/c)|U_i|)^s = 2^s (1+1/c)^s \sum_i |U_i|^s \leq \epsilon \]
by~\eqref{firstphicompareeps}, so $\overline{\dim}^{\Phi_1} F \leq s$, hence $\underline{\dim}^{\Phi_1} F \leq \overline{\dim}^\Phi F$ as required. 

(ii) 
(1) Follows directly from~\eqref{hausdorffdef} and Definition~\ref{maindefinition}. 

(ii) (2) 
Assume that $F$ is compact. %
Let $s>\dim_\mathrm{H} F$, $\epsilon>0$ and $\delta_2 \in (0,1]$, so there exists $\delta_3 \in (0,\delta_2)$ such that $\inf\{\,\Phi(\delta) : \delta \in [\delta_3,\delta_2] \, \} = 0$. There exists a countable cover $\{U_i\}$ of $F$ such that $\sum_i |U_i|^{s} \leq \min\{\delta_3^{s},\epsilon\}$. In particular, $|U_i|\leq \delta_3$. Since $F$ is compact, there is a finite subcover $\{V_i\}$, so $\min_i\{|V_i|\}>0$, and each  $|V_i|\leq \delta_3$. Since $\inf\{\,\Phi(\delta) : \delta \in [\delta_3,\delta_2] \, \} = 0$, there exists $\delta_4 \in [\delta_3,\delta_2]$ such that $\Phi(\delta_4) \in (0,\min_i\{|V_i|\})$. Then $0 \leq \Phi(\delta_4) \leq \min_i\{|V_i|\}) \leq |V_i| \leq \delta_3 \leq \delta_4$ for each $i$, and $\sum_i |V_i|^{s} \leq \sum_i |U_i|^{s} \leq \epsilon$. As $\epsilon$ and $\delta_2$ were arbitrary, $\underline{\dim}^\Phi F \leq s$, so $\underline{\dim}^\Phi F = \dim_\mathrm{H} F$. 

(ii) (3) 
Clearly $\Phi_2$ is admissible and $\underline{\dim}^{\Phi_2} F \leq \underline{\dim}^{\Phi} F$, so it remains to prove the reverse inequality. Let $s > \underline{\dim}^{\Phi_2} F$ and $\epsilon > 0$. Let $\delta_1 > 0$ and let $\delta_0 \in (0,\Phi_2(\min\{\Delta,\delta_1\})/2)$. %
 Then there exists $\delta \in (0,\delta_0)$ and a cover $\{U_i\}$ of $F$ such that $\Phi_2(\delta) \leq |U_i| \leq \delta$ for all $i$, and 
 \begin{equation}\label{lastphicompareeps}
 \sum_i |U_i|^s \leq 2^{-s}(1+1/c)^{-s} \epsilon.
 \end{equation}
  By the definition of $\Phi_2$, there exists $\delta_2 \in [\delta,\Delta]$ such that $\Phi(\delta_2) < 2\Phi_2(\delta)$. But since $\Phi_2(\delta) \leq \Phi(\delta) \leq \delta_0 < \Phi_2(\min\{\Delta,\delta_1\})/2$, %
 it must be the case that $\delta_2 <\min\{\Delta,\delta_1\}$. %
 If $|U_i| \geq \Phi(\delta_2)$ then leave $U_i$ in the cover unchanged. If $|U_i| < \Phi(\delta_2)$ then fix $p_i \in U_i$ and $q_i \in X$ such that $\Phi(\delta_2) \leq d(p_i,q_i) \leq \Phi(\delta_2)/c$; replace $U_i$ in the cover with $U_i \cup \{q_i\}$ and call the new cover $\{V_i\}$. Now, $\Phi(\delta_2) \leq |U_i \cup \{q_i\}| < \delta_2$. Also, $|U_i \cup \{q_i\}| \leq 2(1+1/c)\Phi_2(\delta) \leq 2(1+1/c)|U_i|$. 
 Therefore 
 \[ \sum_i |V_i|^s \leq \sum_i (2(1+1/c)|U_i|)^s = 2^s (1+1/c)^s \sum_i|U_i|^s \leq \epsilon, \]
 by~\eqref{lastphicompareeps}. It follows that $\underline{\dim}^{\Phi} F \leq s$, as required. 
\end{proof}

Another consequence of Corollary~\ref{hardcomparisoncor} is the following relationships between the $\Phi$-intermediate and intermediate dimensions. %

\begin{prop}\label{compareintermediate}
Let $\Phi$ be an admissible function, and let
\begin{equation}\label{compareintermediatedefinitions} \theta_1 \coloneqq \liminf_{\delta \to 0^+} \frac{\log \delta}{\log \Phi(\delta)}; \qquad \theta_2 \coloneqq \limsup_{\delta \to 0^+} \frac{\log\delta}{\log\Phi(\delta)},
\end{equation}
 noting that $0\leq \theta_1 \leq \theta_2 \leq 1$. If $\dim_{\mathrm A} F < \infty$ then the following bounds hold: 
 \begin{itemize}
\item If $0=\theta_2=\lim_{\delta \to 0^+} \frac{\log \delta}{\log \Phi(\delta)}$ then $\underline{\dim}^\Phi F \leq \underline{\dim}_\theta F$ and $\overline{\dim}^\Phi F \leq \overline{\dim}_\theta F$ for all $\theta \in (0,1]$. 

\item If $0=\theta_1<\theta_2$ then $\underline{\dim}_{\theta_2} F \leq \overline{\dim}^\Phi F \leq \overline{\dim}_{\theta_2} F$ (so if $\dim_{\theta_2} F$ exists then $\overline{\dim}^\Phi F = \dim_{\theta_2} F$), and $\underline{\dim}^\Phi F \leq \min\{\overline{\dim}_\theta F,\underline{\dim}_{\theta_2} F\}$ for all $\theta \in (0,1]$. 

\item If $0< \theta_1\leq\theta_2$ then 
\begin{align*}
\underline{\dim}_{\theta_1} F \leq &\underline{\dim}^\Phi F \leq \min\{\overline{\dim}_{\theta_1} F,\underline{\dim}_{\theta_2} F\}, \\*
\max\{\overline{\dim}_{\theta_1} F,\underline{\dim}_{\theta_2} F\} \leq &\overline{\dim}^\Phi F \leq \overline{\dim}_{\theta_2} F. 
\end{align*}

\item If $0<\theta_1=\theta_2$ then $\underline{\dim}^\Phi F = \underline{\dim}_{\theta_1} F$ and $\overline{\dim}^\Phi F =\overline{\dim}_{\theta_1}$. 
\end{itemize}
\end{prop}

\begin{proof}
As an example, we prove $\overline{\dim}^\Phi F \leq \overline{\dim}_{\theta_2} F$ under the assumption that $\theta_2 > 0$; the other bounds are proved similarly. If $\theta_2 = 1$ then this follows from Proposition~\ref{basicbounds}, so assume $\theta_2 \in (0,1)$. Then letting $\eta \in (0,1-\theta_2)$, by the definition of $\theta_2$, 
\[ \limsup_{\delta \to 0^+}\frac{\Phi(\delta)}{\delta^{1/(\theta_2 + \eta)}} = 0 < \infty. \]
Corollary~\ref{hardcomparisoncor}~(i) now gives $\overline{\dim}^\Phi F \leq \overline{\dim}_{\theta_2+\eta} F$. It is straightforward to verify from Theorem~\ref{maincty} that the intermediate dimensions are continuous at $\theta_2 > 0$ (see also Theorem~\ref{intermediatects}), so the result follows upon letting $\eta \to 0^+$. 
\end{proof}

For sets whose upper intermediate dimensions are continuous at $\theta=0$, usually we will not study the $\Phi$-intermediate dimensions, because much information about the general $\Phi$-intermediate dimensions of such sets can be obtained directly from results about their intermediate dimensions and the inequalities from Corollary~\ref{hardcomparisoncor}.

\section{H\"older and Lipschitz maps}\label{holdersection}

\subsection{H\"older distortion}

We now investigate how these dimensions behave under H{\"o}lder and Lipschitz maps and prove bounds which are different from the usual $\dim f(F) \leq \alpha^{-1}\dim F$. 

\begin{thm}\label{holder}
Let $\Phi$ and $\Phi_1$ be admissible functions and let $(X,d_X)$ and $(Y,d_Y)$ be uniformly perfect. 
Let $f \colon F \to Y$ be a H{\"o}lder map with exponent $\alpha \in (0,1]$ for some $F \subseteq X$, assume $\dim_\mathrm{A} f(F) < \infty$, and let $\gamma \in [1,1/\alpha]$. 
Assume that 
\begin{equation}\label{phiholdercondition}
 \Phi_1(\delta) \leq (\Phi(\delta^{1/(\alpha \gamma)}))^\alpha
 \end{equation}%
 for all sufficiently small $\delta$, and suppose $\overline{\dim}^\Phi F < \alpha \dim_\mathrm{A} f(F)$. 

  Then 
\[ \overline{\dim}^{\Phi_1} f(F) \leq \frac{\overline{\dim}^\Phi F + \alpha (\gamma - 1)\dim_\mathrm{A} f(F)}{\alpha \gamma}.\]
The same holds with $\overline{\dim}$ replaced by $\underline{\dim}$ throughout. 
\end{thm}

\begin{proof}
The idea of the proof is to consider a cover of $F$ with diameters in $[\Phi(\delta),\delta]$, consider the cover of $f(F)$ formed by the images under $f$ of this cover, and `fatten' the smallest sets in the new cover to size $\Phi_1(\delta^{\alpha \gamma})$ and break up the largest sets in the new cover to size $\delta^{\alpha \gamma}$. 
Assume that $f$ is $C,\alpha$-H{\"o}lder with $C \geq 1$. Let $\epsilon > 0$. Let  
\[ t > \frac{\overline{\dim}^\Phi F + \alpha (\gamma - 1)\dim_\mathrm{A} f(F)}{\alpha \gamma}. \]
Then there exist $s > \overline{\dim}^\Phi F$ and $a> \dim_\mathrm{A} f(F)$ such that $s < \alpha a$ and 
\[ t > \frac{s + \alpha (\gamma - 1)a}{\alpha \gamma}.\]  
 Define  
\[g(\eta) \coloneqq \frac{\eta s + \alpha a (\gamma - \eta )}{\alpha \gamma}. \] 

 Since $a> \dim_\mathrm{A} f(F)$, there exists $M \in \mathbb{N}$ such that for all $y \in f(F)$ and $0<r<R$, we have $N_r(B(y,R)\cap f(F)) \leq M(R/r)^a$.  
Let $c \in (0,1)$ be such that $X$ and $Y$ are $c$-uniformly perfect. %
For all small enough $\delta$ we have $\Phi(\delta)/\delta < c/2$ and $\Phi_1(\delta)/\delta < c/2$, and there exists a cover $\{U_i\}$ of $F$ such that $\Phi(\delta) \leq |U_i| \leq \delta$ for all $i$, and 
\begin{equation}\label{holdersumbound}
\sum_i |U_i|^s \leq ((C + c^{-1})^{s/\alpha} + M(2C)^{a + \gamma g(1)})^{-1} \epsilon/2.
\end{equation}
Without loss of generality assume $U_i \cap F \neq \varnothing$ for all $i$. Now, $\{f(U_i)\}$ covers $f(F)$, and $|f(U_i)| \leq C|U_i|^\alpha$ for all $i$. 
There are two cases. 

\textbf{Case 1}: Suppose $i$ is such that $|f(U_i)| \leq \delta^{\alpha \gamma}/2$. Fix any $y_i \in f(U_i)$. There exists $y_i' \in Y$ such that $\Phi_1(\delta^{\alpha \gamma}) \leq d_Y(y_i,y_i') \leq \Phi_1(\delta^{\alpha \gamma})/c$, hence $d_Y(y_i,y_i') \leq (\Phi(\delta))^\alpha/c$. Let $V_i \coloneqq f(U_i) \cup \{y_i'\}$. 
By the triangle inequality, %
\begin{equation}\label{holderdiambound1} 
\Phi_1(\delta^{\alpha \gamma}) \leq d_Y(y_i,y_i') \leq |V_i| \leq |f(U_i)| + \Phi_1(\delta^{\alpha \gamma})/c \leq \delta^{\alpha \gamma}.
\end{equation}
Moreover, by the assumption~\eqref{phiholdercondition} about $\Phi_1$, 
\begin{equation}\label{holderubound1} |V_i| \leq |f(U_i)| + \Phi_1(\delta^{\alpha \gamma})/c \leq C|U_i|^\alpha + (\Phi(\delta))^\alpha/c \leq (C+ c^{-1}) |U_i|^\alpha. 
\end{equation}

\textbf{Case 2}: Now suppose that $i$ is such that $\delta^{\alpha \gamma}/2 < |f(U_i)| \leq C\delta^\alpha$. Then $(2C)^{-1/\alpha}\delta^\gamma < |U_i| \leq \delta$ so there exists $\beta_i \in [1,\gamma]$ such that $(2C)^{-1/\alpha} \delta^{\beta_i} < |U_i| \leq \delta^{\beta_i}$. 
Then $\delta^{\alpha \gamma}/2 < |f(U_i)| \leq C\delta^{\alpha \beta_i} \leq C\delta^\alpha$. %
There exists a collection of $M(2C)^a \delta^{\alpha (\beta_i - \gamma)a} \leq M(2C)^a |U_i|^{\alpha a (1-\gamma/\beta_i)}$  %
or fewer balls, each of diameter at most $\delta^{\alpha \gamma}/2$, which cover $f(U_i) \cap f(F)$. For each ball we can add a point in $Y$ whose distance from the centre of the ball is between $\Phi_1(\delta^{\alpha \gamma})$ and $\Phi_1(\delta^{\alpha \gamma})/c$. Each of the new sets, which we call $\{W_{i,j}\}_j$, will satisfy 
\begin{equation}\label{holderdiambound2}
 \Phi_1(\delta^{\alpha \gamma}) \leq |W_{i,j}| \leq  \delta^{\alpha \gamma}.
 \end{equation}
Moreover, 
\begin{equation}\label{holderubound2}
|W_{i,j}| \leq \delta^{\alpha \gamma} =  (2C)^{\gamma/\beta_i} ((2C)^{-1/\alpha}\delta^{\beta_i})^{\alpha \gamma/\beta_i} \leq (2C)^{\gamma/\beta_i} |U_i|^{\alpha \gamma/\beta_i}. 
\end{equation}

Note that $g(\eta)$ is linear and decreasing in $\eta$, so $t>g(1) \geq g(\eta) \geq g(\gamma) = s/\alpha$ for all $\eta \in [1,\gamma]$, and in particular $t > g(\beta_i)$ for all $i$. Therefore using~\eqref{holderubound1} and~\eqref{holderubound2},  
\begin{align*}
\sum_k |V_k|^t &+ \sum_{i,j} |W_{i,j}|^t < \sum_k |V_k|^{s/\alpha} + \sum_{i,j} |W_{i,j}|^{g(\beta_i)} \\
&\leq \sum_k ((C + c^{-1})|U_k|^\alpha)^{s/\alpha} + \sum_i M(2C)^a |U_i|^{\alpha a (1-\gamma/\beta_i)} ((2C)^{\gamma/\beta_i} |U_i|^{\alpha \gamma/\beta_i})^{g(\beta_i)} \\
&\leq (C + c^{-1})^{s/\alpha} \sum_k |U_k|^s + M(2C)^{a + \gamma g(\beta_i)/\beta_i} \sum_i |U_i|^s \\
&\leq \epsilon,
\end{align*}
where the last equality follows from~\eqref{holdersumbound}. 
Also, $\{V_k\}_k \cup \{W_{i,j}\}_{i,j}$ covers $f(F)$, and (noting~\eqref{holderdiambound1} and~\eqref{holderdiambound2}) $\overline{\dim}^{\Phi_1} f(F) \leq t$, as required. 
\end{proof}

We make several comments about Theorem~\ref{holder}. 
\begin{itemize}
\item An important special case is when $\gamma = 1/\alpha$ and $\Phi_1 = \Phi$. Then we can conclude 
\[ \overline{\dim}^\Phi f(F) \leq \overline{\dim}^\Phi F + (1-\alpha)\dim_\mathrm{A} f(F). \]
\item Another special case is for the $\Phi_1$ which satisfy~\eqref{phiholdercondition} with $\gamma = 1$, when we can conclude $\overline{\dim}^{\Phi_1} f(F) \leq \alpha^{-1}\overline{\dim}^\Phi F$. %
\item If $\overline{\dim}^\Phi F \geq \alpha \dim_\mathrm{A} f(F)$ (contrary to the assumption of Theorem~\ref{holder}) then the simple bound $\overline{\dim}^\Phi f(F) \leq \alpha^{-1} \overline{\dim}^\Phi F$ follows immediately. %

\item If $\overline{\dim}^\Phi F < \alpha \dim_\mathrm{A} f(F)$ but we only assume  that~\eqref{phiholdercondition} holds along a subsequence of $\delta \to 0^+$, then we can conclude only that 
\[ \underline{\dim}^{\Phi_1} f(F) \leq \frac{\overline{\dim}^\Phi F + \alpha (\gamma - 1)\dim_\mathrm{A} f(F)}{\alpha \gamma}.\]
\end{itemize}

Setting $\Phi(\delta) = \delta^{1/\theta}$ gives a H{\"o}lder distortion estimate for the intermediate dimensions in Corollary~\ref{holderintermediate}. 
For subsets of Euclidean space, Corollary~\ref{holderintermediate} was noted in~\cite[Section~14.2.1 5.]{Falconer2021intdimsurvey}, and it also follows from the stronger result~\cite[Theorem~3.1]{Burrell2022brownian} which is proven using capacity theoretic methods and dimension profiles, but we include it nonetheless because our proof works for more general metric spaces. 

\begin{cor}\label{holderintermediate}
If $f \colon F \to Y$ is an $\alpha$-H{\"o}lder map with exponent $\alpha \in (0,1]$ and $\dim_\mathrm{A} f(F) < \infty$, then $\overline{\dim}_\theta f(F) \leq \alpha^{-1}\overline{\dim}_\theta F$ and $\underline{\dim}_\theta f(F) \leq \alpha^{-1}\underline{\dim}_\theta F$ for all $\theta \in [0,1]$. 

\end{cor}

\begin{proof} 
These estimates hold for the Hausdorff and lower and upper box dimensions (similar to~\cite[Exercise~2.2 and Proposition~3.3]{Falconer2014main}), so assume that $\theta \in (0,1)$ and let $\Phi(\delta) = \Phi_1(\delta) = \delta^{1/\theta}$. 
If $\overline{\dim}_\theta F \geq \alpha \dim_\mathrm{A} f(F)$ then $\overline{\dim}_\theta f(F) \leq \dim_\mathrm{A} f(F) \leq \alpha^{-1} \overline{\dim}_\theta F$. If $\overline{\dim}_\theta F < \alpha \dim_\mathrm{A} f(F)$ then since 
\[ \Phi_1(\delta) = \Phi(\delta) = \delta^{1/\theta} = ((\delta^{1/\alpha})^{1/\theta})^\alpha = \Phi(\delta^{1/\alpha})^\alpha, \]
the case $\gamma = 1$ of Theorem~\ref{holder} gives that $\overline{\dim}_\theta f(F) \leq \alpha^{-1} \overline{\dim}_\theta F$. 
Similarly, the bound for the lower intermediate dimensions follows from the version of Theorem~\ref{holder} for the lower $\Phi$-intermediate dimensions. %
\end{proof} 

\subsection{Lipschitz stability}

Recall that a map is \emph{Lipschitz} if it is 1-H\"older, and \emph{bi-Lipschitz} if it is a Lipschitz bijection with a Lipschitz inverse. 
Corollary~\ref{philipschitz} shows that the $\Phi$-intermediate dimensions cannot increase under Lipschitz maps. %
It follows that $\overline{\dim}^\Phi$ and $\underline{\dim}^\Phi$ satisfy the important property of being stable under bi-Lipschitz maps. 
Bi-Lipschitz stability has already been proven for the Hausdorff and box dimensions in~\cite[Propositions~2.5 and 3.3]{Falconer2014main} and, for subsets of $\R^n$, for the intermediate dimensions in~\cite[Lemma~3.1]{Fraser2021interpolating}.

\begin{cor}\label{philipschitz}
Let $X$ and $Y$ be underlying spaces, let $F \subseteq X$, let $f \colon F \to Y$ be Lipschitz, and assume that $\dim_\mathrm{A} f(F) < \infty$. Then 
\begin{enumerate}[label=(\roman*)]
\item\label{i:firstphilip} We have $\overline{\dim}^\Phi f(F) \leq \overline{\dim}^\Phi F$ and $\underline{\dim}^\Phi f(F) \leq \underline{\dim}^\Phi F$.

\item\label{i:bilipschitzlabel} If moreover $f$ is bi-Lipschitz then 
$\overline{\dim}^\Phi f(F) = \overline{\dim}^\Phi F$
and
$\underline{\dim}^\Phi f(F) = \underline{\dim}^\Phi F$. 
\end{enumerate}
\end{cor}
In~\ref{i:bilipschitzlabel}, the assumption $\dim_\mathrm{A} f(F) < \infty$ is equivalent to $\dim_\mathrm{A} F < \infty$ since the Assouad dimension is stable under bi-Lipschitz maps. 

\begin{proof} 
If $\overline{\dim}^\Phi F \geq \dim_\mathrm{A} f(F)$ then $\overline{\dim}^\Phi f(F) \leq \dim_\mathrm{A} f(F) \leq \overline{\dim}^\Phi F$ by Proposition~\ref{basicbounds}; if $\overline{\dim}^\Phi F < \dim_\mathrm{A} f(F)$ then the case $\alpha = \gamma = 1$, $\Phi_1 = \Phi$, of Theorem~\ref{holder} gives $\overline{\dim}^\Phi f(F) \leq \overline{\dim}^\Phi F$. The proof that $\underline{\dim}^\Phi f(F) \leq \underline{\dim}^\Phi F$ is similar, and~\ref{i:bilipschitzlabel} follows from~\ref{i:firstphilip}. 
\end{proof}

\section{A mass distribution principle}\label{masssection}

In this section we prove a mass distribution principle for the $\Phi$-intermediate dimensions and a converse result (a Frostman type lemma), which together give an alternative characterisation of the intermediate dimensions. We then prove some applications regarding product sets and finite stability. 

\subsection{A mass distribution principle}

The mass distribution principle is a useful tool to bound dimensions from below by putting a measure on the set. The original version was for the Hausdorff dimension (see~\cite[page~67]{Falconer2014main}), and a version was proved for the intermediate dimensions in~\cite[Proposition~2.2]{Falconer2020firstintermediate}. The following natural generalisation for the $\Phi$-intermediate dimensions holds. 

\begin{lemma}\label{massdistprinc}
 Let $F$ be a subset and let $s,a,c,\delta_0>0$ be positive constants. 
 \begin{enumerate}[label=(\roman*)]
\item If there exists a positive decreasing sequence $\delta_n \to 0$ such that for each $n \in \mathbb{N}$ there exists a Borel measure $\mu_n$ with support $\mathrm{supp}(\mu_n) \subseteq F$ with $\mu_n (\mathrm{supp}(\mu_n)) \geq a$, and such that for every Borel subset $U \subseteq X$ with $\Phi(\delta_n) \leq |U| \leq \delta_n$ we have $\mu_n(U)\leq c |U|^s$, then $\overline{\dim}^\Phi F \geq s$. %

\item If, moreover, for all $\delta \in (0,\delta_0)$ there exists a Borel measure $\mu_\delta$ with support $\mathrm{supp}(\mu_\delta) \subseteq F$ with $\mu_\delta (\mathrm{supp}(\mu_\delta)) \geq a$, and such that for every Borel subset $U \subseteq X$ with $\Phi(\delta) \leq |U| \leq \delta$ we have $\mu_\delta(U)\leq c |U|^s$, then $\underline{\dim}^\Phi F \geq s$. 
\end{enumerate}
\end{lemma}

\begin{proof}
We prove (i); the proof of~(ii) is similar. If $n \in \mathbb{N}$ and $\{U_i\}$ is a cover of $F$ such that $\Phi(\delta_n) \leq |U_i| \leq \delta_n$ for all $i$, then the closures $\overline{U_i}$ are Borel, satisfy $\Phi(\delta_n) \leq |\overline{U_i}| = |U_i| \leq \delta_n$, and cover $\mathrm{supp}(\mu_n)$, so
\begin{equation}\label{massdistprinceqn} 
a \leq \mu_n (\mathrm{supp}(\mu_n)) = \mu_n\left(\bigcup_i \overline{U_i}\right) \leq \sum_i \mu_n(\overline{U_i}) \leq c\sum_i |\overline{U_i}|^s = c\sum_i |U_i|^s. 
\end{equation}
 Therefore $\sum_i |U_i|^s \geq a/c > 0$, so $\overline{\dim}^\Phi F \geq s$. 
 \end{proof}
 
 \subsection{A Frostman type lemma}\label{s:frostman}
 
Another powerful tool in fractal geometry and geometric measure theory is Frostman's lemma, dual to the mass distribution principle. 
Lemma~\ref{frostman} is an analogue of Frostman's lemma for the $\Phi$-intermediate dimensions, generalising~\cite[Proposition~2.3]{Falconer2020firstintermediate} for the intermediate dimensions both to more general functions~$\Phi$ and to more general metric spaces. 

The main difference with the proof of~\cite[Proposition~2.3]{Falconer2020firstintermediate} is that in $\mathbb{R}^n$ there are the dyadic cubes to work with, but here we use the fact that $\dim_\mathrm{A} F < \infty$, and use an analogue of the dyadic cubes constructed in~\cite{Hytonen2010} for general doubling metric spaces. 
We now state a special case of \cite[Theorem~2.2]{Hytonen2010}, using notation from that theorem (note in particular that $\delta$ denotes a certain constant). 
We take the quasi-metric $\rho$ simply to be the metric $d$ restricted to $F$ (so the usual triangle inequality holds and $A_0 = 1$). 
Fix $\delta \coloneqq 1/20$ (in fact any $\delta \in (0,1/12)$ will do). 
Since $\dim_\mathrm{A} F < \infty$, for each $k \in \mathbb{N}$ we have $N_{\delta^k/3}(F) < \infty$. Therefore there exists a finite $\delta^k$-separated subset $\{z_{\alpha}^k\}_\alpha$ of $F$, of maximum possible cardinality. 
Then applying \cite[Theorem~2.2]{Hytonen2010} with $c_0 = C_0 = 1$, $c_1 = 1/3$, $C_1 = 2$, for each $k \in \mathbb{N}$ there exist subsets $Q^k \coloneqq \{Q^k_\alpha\}_\alpha$ of $F$ such that: 
\begin{enumerate}
\item for all $k \in \mathbb{N}$, $F = \bigcup_\alpha Q_\alpha^k$ with the union disjoint;
\item\label{frostmanballinball} 
$B^F (z_\alpha^k,(20)^{-k}/4) \subseteq B^F (z_\alpha^k,c_1 (20)^{-k}) \subseteq Q_\alpha^k 
\subseteq B^F (z_\alpha^k,C_1 (20)^{-k})
= B^F (z_\alpha^k,2 (20)^{-k})$, recalling that $B^F$ denotes the open ball in $F$;
\item if $k,l \in \mathbb{N}$ with $k \leq l$ then for all $\alpha, \beta$, either $Q_\alpha^k \cap Q_\beta^l = \varnothing$ or $Q_\beta^l \subseteq Q_\alpha^k$, and in the latter case, also $B^F (z_\beta^l,2 (20)^{-l}) \subseteq B^F (z_\alpha^k,2 (20)^{-k})$. We call $Q_\alpha^k$ a \emph{parent} of $Q_\beta^l$. 
\end{enumerate}
We say that $Q_\alpha^k$ is a \emph{dyadic cube} with \emph{centre} $z_\alpha^k$. 

\begin{lemma}\label{frostman}
Assume that $\dim_\mathrm{A} F < \infty$. 
\begin{enumerate}[label=(\roman*)]
\item If $\overline{\dim}^\Phi F > 0$ then for all $s \in (0,\overline{\dim}^\Phi F)$ there exists a constant $c \in (0,\infty)$ such that for all $\delta_0 >0$ there exist $\delta' \in (0,\delta_0)$ and a Borel probability measure $\mu_{\delta'}$ with finite support $\mathrm{supp}(\mu_{\delta'}) \subseteq F$ 
such that if $x \in X$ and $\Phi(\delta') \leq r \leq \delta'$ then 
\[\mu_{\delta'} (B(x,r)) \leq c r^s.\]

\item If $\underline{\dim}^\Phi F > 0$ then for all $s \in (0,\underline{\dim}^\Phi F)$ there exists $c \in (0,\infty)$ such that for all sufficiently small $\delta'$ there exists a Borel probability measure $\mu_{\delta'}$ with finite support $\mathrm{supp}(\mu_{\delta'}) \subseteq F$ such that if $x \in X$ and $\Phi(\delta') \leq r \leq \delta'$ then $\mu_{\delta'} (B(x,r)) \leq c r^s$. 
\end{enumerate}
\end{lemma} 

\begin{proof}
We prove~(ii); the proof of~(i) is similar. 
The idea of the proof is to put point masses on an analogue of dyadic cubes of size approximately $\Phi(\delta')$ so that the measure of sets with diameter approximately $\Phi(\delta')$ is controlled by the $\Phi$-intermediate dimension of $F$, and then iteratively reduce the masses so that the mass of larger cubes is not too large either. 
The proof is based on the proof of~\cite[Proposition~2.3]{Falconer2020firstintermediate} for the intermediate dimensions, which is in turn based on~\cite[pages~112--114]{Mattila1995book}. In~\cite[Proposition~2.3]{Falconer2020firstintermediate}, the assumption that the set $F$ is closed is not necessary as it is not used in the proof. 

We use notation from \cite[Theorem~2.2]{Hytonen2010}, as above. 
Let $c_2 \in (0,1)$ be such that $X$ is $c_2$-uniformly perfect. 
Suppose $\underline{\dim}^\Phi F > 0$ and let $s \in (0,\underline{\dim}^\Phi F)$. Then there exists $\epsilon > 0$ such that for all sufficiently small $\delta'$ and all covers $\{U_i\}$ of $F$ satisfying $\Phi(\delta') \leq |U_i| \leq \delta'$ for all $i$, %
\begin{equation}\label{frostmansumlarge}
\sum_i |U_i|^s \geq \epsilon.
\end{equation} 
Let $\delta'$ be small enough such that this is the case, and moreover that $\Phi(\delta')/\delta' < c_2/320$. Define $m = m(\delta')$ to be the largest natural number satisfying $\Phi(\delta') \leq \frac{1}{2}(20)^{-m}$. Define the Borel measure $\mu_m$ by 
\[ \mu_m \coloneqq \sum_{\alpha} 20^{-ms} M_{z_\alpha^k} \]
where $M_{z_\alpha^k}$ is a unit point mass at $z_\alpha^k$. 

Let $l$ be the largest integer such that $8(20^{-(m-l)}) \leq \delta'$, noting that $l \geq 1$. In particular, $|Q_{m-l}| \leq \delta'/2$ for all $Q_{m-l} \in Q^{m-l}$. 
In order to reduce the mass of cubes which carry too much measure, having defined $\mu_{m-k}$ for some $k \in \{0,1,\dotsc,l-1\}$, inductively define the Borel measure $\mu_{m-k-1}$, supported on the same finite set as $\mu_m$, by 
\[ \mu_{m-k-1} |_{Q_{m-k-1}} \coloneqq \min\left\{1, \frac{20^{-(m-k-1) s}}{ \mu_{m-k}(Q_{m-k-1})}\right\}\mu_{m-k}|_{Q_{m-k-1}} \]
for all $Q_{m-k-1} \in Q^{m-k-1}$. By construction, if $k \in \{0,1,\dotsc, l\}$ and $Q_{m-k} \in Q^{m-k}$ then 
\begin{equation}\label{frostmanupperbound} \mu_{m-l}(Q_{m-k}) \leq 20^{-(m-k)s} \leq 4^s c_2^{-s} |Q_{m-k}|^s 
\end{equation}
by condition~\ref{frostmanballinball}. %
Moreover, each $Q_m \in Q^m$ satisfies $\mu_m(Q_m) = 20^{-ms}$. If $k \in \{0,1,\dotsc, l-1\}$ and $Q_{m-k} \in Q^{m-k}$ satisfies $\mu_{m-k}(Q_{m-k}) = 20^{-(m-k)s}$ and $Q_{m-k-1} \in Q^{m-k-1}$ is the parent of $Q_{m-k}$, then by the construction of $\mu_{m-k-1}$, either $\mu_{m-k-1}(Q_{m-k}) = 20^{-(m-k)s}$ or $\mu_{m-k-1}(Q_{m-k-1}) = 20^{-(m-k-1)s}$. Therefore for all $y \in F$ there is at least one $k \in \{0,1,\dotsc, l\}$ and $Q_y \in Q^{m-k}$ with $y \in Q_y$ such that 
\begin{equation}\label{frostmanmulowerbound}
\mu_{m-l}(Q_y) = 20^{-(m-k)s} \geq 4^{-s}|Q_y|^s, 
\end{equation}
where the inequality is by condition~\ref{frostmanballinball}. 

 For each $y \in F$, choosing $Q_y$ such that~\eqref{frostmanmulowerbound} is satisfied and moreover $Q_y \in Q^{m-k}$ for the largest possible $k \in \{0,1,\dotsc, l\}$ yields a finite collection of cubes $\{Q_i\}$ which cover $F$. 
 For each $i$, let $z_i$ be the centre of $Q_i$, and by the uniformly perfect condition there exists $p_i \in X$ such that $\Phi(\delta') \leq d(p_i,z_i) \leq \Phi(\delta')/c_2 \leq \delta'/2$. Letting $U_i \coloneqq Q_i \cup \{p_i\}$, by condition~\ref{frostmanballinball} we have $\Phi(\delta') \leq |U_i| \leq \delta'$. Then $\{U_i\}$ covers $F$, and each $|U_i| \leq |Q_i| + \Phi(\delta')/c_2 \leq (1+1/c_2)|Q_i|$. %
 Therefore by~\eqref{frostmansumlarge} and~\eqref{frostmanmulowerbound},  
\begin{equation}\label{frostmangreaterepsilon}
\mu_{m-l}(F) = \sum_i \mu_{m-l}(Q_i) \geq \sum_i 4^{-s}|Q_i|^s \geq 4^{-s}(1+1/c_2)^{-s} \sum_i |U_i|^s \geq 4^{-s}(1+1/c_2)^{-s} \epsilon. 
\end{equation}
Define $\mu_{\delta'} \coloneqq (\mu_{m-l}(F))^{-1} \mu_{m-l}$, which is clearly a Borel probability measure with finite support $\mathrm{supp}(\mu_{\delta'}) \subseteq F$. 

Now, since $\dim_\mathrm{A} F < \infty$ there exists $C \in \mathbb{N}$ such that for all $p \in F$ and $d>0$, $N_d(B^F(p,13d)) \leq C$. 
Let $x \in X$ and $r \in [\Phi(\delta'), \delta']$. Let $j = j(r)$ be the largest integer in $\{0,1,\dotsc,l\}$ such that $20^{-(m-j+1)} < r$; such an integer exists by the definition of $m$. 
If $B^X (x,r) \cap F = \varnothing$ then $\mu_{\delta'}(B^X (x,r)) = 0$, so suppose that there exists some $x_1 \in B^X (x,r) \cap F$, so $B^X (x,r) \subseteq B^F(x_1,2r)$. 
Suppose $B^X (x,r) \cap Q_{m-j} \neq \varnothing$ for some $Q_{m-j} \in Q^{m-j}$, with centre $z_{m-j}$, say. 
Then there exists $z \in B^X(x,r) \cap Q_{m-j}$, and by condition~\ref{frostmanballinball} and the definition of $j$,  
\[ d(x_1,z_{m-j}) \leq d(x_1,z) + d(z,z_{m-j}) \leq 2r + 2(20)^{-(m-j)} \leq 6(20)^{-(m-j)}. \]
Therefore $z_{m-j} \in B^F (x_1,6(20)^{-(m-j)})$, and the centres of the cubes in $Q_{m-j}$ which intersect $B^X (x,r)$ form a $20^{-(m-j)}$-separated subset of $B^F (x_1,6(20)^{-(m-j)})$. But 
\[ N_{6(20)^{-(m-j)}/13}(B^F (x_1,6(20)^{-(m-j)})) \leq C.\] 
Therefore there are most $C$ such centres, so at most $C$ elements of $Q^{m-j}$ which intersect $B^X(x,r)$. Therefore by~\eqref{frostmanupperbound} and~\eqref{frostmangreaterepsilon} and the definition of $j$,  
\[ \mu_{\delta'} (B^X(x,r)) = (\mu_{m-l}(F))^{-1} \mu_{m-l}(B^X(x,r)) \leq C (\mu_{m-l}(F))^{-1} 20^{-(m-j)s} \leq c r^s, \]
where $c \coloneqq C 4^s (1+1/c_2)^s \epsilon^{-1} (20)^s$, as required. 
\end{proof}

Putting Lemmas~\ref{massdistprinc} and~\ref{frostman} together, we obtain a useful characterisation of the $\Phi$-intermediate dimensions. 

\begin{thm}\label{massfrostman}
If $\Phi$ is an admissible function and $\dim_\mathrm{A} F < \infty$ then 
\begin{align*} (i) \quad \overline{\dim}^\Phi F = \sup\{ s \geq 0 : &\mbox{ there exists } C \in (0,\infty) \mbox{ such that for all } \delta_1 > 0 \\
&\mbox{ there exists } \delta \in (0,\delta_1) \mbox{ and a Borel probability measure } \mu_\delta  \\*
&\mbox{ with support } \mathrm{supp}(\mu_\delta) \subseteq F \mbox{ such that if } U \mbox{ is a Borel subset }\\
&\mbox{ of } X \mbox{ which satisfies } \Phi(\delta) \leq |U| \leq \delta \mbox{ then } \mu_\delta(U) \leq C|U|^s \} 
\end{align*}%

\begin{align*} (ii) \quad \underline{\dim}^\Phi F = \sup\{ s \geq 0 : &\mbox{ there exist } C,\delta_1 \in (0,\infty) \mbox{ such that for all } \delta \in (0,\delta_1) \mbox{ there }\\ 
&\mbox{ exists a Borel probability measure } \mu_\delta \mbox{ with support } \\*
&\mathrm{supp}(\mu_\delta) \subseteq F \mbox{ such that if } U \mbox{ is a Borel subset satisfying }\\
&\Phi(\delta) \leq |U| \leq \delta \mbox{ then } \mu_\delta(U) \leq C|U|^s \} 
\end{align*}
\end{thm}

\begin{proof}
We prove~(ii) using Lemma~\ref{massdistprinc}~(ii) and Lemma~\ref{frostman}~(ii); (i) follows from Lemma~\ref{massdistprinc}~(i) and Lemma~\ref{frostman}~(i) in a similar way. 
We denote by $\mathit{sup}$ the supremum on the right-hand side of the equation~(ii). 
Fix $y \in F$. 
If $s=0$, then letting $C\coloneqq 1$ and letting $\mu_\delta$ be a unit point mass at $y$ for all sufficiently small $\delta$, we see that $\mathit{sup}$ is well-defined and non-negative. Suppose that $\underline{\dim}^\Phi F > 0$ and let $s \in (0,\underline{\dim}^\Phi F)$. Then by the Frostman type Lemma~\ref{frostman}~(ii), there exist constants $c,\delta_1 \in (0,\infty)$ such that for all $\delta \in (0,\delta_1)$ there exists a Borel probability measure $\mu_{\delta}$ with finite support $\mathrm{supp}(\mu_{\delta}) \subseteq F$ such that if $x \in X$ and $\Phi(\delta) \leq r \leq \delta$ then $\mu_{\delta} (B(x,r)) \leq c r^s$. 
If $U$ is a Borel subset of $X$ satisfying $\Phi(\delta) \leq |U| \leq \delta$, then $U \cap F = \varnothing$ implies $\mu_\delta(U) = 0$. Suppose there exists some $x \in U \cap F$. Let $M$ be the doubling constant of $F$. Then $U \cap \mathrm{supp}(\mu_\delta) \subseteq B(x,2|U|)$, so there exist $x_1,\dotsc,x_M \in B^F(x,2|U|)$ such that $U \cap \mathrm{supp}(\mu_\delta) \subseteq B^F(x,2|U|) \subseteq \bigcup_{i=1}^M B^F(x_i,|U|)$. Therefore 
\[\mu_\delta(U) \leq \sum_{i=1}^M \mu_\delta (B^F(x_i,|U|)) = \sum_{i=1}^M \mu_\delta(B^X(x_i,|U|)) \leq C|U|^s,\]
where $C \coloneqq Mc$. %
Thus $s \leq \mathit{sup}$. 

For the reverse inequality, if $\mathit{sup}> 0$ and $t \in (0, \mathit{sup})$ then by the mass distribution principle Lemma~\ref{massdistprinc}~(ii), $t \leq \underline{\dim}^\Phi F$. 
Therefore if $\max\{\mathit{sup}, \underline{\dim}^\Phi F\} > 0$ then in fact $\mathit{sup} = \underline{\dim}^\Phi F$. But both $\mathit{sup}$ and $\underline{\dim}^\Phi F$ are non-negative, so they must always be equal. 
\end{proof} 

\subsection{Product formulae}\label{s:products}

It is a well-studied problem to bound the dimensions of product sets in terms of the dimensions of the marginals. Very often, dimensions come in pairs (dim, Dim) which satisfy $\dim F \leq \mbox{Dim} F$ and 
\begin{equation}\label{dimpairineq} \dim E + \dim F \leq \dim (E \times F) \leq \dim E + \mbox{Dim} F \leq \mbox{Dim} (E \times F) \leq \mbox{Dim} E + \mbox{Dim} F \end{equation}
for all `reasonable' sets $E$ and $F$ and `reasonable' metrics on the product space. Examples are (Hausdorff, packing)~\cite{Howroyd1996}, (lower box, upper box)~\cite{Robinson2013}, (lower, Assouad) and (modified lower, Assouad)~\cite[Corollary~10.1.2]{Fraser2020book} and, for each fixed $\theta \in (0,1)$, (lower spectrum at $\theta$, Assouad spectrum at $\theta$) and (modified lower spectrum at $\theta$, Assouad spectrum at $\theta$)~\cite[Proposition~4.4]{Fraser2018firstassspec}. In Theorem~\ref{producttheorem} we show that for any given $\Phi$ or $\theta$, (lower $\Phi$-intermediate, upper box) and (lower $\theta$-intermediate, upper box) are also such pairs. However, our upper bound for $\overline{\dim}^\Phi (E \times F)$ is $\overline{\dim}^\Phi E + \overline{\dim}_\mathrm{B} F$, rather than the expected $\overline{\dim}^\Phi E + \overline{\dim}^\Phi F$. Theorem~\ref{producttheorem} generalises~\cite[Proposition~2.5]{Falconer2020firstintermediate} on the intermediate dimensions of product sets to more general functions $\Phi$ and more general metric spaces, and also gives an improved lower bound for $\overline{\dim}^\Phi (E \times F)$ and an improved upper bound for $\underline{\dim}^\Phi (E \times F)$. 
We improve the lower bound for $\overline{\dim}^\Phi (E \times F)$ further for self-products in~(iii). 

\begin{thm}\label{producttheorem}
Consider uniformly perfect metric spaces $(X,d_X)$ and $(Y,d_Y)$. Let $d_{X \times Y}$ be a metric on $X \times Y$ such that there exist constants $c_1,c_2 \in (0,\infty)$ such that 
\begin{equation}\label{metriclikesup} 
c_1 \max(d_X,d_Y) \leq d_{X \times Y} \leq c_2 \max(d_X,d_Y). 
\end{equation}
Then if $E \subseteq X$ and $F \subseteq Y$ have finite Assouad dimension, then 
 \begin{enumerate}[label=(\roman*)]
\item $\overline{\dim}^\Phi E + \underline{\dim}^\Phi F 
\leq \overline{\dim}^\Phi (E \times F) 
\leq \overline{\dim}^\Phi E + \overline{\dim}_\mathrm{B} F$;  

\item $\underline{\dim}^\Phi E + \underline{\dim}^\Phi F 
\leq \underline{\dim}^\Phi (E \times F)
\leq \underline{\dim}^\Phi E + \overline{\dim}_\mathrm{B} F$. 

In the case of self-products,~(i) can be improved to %

\item 
 $2 \overline{\dim}^\Phi F \leq \overline{\dim}^\Phi (F \times F) \leq \overline{\dim}^\Phi F + \overline{\dim}_\mathrm{B} F$. 

\end{enumerate}

\end{thm}

Note that~\eqref{metriclikesup} is the same condition as~\cite[(2.4)]{Robinson2013}, and familiar metrics which satisfy this are $d_{X \times Y} \coloneqq \max\{d_X,d_Y\}$ and $d_{X \times Y} \coloneqq (d_X^p + d_Y^p)^{1/p}$ for $p \in [1,\infty)$. 

\begin{proof}
The idea of the proof of the upper bounds is to consider a cover of one of the sets $E$ with diameters in $[\Phi(\delta),\delta]$, and, for each set $U_i$ in that cover, to form a cover of that other set $F$ with all the diameters approximately equal to $|U_i|$, with the number of sets in this cover controlled by $\overline{\dim}_{\mathrm B} F$. We can then cover the product set with approximate squares with sizes between $\Phi(\delta)$ and $\delta$ to obtain the result. 
The idea of the proof of the lower bounds is to use the Frostman type lemma to put measures on each of the marginal sets such that the measure of sets with diameter in $[\Phi(\delta),\delta]$ is controlled by the $\Phi$-intermediate dimensions of the sets, and then apply the mass distribution principle with the product measure on the product set. 

Since $X$ and $Y$ each have more than one point, so does $X \times Y$. 
A straightforward calculation shows that since $(X,d_X)$ is uniformly perfect, so is $(X \times Y,d_{X \times Y})$. 
Another routine calculation shows that since $E$ and $F$ have finite Assouad dimension, so does $E \times F$. 

(i) We first prove the upper bound of~(i), following the proof of the upper bound in~\cite[Proposition~2.5]{Falconer2020firstintermediate}. Let $\epsilon > 0$. Let $c_p \in (0,1)$ be such that $X \times Y$ is $c_p$-uniformly perfect, and without loss of generality assume $0 < c_p < c_1 < 1 < c_2 < \infty$. 
Since $\dim_\mathrm{A} (E \times F) < \infty$ there exists $A \in \mathbb{N}$ such that $N_r(B^{E \times F} (p,4c_2 r)) \leq A$ for all $p \in E \times F$ and $r > 0$. 
 Let $\Delta > 0$ be such that $\Phi(\delta)/\delta < c_p/2$ for all $\delta \in (0,\Delta)$. 
Fix $s > \overline{\dim}^\Phi E$ and $d > \overline{\dim}_\mathrm{B} F$. Let $\delta_1\in (0,\Delta)$ be such that for all $r \in (0,\delta_1)$ there is a cover of $F$ by $r^{-d}$ or fewer sets, each having diameter at most $r$. Let $\delta_0 \in (0,\delta_1)$ be such that for all $\delta \in (0,\delta_0)$ there exists a cover $\{U_i\}$ of $E$ such that $\Phi(\delta) \leq |U_i| \leq \delta$ for all $i$, and 
\begin{equation}\label{productsetepsbound}
\sum_i |U_i|^s \leq A^{-1} (c_2 + c_p^{-1})^{-(s+d)}\epsilon. 
\end{equation}

For such a cover, for each $i$ let $\{U_{i,j}\}_j$ be a cover of $F$ by $|U_i|^{-d}$ or fewer sets, each having diameter $|U_{i,j}| \leq |U_i|$. 
Then for all $i$ and $j$,  
\begin{equation}\label{productbound}
|U_i \times U_{i,j}| \leq c_2 \max\{|U_i|,|U_{i,j}|\} = c_2|U_i| \leq c_2\delta, 
\end{equation}
so $U_i \times U_{i,j}$ can be covered by $A$ sets $\{U_{i,j,k}\}_k$, each having diameter at most $\min\{\delta/2,{|U_i \times U_{i,j}|}\}$. We may assume that each of these sets is non-empty, and fix $p_{i,j,k} \in U_{i,j,k}$. Fix $q_{i,j,k} \in X \times Y$ such that $\Phi(\delta) \leq d_{X \times Y}(p_{i,j,k},q_{i,j,k}) \leq \Phi(\delta)/c_p$. Let $V_{i,j,k} \coloneqq U_{i,j,k} \cup \{q_{i,j,k}\}$, so by the triangle inequality 
\begin{equation}\label{productlimits}
 \Phi(\delta) \leq d_{X \times Y}(p_{i,j,k},q_{i,j,k}) \leq |V_{i,j,k}| \leq \delta/2 + \Phi(\delta)/c_p \leq \delta,
 \end{equation}
 since $\delta < \delta_0 < \Delta$. 
Also, by~\eqref{productbound}, 
\begin{equation}\label{productvbound} |V_{i,j,k}| \leq  c_2|U_i| + \Phi(\delta)/c_p \leq (c_2 + c_p^{-1})|U_i|. 
\end{equation}
Therefore by~\eqref{productvbound} and~\eqref{productsetepsbound},  
\[ \sum_{i,j,k} |V_{i,j,k}|^{s+d} \leq \sum_i A|U_i|^{-d} ((c_2 + c_p^{-1})|U_i|)^{s+d} \leq A (c_2 + c_p^{-1})^{s+d} \sum_i |U_i|^s \leq \epsilon.\]
Also \[E \times F \subseteq \bigcup_{i,j,k} U_{i,j,k} \subseteq \bigcup_{i,j,k} V_{i,j,k}. \]
This gives $\overline{\dim}^\Phi (E \times F) \leq s + d$, so $\overline{\dim}^\Phi (E \times F) \leq \overline{\dim}^\Phi E + \overline{\dim}_\mathrm{B} F$. 

The proof of the lower bound is somewhat similar to the proof of the lower bound in~\cite[Proposition~2.5]{Falconer2020firstintermediate}. First assume $\underline{\dim}^\Phi F = 0$. Fix any $f \in F$. By~\eqref{metriclikesup}, the natural embedding $E \xhookrightarrow{} X \times Y$, $x \mapsto (x,f)$, is bi-Lipschitz onto its image, so by Corollary~\ref{philipschitz} 2. and Proposition~\ref{unprovedprop}~(i),
\[ \overline{\dim}^\Phi E + \underline{\dim}^\Phi F = \overline{\dim}^\Phi E = \overline{\dim}^\Phi (E \times \{f\}) \leq \overline{\dim}^\Phi (E \times F). \]
Now assume that $\overline{\dim}^\Phi E > 0$ and $\underline{\dim}^\Phi F > 0$. Fix $t_1 \in (0,\overline{\dim}^\Phi E)$ and $t_2 \in (0,\underline{\dim}^\Phi F)$. By Lemma~\ref{frostman}~(i) there exists $c_E \in (0,\infty)$ such that for all $\delta_2>0$ there exists $\delta_3 \in (0,\delta_2)$ and a Borel probability measure $\mu_{\delta_3}$ with $\mathrm{supp}(\mu_{\delta_3}) \subseteq E$ such that if $x \in X$ and $\Phi(\delta_3) \leq r_1 \leq \delta_3$ then $\mu_{\delta_3}(B^X(x,r_1)) \leq c_E r_1^{t_1}$. 
By Lemma~\ref{frostman}~(ii) there exist $c_F,\delta_4 \in (0,\infty)$ such that for all $\delta_5 \in (0,\delta_4)$ there exists a Borel probability measure $\nu_{\delta_5}$ with $\mathrm{supp}(\nu_{\delta_5}) \subseteq F$ such that if $y \in Y$ and $\Phi(\delta_5) \leq r_2 \leq \delta_5$ then $\nu_{\delta_5}(B^Y(y,r_2)) \leq c_F r_2^{t_2}$. 
If $\delta_7>0$, then there exists $\delta_6 \in (0,\min\{\delta_7,\delta_4\})$ and Borel probability measures $\mu_{\delta_6}$ and $\nu_{\delta_6}$ as above. Let $\mu_{\delta_6} \times \nu_{\delta_6}$ be the product measure, which satisfies $\mathrm{supp}(\mu_{\delta_6} \times \nu_{\delta_6}) \subseteq E \times F$. 

If $U \subseteq X \times Y$ is Borel and satisfies $\Phi(\delta_6) \leq |U| \leq \delta_6$ then if we fix any $(x,y) \in U$ then 
\begin{equation}\label{productlowerboundsubset} U \subseteq B^{X \times Y}((x,y), 2|U|) \subseteq B^X(x,2|U|/c_1) \times B^Y(y,2|U|/c_1). 
\end{equation}
Fix $x_1,\dotsc,x_C \in E$ and $y_1,\dotsc,y_C \in F$ such that 
\[ \overline{E} \cap B^X(x,2|U|/c_1) \subseteq \bigcup_{i=1}^C B^X(x_i,|U|); \qquad \overline{F} \cap B^Y(y,2|U|/c_1) \subseteq \bigcup_{i=1}^C B^Y(y_i,|U|). \]
Now, 
\begin{align*} (\overline{E} \times \overline{F}) \cap (B^X(x,2|U|/c_1) &\times B^Y(y,2|U|/c_1)) \\
&= (\overline{E} \cap B^X(x,2|U|/c_1)) \times (\overline{F} \cap B^Y(y,2|U|/c_1)) \\
&\subseteq \left(\bigcup_{i=1}^C B^X(x_i,|U|)\right) \times \left(\bigcup_{j=1}^C B^Y(y_j,|U|)\right) \\
&= \bigcup_{i=1}^C \bigcup_{j=1}^C (B^X(x_i,|U|) \times B^Y(y_j,|U|)).
\end{align*}
Then by~\eqref{productlowerboundsubset} and the definition of the product measure,  
\begin{align*}
 (\mu_{\delta_6} \times \nu_{\delta_6}) (U) &\leq (\mu_{\delta_6} \times \nu_{\delta_6}) (B^X(x,2|U|/c_1) \times B^Y(y,2|U|/c_1)) \\
 &\leq (\mu_{\delta_6} \times \nu_{\delta_6}) \left(\bigcup_{i=1}^C \bigcup_{j=1}^C (B^X(x_i,|U|) \times B^Y(y_j,|U|))\right) \\
 &\leq \sum_{i=1}^C \sum_{j=1}^C (\mu_{\delta_6} \times \nu_{\delta_6}) (B^X(x_i,|U|) \times B^Y(y_j,|U|)) \\
 &= C^2 c_E c_F |U|^{t_1+t_2}. 
 \end{align*}
Therefore by the mass distribution principle Lemma~\ref{massdistprinc}~(i), $\overline{\dim}^\Phi (E \times F) \geq t_1 + t_2$, as required. The bound $\overline{\dim}^\Phi (E \times F) \geq \underline{\dim}^\Phi E + \overline{\dim}^\Phi F$ follows similarly. 

(ii) The proof of~(ii) is a straightforward modification of the proof of~(i). 

(iii) The upper bound is just the upper bound of~(i) with $E=F$; the improved bound is the lower bound. Assume $\overline{\dim}^\Phi F > 0$ and let $t \in (0,\overline{\dim}^\Phi F)$. By Lemma~\ref{frostman}~(i) there exists $c_F \in (0,\infty)$ such that for all $\delta_0>0$ there exists $\delta \in (0,\delta_0)$ and a Borel probability measure $\mu_\delta$ with $\mathrm{supp}(\mu_{\delta}) \subseteq F$ such that if $x \in X$ and $\Phi(\delta) \leq r \leq \delta$ then $\mu_{\delta}(B^X(x,r)) \leq c_F r^t$. Then $\mathrm{supp}(\mu_\delta \times \mu_\delta) \subseteq F \times F$, and as in the proof of the lower bound of~(i), if $\Phi(\delta) \leq |U| \leq \delta$ then $(\mu_\delta \times \mu_\delta)(U) \leq C^2 c_F^2 |U|^{2t}$. Therefore by Lemma~\ref{massdistprinc}, $\overline{\dim}^\Phi (F\times F) \geq 2t$, as required. 
\end{proof} 

In the particular case $\Phi(\delta) \coloneqq \frac{\delta}{-\log \delta}$, Proposition~\ref{whenequalsbox} gives $\overline{\dim}^\Phi G = \overline{\dim}_\mathrm{B} G$ and $\underline{\dim}^\Phi G = \underline{\dim}_\mathrm{B} G$ for a subset $G$ of an underlying space $X$. Therefore from~(i) and (ii) we recover the inequalities for the upper and lower box dimensions of product sets in~\cite[Theorem~2.4]{Robinson2013} (which is proven directly, without putting measures on the sets). 
Note also that bounds on the dimensions of products of more than two sets can be obtained by applying Theorem~\ref{producttheorem} iteratively, for example 
\[\overline{\dim}^\Phi (E \times F \times G) \geq \overline{\dim}^\Phi (E \times F) + \underline{\dim}^\Phi G \geq \overline{\dim}^\Phi E +  \underline{\dim}^\Phi F + \underline{\dim}^\Phi G.\]

\subsection{Finite stability}

Our next application of the mass distribution principle is Proposition~\ref{finitestability}, which illustrates an important difference between the upper and lower versions of the dimensions. It was stated in~\cite[Section~14.2.1 2.]{Falconer2021intdimsurvey} that in Euclidean space, the upper intermediate dimensions are finitely stable but the lower intermediate dimensions are not. 

\begin{prop}\label{finitestability}
Let $\Phi$ be an admissible function. 
\begin{enumerate}[label=(\roman*)]
\item The dimension $\overline{\dim}^\Phi$ is finitely stable: we always have 
\[ \overline{\dim}^\Phi(E \cup F) = \max\{\overline{\dim}^\Phi E, \overline{\dim}^\Phi F\}.\] 

\item The dimension $\underline{\dim}^\Phi$ is \emph{not} finitely stable: there exist compact sets $E,F \subset \mathbb{R}$ such that 
\[\underline{\dim}^\Phi(E \cup F) > \max\{\underline{\dim}^\Phi E, \underline{\dim}^\Phi F\}.\]
\end{enumerate}
\end{prop}

\begin{proof}

It is a straightforward exercise to prove~(i) directly from Definition~\ref{maindefinition}, so we prove only~(ii). 
To do so, we take inspiration from~\cite[Exercises~2.8, 2.9]{Falconer2014main}. %
The idea is to construct generalised Cantor sets $E$ and $F$, each of which looks `large' on most scales but `small' on some sequence of scales. We do this in such a way that the sequences of scales where the two sets look small do not even approximately coincide, so for each small $\delta$, either $E$ looks large at \emph{every} scale between $\delta$ and $\Phi(\delta)$, or $F$ looks large at \emph{every} scale between $\delta$ and $\Phi(\delta)$. 

Assume without loss of generality that $\Phi \colon (0,1] \to \mathbb{R}$. 
We inductively define the numbers $k_n \in \{0,1,2,\dotsc \}$ and $e_{10^{k_n}},f_{10^{k_n}}>0$, for $n = 0,1,2,\dotsc$, as follows. Let $k_0\coloneqq 0$, $e_{10^{k_0}}=f_{10^{k_0}}=1$. Having defined $k_n,e_{10^{k_n}},f_{10^{k_n}}$ for some $n = 0,1,2,\dotsc$, there are two cases depending on the parity of $n$. 
If $n$ is even, let $k_{n+1}$ be the smallest integer such that $k_{n+1} > k_n$ and 
\begin{equation}\label{finitestableeqne} (1/3)^{10^{k_{n+1}}-10^{k_n}}f_{10^{k_n}} < \Phi\left((1/5)^{10^{k_n + 1} - 10^{k_n}}(1/3)^{10^{k_n+2}-10^{k_n+1}}e_{10^{k_n}}\right),
\end{equation} 
and let 
\begin{align*}
e_{10^{k_{n+1}}} &\coloneqq (1/5)^{10^{k_n + 1} - 10^{k_n}}(1/3)^{10^{k_{n+1}}-10^{k_n+1}}e_{10^{k_n}}, \\*
f_{10^{k_{n+1}}} &\coloneqq (1/3)^{10^{k_{n+1}}-10^{k_n}}f_{10^{k_n}}.
\end{align*} 
If, on the other hand, $n$ is odd, then let $k_{n+1} > k_n$ be the smallest integer such that 
\begin{equation}\label{finitestableeqnf} (1/3)^{10^{k_{n+1}}-10^{k_n}}e_{10^{k_n}} < \Phi\left((1/5)^{10^{k_n + 1} - 10^{k_n}}(1/3)^{10^{k_n+2}-10^{k_n+1}}f_{10^{k_n}}\right),
\end{equation}
 and let 
 \begin{align*}
 f_{10^{k_{n+1}}} &\coloneqq (1/5)^{10^{k_n + 1} - 10^{k_n}}(1/3)^{10^{k_{n+1}}-10^{k_n+1}}f_{10^{k_n}}, \\*
 e_{10^{k_{n+1}}} &\coloneqq (1/3)^{10^{k_{n+1}}-10^{k_n}}e_{10^{k_n}}.
 \end{align*} 

Now let $E_1 \coloneqq [0,1]$ and for $j \in \mathbb{N}$, if $10^{k_n} < j \leq 10^{k_n+1}$ for some even $n \in \{0,2,4,\dotsc\}$ then obtain $E_j$ by removing the middle 3/5 of each interval in $E_{j-1}$, otherwise obtain $E_j$ by removing the middle 1/3 of each interval in $E_{j-1}$. 
Let $F_1 \coloneqq [2,3]$ and for $j \in \mathbb{N}$, if $10^{k_n} < j \leq 10^{k_n+1}$ for some odd $n \in \{1,3,5,\dotsc\}$ then obtain $F_j$ from removing the middle 3/5 of each interval in $F_{j-1}$, otherwise obtain $F_j$ by removing the middle 1/3 of each interval in $F_{j-1}$. 
Define $E \coloneqq \bigcap_{j=1}^\infty E_j$ and $F \coloneqq \bigcap_{j=1}^\infty F_j$, noting that both are non-empty and compact subsets of $\mathbb{R}$. 
For all $j \in \mathbb{N}$, let $e_j$ and $f_j$ be the lengths of each of the $2^j$ intervals in $E_j$ and $F_j$ respectively, noting that for each $n \in \mathbb{N}$, the two different definitions that we have given for $e_{10^{k_n}}$ and $f_{10^{k_n}}$ agree by induction. The sequences $e_j$ and $f_j$ lie in $(0,1]$ by induction and converge monotonically to 0. 

We now find an upper bound for $\underline{\dim}_\mathrm{B} E$. Let $n \in \mathbb{N}$ be even. Then $E_{10^{k_n+1}}$ is made up of $2^{10^{k_n+1}}$ intervals, each of length $e_{10^{k_n+1}} = (1/5)^{10^{k_n+1}-10^{k_n}}e_{10^{k_n}} \leq (1/5)^{10^{k_n+1}-10^{k_n}}$. Covering $E$ with these intervals, we see that for all $n \in \mathbb{N}$,  
\begin{equation*} 
\frac{\log N_{e_{10^{k_n+1}}} F (E)}{-\log (e_{10^{k_n+1}})} \leq \frac{\log 2^{10^{k_n+1}}}{\log 5^{10^{k_n+1}-10^{k_n}}} = \frac{10\log 2}{9\log 5}.
 \end{equation*}
Therefore $\underline{\dim}_\mathrm{B} E \leq \frac{10\log 2}{9\log 5}$, and similarly using $F_{10^{k_n+1}}$ for $n$ odd to cover $F$ gives $\underline{\dim}_\mathrm{B} F \leq \frac{10\log 2}{9\log 5}$. Therefore 
\begin{equation}\label{stabilityboxbound}
\frac{10\log 2}{9\log 5} \geq \max\{\underline{\dim}_\mathrm{B} E, \underline{\dim}_\mathrm{B} F\}. 
\end{equation}

To bound $\underline{\dim}^\Phi (E \cup F)$ from below, we use the mass distribution principle. 
Define the sequence $(r_n)_{n \geq 0}$ by 
\[r_n \coloneqq \begin{cases} e_{10^{k_n+2}} = (1/5)^{10^{k_n + 1} - 10^{k_n}}(1/3)^{10^{k_n+2}-10^{k_n+1}}e_{10^{k_n}} & \mbox{if } n \mbox{ even,}\\
f_{10^{k_n+2}} = (1/5)^{10^{k_n + 1} - 10^{k_n}}(1/3)^{10^{k_n+2}-10^{k_n+1}}f_{10^{k_n}} & \mbox{if } n \mbox{ odd}.
\end{cases}\]
This sequence is strictly decreasing, because if $n \geq 0$ is even then by~\eqref{finitestableeqne}, 
\[ r_{n+1} = f_{10^{k_{n+1}+2}} < f_{10^{k_{n+1}}} < \Phi(e_{10^{k_n+2}}) \leq e_{10^{k_n+2}} = r_n,\]
and similarly if $n$ is odd then $r_{n+1} < r_n$ by~\eqref{finitestableeqnf}. 
Let $\delta \in (0,r_0)$. Define $n_\delta \in \mathbb{N}$ by $r_{n_\delta}\leq \delta < r_{n_\delta - 1}$. 

There are two cases depending on the parity of $n_\delta$. If $n_\delta$ is even, then let $\mu_\delta$ be any Borel probability measure on $F$ which gives mass $2^{-10^{k_{n_\delta + 1}}}$ to each of the $2^{10^{k_{n_\delta + 1}}}$ intervals in $F_{10^{k_{n_\delta + 1}}}$. Let $U$ be a Borel subset of $\mathbb{R}$ with $\Phi(\delta) \leq |U| \leq \delta$. Define $j \in \mathbb{N}$ (depending on $|U|$) by $f_j \leq |U| < f_{j-1}$. By~\eqref{finitestableeqne},  
\[f_{10^{k_{n_\delta+1}}} \leq \Phi(e_{10^{k_{n_\delta}+2}}) = \Phi(r_{n_\delta}) \leq \Phi(\delta) \leq |U| < f_{j-1}, \]
so $j-1 < 10^{k_{n_\delta+1}}$. Also, $f_j \leq |U| \leq \delta < r_{n_\delta - 1} = f_{10^{k_{n_\delta - 1}+2}}$, so in fact $10^{k_{n_\delta-1}+2} < j \leq 10^{k_{n_\delta+1}}$. 
Therefore by the construction of $F$,  
\begin{equation*}%
f_j \geq \left(\frac{1}{5}\right)^{10^{k_{n_\delta-1}+1}}\left(\frac{1}{3}\right)^{j-10^{k_{n_\delta-1}+1}} > \left(\frac{1}{5}\right)^{j/2}\left(\frac{1}{3}\right)^{j/2}.
\end{equation*}
  Since $U$ has diameter less than $f_{j-1}$, it can intersect at most two of the $2^{j-1}$ intervals in $F_{j-1}$. Therefore $U$ can intersect at most $2(2^{10^{k_{n_\delta + 1}} - j})$ of the $2^{10^{k_{n_\delta + 1}}}$ intervals in $F_{10^{k_{n_\delta + 1}}}$. 
  Therefore 
  \begin{equation*}
   \mu_\delta (U) \leq 2(2^{10^{k_{n_\delta + 1}} - j}) (2^{-10^{k_{n_\delta + 1}}}) = 2\left(\left(\frac{1}{3}\right)^{j/2}\left(\frac{1}{5}\right)^{j/2}\right)^{\frac{2\log 2}{\log 15}} \leq 2f_j^{\frac{2\log 2}{\log 15}} \leq 2|U|^{\frac{2\log 2}{\log 15}}. 
  \end{equation*}
  
  If, on the other hand, $n_\delta$ is odd, then let $\mu_\delta$ be a Borel probability measure on $E$ which gives mass $2^{-10^{k_{n_\delta + 1}}}$ to each of the $2^{10^{k_{n_\delta + 1}}}$ intervals in $E_{10^{k_{n_\delta + 1}}}$. As above, if $\Phi(\delta) \leq |U| \leq \delta$ then $\mu_\delta (U) \leq 2|U|^{\frac{2\log 2}{\log 15}}$. Therefore by the mass distribution principle Lemma~\ref{massdistprinc}~(ii) and Proposition~\ref{basicbounds} and~\eqref{stabilityboxbound},  
  \[\underline{\dim}^\Phi(E \cup F) \geq \frac{2\log 2}{\log 15} > \frac{10\log 2}{9\log 5} \geq \max\{\underline{\dim}_\mathrm{B} E, \underline{\dim}_\mathrm{B} F\} \geq \max\{\underline{\dim}^\Phi E, \underline{\dim}^\Phi F\}. \qedhere \]
\end{proof}
It follows from Propositions~\ref{finitestability} and~\ref{basicbounds} and the fact that the Hausdorff dimension is countably stable that for all admissible functions $\Phi_1$ and $\Phi_2$, the three notions of dimension $\dim_\mathrm{H}$, $\underline{\dim}^{\Phi_1}$ and $\overline{\dim}^{\Phi_2}$ are pairwise-distinct, even just working in $\mathbb{R}$. 

Letting $E,F$ be as in Proposition~\ref{finitestability}, applying the mass distribution principle as in the proof of that result at the scales $\delta \coloneqq f_{10^{k_n+2}}$ shows that 
\[ \dim_\mathrm{H} F \leq \underline{\dim}^\Phi F \leq \frac{10\log 2}{9\log 5} < \frac{2\log 2}{\log 15} \leq \overline{\dim}^\Phi F \leq \overline{\dim}_\mathrm{B} F,\]
 and similarly for $E$. 
 Suppose $F$ is the set corresponding to a $\Phi$ satisfying $\frac{\log \delta}{\log \Phi(\delta)} \to 0$ as $\delta \to 0^+$ (for example $\Phi(\delta) = e^{-\delta^{-0.5}}$). Then by Proposition~\ref{compareintermediate}, 
 \[ \dim_\mathrm{H} F < \frac{2\log 2}{\log 15} \leq \overline{\dim}^\Phi F \leq \overline{\dim}_\theta F\]
  for all $\theta \in (0,1]$, so $\overline{\dim}_\theta F$ is discontinuous at $\theta = 0$. Let $\Phi_1$ be an admissible function such that $\Phi_1(f_{10^{k_n+1}}) \leq \Phi_1(f_{10^{k_{n+2}+1}})$ for all sufficiently large $n$. Then for all sufficiently small $\delta$, there exists an odd integer $n(\delta)$ such that $\Phi(\delta) \leq f_{10^{k_{n(\delta)}+1}} \leq \delta$, and the natural cover of $F_{10^{k_{n(\delta)}+1}}$ with $2^{10^{k_{n(\delta)}+1}}$ intervals gives $\overline{\dim}^{\Phi_1} F \leq \frac{10\log 2}{9\log 5} < \frac{2\log 2}{\log 15}$. 
  This gives an indication of how one might construct the admissible functions from Theorem~\ref{recoverinterpolation} below which recover the interpolation for this particular set.

\section{Recovering the interpolation}\label{recoversection} 

It is clear from~\cite[Proposition~2.4]{Falconer2020firstintermediate} and the proof of Proposition~\ref{finitestability} that there are many compact sets with intermediate dimensions discontinuous at $\theta = 0$. For these sets the intermediate dimensions do not fully interpolate between the Hausdorff and box dimensions. 
The main result of this section, Theorem~\ref{recoverinterpolation}, shows that for every compact set there is indeed a family of functions $\Phi$ for which the $\Phi$-intermediate dimensions interpolate all the way between the Hausdorff and box dimensions of the set. %
Moreover, there exists a family of $\Phi$ which interpolates for both the upper and lower versions of the dimensions, and forms a chain in the partial order introduced in Section~\ref{ctysection}. 
Banaji, Rutar and Troscheit~\cite{BanajiPreprintphiassouad} prove that the Assouad-like dimensions studied in~\cite{Garcia2019-2} fully interpolate between the quasi-Assouad and Assouad dimensions of all non-empty, bounded, doubling metric spaces. 

\begin{thm}\label{recoverinterpolation}
For each non-empty, compact subset $F$, there exists a family 
\[ \{\Phi_s\}_{s \in [\dim_\mathrm{H} F, \overline{\dim}_\mathrm{B} F]}\]
 of admissible functions such that if $s,t$ are such that $\dim_\mathrm{H} F \leq s \leq t \leq \overline{\dim}_\mathrm{B} F$ then the following three conditions hold: 
  \begin{enumerate}[label=(\roman*)]
\item $\overline{\dim}^{\Phi_s} F = s$;
\item $\underline{\dim}^{\Phi_s} F = \min\{s,\underline{\dim}_\mathrm{B} F\}$;
\item $\Phi_s \preceq \Phi_t$. 
\end{enumerate}
\end{thm}
The key definition in the proof is~\eqref{definephis}. The assumption of compactness allows us to take a \emph{finite} subcover in Definition~\ref{hausdorffdef} of Hausdorff dimension, which ensures that $\Phi_s(\delta)$ is well-defined and positive.   
  
\begin{proof} 
Define $\Phi_{\overline{\dim}_\mathrm{B} F}(\delta) \coloneqq \frac{\delta}{-\log \delta}$, so~(i) and (ii) are satisfied for $s = \overline{\dim}_\mathrm{B} F$ by Proposition~\ref{whenequalsbox}. 
We henceforth assume that $\dim_\mathrm{H} F < \overline{\dim}_\mathrm{B} F$, or else there is nothing more to prove.  
The same symbols may take different values in the proofs of parts (i), (ii), (iii). 

(i) %
Let $\Delta \in (0,1/5)$ be such that $0<\frac{\delta}{-\log \delta} < c\delta /3$ for all $\delta \in (0,\Delta)$. %
For now, let $s \in (\dim_\mathrm{H} F, \overline{\dim}_\mathrm{B} F)$. %
For each $\delta \in (0,\Delta)$ there exists a countable cover $\{V_i\}_{i \geq 1}$ of $F$ such that $|V_i| \leq \delta$ for all $i$, and $\sum_i|V_i|^s \leq 2^{-1-2s}$. We may assume that each $V_i$ is non-empty and fix $p_i \in V_i$. Each $V_i \subseteq B(p_i,\max\{2|V_i|,2^{-1-2i/s})\})$, so $\{B(p_i,\max\{2|V_i|,2^{-1-2i/s}\})\}_{i \geq 1}$ is an open cover for $F$. %
Since $F$ is compact, there is a finite subset $\{U_i\} \subseteq \{B(p_i,\max\{2|V_i|,2^{-1-2i/s}\})\}$ which also covers $F$. %
Now, 
\begin{align*}
\sum_i |U_i|^s \leq \sum_{i \geq 1} |B(p_i,\max\{2|V_i|,2^{-1-2i/s}\})|^s &\leq \sum_{i=1}^\infty (2^{-2i/s})^s + \sum_{i \geq 1} (4|V_i|)^s \\
&= 1/3 + 4^s \sum_i |V_i|^s \\
&< 1.
\end{align*} 
Since $\{U_i\}$ is a finite collection of sets, and each has positive diameter as $X$ is uniformly perfect, it follows that $\min_i |U_i| >0$. Therefore $\Phi_s \colon (0,\Delta) \to \mathbb{R}$ is positive and well-defined by  
\begin{equation}\label{definephis}
\begin{aligned}
\Phi_s(\delta) \coloneqq \sup \{ \,x \in [0,\delta/(-\log \delta)] : &\mbox{ there exists a finite cover } \{U_i\} \mbox{ of } F \\*
 & \mbox{ such that } x \leq |U_i| \leq \delta \mbox{ for all } i \mbox{ and } \sum_i|U_i|^s \leq 1 \, \}.
\end{aligned}
\end{equation}
 By construction, $\Phi_s(\delta)/\delta \leq \left(\frac{\delta}{-\log \delta}\right)/\delta \to 0$ as $\delta \to 0^+$, and $\Phi_s$ is increasing in $\delta$, so $\Phi_s$ is admissible. 
  
 We now show that $\overline{\dim}^{\Phi_s} F \leq s$. Given $\eta,\epsilon >0$, define $\delta_0 \coloneqq \min\{ \epsilon^{1/\eta}c^{s/\eta}4^{-s/\eta}, \Delta\}$. Then for all $\delta \in (0, \delta_0)$ there exists a finite cover $\{W_i\}$ of $F$ satisfying $\Phi_s(\delta)/2 \leq |W_i| \leq \delta$ for all $i$, and %
 $\sum_i|W_i|^s \leq 1$. If $|W_i| \geq \Phi_s(\delta)$ then leave $W_i$ in the cover unchanged. If $|W_i| < \Phi_s(\delta)$ then pick any $w_i \in W_i$ and $q_i \in X$ such that $\Phi_s(\delta) \leq d(q_i,w_i) \leq \Phi_s(\delta)/c$. Replace $W_i$ in the cover by $W_i \cup \{q_i\}$. Call the new cover $\{Y_i\}$. By the triangle inequality, 
 \begin{equation*} 
 \Phi_s(\delta) \leq d(q_i,w_i) \leq |W_i \cup \{q_i\}| < \Phi_s(\delta) + \Phi_s(\delta)/c \leq 2\delta/(-c\log \delta)\leq \delta.
 \end{equation*}
 Also 
 \[|W_i \cup \{q_i\}| \leq 2\Phi_s(\delta)/c \leq (4/c)\Phi_s(\delta)/2 \leq 4|W_i|/c. \]
 Therefore 
 \[\sum_i|Y_i|^{s+\eta} \leq \sum_i|Y_i|^s \delta^\eta \leq \delta_0^\eta (4/c)^s \sum_i|W_i|^s \leq \epsilon.\]
  It follows that $\overline{\dim}^{\Phi_s} F \leq s+\eta$, so in fact $\overline{\dim}^{\Phi_s} F \leq s$.
 
 To prove the reverse inequality, assume for a contradiction that $\overline{\dim}^{\Phi_s} F < s$. Then there exists $\delta_1 \in (0,\Delta)$ such that for all $\delta_2 \in (0,\delta_1)$ there exists a cover $\{Z_i\}$ of $F$ such that $\Phi_s(\delta_2) \leq |Z_i| \leq \delta_2$ for all $i$, and $\sum_i |Z_i|^s \leq 3^{-s}c^s$. 
 By Proposition~\ref{whenequalsbox} there exists $\delta_2 \in (0,\delta_1)$ such that $\Phi_s(\delta_2) < \delta_2/(-\log \delta_2)$, and let $\{Z_i\}$ be the cover corresponding to this $\delta_2$, as above. Choose any $z_i \in Z_i$ and let $x_i \in X$ be such that $2|Z_i| \leq d(z_i,x_i) \leq 2|Z_i|/c$. Then by the triangle inequality,  
 \begin{equation*}
 2\Phi_s(\delta_2) \leq 2|Z_i| \leq d(z_i,x_i) \leq |Z_i \cup \{x_i\}| \leq |Z_i| + 2|Z_i|/c \leq (3/c)\delta_2/(-\log \delta_2) <\delta_2. 
 \end{equation*}
 Moreover, $\{Z_i \cup \{x_i\}\}_i$ covers $F$, and 
 \[\sum_i|Z_i \cup \{x_i\}|^s \leq \sum_i (3|Z_i|/c)^s \leq 3^s c^{-s} \sum_i |Z_i|^s \leq 1.\] 
 Therefore 
 \[\Phi_s(\delta_2) \geq \min\{2\Phi_s(\delta_2),\delta_2/(-\log \delta_2)\} > \Phi_s(\delta_2),\]
  a contradiction. Hence $\overline{\dim}^{\Phi_s} F \geq s$ for all $s \in (\dim_\mathrm{H} F, \overline{\dim}_\mathrm{B} F)$, so $\overline{\dim}^{\Phi_s} F = s$. 
 
 Now consider the case $s=\dim_\mathrm{H} F$. Let $N \in \mathbb{N}$ satisfy 
 \[ N > \max\left\{\frac{1}{\overline{\dim}_\mathrm{B} F - \dim_\mathrm{H} F},\frac{1}{\Delta}\right\}. \]
 For $\delta \in (0,1/N]$, let $n \geq N$ be such that $\delta \in (\frac{1}{n+1}, \frac{1}{n}]$, and define 
 \[ \Phi_s(\delta) \coloneqq \min\{\Phi_{s+1/N} (\delta), \dotsc, \Phi_{s+1/n}(\delta)\}.\] 
  Then $\Phi_s(\delta) \leq \Phi_{s+1/N}(\delta) \leq \delta/(-\log \delta)$ for all $\delta \in (0,1/N]$, so $\Phi_s(\delta)/\delta \to 0$ as $\delta \to 0^+$. For all $n \geq N$ and $\delta \in (0,\Delta)$ it holds that $\Phi_{s+1/n}(\delta)>0$, so if $\delta>0$ then $\Phi_s(\delta)>0$. Moreover, if $\delta_1 \leq \delta_2$, say 
 $\delta_1 \in (\frac{1}{n+1}, \frac{1}{n}]$ and $\delta_2 \in (\frac{1}{m+1}, \frac{1}{m}]$ where $n\geq m \geq N$, then
 \begin{equation*}
 \Phi_s(\delta_1) \leq \min\{\Phi_{s+1/N} (\delta_1), \dotsc, \Phi_{s+1/m}(\delta_1)\} \leq \Phi_s(\delta_2)
 \end{equation*}
  by the monotonicity of each $\Phi_{s+1/i}$. Thus $\Phi_s$ is monotonic, so admissible. For all $n \geq N$ and $\delta \in (0,1/n)$, clearly $\Phi_s(\delta) \leq \Phi_{s+1/n}(\delta)$. 
  Therefore by Proposition~\ref{basicbounds} and Corollary~\ref{hardcomparisoncor}~(i),  
  \[ s = \dim_\mathrm{H} F \leq \underline{\dim}^{\Phi_s} F \leq \overline{\dim}^{\Phi_s} F \leq \overline{\dim}^{\Phi_{s+1/n}} F = s+\frac{1}{n}.\]
   Letting $n \to \infty$ gives $\underline{\dim}^{\Phi_s} F = \overline{\dim}^{\Phi_s} F = s = \dim_\mathrm{H} F$, as required. 
   
   (iii) %
   By construction, (iii) holds since if $\dim_\mathrm{H} F \leq s \leq t \leq \overline{\dim}_\mathrm{B} F$ then $\Phi_s(\delta) \leq \Phi_t(\delta)$ for all sufficiently small $\delta$. 
   
    (ii) It suffices to prove $\underline{\dim}^{\Phi_s} F \geq \min\{s,\underline{\dim}_\mathrm{B} F\}$, since the opposite inequality follows from Proposition~\ref{basicbounds} and~(i). 
    If $s = \dim_\mathrm{H} F$ or $s = \overline{\dim}_\mathrm{B} F$ then we are done by Propositions~\ref{basicbounds} and~\ref{whenequalsbox}.  
    Suppose $s \in (\dim_\mathrm{H} F, \underline{\dim}_\mathrm{B} F] \cap (\dim_\mathrm{H} F, \overline{\dim}_\mathrm{B} F)$. %
     Assume for a contradiction that $\underline{\dim}^{\Phi_s} F < s$. Let $t,t'$ be such that $\underline{\dim}^{\Phi_s} F < t < t' < s$. Since $t' < \underline{\dim}_\mathrm{B} F$, there exists $\Delta \in (0,\min\{1,|X|\})$ such that $N_\delta (F) \geq \delta^{-t'}$ for all $\delta \in (0,\Delta)$. Reducing $\Delta$ if necessary, we may assume further that $\frac{\delta^{t-t'}}{(-\log \delta)^t} > (1+2/c)^{-s}$ and $-\log \delta \geq 2(1+2/c)$ for all $\delta \in (0,\Delta)$. 
    Since $t > \underline{\dim}^{\Phi_s} F$, for all $\delta_0 > 0$ there exists $\delta \in (0,\min\{\Delta,\delta_0\})$ and a cover $\{U_i\}$ such that $\Phi_s(\delta) \leq |U_i| \leq \delta$ for all $i$, and 
    \begin{equation}\label{longlowersum}
    (1+2/c)^{-s} \geq \sum_i |U_i|^t \geq \sum_i |U_i|^s.
    \end{equation}
    
     But   
     \[ (1+2/c)^{-s} < \frac{\delta^{t-t'}}{(-\log \delta)^t} = \delta^{-t'}\left(\frac{\delta}{-\log \delta}\right)^t,\]
      so there exists $i$ such that $\delta/(-\log \delta) > |U_i| \geq \Phi_s(\delta)$. %
      If $i$ is such that $|U_i| \geq \min\{2\Phi_s(\delta),\delta/(-\log \delta)\}$ then leave $U_i$ in the cover unchanged. If, however, $i$ is such that $|U_i| <\min\{2\Phi_s(\delta),\delta/(-\log \delta)\}$ then fix $p_i \in U_i$. Fix $q_i \in X$ such that $2\Phi_s(\delta) \leq d(p,q) \leq 2\Phi_s(\delta)/c$, replace $U_i$ in the cover by $U_i \cup \{q_i\}$, and call the new cover $\{V_i\}_i$. In the case $|U_i| <\min\{2\Phi_s(\delta),\delta/(-\log \delta)\}$, 
    \begin{align*}
     2\Phi_s(\delta) \leq d(p_i,q_i) \leq |U_i \cup \{q_i\}| \leq |U_i| + 2\Phi_s(\delta)/c 
     &< 2(1+2/c)\Phi_s(\delta) \\
     &\leq 2(1+2/c)\delta/(-\log \delta) \\
     &\leq \delta.
     \end{align*}
    Then $\min\{\delta/(-\log \delta),2\Phi_s(\delta)\} \leq |V_i| \leq \delta$ for each $i$, and 
    \[ \sum_i |V_i|^s \leq \sum_i ((1+2/c)|U_i|)^s = (1+2/c)^s \sum_i|U_i|^s \leq 1, \]
    by~\eqref{longlowersum}. 
    This means that $\Phi_s(\delta) \geq \min\{2\Phi_s(\delta), \delta/(-\log \delta)\} > \Phi_s(\delta)$, a contradiction. Hence $\underline{\dim}^{\Phi_s} F \geq s$ for all $s \in (\dim_\mathrm{H} F, \underline{\dim}_\mathrm{B} F]$. %
    
   Now suppose $s \in (\underline{\dim}_\mathrm{B} F, \overline{\dim}_\mathrm{B} F)$. %
   By (iii), $\Phi_{\underline{\dim}_\mathrm{B} F} \preceq \Phi_s$, so by what we have just proved, $\min\{s,\underline{\dim}_\mathrm{B} F\} = \underline{\dim}_\mathrm{B} F \leq \underline{\dim}^{\Phi_{\underline{\dim}_\mathrm{B} F}} F \leq \underline{\dim}^{\Phi_s} F$. 
        Together, the cases show that for all $s \in [\dim_\mathrm{H} F,\overline{\dim}_\mathrm{B} F]$ it holds that $\underline{\dim}^{\Phi_s} F \geq \min\{s,\underline{\dim}_\mathrm{B} F\}$ and hence $\underline{\dim}^{\Phi_s} F = \min\{s,\underline{\dim}_\mathrm{B} F\}$, as required.  
\end{proof}

In the definition~\eqref{definephis} of $\Phi_s$, any positive constant would work in place of the constant 1, so there are many different $\Phi_s$ that will work. The family of dimensions $\overline{\dim}^{\Phi_s}$ and $\underline{\dim}^{\Phi_s}$ may not vary continuously for all sets, as shown by the following proposition. 

\begin{prop}\label{interpolatenotcts}
There exist non-empty, compact subsets $F,G$ of $\mathbb{R}$ such that: 

(i) if $(\Phi_s)_{s \in (\dim_\mathrm{H} F, \underline{\dim}_\mathrm{B} F)}$ is a family of admissible functions such that $\overline{\dim}^{\Phi_s} F = s$ for all $s \in (\dim_\mathrm{H} F, \underline{\dim}_\mathrm{B} F)$ then the function $s \mapsto \overline{\dim}^{\Phi_s} G$ is not continuous on $(\dim_\mathrm{H} F, \underline{\dim}_\mathrm{B} F)$, and

(ii) if $(\Psi_s)_{s \in (\dim_\mathrm{H} F, \underline{\dim}_\mathrm{B} F)}$ is such that $\underline{\dim}^{\Psi_s} F = s$ for all $s \in (\dim_\mathrm{H} F, \underline{\dim}_\mathrm{B} F)$ then the function $s \mapsto \underline{\dim}^{\Psi_s} G$ is not continuous on $(\dim_\mathrm{H} F, \underline{\dim}_\mathrm{B} F)$. 
\end{prop}

\begin{proof}

Let $G \coloneqq \{0\} \cup \{1/n : n \in \mathbb{N}\}$, so $\dim_\theta G = \frac{\theta}{1+\theta}$ for all $\theta \in [0,1]$ by~\cite[Proposition~3.1]{Falconer2020firstintermediate}. 
Let $E \subset \mathbb{R}$ be a compact countable set with $\underline{\dim}_\mathrm{B} E = \dim_\mathrm{A} E = 1/4$ and let $F = E \cup G$, so as in~\cite[Example~3]{Falconer2020firstintermediate} $\dim_\mathrm{H} F = 0$ and $\dim_\theta F = \max\left\{ \frac{\theta}{1+\theta},1/4\right\}$ for all $\theta \in (0,1]$. 
We now prove (i) using Proposition~\ref{compareintermediate}; the proof of (ii) is similar. Suppose $(\Phi_s)_{s \in (\dim_\mathrm{H} F, \underline{\dim}_\mathrm{B} F)}$ satisfies $\overline{\dim}^{\Phi_s} F = s$ for all $s \in (\dim_\mathrm{H} F, \underline{\dim}_\mathrm{B} F)$. Then if $s>1/4$ then $\overline{\dim}^{\Phi_s} F = s > 1/4 = \dim_{1/3} F$, so by Proposition~\ref{compareintermediate}, 
\[ \limsup_{\delta \to 0^+} \frac{\log \Phi_s(\delta)}{\log \delta} > 1/3\] and $\overline{\dim}^{\Phi_s} G \geq \dim_{1/3} G = 1/4$. For all $s < 1/4$, $\frac{\log \Phi_s(\delta)}{\log \delta} \to 0$ as $\delta \to 0^+$, so since $\dim_\theta G = \frac{\theta}{1+\theta} \to 0$ as $\theta \to 0$, it follows that $\dim^{\Phi_s} G = 0$. Therefore the function $s \mapsto \overline{\dim}^{\Phi_s} G$ is not continuous at $s=1/4$. 
\end{proof}

We believe that the results of this chapter demonstrate that the $\Phi$-intermediate dimensions give rise to a rich and workable theory in their own right, and there are several further questions about them that we will not explore in this thesis. 
In particular, it would be natural to calculate the $\Phi$-intermediate dimensions of two extreme types of cutout sets. Specifically, decreasing sequences with decreasing gaps such as $\{0\} \cup \{\,\frac{1}{\log n} : n \in \mathbb{N},n \geq 3 \,\}$ are in some sense the least spatially homogeneous in space of all cutout sets corresponding to a given sequence of lengths. On the other hand, homogeneous Moran sets are the most spatially homogeneous, and in this case the $\Phi$-intermediate dimensions will satisfy the natural analogue of Proposition~\ref{p:int-dim} from Chapter~\ref{s:attainable}, page~\pageref{p:int-dim}.

\chapter{Attainable forms of intermediate dimensions}\label{s:attainable}

\section{Introduction}
This chapter considers the general form of intermediate dimensions of sets, and is mostly based on our joint paper~\cite{Banaji2022moran} with A.~Rutar. 
The main result is to obtain a full characterisation of possible intermediate dimension functions for subsets of Euclidean space.
Recall that the upper Dini derivative of a function $f\colon\R\to\R$ at $x$ is given by
\begin{equation}\label{e:Dini-intro}
    \diniu{+}f(x)=\limsup_{\epsilon\to 0^+}\frac{f(x+\epsilon)-f(x)}{\epsilon}.
\end{equation}
We will prove that the following characterisation holds. 
\begin{theorem}\label{it:intro-main-res}
    Let $h\colon[0,1]\to[0,d]$ be any function.
    Then there exists a non-empty bounded set $F\subset\Rd$ with $\dim_\theta F=h(\theta)$ for all $\theta \in [0,1]$ if and only if $h$ is non-decreasing, is continuous on $(0,1]$, and satisfies
    \begin{equation}\label{e:intro-gen-bound}
        \diniu{+}h(\theta)\leq\frac{h(\theta)(d-h(\theta))}{d\theta}
    \end{equation}
    for all $\theta\in(0,1)$.
\end{theorem}
\begin{proof}
This follows immediately from Theorem~\ref{it:general-form} below. 
\end{proof}
We see that the intermediate dimensions can have highly varied behaviour; such behaviour has not been seen in any prior examples.
In particular, without stronger assumptions on the set $F$, very little can be said about the possible forms of the intermediate dimensions.
For example, it follows directly from~\eqref{e:intro-gen-bound} that if $f$ is any non-decreasing Lipschitz function on $[0,1]$, there exists some constants $a>0$, $b\in\R$, and a set $F\subset\R$ such that $\dim_\theta F=af(\theta)+b$ for all $\theta\in[0,1]$.
In particular, the following behaviours for the intermediate dimensions are all possible:
\begin{enumerate}%
    \item Constant on countably many disjoint closed intervals in $[0,1]$, and strictly increasing otherwise.
    \item Strictly concave, strictly convex, or linear and non-constant, on $[0,1]$.
    \item Non-differentiable at each point in a dense subset $E$ of $(0,1)$ with $\dimH E=1$ (in fact, the points of non-differentiability in $(0,1)$ can form an arbitrary $G_{\delta\sigma}$ subset of $(0,1)$ with Lebesgue measure zero~\cite{Zahorski1946lipdiff}).
\end{enumerate}
This resolves all remaining questions asked in Falconer's survey \cite{Falconer2021intdimsurvey}.

To state a stronger form of the main result of this chapter, we define the following class of functions. 
\begin{definition}\label{d:h-class}
    Let $0\leq\lambda\leq\alpha\leq d$.
    If $\lambda<\alpha$, we denote by $\mathcal{H}(\lambda,\alpha)$ the set of functions $h\colon[0,1]\to[\lambda,\alpha]$ which satisfy the following constraints:
    \begin{enumerate}%
        \item\label{i:hdefnondec} $h$ is non-decreasing,
        \item\label{i:hdefcty} $h$ is continuous on $(0,1]$, and
        \item\label{i:hdefdini} for each $\theta\in(0,1)$, 
            \begin{equation}\label{e:h-bound-general}
                \diniu{+}h(\theta)\leq\frac{(h(\theta)-\lambda)(\alpha-h(\theta))}{(\alpha-\lambda)\theta}.
            \end{equation}
    \end{enumerate}
    Otherwise, $\lambda=\alpha$ and we let $\mathcal{H}(\lambda,\alpha)$ be the set consisting only of the constant function $h(\theta)=\alpha$.
\end{definition}
We now state the main result of this chapter precisely.
\begin{theorem}\label{it:general-form}
    Suppose $F\subset\Rd$ has $\dimL F=\lambda$ and $\dimA F=\alpha$.
    Then if $\underline{h}$ and $\overline{h}$ denote the functions $\underline{h}(\theta) = \underline{\dim}_\theta F$ and $\overline{h}(\theta) = \overline{\dim}_\theta F$, it holds that $\underline{h},\overline{h}\in\mathcal{H}(\lambda,\alpha)$, $\underline{h}\leq\overline{h}$, and $\underline{h}(0)=\overline{h}(0)$. 
    
    Conversely, if $0\leq\lambda\leq\alpha\leq d$ and $\underline{h},\overline{h}\in\mathcal{H}(\lambda,\alpha)$ satisfy $\underline{h}\leq\overline{h}$ and $\underline{h}(0)=\overline{h}(0)$, then there exists a compact perfect set $F\subset\Rd$ such that $\dimL F=\lambda$, $\dimA F=\alpha$, $\underline{\dim}_\theta F=\underline{h}(\theta)$, and $\overline{\dim}_\theta F=\overline{h}(\theta)$ for all $\theta\in[0,1]$.
\end{theorem}
\begin{proof}
This follows from Corollaries~\ref{c:int-dim-bound} and~\ref{c:upper-lower-match}.
\end{proof}
This result gives a full characterisation of all possible forms of the upper and lower intermediate dimensions of a bounded set $F\subset\Rd$.
The constraint~\eqref{e:h-bound-general} generalises all previously known bounds~\cite{Falconer2021intdimsurvey,Falconer2020firstintermediate}.
We see that the Assouad and lower dimensions influence the possible forms of the intermediate dimensions in a natural way.
Note that~\eqref{e:h-bound-general} also provides quantitative information about the Assouad and lower dimensions in terms of the intermediate dimensions.
This is in contrast to the box and Hausdorff dimensions, which provide no more information about the Assouad and lower dimensions beyond the usual order constraints.
We can also view the bound in~\eqref{e:h-bound-general} as $(2\theta)^{-1}$ times the harmonic mean of $h(\theta)-\lambda$ and $\alpha-h(\theta)$.
In particular, if $0\leq\lambda'\leq\lambda\leq\alpha\leq\alpha'\leq d$, then $\mathcal{H}(\lambda',\alpha')\supseteq\mathcal{H}(\lambda,\alpha)$.
Of course, by taking $\underline{h}=\overline{h}$ we can also ensure that the intermediate dimensions exist.
Therefore, Theorem~\ref{it:intro-main-res} follows from Theorem~\ref{it:general-form}.

The proof of the bound~\eqref{e:h-bound-general} is given in Corollary~\ref{c:int-dim-bound}, using some similar ideas to the proof of Theorem~\ref{maincty}. 
In Sections~\ref{s:lattice} and~\ref{s:popcorn} we will use it to calculate the intermediate dimensions of some natural classes of sets, including certain lattice sets and the graph of the popcorn function; in particular, the bound~\eqref{e:intro-gen-bound} is attained at all $\theta\in(0,1)$ for the lattices. 
More remarkable than the bound itself, however, is the fact that it is essentially the only constraint that a function must satisfy in order to be realised as the intermediate dimensions of a set. 
In order to establish this converse result, our main strategy is to construct sets which we call \emph{homogeneous Moran sets}.
These sets are analogous to the $2^d$-corner Cantor sets in $\Rd$, except we only require the subdivision ratios to be equal within each stage in the construction, and not necessarily between stages. 
The following nice property was essentially observed in~\cite{Cabrelli1997cantor}: the optimal covers for a homogeneous Moran set can be taken to consist of sets with equal diameter.
This result is given in Lemma~\ref{l:flat-covers}.
A direct application of this result is a convenient formula for the upper intermediate dimensions of these sets, given in Proposition~\ref{p:int-dim}.
Using this formula, in Lemmas~\ref{l:exact-construction} and~\ref{l:Moran-formula}, we present a general strategy to construct homogeneous Moran sets with upper intermediate dimensions given by a certain infimum over a `sliding window' of a function $g$ satisfying certain derivative constraints.
Then for each $h(\theta)$ satisfying the general bounds, in Theorem~\ref{t:upper-h-form} we construct a function satisfying the derivative constraints so that the corresponding Moran set has upper intermediate dimensions given by the prescribed formula.
This establishes Theorem~\ref{it:general-form} for the upper intermediate dimensions.

Finally, in Theorem~\ref{t:gen-h-form}, we construct an inhomogeneous Moran set which, at a fixed scale, looks like a finite union of homogeneous Moran sets each with prescribed upper intermediate dimension $h(\theta)$.
This process is done in such a way to ensure that the intermediate dimensions exist.
Then, taking a disjoint union of this set with the set provided in Theorem~\ref{t:upper-h-form}, we complete the proof of Theorem~\ref{it:general-form}. 
The details are provided in Corollary~\ref{c:upper-lower-match}.
Heuristically, the covering strategy for Corollary~\ref{c:int-dim-bound} will be sharp when the relative covering numbers in the Assouad and lower dimensions are realised uniformly on the entire set for a fixed scale.
In some sense, this motivates the choice of homogeneous Moran sets, which have the maximum possible uniformity at a fixed scale.
The key observation is that inhomogeneity between scales is sufficient to obtain all possible forms of the upper intermediate dimensions. 
In order to prove that the lower and upper intermediate dimensions can be prescribed simultaneously in Section~\ref{ss:pres}, we use a set which behaves like a union of homogeneous Moran sets at each fixed scale, but whose resolution increases as the scale reduces. 

Rutar~\cite{Rutar2022assouad} has used ideas about homogeneous Moran sets from the paper~\cite{Banaji2022moran} on which this chapter is based to obtain a full characterisation for attainable forms of Assouad spectra, building on previous results in~\cite{Fraser2019twospectra}. 

\section{General bounds}\label{s:genbounds}
\subsection{Dini derivatives}
We begin with some standard results on Dini derivatives, which will be useful later in the chapter.
We refer the reader to \cite{Bruckner1994diff} for more details.
\begin{definition}
    Let $g\colon\R\to\R$ be a function.
    Then the \emph{upper right Dini derivative} is given by
    \begin{equation*}
        \diniu{+}g(x)=\limsup_{\epsilon\to 0^+}\frac{g(x+\epsilon)-g(x)}{\epsilon}.
    \end{equation*}
    The lower right Dini derivative is defined with $\liminf_{\epsilon \to 0^+}$ and denoted $\dinil{+}g$, and the left Dini derivatives are defined analogously using $\epsilon \to 0^-$ and denoted $\diniu{-}g$ and $\dinil{-}g$. 
\end{definition}
We first make the following observation.
\begin{lemma}\label{l:c1-bound}
    Let $f$ and $g$ be continuous functions on $[a,b]$ with $\dinil{+}g\leq \diniu{+}f$ and $g(a)=f(a)$.
    Then $g\leq f$.
\end{lemma}
\begin{proof}
    Observe that $\diniu{+}(f-g)=\diniu{+}f-\dinil{+}g\geq 0$ so by \cite[Corollary~11.4.2]{Bruckner1994diff}, $f-g$ is non-decreasing.
    But $(f-g)(a)=0$ so $g\leq f$.
\end{proof}
As an application, we obtain the following analogue of the mean value theorem.
\begin{corollary}\label{c:dmvt}
    Let $g$ be a continuous function on $[a,b]$ and set $s=\frac{g(b)-g(a)}{b-a}$.
    Then for all $\phi\in\{\diniu{+}g,\dinil{+}g,\diniu{-}g,\dinil{-}g\}$,
    \begin{enumerate}%
        \item there exists $x\in[a,b]$ such that $\phi(x)\leq s$, and
        \item there exists $x\in[a,b]$ such that $\phi(x)\geq s$.
    \end{enumerate}
\end{corollary}
\begin{proof}
    We prove that there is some $x$ such that $\dinil{+}g(x)\geq s$; the other cases are similar.
    Without loss of generality, there is some $x_0\in(a,b)$ such that $g(x_0)> g(a)+s(x_0-a)$.
    Suppose for a contradiction $\dinil{+}g(x)\leq s$ for all $x\in[a,x_0]$.
    By Lemma~\ref{l:c1-bound}, $g(x)\leq s(x-a)+g(a)$ for all $x\in[a,x_0]$, contradicting the choice of $x_0$.
\end{proof}
It follows from Corollary~\ref{c:dmvt} that in~\eqref{e:h-bound-general} one could take instead the lower Dini derivative and the class of functions would remain unchanged. 
We now have the following elementary result.
\begin{lemma}\label{l:g-check}
    Let $0\leq\lambda<\alpha\leq d$, let $g\colon\R \to(\lambda,\alpha)$ be continuous, and let $U\subset\R$ be an open set.
    Then the following are equivalent:
    \begin{enumerate}%
        \item $\diniu{+}g(x)\in[\lambda-g(x),\alpha-g(x)]$ for all $x\in U$.
        \item $\dinil{+}g(x)\in[\lambda-g(x),\alpha-g(x)]$ for all $x\in U$.
        \item $\diniu{-}g(x)\in[\lambda-g(x),\alpha-g(x)]$ for all $x\in U$.
        \item $\dinil{-}g(x)\in[\lambda-g(x),\alpha-g(x)]$ for all $x\in U$.
    \end{enumerate}
\end{lemma}
\begin{proof}
    We will see that $\dinil{+}g(x)\in[\lambda-g(x),\alpha-g(x)]$ for all $x\in U$ implies that $\diniu{+}g(x)\in[\lambda-g(x),\alpha-g(x)]$ for all $x\in U$; the remaining implications are similar.
    Suppose for a contradiction there is some $x_0\in U$ such that $\diniu{+}g(x_0)\notin[\lambda-g(x_0),\alpha-g(x_0)]$.
    If $\diniu{+}g(x_0)<\lambda-g(x_0)$ then there is an immediate contradiction, so we assume $\diniu{+}g(x_0)>\alpha-g(x_0)$.
    Then there is some $\epsilon>0$ and $x_1$ such that
    \begin{equation*}
        \frac{g(x_1)-g(x_0)}{x_1-x_0}\geq \alpha-g(x_0)+\epsilon,
    \end{equation*}
    $[x_0,x_1]\subset U$, and $|g(y)-g(x_0)|<\epsilon/2$ for all $y\in[x_0,x_1]$.
    Then by Corollary~\ref{c:dmvt}, there is some $y\in[x_0,x_1]$ such that
    \begin{equation*}
        \dinil{+}g(y)\geq \alpha-g(x_0)+\epsilon>\alpha-g(y)+\frac{\epsilon}{2}
    \end{equation*}
    a contradiction.
\end{proof}
\subsection{Bounding the intermediate dimensions}\label{s:generalbounds}

In this section, we prove general bounds for the intermediate dimensions which improve existing bounds in the literature. 
 Theorem~\ref{intermediatects} is a quantitative continuity result for the intermediate dimensions which improves~\cite[Proposition~2.1]{Falconer2020firstintermediate} and~\cite[(14.2.2)]{Falconer2021intdimsurvey}. The proof of~\cite[Proposition~2.1]{Falconer2020firstintermediate} involves breaking up the largest sets in the cover, while~\cite[(14.2.2)]{Falconer2021intdimsurvey} is proved by `fattening' the smallest sets in the cover. The novelty in the proof of Theorem~\ref{maincty} (from which Theorem~\ref{intermediatects} follows) is to deal with the smallest and largest sets at the same time in such a way that the `cost' of each (in terms of how much the dimension can increase) is the same. %
 
 \begin{thm}\label{intermediatects}
 Suppose $F$ is a non-empty, totally bounded subset of a uniformly perfect metric space with more than one point. Write $\lambda =  \dim_\mathrm{L} F$, $\alpha = \asd F$, and let $\underline{h}$ and $\overline{h}$ denote the functions $\underline{h}(\varphi) = \underline{\dim}_\varphi F$ and $\overline{h}(\varphi) = \overline{\dim}_\varphi F$. Assume that $0 \leq \lambda < \alpha < \infty$ and let $h \in \{\underline{h},\overline{h}\}$ and $0 < \theta \leq \phi \leq 1$. Then 
 \begin{equation}\label{e:thetaphibound}
   h(\theta)  \leq h(\phi)  \leq h(\theta) + \frac{(h(\theta) - \lambda) (\alpha - h(\theta))}{\phi (h(\theta)- \lambda) + \theta( \alpha -h(\theta))}(\phi-\theta). 
   \end{equation} 
Furthermore, $h$ is continuous on $(0,1]$, Lipschitz on $[\theta,1]$ with Lipschitz constant $\frac{\alpha - \lambda}{4\theta}$, and differentiable Lebesgue-almost everywhere on $(0,1)$. 
\end{thm}

\begin{proof}
We prove the version for $\overline{\dim}$; the version for $\underline{\dim}$ is similar. 
The inequality $h(\theta) \leq h(\phi)$ is immediate from the definitions. 
The only non-trivial case of the other inequality is when $0 < \theta < \phi \leq 1$ and $0 < h(\theta) < \alpha$. Define $\Phi(\delta) \coloneqq \delta^{1/\theta}$. If $\phi<1$, define $\Phi_1(\delta) \coloneqq \delta^{1/\phi}$, but if $\phi = 1$ then define $\Phi_1(\delta) \coloneqq \delta/(-\log \delta)$. Then $\overline{\dim}^\Phi F = h(\theta)$ and $\overline{\dim}^{\Phi_1} F = h(\phi)$. Define 
\[\eta \coloneqq \frac{(h(\theta) - \lambda) (\alpha - h(\theta))}{\phi (h(\theta) - \lambda) + \theta( \alpha - h(\theta))}(\phi-\theta). \] 
Write 
\begin{equation*}
\gamma \coloneqq \frac{h(\theta) - \lambda}{h(\theta) + \eta - \lambda}; \qquad \beta \coloneqq \frac{\alpha - h(\theta)}{\alpha - h(\theta) - \eta}. 
\end{equation*}
A direct manipulation now shows that $\beta/\phi = \gamma/\theta$. 
Therefore 
\[ \Phi_1(\delta^\beta) = \delta^{\beta/\phi} = \delta^{\gamma/\theta} = (\Phi(\delta))^{\gamma}.\] %
Thus $\overline{\dim}^{\Phi_1} F \leq \overline{\dim}^\Phi F + \eta$ by Theorem~\ref{maincty}, as required. %

If $0 < \theta \leq \theta' \leq \phi \leq 1$ then 
\begin{equation*}
h(\phi) - h(\theta') 
\leq 
\frac{((\alpha - \lambda)/2)^2}{(\alpha - \lambda)\theta} (\phi - \theta')
\leq \frac{\alpha - \lambda}{4\theta}(\phi - \theta'),
\end{equation*}
proving Lipschitz continuity on $[\theta,1]$. 
Differentiability almost everywhere now follows from Lipschitz continuity and Rademacher's theorem, or alternatively from monotonicity and Lebesgue's theorem. 
\end{proof}

For the convenience of the reader, we now give an alternative direct proof of~\eqref{e:thetaphibound} for subsets of $\Rd$ and $0 < \theta < \phi < 1$ that does not rely on Theorem~\ref{maincty} but instead uses similar ideas to those used in the proof of Theorem~\ref{maincty}. 
 More specifically, we will bound $\overline{\dim}_{\theta+\epsilon}F$ in terms of $\theta$, $\overline{\dim}_\theta F$, and the Assouad and lower dimensions of $F$.
Given an optimal cover for $[\delta^{1/\theta},\delta]$, we want to convert this to a cover for the smaller range of scales $[\delta^{\beta/(\theta+\epsilon)},\delta^\beta]\subset[\delta^{1/\theta},\delta]$.
We then use the Assouad dimension to replace the sets with large diameter with sets with smaller diameter (corresponding to the indices in $I_3$), and the lower dimension to cover the sets with small diameter (corresponding to the indices in $I_1$).
The remaining elements of the cover remain essentially the same.
The parameter $\beta$ is chosen to optimise this process. 

Recall the definition of the lower dimension of a measure from~\eqref{e:lowerdimmeas} on page~\pageref{e:lowerdimmeas}. 
We also recall that $\mathcal{H}(\lambda,\alpha)$ is defined in Definition~\ref{d:h-class}.

\begin{proof}[Proof of~\eqref{e:thetaphibound}]
Assume $F \subset \Rd$ is non-empty and bounded with $\lambda < \alpha$, and let $0 < \theta < \phi < 1$. 
We prove~\eqref{e:thetaphibound} for the upper intermediate dimensions; the proof for the lower version is similar. 
    Let $\epsilon \coloneqq \phi - \theta$ and let $\eta,\beta$ be the unique solutions to the equations
    \begin{align*}
        \alpha - &h(\theta) - \beta (\alpha - h(\theta) - \eta) = 0 & \frac{\beta}{\theta + \epsilon}&(h(\theta) + \eta - \lambda) + \frac{\lambda - h(\theta)}{\theta} = 0.
    \end{align*}
    One can verify that $\eta$ and $\beta$ are given by
    \begin{align*}
        \eta &= \frac{( h(\theta) - \lambda) (\alpha -h(\theta))\epsilon}{(h(\theta) - \lambda)\epsilon + (\alpha - \lambda)\theta} & \beta &= \frac{(h(\theta) - \lambda)\epsilon}{(\alpha - \lambda)\theta}+1.
    \end{align*}
    Now for $s>h(\theta)$, let $s' \in (h(\theta),s)$, $\alpha' > \alpha$ and $\lambda' < \lambda$ satisfy
    \begin{align*}
        \alpha' - &s - \beta (\alpha' - s - \eta) > 0 & \frac{\beta}{\theta + \epsilon}&(s + \eta - \lambda') + \frac{\lambda' - s'}{\theta} > 0.
    \end{align*}
    For all sufficiently small $\delta \in (0,1)$ there exists a $(\delta,\theta)$-cover $\{U_i\}_{i \in I}$ of $F$ whose $s'$-cost is less than 1.
    Define
    \begin{align*}
        I_1 &= \{ \, i \in I : |U_i| < \delta^{\frac{\beta}{\theta + \epsilon}} \, \}\\*
        I_2 &= \{ \, i \in I :  \delta^{\frac{\beta}{\theta + \epsilon}} \leq  |U_i| \leq \delta^\beta / 2 \, \} \\*
        I_3 &= \{ \, i \in I :   |U_i| > \delta^\beta / 2 \, \}.
    \end{align*}
    
    There exists $C > 0$ such that for all $0 < r \leq 2R$, every set of diameter $R$ contained in $F$ can be covered by $\lfloor C(R/r)^{\alpha'} \rfloor$ balls of diameter $r$.
    Then for $k \in I_3$, let
    \begin{equation*}
        B_{k,1}, \dotsc, B_{k,\lfloor C((2|U_k|)/\delta^{\beta})^{\alpha'} \rfloor}
    \end{equation*}
    satisfy 
    \begin{equation*}
        |B_{k,i}|=\delta^\beta\quad\text{and}\quad \mathcal{S}_{\delta^{\beta/(\theta+\epsilon)}}(U_k)\cap F\subset\bigcup_{i=1}^{\lfloor C((2|U_k|)/\delta^{\beta})^{\alpha'} \rfloor}B_{k,i},
    \end{equation*}
    recalling that $\mathcal{S}_r$ denotes the $r$-neighbourhood. 
    Let $z_1,\dotsc,z_K$ be a maximal $4\delta^{\beta/(\theta + \epsilon)}$-separated subset of
    \begin{equation*}
        F \setminus \left( \bigcup_{i \in I_2 \cup I_3} \mathcal{S}_{\delta^{\beta/(\theta+\epsilon)}}(U_i) \right).
    \end{equation*}
    Set
    \begin{align*}
        \mathcal{U}_1 &\coloneqq \{ \, B(z_m, 5\delta^{\frac{\beta}{\theta + \epsilon}}) : 1 \leq m \leq K \, \}\\*
        \mathcal{U}_2 &\coloneqq \{ \, \mathcal{S}_{\delta^{\beta/(\theta+\epsilon)}}(U_j) : j \in I_2 \, \}\\*
        \mathcal{U}_3 &\coloneqq \bigcup_{k\in I_3}\bigl\{ \, B_{k,\ell}:\ell=1,\dotsc,\lfloor C((2|U_k|)/\delta^{\beta})^{\alpha'} \rfloor \, \bigr\}.
    \end{align*}
    Then for sufficiently small $\delta$,
    \begin{equation}\label{e:general-cover}
        \mathcal{U}\coloneqq\mathcal{U}_1\cup\mathcal{U}_2\cup\mathcal{U}_3
    \end{equation}
    is a $(\delta^\beta , \theta + \epsilon)$-cover of $F$. 
    
    We now bound the $(s+\eta)$-cost of $\mathcal{U}$ independently of $\delta$.
    First consider $\mathcal{U}_1$.
    By~\eqref{e:lowerdimexistmeas}, there exists a doubling Borel probability measure $\mu$ with $\supp\mu=\overline{F}$ and $\dimL \mu  \in (\lambda',\lambda]$.
    Let $M$ be a doubling constant for $\mu$.
    In particular, there is $c > 0$ such that if $0 < r < R \leq |F|$ and $x \in F$ then
    \begin{equation*}
        \frac{\mu(B(x,R))}{\mu(B(x,r))} \geq c \left(\frac{R}{r}\right)^{\lambda'}.
    \end{equation*}
    For $m \in \{1,\dotsc,K\}$ let
    \begin{equation*}
        J_m\coloneqq \{ \, i\in I_1:U_i \cap B(z_m,\delta^{\beta/(\theta + \epsilon)}) \neq \varnothing \, \}.
    \end{equation*}
    If $i \in J_m$, fixing $x_{i,m} \in U_i \cap B(z_m,\delta^{\beta/(\theta + \epsilon)})$,
    \begin{align*}
        \mu(U_i) &\leq \mu(B(x_{i,m},2|U_i|)) \leq c^{-1} \mu(B(x_{i,m},2\delta^{\frac{\beta}{\theta + \epsilon}})) \left(\frac{ \delta^{\frac{\beta}{\theta + \epsilon}} }{|U_i|}\right)^{-\lambda'} \\*
        &\leq c^{-1} \mu(B(z_m,4\delta^{\frac{\beta}{\theta + \epsilon}})) \left(\frac{ \delta^{\frac{\beta}{\theta + \epsilon}} }{|U_i|}\right)^{-\lambda'} \leq c^{-1} M^2\mu(B(z_m,\delta^{\frac{\beta}{\theta + \epsilon}})) \left(\frac{ \delta^{\frac{\beta}{\theta + \epsilon}} }{|U_i|}\right)^{-\lambda'}.
    \end{align*}
    Then
    \begin{equation*}
        \mu(B(z_m,\delta^{\frac{\beta}{\theta + \epsilon}}))  \leq \sum_{i\in J_m} \mu(U_i) \leq c^{-1} M^2 \mu(B(z_m,\delta^{\frac{\beta}{\theta + \epsilon}})) \delta^{\frac{-\lambda' \beta}{\theta + \epsilon}} \sum_{i\in J_m} |U_i|^{\lambda'}.
    \end{equation*}
    Note that $\mu(B(z_m,\delta^{\frac{\beta}{\theta + \epsilon}})) > 0$ since $\supp\mu=\overline{F}$.
    Moreover, if $i \in I_1$, then $U_i$ intersects at most one of the balls of radius $\delta^{\beta/(\theta + \epsilon)}$, so for sufficiently small $\delta$,
    \begin{align*}
    \sum_{U \in \mathcal{U}_1} |U|^{s+\eta} &= \sum_{i \in I_1} \big| B(z_i, 5\delta^{\frac{\beta}{\theta + \epsilon}})\big|^{s + \eta}\\
                                           &= \sum_{m = 1}^K (10\delta^{\frac{\beta}{\theta + \epsilon}})^{s + \eta}\leq 10^{s + \eta} c^{-1} M^2 \delta^{\frac{\beta}{\theta + \epsilon}(s + \eta - \lambda')} \sum_{i \in I} |U_i|^{\lambda'} \\
                                                                                   &\leq 10^{s + \eta} c^{-1} M^2 \delta^{\frac{\beta}{\theta + \epsilon}(s + \eta - \lambda')} \delta^{\frac{\lambda' - s'}{\theta}} \sum_{i \in I} |U_i|^{s'} \leq 10^{s + \eta} c^{-1} M^2.
    \end{align*}
    
    Next, consider $\mathcal{U}_2$:
    \begin{equation*}
        \sum_{U \in\mathcal{U}_2} |U|^{s+\eta} =\sum_{j \in I_2} \big|\mathcal{S}_{\delta^{\beta/(\theta+\epsilon)}}(U_j)\big|^{s+\eta} \leq  \sum_{j \in I_2} (3 |U_i|)^{s + \eta} \leq 3^{s + \eta}.
    \end{equation*}
    
    Finally, consider $\mathcal{U}_3$.
    Since $|U_k|\leq\delta$,
    \begin{align*}
         \sum_{U \in\mathcal{U}_3} |U|^{s+\eta} &=\sum_{k \in I_3} \sum_{\ell=1}^{\lfloor C((2|U_k|)/\delta^{\beta})^{\alpha'} \rfloor} |B_{k,\ell}|^{s+\eta}\leq \sum_{k \in I_3} 2^{\alpha'} C |U_k|^{\alpha'} \delta^{-\beta \alpha'} \delta^{\beta (s+\eta)}\\
                                           &\leq 2^{\alpha'} C \sum_{k \in I_3} |U_k|^s \delta^{\alpha' - s - \beta \alpha' + \beta(s+\eta)}\leq 2^{\alpha'} C \sum_{k \in I} |U_k|^s \leq 2^{\alpha'} C.
    \end{align*}
    Thus $\sum_{U \in\mathcal{U}} |U|^{s+\eta}\leq 10^{s + \eta} c^{-1} M^2 + 3^{s + \eta} + 2^{\alpha'} C$ which does not depend on $\delta$.
    Since $s > h$ was arbitrary, we have shown that $h(\theta + \epsilon) \leq h(\theta) + \eta$, as required. 
    \end{proof}
    
    \begin{corollary}\label{c:int-dim-bound}
     Suppose $F$ is a non-empty, totally bounded subset of a uniformly perfect metric space with more than one point, write $\lambda = \dim_\mathrm{L} F$, $\alpha = \asd F$, and assume that $0 \leq \lambda < \alpha < \infty$. 
     Let $\underline{h}$ and $\overline{h}$ denote the functions $\underline{h}(\varphi) = \underline{\dim}_\varphi F$ and $\overline{h}(\varphi) = \overline{\dim}_\varphi F$, and let $h \in \{\underline{h},\overline{h}\}$. 
    Then $h\in\mathcal{H}(\lambda,\alpha)$, and if $h(\theta)\in\{0,\lambda,\alpha\}$ for some $0<\theta\leq 1$ then $h(\theta)$ is constant on $(0,1]$.%
    \end{corollary}
    \begin{proof}
    Rearranging~\eqref{e:thetaphibound} and dividing through by $\phi - \theta$ gives 
    \begin{equation*}\label{e:h-epsilon-bound}
        \frac{h(\phi)-h(\theta)}{\phi - \theta}\leq \frac{ ( h(\theta) - \lambda) (\alpha -h(\theta))}{(h(\theta) - \lambda)(\phi - \theta) + (\alpha - \lambda)\theta} .
    \end{equation*}
    Taking the limit $\phi \to \theta^+$, we verify~\eqref{e:h-bound-general}. The particular cases $h(\theta)\in\{0,\lambda,\alpha\}$ for some $\theta\in(0,1]$ follow directly from~\eqref{e:thetaphibound}.
\end{proof}

Note that by Theorem~\ref{it:general-form} (or more precisely by Theorem~\ref{t:upper-h-form} below), every function $h \in \mathcal{H}(\lambda,\alpha)$ can be realised as the upper intermediate dimensions of some set, so must satisfy~\eqref{e:thetaphibound} by Theorem~\ref{intermediatects}. 
We now give a direct proof of this fact which does not use the intermediate dimensions. 
\begin{prop}
If $0 \leq \lambda < \alpha < \infty$ then every function $h \in \mathcal{H}(\lambda,\alpha)$ satisfies~\eqref{e:thetaphibound} for $0 < \theta \leq \phi \leq 1$. 
\end{prop}
\begin{proof}
All functions in $\mathcal{H}(\lambda,\alpha)$ are non-decreasing, so the bound $h(\theta) \leq h(\phi)$ is immediate. 
Suppose $g \colon (0,1] \to [\lambda,\alpha]$ is continuous on $(0,1]$, differentiable on $(0,1)$ and satisfies 
\[ g'(\theta') = \frac{(g(\theta') - \lambda)(\alpha - g(\theta'))}{(\alpha - \lambda)\theta'} \qquad \mbox{ for all } \theta' \in (0,1).\] 
Solving this differential equation by separation of variables gives $g(\theta') = \frac{A \theta' \alpha + \lambda}{A\theta' + 1}$ for some $A>0$ and all $\theta' \in (0,1]$. 
Therefore the unique solution satisfying $g(\theta) = h(\theta)$ has $A = \frac{h(\theta)-\lambda}{\theta(\alpha - h(\theta))}$ and satisfies 
\begin{equation*}
g(\phi) = h(\theta) + \frac{(h(\theta) - \lambda) (\alpha - h(\theta))}{\phi (h(\theta)- \lambda) + \theta( \alpha -h(\theta))}(\phi-\theta). 
\end{equation*}
Using conditions~\ref{i:hdefcty} and~\ref{i:hdefdini} from Definition~\ref{d:h-class}, we can apply Lemma~\ref{l:c1-bound} to give $h(\phi)\leq g(\phi)$, completing the proof. 
\end{proof} 

The fact that $h(\theta)\in\{0,\lambda,\alpha\}$ implies that $h$ is constant extends results in~\cite{Falconer2020firstintermediate,
Falconer2021intdimsurvey}. 
It implies a certain `mutual dependency' between box and intermediate dimensions: in order to check that the box dimension of a set is 0, it suffices to check the \emph{a priori} weaker condition that the $\theta$-intermediate dimension of the set is 0 at a small $\theta \in (0,1]$. 
It would be interesting to know if there are sets whose box dimension has resisted calculation by other methods but can be calculated in this way. 
Another mutual dependency result between different notions of dimension is that the upper box dimension of a set is 0 if and only if its Assouad spectrum and quasi-Assouad dimensions are 0, which follows from work in~\cite{Fraser2018firstassspec,Garcia2021qa,
Fraser2019twospectra}. 

Falconer noted that his continuity result~\cite[(14.2.2)]{Falconer2021intdimsurvey} shows that $\frac{\overline{\dim}_\theta F}{\theta}$ and $\frac{\underline{\dim}_\theta F}{\theta}$ are monotonically decreasing in $\theta \in (0,1]$, so the graphs of $\theta \to \overline{\dim}_\theta F$ and $\theta \to \underline{\dim}_\theta F$ for $\theta \in (0,1]$ are starshaped with respect to the origin. 
 Corollary~\ref{starshaped} shows that in fact the graphs are \emph{strictly} starshaped, and every half-line from the origin in the first quadrant intersects the graphs in a single point. 
 The following two corollaries again hold for any non-empty, totally bounded subset $F$ of a uniformly perfect metric space with more than one point. 

\begin{cor}\label{starshaped}
As above, let $\underline{h}$ and $\overline{h}$ denote the functions $\underline{h}(\theta) = \underline{\dim}_\theta F$ and $\overline{h}(\theta) = \overline{\dim}_\theta F$, let $h \in \{\underline{h},\overline{h}\}$, and write $\lambda = \dim_\mathrm{L} F$, $\alpha = \asd F$. 
If $0 < h(1) \leq \alpha < \infty$ then $h(\theta)/\theta$ is strictly decreasing in $\theta \in (0,1]$. 
\end{cor} 

\begin{proof} 
The only non-trivial case is when $\lambda < \alpha$. Suppose $0 < \theta < \phi \leq 1$. By Corollary~\ref{c:int-dim-bound}, $h(\theta) > 0$, so by Theorem~\ref{intermediatects} and a direct algebraic manipulation, 
\[ \frac{h(\phi)}{\phi} \leq \frac{1}{\phi}\left( h(\theta) + \frac{(h(\theta) - \lambda) (\alpha -h(\theta))}{\phi (h(\theta) - \lambda) + \theta( \alpha - h(\theta))}(\phi-\theta)   \right) < \frac{h(\theta)}{\theta}. \qedhere \]
\end{proof}
If $0 < \theta \leq \phi \leq 1$, Theorem~\ref{intermediatects} gives an upper bound for $h(\phi)$ in terms of $h(\theta)$ which can be rearranged to 
\begin{align*}
 h(\phi) (\phi (h(\theta) - \lambda) + \theta( \alpha -h(\theta))) \leq h(\theta) &(\phi (h(\theta) - \lambda) + \theta( \alpha -h(\theta))) \\*
 &+ (h(\theta) - \lambda) (\alpha -h(\theta))(\phi-\theta). 
 \end{align*}
Expanding brackets, cancelling terms and rearranging, we obtain what can be thought of as a lower bound for $h(\theta)$ in terms of $h(\phi)$: 
\begin{equation}\label{ctylowerbound}
 h(\theta) \geq \frac{ \alpha \theta ( h(\phi) - \lambda ) +    \phi  \lambda ( \alpha -  h(\phi)) }{ \theta (h(\phi) - \lambda) + \phi ( \alpha - h(\phi))}.
 \end{equation}
Of particular interest is the lower bound for the intermediate dimensions in terms of the box dimension, because the box dimension of many sets is known independently. 
Recall that $\overline{\dim}_1 F=\dimuB F$ and $\underline{\dim}_1 F=\dimlB F$.
\begin{cor}\label{c:lower-bound-general}
    As above, let $\underline{h}$ and $\overline{h}$ denote the functions $\underline{h}(\theta) = \underline{\dim}_\theta F$ and $\overline{h}(\theta) = \overline{\dim}_\theta F$, and let $h \in \{\underline{h},\overline{h}\}$. 
    If $\lambda < \alpha < \infty$, then for all $\theta \in (0,1]$, 
    \begin{equation*}
        h(\theta)\geq\frac{\alpha \theta(h(1)-\lambda)+\lambda(\alpha-h(1))}{\theta(h(1)-\lambda)+(\alpha-h(1))}.
    \end{equation*}
\end{cor}

\begin{proof}
Set $\phi = 1$ in~\eqref{ctylowerbound}.  
\end{proof}
 \begin{figure}[ht]
\center{\includegraphics[width=.8\textwidth]
        {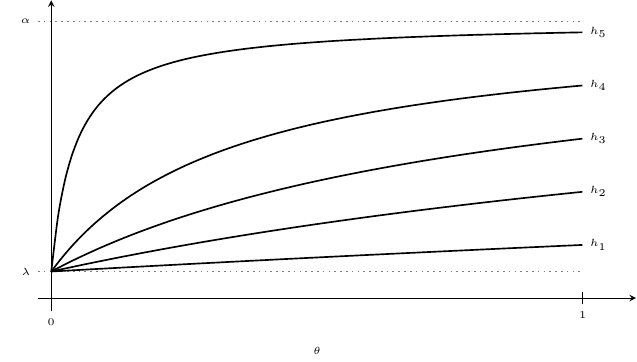}}
        \caption{\label{f:lower-bound-curves}
        Plots of the bound in Corollary~\ref{c:lower-bound-general} for $\lambda=0.05$, $\alpha=0.52$, and box dimensions $h_i(1)=i/10$ for $i=1,2,\dotsc,5$.}
\end{figure}
 Plots of this bound for particular parameters are given in Figure~\ref{f:lower-bound-curves}. 
 We make several remarks about the bound. 
\begin{itemize}
\item Working in $\Rd$, we can replace $\lambda$ by $0$ and $\alpha$ by $d$ to obtain bounds which hold for all subsets. 

\item This bound improves previous general lower bounds in the literature such as \cite[Proposition~2.4]{Falconer2020firstintermediate} and \cite[Corollary~14.4]{Falconer2021intdimsurvey}. 

\item Assume $\lambda < h(1) < \alpha$. Then one can differentiate the bound and show that it is real analytic, strictly increasing, strictly concave, and takes value $\lambda$ at $\theta = 0$ and $h(1)$ at $\theta = 1$. %

\item As $h(1)$ approaches $\alpha$ or $\lambda$ respectively, so does the lower bound pointwise. 

\item For some self-affine Bedford--McMullen carpets, in particular when the maps in the defining iterated function system are very unevenly distributed in the different columns, Corollary~\ref{c:lower-bound-general} can give non-trivial information when $\theta$ is close to 1. For much more on the intermediate dimensions of Bedford--McMullen carpets, we refer the reader to Chapter~\ref{s:bm}. 

\end{itemize}

\subsection{Lattice sets}\label{s:lattice}
In Proposition~\ref{p:lattice} below, we will use the bounds from Section~\ref{s:generalbounds} to calculate the intermediate dimensions of certain lattice sets. Specifically, consider the inversion of the lattice $\{ 1^p,2^p,3^p,\dotsc \}^d$ in the unit $d$-sphere in $\Rd$ and observe that the bound~\eqref{e:intro-gen-bound} is attained at each $\theta\in(0,1)$. 
In the case $d=1$, the sets are just the polynomial sequence sets $\{ 1^{-p},2^{-p},3^{-p},\dotsc \}$ whose intermediate dimensions were calculated in~\cite[Proposition~3.1]{Falconer2020firstintermediate} using a mass distribution principle. Since we have the benefit of the general bound, we do not need to use a mass distribution principle. 
Proposition~\ref{p:lattice} will be useful in Section~\ref{ctdfracsect} related to continued fraction sets.

\begin{prop}\label{p:lattice}
For $d \in \N$ and $p \in (0,\infty)$ define \[ G_{p,d} \coloneqq \{ \, x/||x||^2 : x \in \{ 1^p,2^p,3^p,\dotsc \}^d \, \}. \]
 Then for all $\theta \in [0,1]$, 
\[ \dim_{\theta} G_{p,d} = \frac{d\theta}{p + \theta}. \]
In particular, the intermediate dimensions are continuous at $\theta=0$. 
\end{prop} 

\begin{proof}
We begin with the upper bound. Let $\theta \in (0,1]$. 
For $\delta \in (0,1/10)$ let $n \coloneqq \lceil \delta^{-\theta/(p+\theta)} \rceil$. 
We form a cover $\mathcal{U}$ by covering each point in $\{ \, x/||x||^2 : x \in \{ 1^p,2^p,\dotsc, n^p\}^d \, \}$ with a ball of diameter $\delta$, and covering $[0,n^{-p}]^d$ with $\lesssim (n^{-p}/\delta^\theta + 1)^d$ sets of diameter $\delta^\theta$, where here $\lesssim$ means up to a multiplicative constant independent of $\delta$ and $n$. Then for $s > \frac{d\theta}{p + \theta}$, 
\begin{align*}
 \sum_{U \in \mathcal{U}} |U|^s &\lesssim (n^{-p}/\delta^\theta + 1)^d \delta^{\theta s} + n^d \delta^s \\
 &\lesssim \delta^{dp\theta/(p+\theta)}\delta^{-d \theta}\delta^{\theta s} + \delta^{\theta s} + \delta^{-d\theta/(p+\theta)}\delta^s \\
 &\lesssim \delta^{\theta(s-d\theta/(p+\theta))} + \delta^{s-d\theta/(p+\theta)} \\
 &\lesssim 1,
 \end{align*}
proving $\uid G_{p,d} \leq s$. 

For the lower bound, for $\delta \in (0,1/10)$, let $m \coloneqq \lceil \delta^{-1/(p+1)} \rceil$. 
A direct geometric argument shows that $\{ \, x/||x||^2 : x \in \{ 1^p,2^p,\dotsc, m^p\}^d \, \}$ is a $\gtrsim \delta$-separated set, so if $0 < p < 1$ then 
\[ N_\delta (G_{p,d} ) \geq N_\delta (\{ \, x/||x||^2 : x \in \{ 1^p,2^p,\dotsc, m^p\}^d \, \}) \gtrsim m^d \geq \delta^{-d/(p+1)}.\]
A geometric argument shows that for each $\eta > 0$, 
\begin{equation}\label{e:latticedense} 
\sup_{x \in [0,m^{-p}]^d} \inf_{y \in G_{p,d}} ||x-y|| \lesssim \delta^{1-\eta},
\end{equation} 
so if $p \geq 1$ then 
\[ N_\delta (G_{p,d}) \geq N_\delta (G_{p,d} \cap [0,m^{-p}]^d ) \gtrsim \delta^{d\eta} (m^{-p}/\delta)^d \gtrsim \delta^{-(d/(p+1) - d\eta)}. \]
Therefore $\lbd G_{p,d} \geq d/(p+1)$ for all $p \in (0,\infty)$. 
Moreover, by~\eqref{e:latticedense}, for all $\eta > 0$ sufficiently small, 
\[ N_{\delta^{1-\eta}} ([0,m^{-p}] \cap G_{p,d}) \approx \left(\frac{m^{-p}}{\delta^{1-\eta}}\right)^d, \]
so $\dim_{\mathrm A} G_{p,d} = d$. Furthermore, $G_{p,d}$ has isolated points so has lower dimension~$0$. 
Thus by our general lower bound Corollary~\ref{c:lower-bound-general}, 
\[ \lid G_{p,d} \geq \frac{d \theta \cdot \underline{\dim}_\mathrm{B} G_{p,d} }{d - (1-\theta) \underline{\dim}_\mathrm{B} G_{p,d} } \geq \frac{d\theta \cdot d/(p+1)}{d - (1-\theta) d/(p+1)} = \frac{d\theta}{p+\theta},\]
completing the proof. 
\end{proof}

These sets have isolated points so their lower dimension is $\lambda=0$, and we have shown that the Assouad dimension $\alpha = d$. 
A direct algebraic manipulation shows that the upper bound from~\eqref{e:thetaphibound} and the bound in Corollary~\ref{c:lower-bound-general} are attained: if $h(\theta) = \dim_{\theta} G_{p,d}$ and $0 < \theta \leq \phi \leq 1$, then 
\begin{equation*}
h(\theta) + \frac{(h(\theta)  - \lambda) (\alpha - h(\theta))}{\phi (h(\theta)- \lambda) + \theta( \alpha -h(\theta))}(\phi-\theta) = \frac{d\phi}{p+\phi} = h(\phi).
\end{equation*}
This family of examples also shows that the Lipschitz constant of $d/(4\theta)$ for subsets of $\Rd$ in Theorem~\ref{intermediatects} cannot be improved in general. 
Note also that it follows from Proposition~\ref{p:lattice} and Burrell's~\eqref{e:burrellbrownian} on page~\pageref{e:burrellbrownian} that for all $d \in \N$, $p \in (0,\infty)$ and $\alpha \in (0,1)$, if $B_\alpha \colon \Rd \to \Rd$ is index-$\alpha$ fractional Brownian motion on $\Rd$ then $\ubd B_\alpha(G_{p,d}) < d$.

\subsection{Popcorn-like pyramid sets}\label{s:popcorn}

Another family of sets whose intermediate dimensions can be calculated using the bounds in Section~\ref{s:generalbounds} are related to the \emph{popcorn function}, also known as \emph{Thomae's function}. This is an important example in real analysis, with many interesting properties, such as being Riemann integrable despite not being continuous on any open interval. In fact, it is discontinuous at the rationals but continuous at the irrationals. There are several intriguing connections between the popcorn function and different areas of mathematics~\cite{Gorodetski2023popcorn,Reeder2023lie,
Copar2020} and computer science~\cite{Shu2017}. 

In this section, as well as working with the popcorn function itself, we will consider the following higher-dimensional generalisations of it. 
Throughout the section, $d$ will denote an integer with $d \geq 2$, and $0 < t < \infty$. Then the popcorn pyramid function $f_{t,d} \colon [0,1]^{d-1} \to \mathbb{R}$ is defined by 
\begin{equation}
f_{t,d}(x) = 
\begin{cases}
q^{-t} & \text{if } x=(p_1/q,\dotsc,p_{d-1}/q), q \in \N, \, \forall i: \, p_i \in \{1,\dotsc,q-1\}, \gcd{(p_i,q)} = 1 , \\
0 & \text{otherwise}.
\end{cases}
\end{equation}
The popcorn function itself is $f_{1,2}$. Note that the function is 0 unless all numbers in the product have the same denominator, for example in the $d=3$ case ${f_{t,3}(1/2,1/3)=0}$. The graphs of the functions are denoted by 
\[
G_{t,d} \coloneqq \left\{ \, (x,f_{t,d}(x)) : x\in [0,1]^{d-1} \, \right\}. 
\]
Two of the graphs in the $d=2$ case are shown in Figure~\ref{f:sets}; the graph on the left is that of the popcorn function. 
For completeness, we also include the full set in our analysis, since we will see that the corresponding sets have the same dimensions: 
\[
F_{t,d} \coloneqq \left\{ \, \left( \frac{p_1}{q},\dotsc, \frac{p_{d-1}}{q}, \left(\frac{1}{q}\right)^t \right) : q \in \N, \forall i \, p_i \in \{1,\dotsc, q-1\} \, \right\} \cup ([0,1]^{d-1} \times \{0\} ).
\]
Then $G_{t,d} \subset F_{t,d} \subset [0,1]^d$; for example, the convex hull of $G_{1,3}$ (or $F_{1,3}$) is a square-based pyramid in $\mathbb{R}^3$, and $(1/2,1/2,1/4) \in  F_{1,3} \setminus G_{1,3}$. 
\begin{figure}[ht]
\centering
\subfigure[The popcorn graph $G_{1,2}$]
{\includegraphics[width=.4\textwidth]{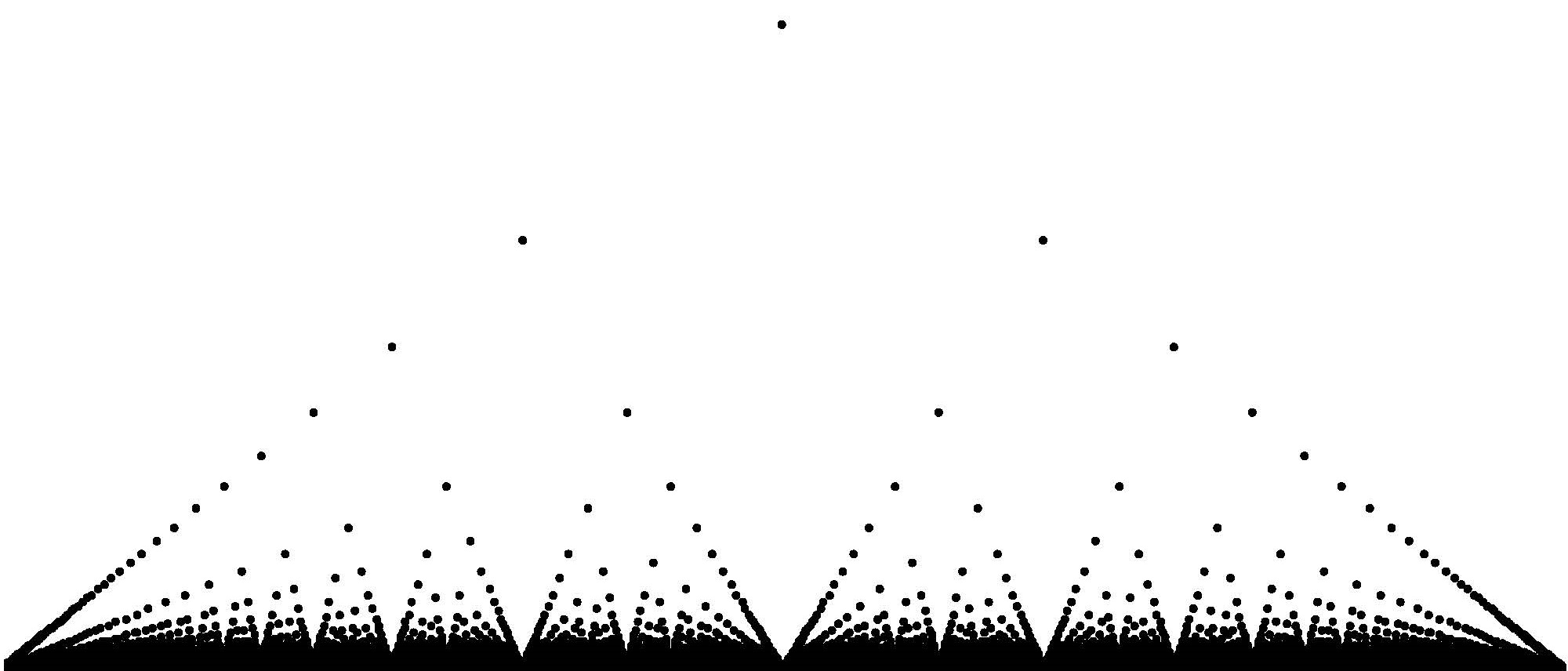}} \qquad
\subfigure[$G_{0.3,2}$]
{\includegraphics[width=.4\textwidth]{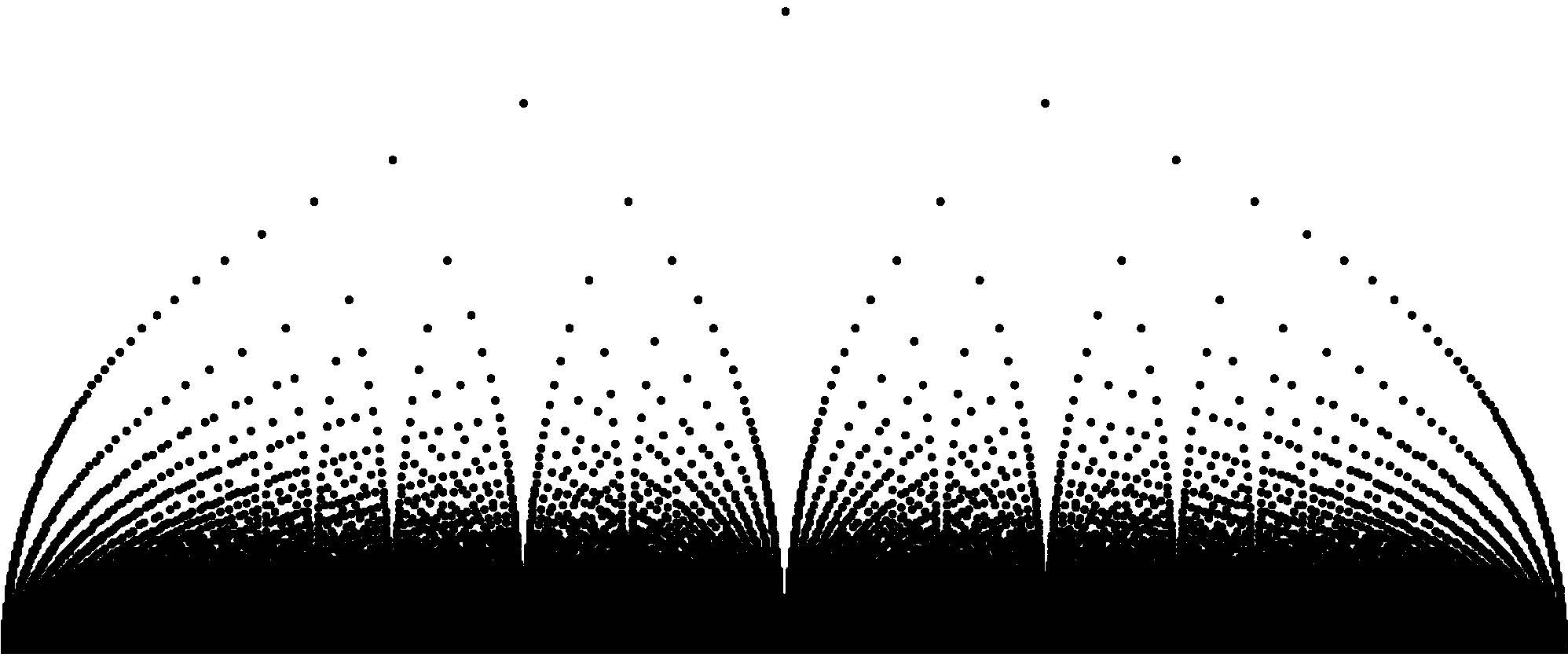}}
\caption{Two popcorn-like graphs}\label{f:sets}
\label{popcorn_graph}
\end{figure}

The sets $G_{t,d}$ and $F_{t,d}$ have an interesting fractal structure, and it is natural to consider different notions of dimension of these sets. This was done in the case $d=2$ in~\cite{Chen2022popcorn,Chen2022tpopcorn}. 
First consider the Assouad dimension. 
\begin{theorem}[Banaji--Chen, Theorem~1.1 from \cite{Banaji2022popcorn}]\label{t:assouad}
We have 
\[ 
\dim_{\mathrm A} G_{t,d} = \dim_{\mathrm A} F_{t,d} = \begin{cases} d & \mbox{ for } 0 < t < \frac{d}{d-1}, \\
d-1 & \mbox{ for } t \geq \frac{d}{d-1}. \end{cases} 
\]
\end{theorem}
It follows from Theorem~\ref{t:assouad} and~\eqref{e:dimrelations} that if $t\geq d/(d-1)$ then the Hausdorff, box and Assouad dimensions are all equal to $d-1$, so in all other results in this section we assume that $t < d/(d-1)$. 
Next, consider box dimension. 
\begin{theorem}[Banaji--Chen, Theorem~1.2 from \cite{Banaji2022popcorn}]\label{t:box}
If $0 < t < d/(d-1)$ then 
\[ \dim_{\mathrm B} G_{t,d} = \dim_{\mathrm B} F_{t,d} = \frac{d^2}{d + t}. \]
\end{theorem}
Note that $\dim_{\mathrm B} G_{t,d}$ is continuous in $t$ but $\asd G_{t,d}$ is not, since Assouad dimension depends sensitively on the local scaling behaviour of the set. 
In order to keep this thesis to a reasonable length, we omit the proofs of Theorem~\ref{t:assouad} (in the $t \geq d/(d-1)$ case) and Theorem~\ref{t:box}, and refer the reader to \cite{Banaji2022popcorn}. 
The lower bound of Theorem~\ref{t:box} uses the Chung--Erd\H{o}s inequality from probability theory, the estimate $\phi(n) \gtrsim n/(\log \log n)$ for Euler's totient function $\phi(n)$, and higher-dimensional Duffin--Schaeffer type estimates from Diophantine approximation. Special cases were proved before the paper~\cite{Banaji2022popcorn} in~\cite{Chen2022popcorn,Chen2022tpopcorn}.

Next, we use Theorem~\ref{t:box} (more precisely, we only use the lower bound) as a black box to calculate the intermediate dimensions. 
Our proof also uses the general bounds from Section~\ref{s:generalbounds}; for a direct proof using similar ideas to those used in the proof of those bounds (fattening the small sets and breaking up the large ones), we refer the reader to~\cite{Banaji2022popcorn}. 
\begin{theorem}\label{t:int}
If $0 < t < d/(d-1)$ then 
\[ 
\dim_{\theta} G_{t,d} = \dim_{\theta} F_{t,d} =\begin{cases}
d-1 & \mbox{ for }0 \leq \theta \leq \frac{(d-1)t}{d}, \\
\frac{d^2 \theta}{d\theta + t} & \mbox{ for }\frac{(d-1)t}{d} < \theta \leq 1.
\end{cases}
\] 
\end{theorem}
\begin{proof}
Note that ${\dim_{\mathrm L} F = 0}$ since $F$ has isolated points, and ${\asd F \leq d}$. Therefore combining Theorem~\ref{t:box} with the bound from Corollary~\ref{c:lower-bound-general} proves the lower bound for $\theta > (d-1)t/d$. 
The idea for the upper bound for this range of $\theta$ is to fix a small number $\delta > 0$ and separate $F_t$ into two parts: 
\[ F_{t,d}^{(1)} = F_{t,d} \cap ([0,1]^{d-1} \times [0,\delta^{dt/(d\theta + t)}]); \qquad F_{t,d}^{(2)} = F_{t,d} \cap ([0,1]^{d-1} \times (\delta^{dt/(d\theta + t)},1]). \]
Covering $[0,1]^{d-1} \times [0,\delta^{dt/(d\theta + t)}]$ with balls of size $\delta$ gives   
\[ N_\delta (F_{t,d}^{(1)}) \lesssim \delta^{-(d-1)} \delta^{dt/(d\theta + t) - 1} = \delta^{-d^2\theta/(d\theta+t)}. \]
It follows from a simple cardinality estimate that 
\[ \# F_{t,d}^{(2)} \lesssim (\delta^{-d/(d\theta + t)})^d = \delta^{-d^2/(d\theta + t)}. \]
We can cover each point in $F_{t,d}^{(2)}$ with a ball of size $\delta^{1/\theta}$ and the result now follows from the estimate  
\[ N_\delta (F_{t,d}^{(1)}) \cdot \delta^{d^2\theta/(d\theta + t)} + \# F_{t,d}^{(2)} \cdot (\delta^{1/\theta})^{d^2\theta/(d\theta + t)} \lesssim 1. \]
This proves the upper bound. 

Since the intermediate dimensions are non-decreasing in $\theta$, it follows that $\uid F_{t,d} \leq {\overline{\dim}}_{(d-1)t/d} F_{t,d} \leq d-1$ for all $\theta \in [0,(d-1)t/d]$. Moreover, $\lid G_{t,d} \geq \dim_{\mathrm H} G_{t,d} = d-1$ for all $\theta \in [0,1]$. Since $G_{t,d} \subset F_{t,d}$, we have $\lid G_{t,d} \leq \lid F_{t,d} \leq \uid F_{t,d}$ and $\lid G_{t,d} \leq \uid G_{t,d} \leq \uid F_{t,d}$, and Theorem~\ref{t:int} follows. 
\end{proof}
Note that one could alternatively have proved the upper bound for the intermediate dimensions by combining the simple special case $\dim_{(d-1)t/d} F_{t,d} \leq d-1$ with~\ref{e:thetaphibound}. 
The $t<d/(d-1)$ case of Theorem~\ref{t:assouad} on the Assouad dimension follows by combining Theorem~\ref{t:int} with the general bound~\eqref{c:lower-bound-general}, and the $t \geq d/(d-1)$ case is proved in~\cite{Banaji2022popcorn}. 
As a special case of Theorem~\ref{t:int}, we obtain a formula for the intermediate dimensions of the graph of the popcorn function. 
\begin{corollary}
\[ \dim_{\theta} G_{1,2} = \dim_{\theta} F_{1,2} = \begin{cases}
1 & \mbox{ for } 0 \leq \theta \leq \frac{1}{2}, \\
\frac{4 \theta}{2\theta + 1} & \mbox{ for } \frac{1}{2} < \theta \leq 1.
\end{cases} 
\] 
\end{corollary}
\begin{proof}
This is immediate from Theorem~\ref{t:int}.
\end{proof}

We see that the intermediate dimensions are constant and equal to the Hausdorff dimension until a phase transition at $\theta = (d-1)t/d$ (which is $\theta = 1/2$ for the popcorn function) and then strictly increasing, concave and analytic. 
This form has not previously been seen in `natural' examples of sets, though we will see in Section~\ref{ctdfracsect} that the intermediate dimensions of certain sets defined using continued fractions have a similar form. 
Note that the graph of the intermediate dimensions is neither convex on the whole domain nor concave on the whole domain. 
The fact that the phase transition takes place at $\theta = 1/2$ for the popcorn function means that for the dimension to increase above the Hausdorff dimension, the sizes of the covering sets need to be restricted to lie in intervals that are smaller than $[\delta^2,\delta]$. 

Recalling the discussion in Section~\ref{s:holderintro}, in the following corollary of Theorem~\ref{t:int}, we apply results of Burrell~\cite{Burrell2022brownian} to give bounds for the box dimension of images of the sets $G_{t,d}$ and $F_{t,d}$ under fractional Brownian motion. 
\begin{corollary}
Fix $d \in \mathbb{N}$ with $d \geq 2$, ${0 < t < d/(d-1)}$ and ${\alpha > (d-1)/d}$. 
If ${B_\alpha \colon \Rd \to \Rd}$ is index-$\alpha$ fractional Brownian motion then the following hold almost surely: 
\begin{align*}
\dim_{\mathrm H} B_\alpha(G_{t,d}) = \dim_{\mathrm H} B_\alpha(F_{t,d}) &= \frac{d-1}{\alpha}, \\
\ubd B_\alpha(G_{t,d}) \leq \ubd B_\alpha(F_{t,d}) &< d.
\end{align*}
\end{corollary}
\begin{proof}
The value of the Hausdorff dimension of the fractional Brownian image is a direct consequence of Kahane's general results~\cite[Chapter~18]{Kahane1985fractbrown}, since $G_{t,d}$ and $F_{t,d}$ are Borel. The box dimension result follows from Burrell's~\eqref{e:burrellbrownian}, since the intermediate dimensions of $G_{t,d}$ and $F_{t,d}$ are continuous at $\theta = 0$ by Theorem~\ref{t:int}. 
\end{proof}
It is interesting to note that the condition $\alpha > (d-1)/d$ does not depend on $t$, even though the box dimension of the sets $G_{t,d}$ and $F_{t,d}$ does depend on $t$. 

Corollary~\ref{c:holder} shows that for the sets $G_{t,d}$ and $F_{t,d}$, the intermediate dimensions for $\theta \in (0,1)$ can give better information than either the Hausdorff or box dimensions. 
The bound achieved by different values of $\theta$ for a certain choice of parameters is shown in Figure~\ref{f:holder}. 

\begin{corollary}\label{c:holder}
Suppose $0 < t_1 < t_2 \leq d/(d - 1)$. Then if $f \colon G_{t_2,d} \to \mathbb{R}^{d}$ satisfies $f(G_{t_2,d}) \supseteq G_{t_1,d}$ and is $\alpha$-H\"older, then 
\[ \alpha \leq \frac{(d - 1)t_2 + t_1}{d t_2}. \]
The same holds if $G_{t_1,d}$ is replaced by $F_{t_1,d}$ or $G_{t_2,d}$ is replaced by $F_{t_2,d}$. 
\end{corollary}
\begin{proof}
If $\theta = (d-1)t_2/d$ then 
\[ \alpha \leq \frac{\uid G_{t_2,d}}{\uid f(G_{t_2,d})} \leq \frac{\dim_{\theta} G_{t_2,d}}{\dim_{\theta} G_{t_1,d}} = \frac{(d - 1)t_2 + t_1}{d t_2}. \qedhere \]
\end{proof}
It is straightforward to see that the value of $\theta$ which gives the best bound for $\alpha$ in the proof of Corollary~\ref{c:holder} is indeed $\theta = (d-1)t_2/d$ (the largest value $\theta$ for which $\dim_\theta G_{t_2,d} = d-1$). It may be of interest to determine whether the bounds in Corollary~\ref{c:holder} are sharp, but we will not pursue this. 
It follows from Corollary~\ref{c:holder} (and directly from Theorem~\ref{t:box}) that if $0 < t_1 < t_2 \leq d/(d - 1)$ then $G_{t_1,d}$ and $G_{t_2,d}$ are not bi-Lipschitz equivalent. 
\begin{figure}[ht]
\center{\includegraphics[width=1\textwidth]
        {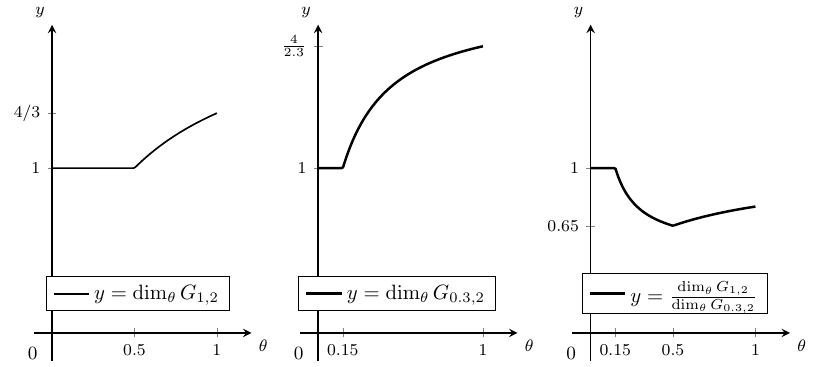}}
        \caption{\label{f:holder}
        Graph of the intermediate dimensions of the popcorn sets from Figure~\ref{f:sets} on page~\pageref{f:sets}, and their ratio (which gives upper bounds on the possible H\"older exponents of surjective maps from $G_{1,2}$ to $G_{0.3,2}$). 
 }
\end{figure}

\section{Moran sets}
\subsection{Definition and dimensions of homogeneous Moran sets}
In this section, we prove the converse direction to Theorem~\ref{it:general-form}. 
The main objects which we use to do so are what we call homogeneous Moran sets. 
The construction of these sets is analogous to the usual $2^d$-corner Cantor set, except that the subdivision ratios need not be the same at each level. 
Because of the nature of these Moran sets, in this chapter it is convenient to make two changes to notation when we are in $\Rd$ which will not affect any definitions or results related to dimensions (see \cite[Equivalent definitions~2.1]{Falconer2014main}). 
Specifically, we fix some $d\in\N$ and work in $\Rd$ equipped with the max norm rather than the Euclidean norm, and we define $N_\delta(F)$ to be the smallest number of sets of diameter $\delta$ needed to cover $F$. 

The construction of homogeneous Moran sets is as follows. 
Fix $\mathcal{I}=\{0,1\}^d$.
We write $\mathcal{I}^*=\bigcup_{n=0}^\infty\mathcal{I}^n$, and denote the word of length $0$ by $\varnothing$.
Suppose we are given a sequence $\bm{r}=(r_n)_{n=1}^\infty$ with $0<r_n\leq 1/2$ for each $n\in\N$.
Then for each $n$ and $\bm{i}\in\mathcal{I}$, we define $S^n_{\bm{i}}\colon\Rd\to\Rd$ by
\begin{equation*}
    S^n_{\bm{i}}(x)\coloneqq r_n x+b^n_{\bm{i}},
\end{equation*}
where $b^n_{\bm{i}}\in\Rd$ has
\begin{equation*}
    (b^n_{\bm{i}})^{(j)} =
    \begin{cases}
        0 &\mbox{ if } \bm{i}^{(j)}=0,\\
        1-r_n & \mbox{ if } \bm{i}^{(j)}=1,
    \end{cases}
\end{equation*} 
recalling that $y^{(j)}$ denotes the $j$\textsuperscript{th} coordinate of a point $y\in\Rd$. 
Given $\sigma=(\bm{i}_1,\dotsc,\bm{i}_n)\in\mathcal{I}^n$, we write $S_\sigma=S^1_{\bm{i}_1}\circ\dotsb\circ S^n_{\bm{i}_n}$.
Then set
\begin{equation}\label{e:definehomogmoran}
    C_n=\bigcup_{\sigma\in\mathcal{I}^n}S_\sigma([0,1]^d)\qquad\text{and}\qquad C=C(\bm{r})\coloneqq\bigcap_{n=1}^\infty C_n.
\end{equation}
We refer to the set $C$ as a \emph{homogeneous Moran set}.
Note that $C_n$ consists of $2^{dn}$ hypercubes each with diameter $\rho_n\coloneqq r_1\dotsm r_n$ (with respect to the max norm). 

Given $\delta>0$, let $k=k(\delta)$ be such that $\rho_k\leq \delta < \rho_{k-1}$.
We then define
\begin{equation}\label{e:morancovernumber}
    s(\delta)=s_{\bm{r}}(\delta)\coloneqq\frac{k(\delta)\cdot d\log 2}{-\log\delta}.
\end{equation}
One can interpret $s(\delta)$ as the best candidate for the `box dimension at scale $\delta$.' 
We now prove the following key covering lemma for intermediate dimensions.
This result essentially shows that the optimal covers for a homogeneous Moran set can be taken to consist of balls all of the same diameter.
\begin{lemma}\label{l:flat-covers}
    Let $\theta\in(0,1]$ be arbitrary.
    Then for all $\delta>0$ sufficiently small, with $t=\inf_{\phi\in[\delta^{1/\theta},\delta]}s(\phi)$,
    \begin{equation*}
        4^{-d}\leq \inf\Bigl\{ \, \sum_{U \in\mathcal{U}} |U|^{t}:\mathcal{U}\text{ is a $(\delta,\theta)$-cover of $C$} \, \Bigr\}\leq 1.
    \end{equation*}
\end{lemma}
\begin{proof}
    We first prove the lower bound. 
    Let $\mu$ denote the uniform Bernoulli measure on $C$, i.e. the measure which gives mass $2^{-dn}$ to each hypercube in $C_n$ for all $n$. 
    Let $U$ be a set with $\delta^{1/\theta}\leq |U|\leq\delta$, and let $k$ be such that $\rho_{k}\leq |U|<\rho_{k-1}$.
    Note that $|U|^{s(|U|)}=2^{-kd}$.
    Then since $U$ intersects at most $4^d$ hypercubes in $C_k$,
    \begin{equation*}
        \mu(U)\leq 4^d\cdot 2^{-kd}= 4^d\cdot |U|^{s(|U|)}\leq 4^d|U|^t.
    \end{equation*}
    In particular, if $\mathcal{U}$ is an arbitrary $(\delta,\theta)$-cover of $C$,
    \begin{equation*}
        1=\mu(C)\leq\sum_{U\in\mathcal{U}}\mu(U)\leq 4^d\sum_{U\in\mathcal{U}}|U|^t
    \end{equation*}
    so that $\sum_{U \in\mathcal{U}} |U|^{t}\geq 4^{-d}$. 
    
    For the upper bound, since $s(\delta)$ is continuous and increasing on each interval $[\rho_k,\rho_{k-1})$, there is $\phi\in[\delta^{1/\theta},\delta]$ such that $t=s(\phi)=\frac{k(\phi)\cdot d\log 2}{-\log\phi}$.
    For each $y=(j_1,\dotsc,j_d)\in\{0,1\}^d$ and $\sigma\in\mathcal{I}^*$, let $E_{\sigma,\phi}(y)$ denote the hypercube with side length $\phi$ contained in $S_\sigma([0,1]^d)$, with edges aligned with the coordinate axes, and containing the point $S_\sigma(y)$.
    Since $\phi\geq\rho_{k(\phi)}$,
    \begin{equation*}
        \mathcal{V}\coloneqq\bigcup_{\sigma\in\mathcal{I}^{k(\phi)-1}}\{ \, E_{\sigma,\phi}(y):y\in\{0,1\}^d \, \}
    \end{equation*}
    is a cover for $C_{k(\phi)}$, and therefore $C$, consisting of $2^{k(\phi)\cdot d}$ hypercubes each with diameter $\phi$.
    Thus $\sum_{V \in\mathcal{V}} |V|^{t}=1$.
\end{proof}
A direct application is the following formula for the intermediate dimensions of $C$.
\begin{prop}\label{p:int-dim}
    Let $C$ be a homogeneous Moran set as above. For all $\theta\in(0,1]$,
    \begin{equation*}
        \overline{\dim}_\theta C = \limsup_{\delta\to 0}\bigl(\inf_{\phi\in[\delta^{1/\theta},\delta]}s(\phi)\bigr), \qquad \mbox{in particular } \ubd C = \limsup_{\delta \to 0} s(\delta),
    \end{equation*}
    and
    \begin{equation*}
        \dimH C = \underline{\dim}_\theta C = \underline{\dim}_{\mathrm B} C=\liminf_{\delta\to 0}s(\delta).
    \end{equation*}
\end{prop}
\begin{proof}
This is immediate from Lemma~\ref{l:flat-covers}. 
\end{proof}

\subsection{Constructing homogeneous Moran sets}\label{ss:Moran-cnst}

In this section, we establish a general strategy for constructing homogeneous Moran sets, which will be used to show that the bounds from Section~\ref{s:generalbounds} are sharp. 
We introduce the following definition, which is in some sense analogous to the definition of $\mathcal{H}(\lambda,\alpha)$.
\begin{definition}\label{d:g-class}
    Given $0\leq\lambda<\alpha\leq d$, we write $\mathcal{G}(\lambda,\alpha)$ to denote the functions $g\colon\R\to(\lambda,\alpha)$ which are continuous and satisfy
    \begin{equation*}
        \diniu{+}g(x)\in[\lambda-g(x),\alpha-g(x)]
    \end{equation*}
    for all $x\in\R$.
\end{definition}
We will essentially show that for any function $g\in\mathcal{G}(0,d)$, there exists a homogeneous Moran set such that $s(\delta)\approx g(\log\log(1/\delta))$.
The transformation $\delta\mapsto\log\log(1/\delta)$ is useful since it converts the exponentiation map $\delta\mapsto\delta^{1/\theta}$ into addition $x\mapsto x+\log(1/\theta)$.
In order to construct such a set, it suffices to define the corresponding contraction ratios by `discretising' the function $g$.
In particular, in Lemma~\ref{l:exact-construction}, we show that there exists a sequence of contractions $\bm{r} = (r_n)_{n=1}^{\infty}$ such that the corresponding covering numbers $s_{\bm{r}}(\delta)$ (recall~\eqref{e:morancovernumber}) are close to $g(\log\log(1/\delta))$ in the precise sense given in~\eqref{e:g-disc}.
Of course, depending on the choice of the function $g$, this bound may be impossible to attain for small $x$.
Thus we begin with a function $\tilde{g}$ and then translate it by some constant amount.
The contraction ratios are then used to define a corresponding Moran set $C$, and~\eqref{e:g-disc} is useful to prove dimension results for the Moran set $C$.
Then, in Sections~\ref{ss:pres-upper} and~\ref{ss:pres}, we use this technique to construct Moran sets with the desired properties.
In Theorem~\ref{t:upper-h-form}, we construct the function $g$ depending on some $h\in\mathcal{H}(\lambda,\alpha)$ such that the corresponding homogeneous Moran set has the desired dimension formulas.
This construction is also used in Theorem~\ref{t:gen-h-form}, where we use the sequence of contraction ratios provided by Lemma~\ref{l:exact-construction} directly.
Here, translations of the function $g$ are used to define an \emph{inhomogeneous} Moran set which `locally' looks like the homogeneous Moran set $C$, but with a much greater amount of uniformity between scales (so that the intermediate dimensions exist).
Finally, these results are combined in Corollary~\ref{c:upper-lower-match} to obtain a proof of Theorem~\ref{it:general-form}. 

We now begin to describe our general strategy for constructing homogeneous Moran sets. 
\begin{lemma}\label{l:scale-bound}
    Let $0\leq\lambda<\alpha\leq d$ and let $g \colon \R \to (\lambda,\alpha)$.
    Then $g\in\mathcal{G}(\lambda,\alpha)$ if and only if for all $x_0\in\R$ and $x>0$,
    \begin{equation*}
        \lambda-(\lambda-g(x_0))\exp(-x)\leq g(x_0+x)\leq \alpha-(\alpha-g(x_0))\exp(-x).
    \end{equation*}
\end{lemma}
\begin{proof}
    This is a direct application of Lemma~\ref{l:c1-bound}.
\end{proof}
\begin{definition}
    Given a sequence of functions $(f_k)_{k=1}^\infty$ each defined on some interval $(0,a_k]$, the \emph{concatenation} of $(f_k)_{k=1}^\infty$ is the function $f\colon(-\infty,\sum_{k=1}^\infty a_k)\to\R$ given as follows: for each $x>0$ with $\sum_{j=0}^{k-1} a_j<x\leq\sum_{j=0}^{k}a_j$ where $a_0=0$, we define
    \begin{equation*}
        f(x)=f_k\left(x-\sum_{j=0}^{k-1} a_j\right),
    \end{equation*}
    and for $x\leq 0$ we define $f(x)=f_1(0)$.
\end{definition}
Given a function $g\in\mathcal{G}(\lambda,\alpha)$ and $w\in\R$, we define the \emph{offset} $\kappa_w(g)\in\mathcal{G}(\lambda,\alpha)$ by
\begin{equation*}
    \kappa_w(g)(x)=\begin{cases}
        g(x-w) &: x\geq w,\\
        g(0) &: x< w.
    \end{cases}
\end{equation*}
We also say that a function $g\in\mathcal{G}(\lambda,\alpha)$ is \emph{rapidly decreasing} if there exists $y\in\R$ and a constant $C>0$ such that for all $x\geq y$,
\begin{equation}\label{e:rapidly-decreasing}
    g(x)\leq g(y)\exp(y-x)+C\exp(-x).
\end{equation}
Note that for all $w\in\R$, $\kappa_w(g)$ is rapidly decreasing if and only if $g$ is rapidly decreasing.

The following technical lemma is stated to be useful in the proof of Theorem~\ref{t:gen-h-form}, where many offsets of the same function will be required.
\begin{lemma}\label{l:exact-construction}
    Let $0\leq\lambda<\alpha\leq d$ and let $\tilde{g}\in\mathcal{G}(\lambda,\alpha)$.
    Suppose $\tilde{g}$ is not rapidly decreasing.
    Then there is a constant $w_0\in\R$ depending only on $\tilde{g}(0)$ and $d$ such that for all $w\geq w_0$, there exists a sequence $\bm{r}\coloneqq (r_j)_{j=1}^\infty \in (0,1/2]^{\N}$ such that $g\coloneqq\kappa_{w}(\tilde{g})$ satisfies
    \begin{equation}\label{e:g-disc}
        |s_{\bm{r}}(\exp(-\exp(x)))-g(x)|\leq d\log(2)\cdot\exp(-x)
    \end{equation}
    for all $x\geq w_0$.
\end{lemma}
\begin{proof}
	Noting that $\tilde{g}(0)\in(0,d)$, choose $r_1$ such that $\frac{2d\log(2)}{\log(1/r_1)}=\tilde{g}(0)$.
    Then let $w_0=\log\log(1/r_1)$, let $w\geq w_0$ be arbitrary, and let $g=\kappa_w(\tilde{g})$.
	Since $\tilde{g}$ is not rapidly decreasing, $g$ is also not rapidly decreasing, so by \eqref{e:rapidly-decreasing}, for every $y\in\R$ there exists a minimal $\psi(y)>y$ such that
    \begin{equation*}
        g(y)\exp(y-\psi(y))=g(\psi(y))-d\log(2)\cdot\exp(-\psi(y)).
    \end{equation*}
    By Lemma~\ref{l:scale-bound}, for all $y\in \R$, $g(\psi(y))\leq d-(d-g(y))\exp(-\psi(y) + y)$, which after algebraic manipulation yields $\psi(y) \geq \log ( e^y + \log 2 )$, i.e. $\exp(-\exp(\psi(y))) \leq \exp(-\exp(y))/2$. 
    
    Now set $x_1=w_0$ and, inductively, set $x_{k+1}=\psi(x_k)$ for each $k\in\N$.
    Let $\rho_k=\exp(-\exp(x_k))$ denote the corresponding scales (noting that $\rho_1 = r_1$), and set $r_k\coloneqq\rho_k/\rho_{k-1}$ for $k\geq 2$.
    Observe that $\rho_{k+1}\leq\rho_k/2$ and $r_k \in (0,1/2]$ for all $k \in \N$. 
    Thus for $0<\delta\leq r_1$, if $k$ is such that $\rho_k<\delta\leq\rho_{k-1}$, we set
    \begin{equation*}
        \overline{s}(\delta)=\frac{kd\log 2}{-\log\delta}.
    \end{equation*}
    
    It suffices to prove by induction that for each $k\in\N$ we have $\overline{s}(\rho_k)=g(x_k)$, and
    \begin{equation}\label{e:bound}
        g(x)-d\log(2)\exp(-x)\leq\overline{s}(\exp(-\exp(x)))\leq g(x) \qquad \mbox{ for all } x\in[x_1,x_k]. 
    \end{equation}
    We first note that $\overline{s}(\rho_1)=g(x_1) = \tilde{g}(0)$ by construction. 
    In general, suppose the hypothesis holds for $k\in\N$.
    By the definition of $\psi$ and the fact that $g(x_k)=\overline{s}(\rho_k)$,
    \begin{align*}
        g(x_{k+1}) &=\overline{s}(\rho_k)\exp(-x_{k+1}+x_k)+d\log(2)\exp(-x_{k+1})\\
                   &=\frac{d(k+1)\log 2}{\exp(x_k)}\cdot\exp(-x_{k+1})\exp(x_k)+d\log(2)\exp(-x_{k+1})\\
                   &= \frac{d(k+2)\log 2}{\exp(x_{k+1})}\\*
                   &= \overline{s}(\rho_{k+1}).
    \end{align*}
    Moreover, by Lemma~\ref{l:scale-bound}, $g(x)\geq g(x_k)\exp(-x+x_k)$ for all $x\geq x_k$, so~\eqref{e:bound} follows for $x\in[x_k,x_{k+1}]$ by the minimality of $x_{k+1}$ in the definition of $\psi$. 
\end{proof}

We make several observations about Lemma~\ref{l:exact-construction}, which be will used in the proof of Theorems~\ref{t:upper-h-form} and~\ref{t:gen-h-form}. 
\begin{rem}\label{r:exactconstruction}
\begin{enumerate}[label=(\roman*)]
\item\label{i:limsuppositive} If the continuous function $\tilde{g}$ satisfies $\limsup_{x\to\infty} \tilde{g}(x)>0$ then $\tilde{g}$ is not rapidly decreasing. 

\item
    If, contrary to the assumption of Lemma~\ref{l:exact-construction}, $g$ is rapidly decreasing, then the function $g$ is decays faster than any function $s_{\bm{r}}$ for a sequence $\bm{r}  \in (0,1/2]^{\N}$.
    
\item Since $s(\delta)$ has discontinuities of size $\frac{d\log 2}{\log(1/\delta)}$, the gap of $\overline{s}(\exp(-\exp(x)))$ between the lower and upper bounds in~\eqref{e:bound} cannot be reduced. 
\end{enumerate}
\end{rem}
We now use the sequence $\bm{r}$ constructed in the previous lemma to define a homogeneous Moran set $C$, and prove that it satisfies the correct properties.
Recall that $\mathcal{G}$ is defined in Definition~\ref{d:g-class} and homogeneous Moran sets are defined in~\eqref{e:definehomogmoran}. 
\begin{lemma}\label{l:Moran-formula}
    Let $g\in\mathcal{G}(0,d)$, and suppose $\bm{r}=(r_j)_{j=1}^\infty  \in (0,1/2]^{\N}$ is such that 
    \begin{equation}\label{e:tight-bound}
        |s_{\bm{r}}(\exp(-\exp(x)))-g(x)|\leq d\log(2)\cdot\exp(-x)
    \end{equation}
    for all $x$ sufficiently large. 
    Then the corresponding homogeneous Moran set $C=C(\bm{r})\subset\R^d$ satisfies:
    \begin{enumerate}
        \item\label{i:upperintmoran} $\displaystyle\overline{\dim}_\theta C =\limsup_{x\to\infty}\left(\inf_{y\in[x,x+\log(1/\theta)]}g(y)\right)$ for $\theta\in(0,1]$,
        \item\label{i:lowerintmoran} $\displaystyle\underline{\dim}_\theta C=\dimH C=\liminf_{x\to\infty}g(x)$ for $\theta\in(0,1]$,
        \item\label{i:assouadmoran} $\dimA C\leq\limsup_{x\to\infty}(\diniu{+}g(x)+g(x))$, and
        \item\label{i:lowermoran} $\dimL C\geq\liminf_{x\to\infty}(\diniu{+}g(x)+g(x))$.
    \end{enumerate}
    Moreover, suppose $\psi \colon \R \to \R^+$ is any function such that $\exp(\psi(x))-\exp(x) \to \infty$ as $x \to \infty$. 
    Then
    \begin{enumerate}[resume]
        \item\label{i:A-dim-block} $\displaystyle\dimA C\geq\limsup_{x\to\infty}\left(\inf_{y\in[x,\psi(x)]}(\diniu{+}g(y)+g(y))\right)$, and
        \item\label{i:L-dim-block} $\displaystyle\dimL C\leq\liminf_{x\to\infty}\left(\sup_{y\in[x,\psi(x)]}(\diniu{+}g(y)+g(y))\right)$.
    \end{enumerate}
\end{lemma}
\begin{proof}
    We first observe that~\ref{i:upperintmoran} and~\ref{i:lowerintmoran} follow immediately from Proposition~\ref{p:int-dim}.
    We verify~\ref{i:assouadmoran} and~\ref{i:A-dim-block}; \ref{i:lowermoran} and~\ref{i:L-dim-block} are given by an analogous argument.
    
        We first establish a general formula for the Assouad dimension of $C$ in terms of the numbers $s_{\bm{r}}(\delta)$.
    Suppose $0<\delta_1\leq\delta_2$ are arbitrary.
    Then the number of subdivision steps between scales $\delta_1$ and $\delta_2$, up to an error of size $2$, is
    \begin{equation*}
        \frac{s_{\bm{r}}(\delta_1)\log(1/\delta_1)-s_{\bm{r}}(\delta_2)\log(1/\delta_2)}{d\log 2}.
    \end{equation*}
    Thus there is a bounded function $h(x,\delta_1,\delta_2)$ such that
    \begin{equation*}
        \frac{\log N_{\delta_1}(B(x,\delta_2)\cap C)}{\log(\delta_2/\delta_1)}=\frac{s_{\bm{r}}(\delta_1)\log(1/\delta_1)-s_{\bm{r}}(\delta_2)\log(1/\delta_2)+h(x,\delta_1,\delta_2)}{\log(1/\delta_1)-\log(1/\delta_2)}.
    \end{equation*}
    Therefore by the definition of the Assouad dimension, for all $\delta_0\in(0,1)$,
    \begin{equation}\label{e:assouadlim}
        \dimA C=\lim_{\epsilon\to 0}\sup_{\substack{0<\delta_1<\delta_2<\delta_0\\\delta_1\leq\epsilon\delta_2}}\frac{s_{\bm{r}}(\delta_1)\log(1/\delta_1)-s_{\bm{r}}(\delta_2)\log(1/\delta_2)}{\log(1/\delta_1)-\log(1/\delta_2)}.
    \end{equation}
    Now we may inductively choose sequences of positive numbers $(\delta_{1,n})_{n=1}^\infty$ and $(\delta_{2,n})_{n=1}^\infty$ such that $\delta_{1,n}/\delta_{2,n}$ and $\delta_{2,n}$ converge to $0$, and
    \begin{equation*}
        \frac{s_{\bm{r}}(\delta_{1,n})\log(1/\delta_{1,n})-s_{\bm{r}}(\delta_{2,n})\log(1/\delta_{2,n})}{\log(1/\delta_{1,n})-\log(1/\delta_{2,n})} \in \Big(\dimA C - \frac{1}{n},\dimA C + \frac{1}{n}\Big)
    \end{equation*}
    for all $n \in \N$.

    For $0<x<y$, let
    \begin{equation*}
        \Phi(x,y)\coloneqq \frac{s_{\bm{r}}(\exp(-\exp(y)))-s_{\bm{r}}(\exp(-\exp(x)))}{1-\exp(x-y)}+s_{\bm{r}}(\exp(-\exp(x))).
    \end{equation*}
    Moreover, write $x_n=\log\log(1/\delta_{2,n})$ and $y_n=\log\log(1/\delta_{1,n})$.
    Next, let $\mathcal{W}$ denote the family of functions $\psi\colon\R\to\R^+$ such that $\lim_{x\to\infty}(\exp(\psi(x))-\exp(x))=\infty$.
    The condition that $\delta_{1,n}/\delta_{2,n}$ converges to $0$ is equivalent to $\exp(y_n)-\exp(x_n)$ diverging to infinity.
    Thus we may choose a function $\psi_0 \in\mathcal{W}$ so that $\psi_0(x_n)=y_n$ for infinitely many $n$. Then with some rearrangement using~\eqref{e:assouadlim} and the definition of $\Phi$,
    \begin{equation*}
        \limsup_{x\to\infty}\Phi(x,\psi_0(x)) \geq \dimA C.
    \end{equation*}
    Conversely, if $\psi \in\mathcal{W}$ is arbitrary, applying the substitutions $\delta_2=\exp(-\exp(x))$ and $\delta_1=\exp(-\exp(\psi(x)))$ and using the fact that $\delta_1/\delta_2$ converges to $0$ as $x \to \infty$ gives
    \begin{equation*}
        \limsup_{x\to\infty}\Phi(x,\psi(x))\leq\dimA C.
    \end{equation*}
    Therefore
    \begin{equation}\label{e:assouad-formula}
        \dimA C=\sup_{\psi\in\mathcal{W}}\limsup_{x\to\infty}\Phi(x,\psi(x)),
    \end{equation}
    and moreover the supremum is attained.

    To conclude the preliminaries, we also note, for $0<x<y$,
    \begin{equation}\label{e:error-bound}
        \frac{\exp(-y)+\exp(-x)}{1-\exp(x-y)}+\exp(-x)=\frac{2}{\exp(y)-\exp(x)}+2\exp(-x).
    \end{equation}
    This bound will be used to control the error resulting from~\eqref{e:tight-bound}.

    We now prove~\ref{i:assouadmoran}.
    Write
    \begin{equation*}
        \overline{\alpha} \coloneqq \limsup_{x\to\infty}(\diniu{+}g(x)+g(x)),
    \end{equation*}
    and let $\epsilon > 0$.
    Then there exists $M_\epsilon>0$ such that for all $x \geq M_\epsilon$ we have $\diniu{+}g(x)+g(x) \leq\overline{\alpha} + \epsilon$.
    For $x\geq M_\epsilon$, define $\overline{g}_x\colon [x,\infty) \to \R$ by
    \begin{equation*}
        \overline{g}_x(y) \coloneqq \overline{\alpha} + \epsilon - (\overline{\alpha} + \epsilon - g(x))\exp(x-y).
    \end{equation*}
    Then $g(x) = \overline{g}_x(x)$, and
    \begin{equation*}
        \overline{g}_x'(y) + \overline{g}_x(y) = \overline{\alpha} + \epsilon \geq \diniu{+}g(y)+g(y)
    \end{equation*}
    for all $y>x$.
    It follows from Lemma~\ref{l:c1-bound} that $g(y) \leq \overline{g}_x(y)$ for all $y \geq x$.
    Now taking a function $\psi_0$ which attains the supremum in~\eqref{e:assouad-formula}, for all $x\geq M_\epsilon$, using~\eqref{e:tight-bound} and~\eqref{e:error-bound} combined with the condition on $\psi_0$,
    \begin{align*}
        \Phi(x,\psi_0(x)) \leq{}& \frac{g(\psi_0(x)) - g(x)}{1-\exp(x-\psi_0(x))} + g(x)\\
              &+2d\log(2)\left(\frac{1}{\exp(\psi_0(x))-\exp(x)}+\exp(-x)\right).
    \end{align*}
    Moreover, since $\psi_0\in\mathcal{W}$,
    \begin{equation*}
        \limsup_{x\to\infty}2d\log(2)\left(\frac{1}{\exp(\psi_0(x))-\exp(x)}+\exp(-x)\right)=0.
    \end{equation*}
    Thus
    \begin{align*}
        \limsup_{x\to\infty}\Phi(x,\psi_0(x)) &\leq\limsup_{x\to\infty}\left(\frac{g(\psi_0(x)) - g(x)}{1-\exp(x-\psi_0(x))} + g(x)\right)\\
                                           &\leq\limsup_{x\to\infty}\left(\frac{\overline{g}_x(\psi_0(x)) - g(x)}{1-\exp(x-\psi_0(x))} + g(x)\right)\\
                                           &=\overline{\alpha}+\epsilon.
    \end{align*}
    But $\epsilon>0$ was arbitrary, giving the claim.

    Finally, we prove~\ref{i:A-dim-block}.
    Fix any $\psi \in \mathcal{W}$ and $\epsilon>0$, and write
    \begin{equation*}
        \underline{\alpha}=\limsup_{x\to\infty}\left(\inf_{y\in[x,\psi(x)]}(\diniu{+}g(y)+g(y))\right).
    \end{equation*}
    Get a sequence $(x_k)_{k=1}^\infty$ diverging to infinity such that for all $k\in\N$,
    \begin{equation*}
        \inf_{y\in[x_k,\psi(x_k)]}(\diniu+ g(y)+g(y))\geq \underline{\alpha}-\epsilon.
    \end{equation*}
    As above, define $\underline{g}_k\colon[x_k,\infty)\to\R$ by
    \begin{equation*}
        \underline{g}_k(x) \coloneqq \underline{\alpha} + \epsilon - (\underline{\alpha} + \epsilon - g(x_k))\exp(x_k-x).
    \end{equation*}
    Then $g(x_k)=\underline{g}_k(x_k)$ and $g(x)\geq\underline{g}_k(x)$ for all $x\in[x_k,\psi(x_k)]$.
    Thus the same computations as before yield that
    \begin{align*}
        \limsup_{x\to\infty}\Phi(x,\psi_0(x))&\geq\limsup_{k\to\infty}\Phi(x_k,\psi(x_k))\\
                                          &\geq\underline{\alpha}-\epsilon-\frac{2d\log 2}{\exp(\psi(x_k))-\exp(x_k)}-2d\log(2)\exp(-x_k).
    \end{align*}
    Since $\exp(\psi(x_k))-\exp(x_k)$ diverges to infinity and $\epsilon>0$ was arbitrary, the claimed inequality follows.
\end{proof}
    In general,~\ref{i:assouadmoran} and~\ref{i:lowermoran} will not be equalities since one would require more robust regularity assumptions about the function $g$. 

\subsection{Prescribing the upper intermediate dimensions}\label{ss:pres-upper}
Now, using the general construction in the previous section, we show how to construct homogeneous Moran sets with upper intermediate dimensions given by a function $h\colon[0,1]\to(0,d)$.
The main idea is to construct functions which we call mountains, which have the property that there are exactly two points $\{x,x+\log(1/\theta)\}$ which have value $h(\theta)$.
This ensures that the limit supremum of infima over windows $[x,x+\log(1/\theta)]$ is exactly $h(\theta)$.
Figure~\ref{f:bump-ctr} depicts this construction.
 \begin{figure}[ht]
\center{\includegraphics[width=.98\textwidth]
        {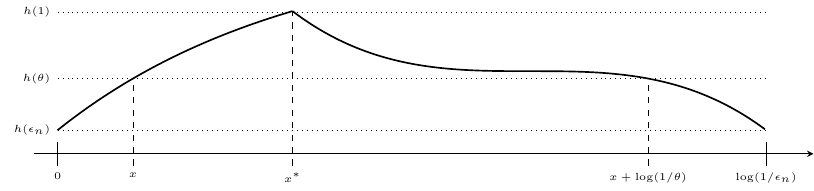}}
        \caption{\label{f:bump-ctr}
        The construction of the mountain $f_n$.}
\end{figure}
    
\begin{theorem}\label{t:upper-h-form}
    Let $0\leq\lambda\leq\alpha\leq d$ and let $h\in\mathcal{H}(\lambda,\alpha)$.
    Then there exists a homogeneous Moran set $C$ such that $\dimL C=\lambda$, $\dimA C = \alpha$, $\underline{\dim}_\theta C=h(0)$, and
    \begin{equation*}
        \overline{\dim}_\theta C=h(\theta)
    \end{equation*}
    for all $\theta\in[0,1]$.
\end{theorem}
\begin{proof}
    We will assume that $\lambda\leq h(0)<h(\theta)<\alpha$ for all $\theta\in(0,1]$, and that there exists $\theta_0 > 0$ such that $h(\theta_0) < h(1)$. 
        The other cases are easier and can be proven with minor modifications, except the case $h(1) = 0$, which we address separately at the end of the proof. 
    
    Let $(\epsilon_n)_{n=1}^\infty\subset(0,1)$ converge monotonically to 0 and $(\gamma_n)_{n=1}^\infty\subset(\lambda,\alpha)$ converge monotonically to $h(0)$ in such a way that $\gamma_n < h(\epsilon_{n+1}) < h(1)$ for all $n \in \N$, and $\gamma_n + 1/n \leq h(\epsilon_{n+1})$ for all $n$ sufficiently large.
    Note that if $h(0)>\lambda$ we can take $\gamma_n=h(0)$ for all $n$.
    We will define functions which we refer to as \emph{mountains} $f_n$ and \emph{valleys} $e_n$; a valley will be used to connect two mountains. 
    Graphical representations of the functions $f_n$ and $e_n$ are given in Figures~\ref{f:bump-ctr} and~\ref{f:connector-ctr} respectively.
    Then we will define a function $g$ by concatenating the $f_n$ and $e_n$, and the corresponding Moran set $C$ will be defined using the sequence given in Lemma~\ref{l:exact-construction}.
    The functions $f_n$ will ensure that $\overline{\dim}_\theta C=h(\theta)$ for $\theta>0$, and the functions $e_n$ will ensure that $g$ is continuous, $\dimH C=h(0)$, $\dimL C=\lambda$, and $\dimA C=\alpha$. 
    \begin{proofpart}
        Construction of the mountains $f_n\colon[0,\log(1/\epsilon_n)]\to[h(\epsilon_n),h(1)]$ for $n\in\N$.
    \end{proofpart}
    First set
    \begin{equation*}
        x^*\coloneqq\log\left(\frac{\alpha-h(\epsilon_n)}{\alpha-h(1)}\right)
    \end{equation*}
    and for $x\in [0,x^*]$ define $f_n(x)=\alpha-(\alpha-h(\epsilon_n))\exp(-x)$. 
    Observe that $f_n(0)=h(\epsilon_n)$, $f_n(x^*)=h(1)$, and
    \begin{equation}\label{e:0-star-bound}
        \dinil{-}f_n(x)=(\alpha-h(\epsilon_n))\exp(-x) = \alpha-f_n(x)
    \end{equation}
    for $x\in(0,x^*]$.
    Now for $x\in[0,x^*]$, if $\theta \in (0,1]$ is such that $h(\theta)=f_n(x)$, we define $f_n(x+\log(1/\theta))=h(\theta)$.
    This is well-defined since $h$ is non-decreasing and continuous.
    In particular, $f_n(x)$ is non-increasing and continuous on $[x^*,\log(1/\epsilon_n)]$ with $f_n(\log(1/\epsilon_n))=f_n(0)=h(\epsilon_n)$.
    
    We now wish to bound $\dinil{-}f_n(x+\log(1/\theta))$ for $x\in(0,x^*]$.
    First, note that
    \begin{equation*}
        x=\log\left(\frac{\alpha-h(\epsilon_n)}{\alpha-h(\theta)}\right).
    \end{equation*}
    Then rearranging~\eqref{e:h-bound-general}, we obtain
    \begin{equation*}
        \diniu{+}h(\theta)\leq (h(\theta)-\lambda)\left(\frac{\diniu{+}h(\theta)}{h(\theta)-\alpha}+\frac{1}{\theta}\right).
    \end{equation*}
    Since $h(\theta)<\alpha$, $\diniu{+}h(\theta)<\frac{\alpha-h(\theta)}{\theta}$ so that $\frac{\diniu{+}h(\theta)}{h(\theta)-\alpha}+\frac{1}{\theta}>0$.
    Therefore 
    \begin{equation}\label{e:hbd-2}
        \frac{\diniu{+}h(\theta)}{\frac{\diniu{+}h(\theta)}{\alpha-h(\theta)}-\frac{1}{\theta}}\geq\lambda-h(\theta)=\lambda-f_n(x+\log(1/\theta)).
    \end{equation}
    But if $x$ and $\theta$ are related as above, $x+\log(1/\theta)$ is a smooth function of $\theta$ and $h(\theta)$, and $h(\theta)=f_n(x+\log(1/\theta))$, and $x$ decreases as $\theta$ decreases. Thus  
    \begin{equation*}
        \diniu{+}h(\theta)= \dinil{-}f_n(x+\log(1/\theta))\cdot\left(\frac{\diniu{+}h(\theta)}{\alpha-h(\theta)}-\frac{1}{\theta}\right)
    \end{equation*}
    which when combined with~\eqref{e:hbd-2} yields $\dinil{-}f_n(x+\log(1/\theta))\geq \lambda-f_n(x+\log(1/\theta))$.
    Note that we have shown that
    \begin{equation*}
        \dinil{-}f_n(x)\in[\lambda-f_n(x),\alpha-f_n(x)]
    \end{equation*}
    for all $x\in(0,\log(1/\epsilon_n)]$.
    \begin{proofpart}
        Construction of the valleys $e_n\colon[0,w_n]\to[\gamma_n,h(\epsilon_n)]$ where $w_n$ is given in~\eqref{e:wn-def} for $n\in\N$.
    \end{proofpart}
    Set
    \begin{equation*}
        w^*\coloneqq\log\left(\frac{h(\epsilon_n)-\lambda}{\gamma_n-\lambda}\right)
    \end{equation*}
    and for $x\in[0,w^*]$ define $e_n(x)=\lambda-(\lambda-h(\epsilon_n))\exp(-x)$.
    Observe that $e_n(w^*)=\gamma_n$.
    Let
    \begin{equation}\label{e:wn-def}
        w_n\coloneqq w^*+\log\left(\frac{\alpha-\gamma_n}{\alpha-h(\epsilon_{n+1})}\right)
    \end{equation}
    and for $x\in[w^*,w_n]$ define $e_n(x)=\alpha-(\alpha-\gamma_n)\exp(-x+w^*)$.
    Of course, $e_n(w_n)=h(\epsilon_{n+1})$.
    It is clear that $\dinil{-}e_n(x)=\lambda-e_n(x)$ for $x\in(0,w^*]$ and $\dinil{-}e_n(x)=\alpha-e_n(x)$ for all $x\in(w^*,w_n]$.
\begin{figure}[ht]
\center{\includegraphics[width=.98\textwidth]
       {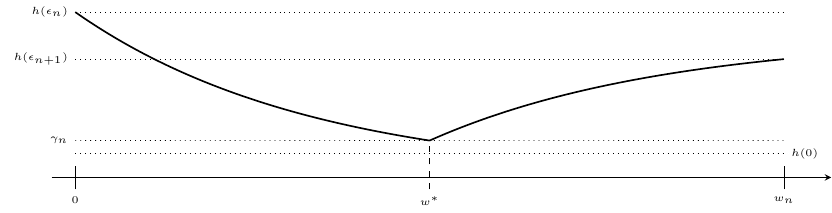}}
       \caption{\label{f:connector-ctr}
       The construction of the valley~$e_n$.}
\end{figure}
    
    \begin{proofpart}
        Construction of $g\in\mathcal{G}(\lambda,\alpha)$ and the corresponding Moran set $C$.
    \end{proofpart}
    Let $\tilde{g}$ denote the concatenation of the sequence $(f_1,e_1,f_2,e_2,\dotsc)$. 
    By Lemma~\ref{l:g-check}, $\tilde{g}$ satisfies the hypotheses of Lemma~\ref{l:exact-construction} (note that $g$ is not rapidly decreasing since since $\limsup_{x\to\infty}\tilde{g}(x)>0$), and get a corresponding function $g$ and sequence $\bm{r}$.
    Let~$g$ and~$\bm{r}$ be the function and sequence respectively given by this lemma. 
    Note that $g\in\mathcal{G}(\lambda,\alpha)$, and let $C=C(\bm{r})$ denote the corresponding Moran set.
    That $\overline{\dim}_\theta C=h(\theta)$ for $\theta\in(0,1]$ follows by definition of the functions $f_n$ and the fact that
    \begin{equation*}
        \lim_{n\to\infty}\left(\sup_{x\in[0,w_n]}e_n(x)\right)\leq\lim_{\theta\to 0}h(\theta).
    \end{equation*}
    Moreover, Lemma~\ref{l:Moran-formula} directly gives that $\dimH C=\underline{\dim}_\theta C = h(0)$ for $\theta\in[0,1]$, $\lambda\leq\dimL C$, and $\dimA C\leq\alpha$.
    
        To see that $\dimA C\geq\alpha$, note that the derivative of the strictly increasing part of each mountain is uniformly bounded above by $\alpha$. Therefore, for all $n \in \N$, the length of the domain of the $n$\textsuperscript{th} mountain can be uniformly bounded below:
    \begin{equation*}
        \log \left(\frac{1}{\epsilon_n}\right) \geq \log \left(\frac{1}{\epsilon_1}\right) > \frac{h(1)-h(\epsilon_1)}{\alpha} > 0.
    \end{equation*}
    Similarly, for all $n$ sufficiently large, the length $w_n$ of the domain of the $n$\textsuperscript{th} valley can be bounded below by $1/(\alpha n)$.
    Therefore there exists $\delta>0$ and a sequence $(b_m)_{m=1}^\infty$ such that for all $m \in \N$ we have $b_m \geq \delta m$, and $\diniu{+}g(x)+g(x)=\alpha$ for all $x\in[b_m,b_m+\delta/m]$.
    Then for all $m \in \N$,
    \begin{equation*}
        \exp\left(b_m + \frac{\delta}{m}\right) - \exp(b_m) =  \left(\exp\left(\frac{\delta}{m}\right) - 1\right) \cdot \exp(b_m) \geq \frac{\delta \cdot \exp(\delta m)}{m}
    \end{equation*}
    which diverges to $\infty$ as $m \to \infty$. 
    Thus we can define a function $\psi \colon \R \to \R^+$ such that $\psi(b_m) = b_m + \delta/m$ for all $m\in\N$ and $\lim_{x\to\infty}(\exp(\psi(x))-\exp(x))=\infty$.
    In particular, $D^+g(y) + g(y) = \alpha$ for all $y \in [b_m,\psi(b_m)]$ and infinitely many $m$.
    By part~\ref{i:A-dim-block} of Lemma~\ref{l:Moran-formula}, it follows that $\dimA C\geq\alpha$.
    An analogous application of part~\ref{i:L-dim-block} gives that ${\dimL C\leq\lambda}$. 
    
    Finally, we address the case when $h(1)=0$. We avoid using Lemma~\ref{l:exact-construction}; an alternative strategy would be to apply this lemma to a carefully-chosen function $\tilde{g}$ which is not rapidly decreasing. 
    By the same arguments as \cite[Lemma~3.2]{Olson2016cantor}, if $M$ is any homogeneous Moran set with notation as above, then for all $K\in\N$,
    \begin{equation}\label{e:assouad-formula-ors}
        \dimA M=\limsup_{n\to\infty}\sup_{k\geq K}\frac{n d\log 2}{\log(\rho_k/\rho_{k+n})}.
    \end{equation}
    Therefore if $\alpha=0$, then choosing the sequence $\bm{r} \in (0,1/2]^{\N}$ inductively such that $\rho_{n+1} \leq \rho_n^{\exp(n)}$ for all $n \in \N$, we have $\dim_{\mathrm A} M = 0$. 
    Now suppose that $\alpha \in (0,d]$. Let $n_1 = 1$, and for $k \in \N$, inductively define $n_{k+1} \coloneqq n_{k} + k$. 
    Define a homogeneous Moran set $M$ which satisfies $\rho_{n_k} \leq \rho_{n_k - 1}^{\exp(n_k)}$ for all $k \in \N$, and $r_j = 2^{-d/\alpha}$ for all integers $j$ which are not of the form $n_k$. 
    It follows directly from~\eqref{e:assouad-formula-ors} that $\dim_{\mathrm A} M = \alpha$, and after a short calculation, Proposition~\ref{p:int-dim} gives that $\overline{\dim}_{\mathrm B} M = 0$, as required. 
\end{proof}

\subsection{Prescribing the intermediate dimensions}\label{ss:pres}
We can get more varied behaviour for the lower intermediate dimensions by taking a finite union of Moran sets, as illustrated by the following result. 
\begin{prop}\label{p:finite-union}
    Suppose $g_i\in\mathcal{G}(0,d)$ for $i=1,\dotsc,m$ have corresponding sequences $\bm{r}_i \in (0,1/2]^{\N}$ satisfying
    \begin{equation*}
        |g_i(x)-s_{\bm{r}_i}(\exp(-\exp(x)))|\leq d\log(2)\exp(-x).
    \end{equation*}
    Let $M$ be a disjoint union of translations of the homogeneous Moran sets $C(\bm{r}_i)$.
    Then for $\theta\in(0,1]$,
    \begin{enumerate}
        \item $\displaystyle\overline{\dim}_\theta M=\limsup_{x\to\infty}\max_{i=1,\dotsc,m}\bigl(\inf_{y\in[x,x+\log(1/\theta)]}g_i(y)\bigr)$,
        \item $\displaystyle\underline{\dim}_\theta M=\liminf_{x\to\infty}\max_{i=1,\dotsc,m}\bigl(\inf_{y\in[x,x+\log(1/\theta)]}g_i(y)\bigr)$,
        \item $\displaystyle\dimH M=\max_{i=1,\dotsc,m}\liminf_{x\to\infty}g_i(x)$.
    \end{enumerate}
\end{prop}
\begin{proof}
It is a straightforward exercise to verify these dimension formulae.
\end{proof}
Suppose $h(\theta)$ satisfies $h(\epsilon)=h(0)$ for some $\epsilon>0$, and let $g$ denote the infinite concatenation of a mountain $f\colon[0,\log(1/\epsilon)]\to(0,d)$ constructed as in Theorem~\ref{t:upper-h-form}.
If $C$ denotes the corresponding Moran set, then $\overline{\dim}_\theta C=h(\theta)$.
Now suppose $N$ is large, and define functions $g_i\coloneqq\kappa_{w_i}(g)$ where $w_i=\frac{(i-1)}{N}\log(1/\epsilon)$ for each $i\in\{1,\dotsc,N\}$.
Write $A=d\log(1/\epsilon)$.
Then if $x$ is arbitrary, since the $g_i$ are Lipschitz continuous with constant $d$, there is some $i$ depending on $x$ such that
\begin{equation*}
    \inf_{y\in[x,x+\log(1/\theta)]}g_i(y)\geq h(\theta)-\frac{A}{N}
\end{equation*}
for all large $x$.
In particular, if $M$ denotes the set given by Proposition~\ref{p:finite-union}, this implies that
\begin{equation*}
    h(\theta)-\frac{A}{N}\leq\underline{\dim}_\theta M\leq\overline{\dim}_\theta M= h(\theta).
\end{equation*}
In other words, by taking a finite union of homogeneous Moran sets, we can ensure that the upper and lower intermediate dimensions are arbitrarily close.

Motivated by this observation, we now construct a set such that the intermediate dimensions exist and are given by a prescribed formula $h(\theta)$.
At a fixed scale $\delta>0$, the set $M$ will look like a finite union of Moran sets each with the same upper intermediate dimensions.
As $\delta$ goes to zero, the resolution increases, so that the intermediate dimensions exist.
The construction here is mildly complicated by the fact that the mountains $f_n$ and valleys $e_n$ can have arbitrarily large support if $h(\theta)>h(0)$ for all $\theta>0$.
\begin{theorem}\label{t:gen-h-form}
    Let $0\leq\lambda\leq\alpha\leq d$ and let $h\in\mathcal{H}(\lambda,\alpha)$.
    Then there exists a compact perfect set $M$ such that $\dimL M=\lambda$, $\dimA M = \alpha$ and
    \begin{equation*}
        \dim_\theta C=h(\theta)
    \end{equation*}
    for all $\theta\in[0,1]$.
\end{theorem}
\begin{proof}
    As in the proof of Theorem~\ref{t:upper-h-form}, we will assume that $\lambda\leq h(0)<h(\theta)<\alpha$ for all $\theta\in(0,1]$, and that there exists $\theta_0 > 0$ such that $h(\theta_0) < h(1)$. 
    The remaining cases follow by similar, but slightly easier, arguments. 
    \begin{proofpart}
        Construction of the set $M$.
    \end{proofpart}
    As in the proof of Theorem~\ref{t:upper-h-form}, fix non-increasing sequences $(\epsilon_n)_{n=1}^\infty$ and $(\gamma_n)_{n=1}^\infty$ and construct corresponding mountains $(f_n)_{n=1}^\infty$ defined on intervals $[0,z_n]$ and valleys $(g_n)_{n=1}^\infty$ defined on intervals $[0,w_n]$, where $z_n=\log(1/\epsilon_n)$ and $w_n$ is defined as in~\eqref{e:wn-def}. 
    We may choose $\epsilon_n$ and $\gamma_n$ such that $w_n+z_n=2^n$.
    Let $\Psi=\{0,1\}\times\{0,1,2,3\}$ and let $\Psi^*=\bigcup_{n=0}^\infty\Psi^n$.
    We first associate to each $\eta\in\Psi^*$ a number $a(\eta)\in[0,\infty)$ as follows.
    Given $k\in\N$ and $i=(u,v)\in\Psi$, we define
    \begin{equation*}
        \psi(k,i)= u 2^{-k}+v 4^{k-1}
    \end{equation*}
    and then for $\eta=(i_1,\dotsc,i_k)$, we set
    \begin{equation*}
        a(\eta)=\sum_{n=1}^k\psi(n,i_n).
    \end{equation*}
    Observe that $a(\Psi^k)=\{j2^{-k}:j\in\Z\}\cap[0,4^{k})$.
    
    For $k\in\N$ and $i\in\Psi$, we define $c_{k,i}(x)=h(\epsilon_k)$ for all $x\in[0,\psi(k,i)]$.
    Now for each $\eta=(i_1,\dotsc,i_n)\in\Psi^*$, let $\tilde{g}_\eta$ denote the concatenation of the sequence
    \begin{equation*}
        (f_1,e_1,c_{1,i_1},f_2,e_2,c_{2,i_2},\dotsc,f_n,e_n,c_{n,i_n},f_{n+1},e_{n+1},f_{n+2},e_{n+2},\dotsc).
    \end{equation*}
    Set $g_\eta\coloneqq\kappa_{w_0}(\tilde{g}_\eta)$, where $w_0$ is guaranteed by Lemma~\ref{l:exact-construction}, and $w_0$ does not depend on the choice of $\eta$ since $g_\eta(0)=f_1(0)$ for all $\eta$, and moreover $w_0$ can be taken to be arbitrarily large.

    There is a sequence $\bm{r}(\eta)\coloneqq(r_j(\eta))_{j=1}^\infty  \in (0,1/2]^{\N}$ such that for all $x\geq w_0$,
    \begin{equation*}
        |s_\eta(\exp(-\exp(x)))-g_{\eta}(x)|\leq d\log(2)\cdot\exp(-x),
    \end{equation*}
    where $s_\eta\coloneqq s_{\bm{r}(\eta)}$.
    Let $\varnothing$ denote the word of length $0$, and let $\rho_k=r_1(\varnothing)\dotsm r_k(\varnothing)$.
    For $k\geq 0$, let 
    \[ y_k=w_0+\sum_{i=1}^k(w_i+z_i)=w_0+2^{k+1}-1.\]
    Then let $n_k$ be the maximal index such that $\log\log(1/\rho_{n_k})\leq y_k$. %
    Choosing $w_0$ large, we may assume that $n_k\geq 3k$ for all $k\in\N$.
    Let $\mathcal{I}=\{0,1\}^d$ and let $L\colon\mathcal{I}^3\to\Psi$ be given by $L(\bm{i},\bm{j},\bm{k})=(\bm{i}^{(1)},\bm{j}^{(1)}+2(\bm{k}^{(1)}))$.
    For $\ell\in\N$, we let $k_\ell$ denote the maximal index such that $n_{k_\ell}\leq\ell$.
    We then define a map $\Lambda\colon\mathcal{I}^*\to\Psi^*$ by
    \begin{equation*}
        \Lambda(\bm{i}_1,\dotsc,\bm{i}_\ell) = (L(\bm{i}_1,\bm{i}_2,\bm{i}_3),L(\bm{i}_4,\bm{i}_5,\bm{i}_6),\dotsc,L(\bm{i}_{3(k_\ell-1)+1},\bm{i}_{3(k_\ell-1)+2},\bm{i}_{3(k_\ell-1)+3})).
    \end{equation*}
    This is well-defined since $\ell\geq n_{k_\ell}\geq 3k_\ell$.
    
    We now construct our inhomogeneous Moran set $M$ as follows.
    Given a word $\sigma=(\bm{i}_1,\dotsc,\bm{i}_\ell)\in\mathcal{I}^\ell$, let $\eta=\Lambda(\sigma)$.
    We then set $S_\sigma=S^1_{\bm{i}_1,\eta}\circ\dotsb\circ S^\ell_{\bm{i}_\ell,\eta}$, where $S^i_{\bm{i},\eta}(x)=r_i(\eta)\cdot x+b^i_{\bm{i}}(\eta)$ with
    \begin{equation*}
        b^i_{\bm{i}}(\eta)^{(j)} \coloneqq
        \begin{cases}
            0 &\mbox{ if } \bm{i}^{(j)}=0,\\
            1-r_i(\eta) &\mbox{ if } \bm{i}^{(j)}=1.
        \end{cases}
    \end{equation*}
    We now set
    \begin{equation*}
        M_\ell \coloneqq \bigcup_{\sigma\in\mathcal{I}^\ell}S_\sigma([0,1]^d).
    \end{equation*}
    Note that if $\sigma$ is a prefix of $\tau$, then $\Lambda(\sigma)$ is a prefix of $\Lambda(\tau)$ and therefore $S_\sigma([0,1]^d)\supseteq S_\tau([0,1]^d)$.
    Thus $M_0\supseteq M_1\supseteq\dotsb$, so the set
    \begin{equation*}
        M\coloneqq\bigcap_{\ell=0}^\infty M_\ell
    \end{equation*}
    is non-empty.
    
    Intuitively, at a fixed scale $\delta$, $M$ looks like a union of $8^k$ homogeneous Moran sets corresponding to the sequences $\bm{r}(\eta)$ for $\eta\in\Psi^k$.
    We can make this precise in the following sense.
    For $\eta\in\Psi^k$, we define
    \begin{align*}
        \mathcal{B}_k(\eta) &= \{ \, (\sigma_1,\dotsc,\sigma_k)\in\mathcal{I}^{3k}:L(\sigma_i)=\eta_i\text{ for each }1\leq i\leq k \, \},\\*
        J_\eta &= \bigcup_{\sigma\in\mathcal{B}_k(\eta)}S_\sigma([0,1]^d).
    \end{align*}
    Let $C_\eta\coloneqq C(\bm{r}(\eta))$ denote the homogeneous Moran set corresponding to the function $g_\eta$.
    Let $\ell\in\N$ satisfy
    \begin{equation*}
        y_k+a(\eta^-)<\log\log(1/(r_1(\eta)\dotsm r_\ell(\eta)))\leq y_{k+1}+a(\eta),
    \end{equation*}
    where $\eta^-\in\Psi^{k-1}$ is the unique prefix of $\eta$.
    Since $g_\varnothing(y_{k})=g_\eta(y_{k}+a(\eta^-))$, if $\sigma\in \mathcal{I}^\ell$, then $\eta$ is a prefix of $\Lambda(\sigma)$.
    Moreover, if $\tau\in\Psi^*$ is a word with $\eta$ as a prefix, then $g_\tau(x)=g_\eta(x)$ for all $x\leq y_{k+1}+a(\eta)$.
    Thus for any such $\ell$, we have
    \begin{equation}\label{e:Ml-formula}
        M_\ell\cap J_\eta=(C_\eta)_\ell\cap J_\eta.
    \end{equation}
    But then if $\eta'$ is a prefix of $\eta$, then $r_\ell(\eta')=r_\ell(\eta)$ for all $\ell$ such that
    \begin{equation}\label{e:eta-valid}
        \log\log(1/(r_1(\eta)\dotsm r_\ell(\eta)))\leq y_{k+1}+a(\eta).
    \end{equation}
    Thus~\eqref{e:Ml-formula} holds for all $\ell$ satisfying~\eqref{e:eta-valid}.
    We also note that $(C_\eta)_\ell\cap J_\eta$ consists of exactly $2^{d\ell-3k}$ hypercubes with diameter $r_1(\eta)\dotsm r_\ell(\eta)$.
    
    \begin{proofpart}
        Proof that $\dim_\theta M=h(\theta)$ for $\theta\in(0,1]$.
    \end{proofpart}
    Fix $\theta\in(0,1]$.
    We first show that $\overline{\dim}_\theta M\leq h(\theta)$.
    Let $\delta$ be sufficiently small such that $\delta\leq\rho_{k_0}$ where $\epsilon_{k_0}\leq\theta$.
    Now let $k$ be such that $\rho_{n_{k}}<\delta^{1/\theta}$.
    It now follows by the same argument as Lemma~\ref{l:flat-covers} that for each $\eta\in\Psi^k$, with $s_\eta\coloneqq\inf_{\phi\in[\delta^{1/\theta},\delta]}s_\eta(\phi)$,
    \begin{align*}
        \inf\Bigl\{ \, \sum_{U \in \mathcal{U}} |U|^{s_\eta} : \mathcal{U}\text{ is a $(\delta,\theta)$-cover of $(C_\eta)_{\ell(\eta)}\cap J_\eta$} \, \Bigr\}\leq 8^{-k},
    \end{align*}
    where $\ell(\eta)$ is minimal such that $r_1(\eta)\dotsm r_\ell(\eta)\leq\delta^{1/\theta}$.
    But $\ell(\eta)$ satisfies~\eqref{e:eta-valid} since $\rho_{n_{k}}<\delta^{1/\theta}$, so $M\subseteq \bigcup_{\eta\in\Psi^k}(C_\eta)_\ell\cap J_\eta$.
    Therefore, $s_\eta\leq h(\theta)+d\log(2)\cdot\exp(-y_{n_k})$.
    This implies that $\overline{\dim}_\theta M\leq h(\theta)$.
    
    Now fix $\epsilon>0$; we will show that $\underline{\dim}_\theta M\geq h(\theta)-(2+d)\epsilon$.
    The various variables in this proof are depicted in Figure~\ref{f:offset-diag}.
    Let $k$ be such that $2^{-k}\leq\epsilon$.
    Let $\delta>0$ be small and let $x\coloneqq\log\log(1/\delta)$.
    We may assume that
    \begin{enumerate}%
        \item $d\log(2)\exp(-x)\leq\epsilon$,
        \item $x\geq y_k$, and
        \item $x\geq y_m$ for some $m$ with $\epsilon_m\leq\theta$.
    \end{enumerate}
    
    For each $m\in\N$, there is some $v_m$ such that $f_m(v_m)=h(\theta)$.
    Equivalently, $g_\varnothing(y_m+v_m)=h(\theta)$.
    Let $m$ be maximal such that $y_m+v_m\leq x$.
    Since $y_{m+1}+v_{m+1}-(y_m+v_m)\leq 4^m$, there is some $\eta_0\in\Psi^m$ such that $|a(\eta_0)-(x-y_m-v_m)|\leq 2^{-k}$.
    Then since $\diniu{+}g_{\eta}(x)\in[-d,d]$ for all $x\in\R$,
    \begin{equation*}
        \inf_{\phi\in[\delta^{1/\theta},\delta]}s_\eta(\phi)\geq \inf_{y\in[x,x+\log(1/\theta)]}g_\eta(y)-\epsilon\geq h(\theta)-(1+d)\epsilon.
    \end{equation*}
    Set $s=h(\theta)-(2+d)\epsilon$.
    Again by the same argument as Lemma~\ref{l:flat-covers}, since $x+\log(1/\theta)<y_{m+1}+v_{m+1}< y_{m+2}$, with $\eta\in\Psi^{m+1}$ satisfying $g_{\eta_0}=g_\eta$, we have
    \begin{align*}
        C\cdot\frac{\delta^{-\epsilon}}{8^m}&\leq \inf\Bigl\{ \, \sum_{U \in\mathcal{U}} |U|^{s} :\mathcal{U}\text{ is a $(\delta,\theta)$-cover of $M\cap J_\eta$} \, \Bigr\}\\*                                                        &\leq \inf\Bigl\{ \, \sum_{U \in\mathcal{U}} |U|^{s}:\mathcal{U}\text{ is a $(\delta,\theta)$-cover of $M$} \, \Bigr\}
    \end{align*}
    for some constant $C>0$ independent of $\delta$.
    But $x\geq y_m\geq 2^m-1$, so
    \begin{equation*}
        \frac{\delta^{-\epsilon}}{8^m}\geq\frac{\bigl(\exp(\exp(2^m-1))\bigr)^\epsilon}{8^m} \xrightarrow[m \to \infty]{} \infty
    \end{equation*}
    as required.
\begin{figure}[ht]
\center{\includegraphics[width=.99\textwidth]{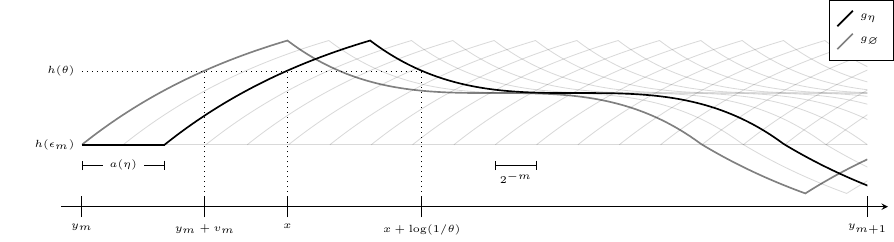}}
        \caption{\label{f:offset-diag} Choice of $g_\eta$ for the lower bound of $\dim_\theta M$.}
\end{figure}

    \begin{proofpart}
        Proof that $\dimH M=h(0)$, $\dimL M=\lambda$, and $\dimA M=\alpha$.
    \end{proofpart}
    It is clear that $\dimH M\geq h(0)$ since $\liminf_{\delta\to 0}s_\eta(\delta)\geq h(0)$ for all $\eta\in\Psi^*$.
    Conversely, let $\epsilon>0$; we will show that $\dimH M\leq h(0)+2\epsilon$.
    Let $n_0$ be sufficiently large such that $\gamma_{n_0}\leq h(0)+\epsilon$.
    Then let $\delta>0$ be sufficiently small such that with $x=\log\log(1/\delta)$, we have $x\geq y_{n_0+1}$ and $d\log(2)\cdot\exp(-x)\leq\epsilon$.
    Let $m$ be such that $x<y_m$.
    For each $\eta\in\Psi^m$, by the choice of $n_0$, there exists some $x\leq x_\eta<y_{m+1}+a(\eta)$ such that
    \begin{equation*}
        g_\eta(x_\eta)= \gamma_m\leq h(0)+\epsilon.
    \end{equation*}
    Then by the same argument as Lemma~\ref{l:flat-covers}, since $d\log(2)\cdot\exp(-x)\leq\epsilon$, with $s=h(0)+2\epsilon$ and with $\ell(\eta)$ minimal so that $\log\log(1/\rho_{\ell(\eta)})\geq x_\eta$, we have
    \begin{align*}
        \inf\Bigl\{ \, \sum_{U \in\mathcal{U}} |U|^{s}:\mathcal{U}\text{ is a $(\delta,0)$-cover of $(C_\eta)_{\ell(\eta)}\cap J_\eta$} \, \Bigr\}\leq 8^{-m}.
    \end{align*}
    Moreover, for $x$ sufficiently large, we can ensure that $\log\log(1/\rho_{\ell(\eta)})\leq y_{m+1}\leq y_{m+1}+a(\eta)$.
    Thus $M\subseteq \bigcup_{\eta\in\Psi^m}(C_\eta)_{\ell(\eta)}\cap J_\eta$, so
    \begin{equation*}
        \inf\Bigl\{\sum_{U \in\mathcal{U}} |U|^{s} : \mathcal{U}\text{ is a $(\delta,0)$-cover of $M$}\Bigr\}\leq 1.
    \end{equation*}
    But $\delta>0$ was arbitrary, so $\dimH M\leq h(0)+2\epsilon$, as required.
    
    Now we will show that $\dimA M=\alpha$; the proof that $\dimL M=\lambda$ follows similarly.
    Observe that there is some $\delta>0$ such that $\diniu{+}g_\varnothing(x)+g_\varnothing(x)=\alpha$ for all $m\in\N$ and $x\in[y_{m+1}-\delta,y_{m+1}]$.
    Let $\tau=\{(0,0),(0,0,),\dotsc,(0,0)\}\in\Psi^m$ and observe that $g_\varnothing=g_\tau$.
    Then if $\log\log(1/\rho_\ell)\in[y_{m+1}-\delta,y_{m+1}]$, we have $M_\ell\cap J_\tau = C_\varnothing\cap J_\tau$.
    Thus $\dimA M\geq\alpha$ follows by the same computation from Theorem~\ref{t:upper-h-form}.
    
    Conversely, it suffices to show that for all $\epsilon>0$ there exist $a,\ell_0 > 0$ such that for all $\ell \geq \ell_0$ and $I\in M_\ell$ with $|I|=R$, and all $r \in (0,aR)$, we have 
    \begin{equation*}
        N_r(I\cap M)\leq (R/r)^{\alpha+2 \epsilon}.
    \end{equation*}
    Let $r>0$ and let $m,k$ be minimal such that $\log\log(1/r)\leq\log\log(1/\rho_m)\leq y_k$.
    First suppose $\ell\geq 3k$, and suppose $I \in M_\ell$ satisfies $|I|=R$. 
    Then there exists a unique $\eta\in\Psi^k$ such that $I\subset J_\eta$, so 
    \begin{equation*}
        (C_\eta)_j\cap I\cap M_j=I\cap M_j
    \end{equation*}
    for all $\ell\leq j\leq m$.
    Then since $\diniu{+}g_\eta+g_\eta\leq\alpha$, a similar computation to the proof of Lemma~\ref{l:Moran-formula} gives that for all $\epsilon>0$ there exists $\ell_0$ depending only on $\epsilon$ and $M$ such that if we additionally assume that $\ell \geq \ell_0$, then
    \begin{equation*}
        N_r(I\cap M)\leq (2^d)^{m-\ell}\leq\left(r_{\ell+1}(\eta)\dotsm r_{m}(\eta)\right)^{-\bigl(\alpha+\frac{2d}{m-\ell}\bigr)}\leq(R/r)^{\alpha+\frac{2d}{m-\ell}}.
    \end{equation*}
    
    For the other case, suppose $\ell<3k$.
    Let $\eta\in\Psi^k$ satisfy $J_\eta\cap I\neq\varnothing$ and let $\sigma\in\mathcal{I}^{3k}$ satisfy $S_\sigma([0,1]^d)\subseteq I\cap J_\eta$.
    Again,
    \begin{equation*}
        N_r(I\cap(C_\eta)_m)\leq(2^d)^{m-\ell}\leq(R/r)^{\alpha+\frac{2d}{m-\ell}},
    \end{equation*}
    so 
    \begin{align*}
        N_r(I\cap(C_\eta)_m\cap S_\sigma([0,1])^d)&\leq (2^d)^{m-3k}\\
        &=(2^d)^{\ell-3k}(2^d)^{m-\ell}\\*
        &\leq (2^d)^{\ell-3k}(R/r)^{\alpha+\frac{2d}{m-\ell}}.
    \end{align*}
    But $I\cap(C_\eta)_m\cap S_\sigma([0,1]^d)=I\cap M_m\cap S_\sigma([0,1]^d)$ and there are precisely $(2^d)^{3k-\ell}$ words $\sigma$, so 
    \begin{equation*}
        N_r(I\cap M)\leq N_r(I\cap M_m)\leq (R/r)^{\alpha+\frac{2d}{m-\ell}}.
    \end{equation*}
    We can therefore choose $a$ small enough so that in either case $N_r(I\cap M)\leq (R/r)^{\alpha+2 \epsilon}$, as required.
\end{proof}
Using this construction, along with the preceding construction for the upper intermediate dimensions, we can now simultaneously prescribe the upper and lower intermediate dimensions.
\begin{corollary}\label{c:upper-lower-match}
    Let $0\leq\lambda\leq\alpha\leq d$ and let $\underline{h},\overline{h}\in\mathcal{H}(\lambda,\alpha)$ satisfy $\underline{h}(0)=\overline{h}(0)$ and $\underline{h}\leq\overline{h}$.
    Then there exists a compact perfect set $M$ such that $\dimL M=\lambda$, $\dimA M = \alpha$ and
    \begin{align*}
        \underline{\dim}_\theta C&=\underline{h}(\theta) & \overline{\dim}_\theta C&=\overline{h}(\theta)
    \end{align*}
    for all $\theta\in[0,1]$.
\end{corollary}
\begin{proof}
    Let $E,F$ be disjoint compact perfect sets such that $\dimL E=\dimL F=\lambda$, $\dimA E=\dimA F=\alpha$, $\dimH E=\dimH F=\underline{h}(0)=\overline{h}(0)$, and for $\theta\in(0,1]$
    \begin{equation*}
        \underline{\dim}_\theta F\leq\dim_\theta E=\underline{h}(\theta)\leq\overline{h}(\theta)=\overline{\dim}_\theta F.
    \end{equation*}
    For example, such a set $E$ is provided by Theorem~\ref{t:gen-h-form} and such a set $F$ is provided by Theorem~\ref{t:upper-h-form}.
    Let $M=E\cup F$.
    Then $\dimL M=\min\{\dimL E,\dimL F\}=\lambda$, $\dimA M=\max\{\dimA E,\dimA F\}=\alpha$,
    \begin{equation*}
        \underline{h}(\theta)=\underline{\dim}_\theta E\leq \underline{\dim}_\theta M\leq\max\{\overline{\dim}_\theta E,\underline{\dim}_\theta F\}=\underline{h}(\theta),
    \end{equation*}
    and
    \begin{equation*}
        \overline{\dim}_\theta M=\max\{\overline{\dim}_\theta E,\overline{\dim}_\theta F\}=\overline{h}(\theta)
    \end{equation*}
    for $\theta\in(0,1]$.
    Thus $M$ satisfies the requirements.
\end{proof}

\chapter{Infinitely generated attractors}\label{s:infinite}

\section{Introduction}

\subsection{Background}

This chapter focuses on infinite IFSs and is based on our joint paper~\cite{Banaji2021infinite} with J.~M.~Fraser. 
The dimension theory of limit sets of finite iterated function systems (IFSs) has been studied extensively since Hutchinson's paper~\cite{Hutchinson1981attractor}. In a seminal 1996 paper~\cite{Mauldin1996iifs} Mauldin and Urbański extended the theory to infinite iterated function systems (IIFSs) consisting of countably many contractions, with the contraction ratios uniformly bounded above by some $\rho < 1$. The dimension theory of IIFSs has been studied further in~\cite{Mauldin1999ctdfrac,Ngai2016iifs,
Chu2020iifs,Banaji2022assouad,
Kaenmaki2014infiniteaffine,
Mauldin1995iifssurvey:Fractals1995} and many other works. 
Mauldin and Urbański paid particular attention to (infinite) conformal iterated function systems (CIFSs, defined in Definition~\ref{cifs}), where the contractions are conformal and are sufficiently separated. 
Approximations to one such limit set are shown in Figure~\ref{f:infiniteselfsim}. 
\begin{figure}[ht]
\center{\includegraphics[width=.9\textwidth]{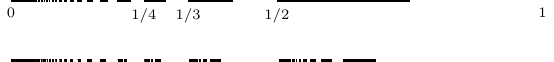}}
\caption{\label{f:infiniteselfsim}
First and second level cylinders for an infinitely generated self-similar set.
}
\end{figure}

There are many similarities, but also many differences, between finite and infinite iterated function systems. One notable difference is that Hausdorff and box dimension coincide for the limit set of every finite CIFS, but can differ for infinite CIFSs, as the presence of infinitely many maps can cause the limit set to have greater inhomogeneity in space. 
In particular, Mauldin and Urbański showed that for a CIFS the Hausdorff dimension can be determined from a certain topological pressure function defined in~\eqref{MUpressure} below (see \cite[Theorem~3.15]{Mauldin1996iifs}). 
The same authors proved that the upper box and packing dimensions are given by the maximum of the Hausdorff dimension of the limit set and the upper box dimension of images of any given point under the maps in the CIFS (noting that the box dimension of a countable set, unlike the Hausdorff dimension, can be strictly positive), see \cite[Theorem~2.11]{Mauldin1999ctdfrac}. 
They applied their results to sets of irrational numbers whose continued fraction expansions have restricted entries, as these are limit sets of an appropriate CIFS (see Section~\ref{ctdfracsect}).  
 
 The following result describes the Assouad type dimensions of the limit set of a CIFS. 
 The Assouad spectra can display interesting behaviour, such as having two phase transitions. 
 \begin{theorem}[Banaji--Fraser~\cite{Banaji2022assouad}]\label{t:conformalassouad}
 Let $F$ be the limit set of an infinite CIFS and let $P$ be the set of fixed points of the contractions. Then for all $\theta \in (0,1)$, 
 \begin{align*}
 \max\{\dim_{\mathrm H} F,\uasp P \} \leq \asp F &= \uasp F \\*
 &\leq \max_{\phi \in [\theta,1]} \frac{(\phi^{-1} - 1) \overline{\dim}_\mathrm{A}^\phi P + (\theta^{-1} - \phi^{-1}) \ubd F}{\theta^{-1} - 1},
 \end{align*}
 and these bounds are sharp in general. 
 If we assume the additional separation condition that $\overline{S_i}(V) \cap \overline{S_j}(V) = \varnothing$ for all distinct $i,j \in I$ (using notation from Definition~\ref{cifs} below), then 
 \[ 
 \dim_{\mathrm A} F = \max\{ \dim_{\mathrm H} F, \dim_{\mathrm A} P \}.
 \]
 \end{theorem}
 To prevent this thesis from becoming unreasonably long, we omit the proof of Theorem~\ref{t:conformalassouad} and instead refer the reader to~\cite{Banaji2022assouad}. 
In this chapter we study the intermediate dimensions of limit sets of infinite iterated function systems.

 \subsection{Structure of chapter and discussion of results}
 
 In Section~\ref{setting} we introduce notation, define limit sets, and define the notions of dimension we will be working with. We introduce and prove basic properties about the topological pressure function that we will use to obtain bounds for dimensions of the limit sets. We also define conformal iterated function systems (CIFSs) and prove geometric consequences of the definition of a CIFS that we will use in the proof of the main result of this chapter, Theorem~\ref{mainint}. 
 
 In Section~\ref{mainsect} we prove Theorem~\ref{inttypeub}, which gives upper bounds for the Hausdorff, box, intermediate and $\Phi$-intermediate dimensions in terms of the topological pressure function that hold in the very general setting of arbitrary IIFSs (without any conformality assumptions or separation conditions). The proof is an induction argument, using efficient covers at larger scales to construct efficient covers at smaller scales. In the conformal setting, Mauldin and Urbański~\cite{Mauldin1996iifs,Mauldin1999ctdfrac} proved results for the Hausdorff and upper box dimensions. We use our upper bound to prove the main result of this chapter, a simplified version of which we now state. 
 \begin{theorem*}[See Theorem~\ref{mainint} for a stronger statement]
 If $F$ is the limit set of an (infinite) CIFS (defined in Definition~\ref{cifs}) and $P$ is the set of fixed points of the contractions then for all $\theta \in [0,1]$, 
 \[ \uid F = \max\{\dim_\mathrm{H} F, \uid P\}.\] 
 \end{theorem*}
 Our methods also apply to infinite parabolic IFSs. 
 In Example~\ref{proj} we consider an example with intermediate dimensions continuous at $\theta=0$ and apply a result of Burrell, Falconer and Fraser~\cite{Burrell2021projections} to give an upper bound for the upper box dimension of orthogonal projections. 
 
 In Section~\ref{ctdfracsect} we apply our results to give a formula in Theorem~\ref{ctdfracintthm} for the upper intermediate dimensions of sets of irrational numbers whose continued fraction expansions have restricted entries. %
Recalling the discussion in Section~\ref{s:intdimsintro}, we show that the intermediate dimensions of continued fraction sets are continuous at $\theta=0$ apply our results and those in~\cite{Burrell2022brownian} to give information related to H\"older distortion and fractional Brownian images of continued fraction sets. 
We also obtain similar results in Section~\ref{compsect} for sets of complex numbers which have complex continued fraction expansions with restricted entries.

 In Section~\ref{genericsect} we consider the limit sets of `generic' IIFSs, in the same vein as the seminal paper~\cite{Falconer1988generic} where Falconer considered the generic dimension of a (finitely-generated) self-affine set by fixing a set of matrices and randomising the translates in a suitable way. We show that under certain conditions, the limit set of an IIFS with `generic' translates is somewhere dense, and so in particular the box and intermediate dimensions equal the ambient spatial dimension, where `generic' can mean either almost surely with respect to a natural measure, or comeagre with respect to a natural topology. %
 This is in stark contrast to the Hausdorff dimension, which K\"aenm\"aki and  Reeve~\cite{Kaenmaki2014infiniteaffine} showed satisfies an analogue of Falconer's affinity dimension formula for a generic IIFS of affine contractions. 

In this chapter, as in~\cite{Mauldin1996iifs,Mauldin1999ctdfrac}, the separation condition we assume in Definition~\ref{cifs} for a CIFS is the open set condition (OSC). Ngai and Tong~\cite{Ngai2016iifs} and Chu and Ngai~\cite{Chu2020iifs} study the Hausdorff, box and packing dimensions of the limit sets of IIFSs with overlaps that do not satisfy the OSC but do satisfy suitable extensions of the weak separation condition. It is therefore natural to ask (though we will not pursue this) what can be said about the intermediate or $\Phi$-intermediate dimensions of the limit sets of infinite iterated function systems with overlaps that do not satisfy the OSC but perhaps satisfy weaker separation conditions such as the extensions of the weak separation condition considered in~\cite{Ngai2016iifs}. 

\section{Infinite IFSs and pressure functions}\label{setting}

We will work with infinite iterated function systems, defined as in~\cite{Mauldin1996iifs} as follows. 

\begin{defn}\label{iifs}
Let $d \in \N$ and let $X$ be a compact, connected subset of $\Rd$ with more than one point, %
equipped with the metric induced by the Euclidean norm $||\cdot ||$. We say that an \emph{infinite iterated function system (IIFS)} on $X$ is a collection of maps $S_i \colon X \to X$, $i \in I$, where $I$ is a countable index set, such that there exists $\rho \in [0,1)$ such that 
\[ ||S_i(x) - S_i(y)|| \leq \rho ||x-y|| \quad \mbox{for all } x,y \in X \mbox{ and } i \in I.\]
\end{defn}

The assumption that the maps are uniformly contracting will be important when defining the limit set. We now introduce some notation. Define $I_0 \coloneqq \{\varnothing\}$ and $I^* \coloneqq \bigcup_{i=1}^\infty I^n$. We call elements of $I^*$ \emph{finite words} and elements of $I^\N$ \emph{infinite words}. We usually denote words by the letter $w$, and we write $w = i_1 \dotsb i_n$ and $w=i_1i_2\dotsb$ instead of $w= (i_1,\dotsc,i_n)$ and $w=(i_1,i_2,\dotsc)$ respectively. We say that a word in $I^n$ has \emph{length} $n$, and an infinite word has \emph{length} $\infty$. If $w \in I^* \cup I^\N$ and $n \in \N$ does not exceed the length of $w$ then we write $w|_n \coloneqq w_1 \dotsb w_n \in I^n$, and $w|_0 \coloneqq \varnothing$. If $w \in I_0 \cup I^* \cup I^\N$ and $v \in I_0 \cup I^*$ then we say that $v$ is a \emph{prefix} of $w$ if there exists $n \in \{0,1,2,\dotsc\}$ such that $v = w|_n$. 
For $w \in I^n$ we define 
\[ S_w \coloneqq S_{w_1} \circ \dotsb \circ S_{w_n},\]  
and we define $S_\varnothing$ to be the identity function on $X$.

Mauldin and Urbański~\cite{Mauldin1996iifs,Mauldin1999ctdfrac} study the IIFSs in Definition~\ref{cifs}, where the contractions are assumed to extend to \emph{conformal} maps (see~\ref{conformal} below for a formal definition), which means that locally they preserve angles. This assumption is crucial for Mauldin and Urbański's formulae for the Hausdorff and upper box dimensions of the limit set (though there has been some subsequent work on IIFSs consisting of affine contractions~\cite{Kaenmaki2014infiniteaffine,
KaenmakiPreprintinfiniteaffine}). 
The conformality assumption will also be crucial when we prove a formula for the intermediate dimensions in Section~\ref{conformalsect}. In one dimension, conformal maps are simply functions with non-vanishing H{\"o}lder continuous derivative. In two dimensions, they are holomorphic functions with non-vanishing derivative on their domain. In dimension three and higher, by a theorem of Liouville (1850) they have a very restricted form: they are M{\"o}bius transformations, so can be composed from homotheties, isometries, reflections in hyperplanes, and inversions in spheres. 
Recall that $\mathcal{L}_d$ denotes $d$-dimensional Lebesgue measure. 

\begin{defn}\label{cifs}%
A \emph{conformal iterated function system (CIFS)} is an IIFS (as in Definition~\ref{iifs}) which satisfies the following additional properties: 
\begin{enumerate}[label=(\roman*)]

\item\label{osc} (Open set condition (OSC)) 
The set $X$ has non-empty interior $U \coloneqq \mathrm{Int}_{\Rd} X$, and $S_i(U) \subset U$ for all $i \in I$ and $S_i(U) \cap S_j(U) = \varnothing$ for all $i,j \in I$ with $i \neq j$. 

\item\label{cone} (Cone condition) $\inf_{x \in X} \inf_{r \in (0,1)} \mathcal{L}_d (B(x,r) \cap \mathrm{Int}_{\Rd} X)/r^d > 0$. 

\item\label{conformal} (Conformality) There exists an open, bounded, connected subset $V \subset \Rd$ such that $X \subset V$ and such that for each $i \in I$, $S_i$ extends to a $C^{1+\epsilon}$ diffeomorphism from $V$ to an open subset of $V$ which is \emph{conformal}, so for all $x \in V$ the differential $S_i'|_x$ exists, is non-zero, is a similarity map (so $||S_i'|_x (y)|| = ||S_i'|_x||\cdot||y||$ for all $y \in \Rd$), and is $\epsilon$-H{\"o}lder continuous in $x$. Moreover, there exists $\rho \in (0,1)$ such that $||S_i'|| <\rho$ for all $i \in I$, where $||\cdot||$ is the supremum norm over $V$. %

\item\label{bdp} (Bounded distortion property (BDP)) There exists $K>0$ such that $||S_w'|_y|| \leq K||S_w'|_x||$ for all $x,y \in V$ and $w \in I^*$. 

\end{enumerate}
\end{defn}
Mauldin and Urbański~\cite[(2.8)]{Mauldin1996iifs} use a stronger form of the cone condition, but note on page~110 that~\ref{cone} is sufficiently strong for their aims. Both forms of this technical condition will be satisfied if $X$ is `reasonable,' for example if $X$ is convex or has a smooth enough boundary. 
In the published version of~\cite{Banaji2021infinite}, there is a typo in the BDP (it should read `for all $x,y \in V$' rather than `for all $x,y \in X$'). 

For every IIFS, since $|S_{w|_n}(X)| \leq \rho^n |X|$ by the uniform contractivity, the map 
\[ \pi \colon I^\N \to X, \qquad \pi(w) \coloneqq \bigcap_{n=1}^\infty S_{w|_n}(X) \]
is well-defined and continuous. 
We are interested in the following set, which will often be fractal in nature. 
\begin{defn}
The \emph{limit set} or \emph{attractor} of an IIFS is defined by
\[ F \coloneqq \pi(I^\N) = \bigcup_{w \in I^\N} \bigcap_{n=1}^\infty S_{w|_n}(X). \]
\end{defn}
For $w \in I^n$ define $F_w = F_{S_w} \coloneqq S_w(F)$ and $X_w = X_{S_w} \coloneqq S_w(X)$. %
Now, $F$ is clearly non-empty and satisfies the relation 
\begin{equation}\label{attractor} F = \bigcup_{i \in I} F_i. 
\end{equation}
It is the largest (by inclusion) of many sets which satisfy~\eqref{attractor}. If $I$ is finite then $F$ is compact (and is indeed the only non-empty compact set which satisfies~\eqref{attractor} by Hutchinson's application of the Banach contraction mapping theorem~\cite{Hutchinson1981attractor}), but if $I$ is infinite then $F$ will not generally be closed. When $I$ is finite, the limit set $F$ equals the closure of the set of fixed points of all finite compositions of maps in the IFS, and also satisfies $F = \cap_{n=1}^\infty S^n(X)$ (where $S(E) \coloneqq \cup_{i \in I} S_i(E)$ for $E \subseteq X$), but when $I$ is infinite these sets may strictly contain $F$. 
Some cylinder sets in the construction of an infinitely generated self-similar set are shown in Figure~\ref{f:infiniteselfsim} on page~\pageref{f:infiniteselfsim}. 

We define some more quantities that will enable us to define a topological pressure function for the system. For every IIFS, for $w \in I_n$ define 
\begin{align*}
r_w = r_{S_w} &\coloneqq \inf_{x,y \in X, x \neq y} \frac{||S_w(x)-S_w(y)||}{||x-y||}; \\* 
R_w = R_{S_w} &\coloneqq \sup_{x,y \in X, x \neq y} \frac{||S_w(x)-S_w(y)||}{||x-y||},
\end{align*}
noting that $0 \leq r_w \leq R_w \leq \rho$. The value $R_w$ is the smallest possible Lipschitz constant for $S_w$, and these constants are clearly submultiplicative; $R_{vw} \leq R_v R_w$ for all $v,w \in I^*$. 
For $n \in \N$ define $M_n \coloneqq \{ \, S_w : w \in I^n \, \}$, and for $t \in [0,\infty)$ define 
\begin{equation}\label{e:definephin} 
\phi_n(t) \coloneqq \sum_{\sigma \in M_n} R_\sigma^t \in [0,\infty].
\end{equation}
 Note that as in~\cite{Ngai2016iifs}, we sum over $\sigma \in M_n$ instead of $w \in I^n$ so that distinct words $w$ that give rise to the same $S_w$ contribute only one term in the sum (exact overlaps are removed). 
\begin{lma}\label{fekete}
For every IIFS, for all $t \in (0,\infty)$, using the convention $\log \infty = \infty$ and $\log 0 = -\infty$, 
\[ \frac{1}{n}\log \phi_n(t) \xrightarrow[n \to \infty]{} \inf_{n \in \N} \frac{1}{n}\log \phi_n(t) \in [-\infty,\infty]. \]
\end{lma}
\begin{proof}
By the submultiplicativity of the Lipschitz constants, if $n,m \in \N$ then 
\[ \log \phi_{n+m} \leq \log \left(\sum_{\sigma \in M_n} \sum_{\tau \in M_m} (R_\sigma R_\tau)^t\right) = \log \left(\sum_{\sigma \in M_n}R_\sigma^t \sum_{\tau \in M_m} R_\tau^t\right) = \log \phi_n + \log \phi_m. \]
Therefore the sequence $(\log \phi_n)_{n=1}^\infty$ is subadditive, so the claim follows from Fekete's lemma. %
\end{proof}
In light of Lemma~\ref{fekete} we can make the following definition, which will later be used when giving bounds and formulae for the different notions of dimension for the limit sets. 
\begin{defn}
For an IIFS, define the \emph{(topological) pressure function} $\overline{P} \colon (0,\infty) \to [-\infty,\infty]$ by 
\[ \overline{P}(t) \coloneqq \lim_{n \to \infty} \frac{1}{n}\log \phi_n(t) = \inf_{n \in \N} \frac{1}{n}\log \phi_n(t).\]
\end{defn}

\begin{lma}\label{pressurestrict}
For every IIFS, $\overline{P}$ is a decreasing function, and if $0<t<s<\infty$ and $\overline{P}(t) \in \mathbb{R}$ then $\overline{P}(t)>\overline{P}(s)$. %
\end{lma}

\begin{proof}
For all $n \in \N$, for all $w \in I^n$, $R_w \leq \rho^n < 1$, so $\phi_n$ is a decreasing function. Therefore $\overline{P}$ is a decreasing function. If $0<t<s<\infty$ and $\overline{P}(t) \in \mathbb{R}$ then for all $n \in \N$, $\phi_n(s) \leq \rho^{(s-t)n} \phi_n(t)$ so $\frac{1}{n} \log \phi_n(s) \leq (s-t) \log \rho + \frac{1}{n} \log \phi_n(t)$, hence $\overline{P}(s) \leq (s-t) \log \rho + \overline{P}(t) < \overline{P}(t)$, as required. 
\end{proof}

\begin{defn}\label{d:h}
Throughout this chapter, using the convention that $\inf \varnothing = \infty$, we define the \emph{finiteness parameter} of the system 
\[ \theta_S \coloneqq \inf\{ \, t > 0 : \overline{P}(t) < \infty \, \} \in [0, \infty], \]
and the quantity 
\[ h \coloneqq \inf\{ \, t > 0 : \overline{P}(t) < 0\, \} \in [0, \infty]. \]
\end{defn}
We use the letter $h$ because we will see that it is related to the Hausdorff dimension of the limit set. 
For all $n \in \N$ and $t \in [0,\infty)$, 
\[ \phi_n(t) \leq \sum_{w \in I^n}R_w^t \leq \sum_{i_1\dotsb i_n \in I^n} \prod_{k=1}^n R_{i_k}^t.\]
Therefore if $\phi_n(t)$ is replaced by either $\sum_{w \in I^n}R_w^t$ or $\sum_{i_1\dotsb i_n \in I^n} \prod_{k=1}^n R_{i_k}^t$ in the definition of the pressure function then the resulting functions would overestimate $\overline{P}(t)$. Thus the infimal values of $t>0$ for which these new functions are negative provide upper bounds for $h$. These may be easier to compute than $h$ itself. For all $n \in \N$, the function $\frac{1}{n} \log \phi_n(t)$ can also be used to give an upper bound, which will be very good when $n$ is large, by Lemma~\ref{fekete}.

We now establish several geometric facts that hold for all CIFSs and will be important when proving a formula for the intermediate dimensions of the limit set in Section~\ref{conformalsect}. The following geometric property was proved in~\cite[page~111]{Mauldin1996iifs}, recalling that $K$ is the constant from the bounded distortion principle. 
\begin{equation}\label{e:mubdp} 
\qquad S_w(B(x,r)) \supseteq B(S_w(x),K^{-1}||S_w'||r) \quad \mbox{for all } x \in X, r \in (0,\mathrm{dist}(X,\partial V)], w \in I^*. 
\end{equation}

The following lemma says that the Lipschitz constants are comparable to the norm of the derivatives of the corresponding map. 
\begin{lma}\label{diameterslemma}
For every CIFS there exists $D \geq 1$ such that for all $w \in I^*$, 
\[ D^{-1} ||S_w'|| \leq r_w \leq R_w \leq D||S_w'||. \]
\end{lma}

\begin{proof}

The proof is similar to the proofs of some of the consequences of the bounded distortion principle in~\cite[page~110--111]{Mauldin1996iifs}. %
For the upper bound, note that $\mathrm{dist}(X,{\Rd\setminus V}) > 0$ since $X$ is compact and disjoint from the closed set ${\Rd\setminus V}$. Let $w \in I^*$. If $B$ is a ball of radius at most $\mathrm{dist}(X,{\Rd\setminus V})$ centred at a point in $X$ and $x,y \in B$ then by the mean value inequality $||S_w(x)-S_w(y)|| \leq ||S_w'|| \cdot ||x-y||$. Since $X$ is compact and connected, it can be covered by a finite chain of open balls $B_1,\dotsc,B_q$ centred at points in $X$ and with radii at most $\mathrm{dist}(X,\Rd\setminus V)/(2K)$ (chain in the sense that $B_i \cap B_{i+1} \neq \varnothing$ for $i=1,2,\dotsc,q-1$). 
Suppose 
\[ D \geq \max\left\{q, K, \frac{K|X|}{\mathrm{dist}(X,\Rd \setminus V)}\right\}.\]
Then since $D \geq q$, the upper bound $R_w \leq D||S_w'||$ holds. 

Trivially, $r_w \leq R_w$. For the lower bound, if $x,y \in X$ and $||x-y|| \leq \mathrm{dist}(X,\Rd \setminus V)$ then by~\eqref{e:mubdp} and the bijectivity of $S_w$, 
 \[||S_w(x) - S_w(y)|| \geq K^{-1}||S_w'||\cdot ||x-y|| \geq D^{-1}||S_w'||\cdot||x-y||.\]%
  If $x,y \in X$ and $|X| \geq ||x-y|| > \mathrm{dist}(X,\Rd \setminus V)$ then since 
  \[ S_w(B(x,\mathrm{dist}(X,\Rd \setminus V))) \supseteq B(S_w(x),K^{-1}||S_w'||\mathrm{dist}(X,\Rd \setminus V)),\]
   it holds that 
 \begin{align*} 
 ||S_w(x) - S_w(y)|| \geq K^{-1}||S_w'||\mathrm{dist}(X,\Rd \setminus V) &\geq K^{-1}||S_w'||\mathrm{dist}(X,\Rd \setminus V) ||x-y||\cdot |X|^{-1} \\*
 &\geq D^{-1}||S_w'||\cdot ||x-y||.
 \end{align*}
  Therefore the lower bound $r_w \geq D^{-1} ||S_w'||$ holds, as required. 
\end{proof}

Lemma~\ref{coneimplication} essentially says that the cone condition holds not just for the set $X$ itself but also for its images under the conformal map corresponding to any given finite word. The purpose of Lemma~\ref{coneimplication} is to prove Lemma~\ref{fitin}. 
\begin{lma}\label{coneimplication}
For all CIFSs, with $D$ as in Lemma~\ref{diameterslemma}, 
\[ \inf_{w \in I^*} \inf_{x \in S_w(X)} \inf_{r \in (0,D||S_w'||)} r^{-d} \cdot \mathcal{L}_d (B(x,r) \cap S_w(\mathrm{Int}_{\Rd} X)) > 0. \]
\end{lma}

\begin{proof}

Write $U=\mathrm{Int}_{\Rd} X$. Let $n \in \N$ and let $w \in I^n$. The idea is that given any ball centred on $S_w(X)$ whose diameter is not too large we can find a large enough ball centred on $X$ that is mapped into it under $S_w$, a uniform proportion of which intersects $U$ by the cone condition. The measure of the image of this part under $S_w$ is large enough by conformality and the BDP. 

 Consider an arbitrary point in $S_w(X)$, which we can write as $S_w(x)$ for some $x \in X$, and let $r \in (0,D||S_w'||)$. 
 By the cone condition there exists $c>0$ such that $\mathcal{L}_d (B(x,r) \cap U)/r^d > c$ for all $x \in X$ and $r \in (0,1)$. By the upper bound of Lemma~\ref{diameterslemma}, $S_w(B(x,r/(D||S_w'||))) \subseteq B(S_w(x),r)$. Now, $\mathcal{L}_d (B(x,r/(D||S_w'||)) \cap U) r^{-d} D^d||S_w'||^d > c$, so by the inner regularity of the Lebesgue measure, there exists a compact $C \subset B(x,r/(D||S_w'||)) \cap U$ such that $\mathcal{L}_d(C) > cr^dD^{-d}||S_w'||^{-d}$. 
 Since $C$ is compact and disjoint from the closed set $\Rd \setminus (B(x,r/(D||S_w'||)) \cap U)$, it follows that $\mathrm{dist}(C,\Rd \setminus (B(x,r/(D||S_w'||)) \cap U)) > 0$. Let $n \in \N$ be large enough so that 
\[2^{-n} < \min\{\mathrm{dist}(C,\Rd \setminus (B(x,r/(D||S_w'||)) \cap U))/\sqrt{d},\mathrm{dist}(X,\Rd\setminus V)\}.\] 
Define $c_d \in (0,1)$ by 
\[ c_d \coloneqq \frac{\mathcal{L}_d(B(0,1))}{\mathcal{L}_d([-1,1]^d)}.\]
Then the balls of diameter $2^{-n}$ inside each of the dyadic cubes of sidelength $2^{-n}$ which intersect $C$ form a disjoint collection of balls inside $B(x,r/(D||S_w'||)) \cap U$ whose total Lebesgue measure is greater than $cc_dr^dD^{-d}||S_w'||^{-d}$. By~\eqref{e:mubdp}, the image of each of these balls under $S_w$ contains a ball of radius $K^{-1}||S_w'||2^{-(n+1)}$. These balls are disjoint subsets of $B(S_w(x),r) \cap S_w(U)$ whose total Lebesgue measure is greater than $c c_d r^dD^{-d}||S_w'||^{-d}||S_w'||^{d}K^{-d} = c c_d r^dD^{-d}K^{-d}$. 
Therefore 
\[ \inf_{w \in I^*} \inf_{x \in S_w(X)} \inf_{r \in (0,D||S_w'||)} \mathcal{L}_d (B(x,r) \cap S_w(\mathrm{Int}_{\Rd} X))/r^d \geq c c_d D^{-d}K^{-d} > 0, \]
as required. 
\end{proof}

Lemma~\ref{fitin} says that it is impossible for too many cylinder sets that are larger than a given size to cluster together and intersect a region that is smaller than that size. This will be useful when proving facts about dimensions, as it shows that a set of a given size cannot cover another set which intersects too many cylinders that are larger than that given size. 
Lemma~\ref{fitin} and its proof are an application of Lemma~\ref{coneimplication} and the OSC, and are similar to (2.2) in~\cite[Proposition~2.9]{Mauldin1999ctdfrac}. The use of the cone condition to obtain Lemma~\ref{fitin} is similar to~\cite[Theorem~4.9]{Graf1988iifs}. 

\begin{lma}\label{fitin}%
For any CIFS there exists $M \in \N$ such that for all $z \in \Rd$ and $r>0$, if $w_1,\dotsc,w_l$ are distinct words in $I^*$ such that for all $i,j \in \{1,\dotsc,l\}$, $w_i$ is not a prefix of $w_j$, and for all $i \in \{1,\dotsc,l\}$ it holds that $B(z,r) \cap S_{w_i}(X) \neq \varnothing$ and $|S_{w_i}(X)| \geq r/2$, then $l \leq M$. 
\end{lma}

\begin{proof}
Write $U=\mathrm{Int}_{\Rd} X$, and let $k_d$ be the $d$-dimensional Lebesgue measure of a ball in $\Rd$ of unit radius. 
By Lemma~\ref{coneimplication} there exists $c>0$ %
such that 
\[ \inf_{w \in I^*} \inf_{x \in S_w(X)} \inf_{r \in (0,D||S_w'||\cdot |X|)} \mathcal{L}_d (B(x,r) \cap S_w(U))/r^d > c.\] %
For each $i=1,\dotsc,l$ there exists $x_i \in X$ such that $S_{w_i}(x_i) \in B(z,r)$. By the upper bound from Lemma~\ref{diameterslemma}, $r/2 \leq |S_{w_i}(X)| \leq D||S_w'||\cdot |X|$, so $\mathcal{L}_d(B(S_{w_i}(x_i),r/2) \cap S_w(U))2^d r^{-d} > c$. Since no $w_i$ is a prefix of any $w_j$, by the OSC the $l$ sets $B(S_{w_i}(x_i),r/2) \cap S_w(U)$ are disjoint subsets of $B(z,2r)$, each having $d$-dimensional Lebesgue measure at least $c r^d 2^{-d}$. Therefore $l c r^d 2^{-d} \leq \mathcal{L}_d(B(z,2r)) = k_d 2^d r^d$, so if we let $M \coloneqq k_d 2^{2d} c^{-1}$ then $l \leq M$, as required. 
\end{proof}

Lemma~\ref{fitin} shows in particular that for all CIFSs, for all $n \in \N$, 
\[ \# \{ \, w \in I^n : B(z,r) \cap S_w(X) \neq \varnothing \textup{ and } |S_w(X)| \geq r/2 \, \} \leq M,\] 
so the family $\{ \, S_w(X) : w \in I^n \, \}$ is \emph{pointwise finite} in the sense that each element of $X$ belongs to at most finitely many elements of this family. Therefore the limit set satisfies 
\[ F = \bigcap_{n=1}^\infty \bigcup_{w \in I^n} S_w(X), \]
and so is a Borel subset of $X$ in the class $F_{\sigma \delta}$. 
Mauldin and Urbański noted this in~\cite[(2.5)]{Mauldin1996iifs} and also showed that the limit set need not be in the class $G_\delta$ (i.e. it need not be a countable intersection of open sets). 
Lemma~\ref{diameterslemma} says that the Lipschitz constants are comparable to the norm of the derivative of the corresponding map; Lemma~\ref{2nddiams} uses this to show that the sizes of the corresponding cylinder sets are also comparable. 
\begin{lma}\label{2nddiams}
For every CIFS there exists $D \geq 1$ such that for all $w \in I^*$, 
\[ D^{-1} ||S_w'|| \leq |F_w| \leq |S_w(X)| \leq D||S_w'||. \]
\end{lma}

\begin{proof}

Lemma~\ref{fitin} shows in particular that the family $\{ \, S_i : i \in I \, \}$ is pointwise finite, so since $I$ is infinite, $F$ has positive diameter. Therefore the result follows from Lemma~\ref{diameterslemma} if we increase $D$ as required. 
\end{proof}

For a CIFS, for each $n \in \N$ define $\psi_n \colon [0,\infty) \to \mathbb{R}$ by 
\[ \psi_n(t) = \sum_{w \in I^n} ||S_w'||^t.\] Mauldin and Urbański~\cite[Section~3]{Mauldin1996iifs} define the pressure function by 
\begin{equation}\label{MUpressure}
\lim_{n \to \infty}\frac{1}{n}\log \psi_n(t),
\end{equation} 
 showing that this limit always exists in $[-\infty,\infty]$ and proving many properties about this function. Lemma~\ref{pressureequal} shows that this coincides with our definition for the pressure function $\overline{P}(t)$, and is in particular independent of the open set $V$.

\begin{lma}\label{pressureequal}
For all CIFSs, for all $t \in (0,\infty)$, 
\[ \frac{1}{n}\log \psi_n(t) \to \overline{P}(t) \textup{ as } n \to \infty. \]
\end{lma}

\begin{proof}
The OSC means that there are no exact overlaps, so $\phi_n(t) = \sum_{w \in I^n} R_w^t$ for all $n \in \N$ and $t \in (0,\infty)$. Therefore by Lemma~\ref{diameterslemma},
\[ D^{-t}\phi_n(t) \leq \psi_n(t) \leq D^t \phi_n(t). \]
Taking logarithms and dividing through by $n$ gives 
\[-\frac{1}{n}t\log D + \frac{1}{n}\log \phi_n(t) \leq \frac{1}{n}\log \psi_n(t) \leq \frac{1}{n}t\log D + \log \phi_n(t),\]
and the result follows upon taking the limit $n \to \infty$. 
\end{proof} 

One of the properties proved in~\cite[Section~3]{Mauldin1996iifs} for a CIFS is that the finiteness parameter for the pressure function is the same as the value at which each of the finite-level approximations to the pressure function becomes finite: $\theta_S = \inf\{ \, t > 0 : \psi_n(t) < \infty \, \}$ for all $n \in \N$. %
Note that for a CIFS where each of the maps $S_i$ are similarities, which means that there exists $c_i \in (0,1)$ such that $||S_i(x)-S_i(y)|| = c_i||x-y||$ for all $x,y \in X$, the quantity $h$ from Definition~\ref{d:h} satisfies the simple form 
\begin{equation}\label{hausdorffsimilarity} h = \inf\left\{ \, t > 0 : \sum_{i \in I} c_i^t < 1 \, \right\}.
\end{equation}

\section{Dimension results}\label{mainsect}

\subsection{General upper bounds}

 In Theorem~\ref{inttypeub} we provide general upper bounds for the Hausdorff, box, intermediate and $\Phi$-intermediate dimensions of the limit set of an arbitrary IIFS.

\begin{thm}\label{inttypeub}
For a IIFS with limit set $F$ and notation as above, 
\begin{enumerate}[label=(\roman*)]
\item\label{hausdorffub} 
$\dim_\mathrm{H} F \leq h$
\item\label{boxub} $\ubd F \leq \max\{h,\lim_{n \to \infty} \inf\{ \, \ubd  P : P \subseteq X \textup{ and } \forall w \in I^n, \, P \cap S_w(X) \neq \varnothing \,  \} \}$
\item\label{intub} For all $\theta \in [0,1]$,  
\[ \uid F \leq \max\{h,\lim_{n \to \infty} \inf\{ \, \uid  P : P \subseteq X \textup{ and } \forall w \in I^n, \, P \cap S_w(X) \neq \varnothing \,  \} \}\]
\item\label{phiub} If $\Phi$ is monotonically admissible (recall Definition~\ref{admissible} from page~\pageref{admissible}) then 
\[ \upd F \leq \max\{h,\lim_{n \to \infty} \inf\{ \, \upd  P : P \subseteq X \textup{ and } \forall w \in I^n, \, P \cap S_w(X) \neq \varnothing \, \} \}\]
\end{enumerate}

\end{thm}

Since $\dim_\mathrm{P} F \leq \ubd F$ always holds,~\ref{boxub} also gives an upper bound for the packing dimension. 
Note that the above upper bounds hold even when there are overlaps of cylinders, and for contractions which are not differentiable and do not satisfy any bi-Lipschitz or bounded distortion condition. However, in some such cases $h$ can overestimate $\dim_\mathrm{H} F$ significantly, and may even be infinite. 

\begin{proof} 

All the bounds are trivial if $h=\infty$, so assume $h<\infty$.

\ref{hausdorffub}. The proof is similar to the proof of the first part of~\cite[Theorem~3.15]{Mauldin1996iifs}. Let $s>h$. By Lemma~\ref{pressurestrict} and the definition of $h$, $\overline{P}(s)<0$. Therefore there exists $N \in \N$ such that $\frac{1}{n}\log \phi_n(s) < \overline{P}(s)/2$, so $\phi_n(s) < e^{n\overline{P}(s)/2}$, for all $n \geq N$. Therefore 
\[ \sum_{\sigma \in M_n} |\sigma(X)|^s \leq |X|^s \phi_n(s) <  |X|^s e^{n\overline{P}(s)/2} \xrightarrow[n \to \infty]{} 0.\]
 But $\{\, \sigma(X) : \sigma \in M_n \, \}$ forms a $|X|\rho^n$-cover of $F$, and $|X|\rho^n \to 0$ as $n \to \infty$, so this means that the $s$-dimensional Hausdorff measure of $F$ is 0. Thus $\dim_\mathrm{H} F \leq s$. Letting $s \to h^+$ gives $\dim_\mathrm{H} F \leq h$, as required. 

\ref{boxub}. follows from the case $\theta=1$ of~\ref{intub}. 

\ref{intub}. The proof is motivated by the proof of~\cite[Lemma~2.8]{Mauldin1999ctdfrac}, which gives a result for the box dimension in the less general setting of a CIFS. 
We will consider $\delta \in \left(\frac{1}{n+1},\frac{1}{n}\right]$ and induct on $n$. The idea is that if we fix a large enough $q \in \N$, the level-$q$ cylinders with size $\lesssim \delta$ can be covered efficiently using a cover of a set $P$ corresponding to level~$q$, and the cylinders with size $\gtrsim \delta$ can be covered efficiently using images of efficient covers of $F$ with larger diameters that are assumed to exist by the inductive hypothesis, and the fact that $\overline{P}(s) < 0$ if $s > h$.  

By definition $\overline{\dim}_{0} = \dim_\mathrm{H}$, so since we can take $P$ to be a countable set (with Hausdorff dimension 0) for all $n \in \N$, the case $\theta = 0$ follows from~\ref{hausdorffub}. Henceforth suppose $\theta \in (0,1]$. Let 
\[ s > \max\{h,\lim_{n \to \infty} \inf\{ \, \uid  P : P \subseteq X \textup{ and } \forall w \in I^n, \, P \cap S_w(X) \neq \varnothing \,  \} \}. \]
Since $s>h$, it holds that $\overline{P}(s)<0$. Therefore there exists $Q \in \N$ such that $\frac{1}{q}\log \phi_q(s) < \overline{P}(s)/2$ for all $q \geq Q$. Fix $q \geq Q$ large enough such that $\phi_q(s) \leq 1/2$ %
and $\rho^q < 1/4$, so $R_w < 1/4$ for all words $w$ of length at least $q$. By the definition of $s$, increasing $q$ further if necessary, we may assume there exists a subset $P_q \subseteq X$ such that $P_q \cap S_w(X) \neq \varnothing$ for all $w \in I^q$, and $\uid P_q < s$. 
Therefore there exists $A>0$ such that for all $\delta \in (0,1]$ there exists a cover $\{V_j\}$ of $P_q$ such that $\delta \leq |V_j| \leq \delta^\theta$ for all $j$, and $\sum_j |V_j|^s \leq A$. 
Let $A_{d,1+2|X|}$ be as in~\eqref{doublingconst}. 
Fix any $B>\frac{A_{d,1+2|X|} A}{1-\phi_q(s)}$ large enough so that for all $\delta \in (1/2,1]$ there exists a cover $\{U_i^\delta\}_i$ of $F$ such that $\delta \leq |U_i| \leq \delta^\theta$ for all $i$, and $\sum_i |U_i|^s \leq B$. It suffices to show that $\uid F \leq s$, which follows from the following claim. 

\textbf{Claim:} For all $n \in \N$, for all $\delta \in \left(\frac{1}{n+1},\frac{1}{n}\right]$ there exists a cover $\{U_i^\delta\}_i$ of $F$ such that $\delta \leq |U_i| \leq \delta^\theta$ for all $i$, and $\sum_i |U_i|^s \leq B$. 

\textbf{Proof of claim:} We prove the claim by induction on $n$. The claim holds for $n=1$ by the definition of $B$. Let $n \in \N$, $n >1$, and assume the claim holds for $1,2,\dotsc,n-1$. Let $\delta \in \left(\frac{1}{n+1},\frac{1}{n}\right]$. By the definition of $A$ there exists a cover $\{V_j\}$ of $P_q$ such that $\delta \leq |V_j| \leq \delta^\theta$ for all $j$, and $\sum_j |V_j|^t \leq A$. 
By the definition of $A_{d,1+2|X|}$, for all $j$ there exist $V_{j,1},\dotsc,V_{j,A_{d,1+2|X|}} \subseteq \Rd$, each of diameter 
\[\max\{\delta,|\mathcal{S}_{|X|\delta}(V_j)|/(1+2|X|)\},\]
 such that 
 \[\mathcal{S}_{|X|\delta}(V_j) \subseteq \bigcup_{k=1}^{A_{d,1+2|X|}} V_{j,k}.\]
  By the triangle inequality, 
\[|\mathcal{S}_{\delta |X|} (V_j)| \leq |V_j| + 2|X|\delta \leq (1+2|X|)|V_j| \leq (1+2|X|)\delta^\theta,\]
so $\delta \leq |V_{j,k}| \leq \delta^\theta$ and $|V_{j,k}| \leq |V_j|$ for all $j,k$. 
Recalling that $M_q$ is the set of maps corresponding to words of length $q$, let $C_{\delta} \coloneqq \{ \, \tau \in M_q : |F_\sigma| \leq |X|\delta \, \}$. 
Since $\{V_j\}$ covers $P_q$, $\{\mathcal{S}_{|X|\delta}(V_j)\}$ covers $\cup_{\tau \in C_\delta} F_\tau$, so $\cup_{k=1}^{A_{d,1+2|X|}} V_{j,k}$ covers $\cup_{\tau \in C_\delta} F_\tau$. 

Now suppose $\sigma \in M_q \setminus C_\delta$, so $|X|\delta < |F_\sigma| \leq |X|R_\sigma$, so $\delta/R_\sigma < 1$, and since $R_\sigma <1/4$, 
\[ \frac{\delta}{R_\sigma} \geq \frac{1}{(n+1)R_\sigma} > \frac{4}{n+1} > \frac{1}{n}.\] 
Therefore by the inductive assumption there exists a cover $\{U_i^{\delta/R_\sigma}\}$ of $F$ such that $\delta/R_\sigma \leq |U_i^{\delta/R_\sigma}| \leq (\delta/R_\sigma)^\theta$ for all $i$ and $\sum_i |U_i^{\delta/R_\sigma}|^s \leq B$. 
For each $i$, let $W_{\sigma,i}$ be a set with diameter 
\begin{equation}\label{wsigma} 
|W_{\sigma,i}| = \max\{|S_\sigma(U_i^{\delta/R_\sigma})|,\delta\}
\end{equation}
 such that $S_\sigma(U_i^{\delta/R_\sigma}) \subseteq W_{\sigma,i}$. Since $\{S_\sigma (U_i^{\delta/R_\sigma})\}_i$ covers $F_\sigma$, also $\{W_{\sigma,i}\}_i$ covers $F_\sigma$. By the definition of $R_\sigma$, $|S_\sigma(U_i^{\delta/R_\sigma})| \leq R_\sigma|U_i^{\delta/R_\sigma}|$ for all $j$, and also $\delta = R_\sigma \delta/R_\sigma \leq R_\sigma |U_i^{\delta/R_\sigma}|$, so by~\eqref{wsigma},  
\begin{equation}\label{intkeybound} \delta \leq |W_{\sigma,i}| \leq R_\sigma|U_i^{\delta/R_\sigma}| \leq R_\sigma (\delta/R_\sigma)^\theta \leq \delta^\theta.
\end{equation}
 The last inequality (which is crucial to the argument) holds since $R_\sigma < 1$. 
 
 Now, $\{V_{j,k}\} \cup \{W_{\sigma,i}\}$ is a cover of $F$ and the diameter of each of these sets lies in the interval $[\delta,\delta^\theta]$. 
Moreover, since $|V_{j,k}| \leq |V_j|$ and by~\eqref{intkeybound}, 
\begin{align*}
\sum_j \sum_{k=1}^{A_{d,1+2|X|}} |V_{j,k}|^s + \sum_{\sigma \in M_q \setminus C_\delta} \sum_i |W_{\sigma,i}|^s 
&\leq A_{d,1+2|X|} \sum_j |V_j|^s + \sum_{\sigma \in M_q \setminus C_\delta} R_\sigma^s \sum_i |U_i^{\delta/R_\sigma}|^s \\
&\leq A_{d,1+2|X|} A + B\phi_q(s) \\
&\leq B,
\end{align*}
so the claim holds by induction. 

\ref{phiub}. By Lemma~\ref{phiinvertible} from page~\pageref{phiinvertible} we may assume without loss of generality that $\Phi$ is invertible. Then~\ref{phiub} holds by almost exactly the same proof as~\ref{intub}, with $\uid$, $\delta^\theta$ and $(\delta/R_w)^\theta$ replaced by $\upd$, $\Phi^{-1}(\delta)$ and $\Phi^{-1}(\delta/R_w)$ respectively throughout. In place of~\eqref{intkeybound}, the key inequality $R_w \Phi^{-1}(\delta/R_w) \leq \Phi^{-1}(\delta)$ holds since $\Phi(\delta)/\delta \searrow 0$ monotonically as $\delta \to 0^+$ by assumption. 
\end{proof}

\subsection{Precise formulae for conformal iterated function systems}\label{conformalsect}

In order to use the upper bounds in Theorem~\ref{inttypeub} to prove a simpler formula for intermediate dimensions of the limit set of a CIFS in Theorem~\ref{mainint} as the maximum of the Hausdorff dimension and the intermediate dimensions of the fixed points, we need further lemmas. 
The following lemma and proof are similar to~\cite[Proposition~2.9]{Mauldin1999ctdfrac} for the box dimension. 

\begin{lma}\label{samewithinlevel}%
Fix any CIFS, let $n \in \N$, and assume that $P$ and $Q$ are both subsets of $\cup_{w \in I^n} S_w(X)$ satisfying 
\begin{align*} 
&0 < \inf_{w \in I^n} \# (P \cap S_w(X))   \leq  \sup_{w \in I^n} \# (P \cap S_w(X)) < \infty, \\*
&0 < \inf_{w \in I^n} \# (Q \cap S_w(X))   \leq  \sup_{w \in I^n} \# (Q \cap S_w(X)) < \infty. 
\end{align*}
Then $\uid P = \uid Q$ for all $\theta \in [0,1]$, and $\upd P = \upd Q$ for all monotonically admissible functions $\Phi$. 
The same holds with $\overline{\dim}$ replaced by $\underline{\dim}$ throughout. 
\end{lma}

\begin{proof} 
The idea is to use an efficient cover of $Q$ at scale $\delta$ to construct an efficient cover of $P$ at scale $\delta$. Elements of $P$ in cylinders of size $\lesssim \delta$ can be covered using the cover of the element of $Q$ in the same cylinder, and the conditions of a CIFS (via Lemma~\ref{fitin}) mean that each element of the cover can intersect only a bounded number of cylinders that are larger than the covering set in question. 

Since for each $n \in \N$, $\{ \, S_w : w \in I^n \, \}$ forms a CIFS with the same limit set, we may henceforth assume without loss of generality that $n=1$. 
Let $A_{d,3}$ be as in~\eqref{doublingconst} on page~\pageref{doublingconst} and let $M$ be as in Lemma~\ref{fitin}. 
Let 
\[ C \coloneqq \max\Big\{    \sup_{w \in I^n} \# (P \cap S_w(X)),  \sup_{w \in I^n} \# (Q \cap S_w(X))    \Big\}. \]
If $\theta = 0$ then $\uid P = \uid Q = 0$ because $P$ and $Q$ are countable, so henceforth assume that $\theta \in (0,1]$. 

\textbf{Claim:} Given any $\delta > 0$, if $\{U_j\}$ is a cover of $Q$ such that $\delta \leq |U_j| \leq \delta^\theta$ for all $j$ then there exists a cover $\{V_m\}$ of $P$ such that $\delta \leq |V_m| \leq \delta^\theta$ for all $m$, and 
\[ \sum_m |V_m|^s \leq (A_{d,3} + M C)\sum_j |U_j|^s \] 
for all $s\geq 0$. 

\textbf{Proof of claim:} For each $j$, if $i \in I$ is such that $|S_i(X)| \leq |U_j|$ and $S_i(X) \cap U_j \neq \varnothing$, then 
\[ S_i(X) \subseteq \mathcal{S}_{|U_j|}(U_j) \subseteq \bigcup_{l=1}^{A_{d,3}} \mathcal{S}_{|U_j|}(U_j)_l,\]
 where $\mathcal{S}_{|U_j|}(U_j)$ is the neighbourhood set, which has diameter $3|U_j|$. 
 By Lemma~\ref{fitin} there exist $i_1, \dotsc, i_M \in I$, not necessarily distinct, such that $S_{i_k}(X) \cap U_j \neq \varnothing$ for $k=1,\dotsc,M$, and such that if $i \in I \setminus \{i_1,\dotsc,i_M\}$ and $|S_i(X)| > |U_j|$ then $S_i(X) \cap U_j = \varnothing$. 
 If $k=1,\dotsc,M$ then we can cover $P \cap S_{i_k}(X)$ by $C$ balls $\{B_{j,k,p}\}_{p=1}^C$, each of diameter $|U_j|$. Since $\{U_j\}$ covers $Q$, 
\[P \subseteq \bigcup_j \left( \bigcup_{l=1}^{A_{d,3}} \mathcal{S}_{|U_j|}(U_j)_l \cup \bigcup_{k=1}^M \bigcup_{p=1}^C B_{j,k,p} \right).\]
 Each element of this cover of $P$ has diameter in the interval $[\delta,\delta^\theta]$ by construction. 
 Moreover, 
 \[ \sum_j \left(\sum_{l=1}^{A_{d,3}} |\mathcal{S}_{|U_j|}(U_j)_l|^s + \sum_{k=1}^M \sum_{p=1}^C |B_{j,k,p}|^s\right) = (A_{d,3} + M C)\sum_j |U_j|^s, \]
 proving the claim. 
 
 The claim shows that $\uid P \leq \uid Q$ and $\lid P \leq \lid Q$. The reverse inequalities hold by symmetry, so $\uid P = \uid Q$ and $\lid P = \lid Q$, as required. 
 If $\Phi$ is a monotonically admissible function then by Lemma~\ref{phiinvertible} on page~\pageref{phiinvertible} we may assume without loss of generality that $\Phi$ is invertible. 
 Then the same proof works with $\delta^\theta$, $\uid$ and $\lid$ replaced by $\Phi^{-1}(\delta)$, $\upd$ and $\lpd$ respectively throughout. 
 \end{proof}

Lemma~\ref{changelevel} shows that the upper intermediate dimensions of a set of points corresponding to the $n$-th level cylinders are either all bounded above by the finiteness parameter, and hence the Hausdorff dimension of the limit set, or they all equal the upper intermediate dimensions of the level-1 fixed points. 
We will combine this lemma with the upper bounds in Theorem~\ref{inttypeub} (which considers arbitrarily deep levels) to prove that the dimensions in fact depend only on the level-1 fixed points (and the Hausdorff dimension). 
Mauldin and Urbański prove in~\cite[Lemma~2.10]{Mauldin1999ctdfrac} that the upper box dimension of the level-1 iterates of a given point is greater than or equal to the finiteness parameter $\theta_S$, and deduce that it equals the box dimension of the $n$-th level iterates for all $n \in \N$. The intermediate dimensions, on the other hand, will \emph{not} always exceed the finiteness parameter, so we cannot make the same conclusion for the intermediate dimensions in Lemma~\ref{changelevel}. 

 \begin{lma}\label{changelevel}
 Consider a CIFS, and suppose that for each $n \in \N$, $P_n \subseteq \cup_{w \in I^n} S_w(X)$ is any set satisfying $0 < \inf_{w \in I^n} \# (P_n \cap S_w(X))   \leq  \sup_{w \in I^n} \# (P_n \cap S_w(X)) < \infty$. Then 
 
 \begin{enumerate}[label=(\roman*)]
 \item\label{changeleveluid} for all $\theta \in [0,1]$, either $\uid P_n \leq \theta_S \leq h$ for all $n \in \N$ or $\uid P_n = \uid P_1$ for all $n \in \N$. 
 \item\label{changelevelupd} If $\Phi$ is monotonically admissible then either $\upd P_n \leq \theta_S \leq h$ for all $n \in \N$ or $\upd P_n = \upd P_1$ for all $n \in \N$. 
 \end{enumerate}
 \end{lma}%
 
 \begin{proof}
 \ref{changeleveluid}. This is true for the Hausdorff dimension because each $P_n$ is countable, so henceforth suppose $\theta \in (0,1]$. By Lemma~\ref{samewithinlevel}, $\uid P_n$ does not depend on the particular set $P_n$, so $\uid P_1 \leq \uid P_2 \leq \dotsb$. 
 By Lemmas~\ref{samewithinlevel} and~\ref{fitin}, we can henceforth fix $x \in X$ and assume without loss of generality that $P_n \coloneqq \{ \, S_w(x) : w \in I^n \, \}$ for all $n \in \N$. 
 It suffices to prove that $\uid P_n \leq \max\{\theta_S, \uid P_1\}$ for all $n \in \N$, which follows from the following claim. 
 
 \textbf{Claim:} For all $s>\max\{\theta_S, \uid P_1\}$, for all $n \in \N$ there exists $B_n \in (0,\infty)$ such that for all $\delta \in (0,1]$ there exists a cover $\{U_j^{\delta,n}\}_j$ of $P_n$ such that $\delta \leq |U_j^{\delta,n}| \leq \delta^\theta$ for all $i$ and $\sum_j |U_j^{\delta,n}|^s \leq B_n$. 
 
 \textbf{Proof of claim:} Fix $s>\max\{\theta_S, \uid P_1\}$. We prove the claim by induction on $n$. Suppose $n>1$ and assume the claim holds for $1,2,\dotsc,n-1$. 
 Let $\delta \in (0,1]$. 
 By the definition of $A_{d,1+2|X|}$ in~\eqref{doublingconst}, for all $j$ there exist $U_{j,1}^{\delta,n-1},\dotsc,U_{j,A_{d,1+2|X|}}^{\delta,n-1} \subseteq \Rd$, each of diameter 
 \[ \max\left\{\delta,\frac{|\mathcal{S}_{|X|\delta}(U_j^{\delta,n-1})|}{1+2|X|}\right\},\]
  such that 
  \[\mathcal{S}_{|X|\delta}(U_j^{\delta,n-1}) \subseteq \bigcup_{k=1}^{A_{d,1+2|X|}} U_{j,k}^{\delta,n-1}.\]
   By the triangle inequality, 
\[|\mathcal{S}_{\delta |X|} (U_j^{\delta,n-1})| \leq |U_j^{\delta,n-1}| + 2|X|\delta \leq (1+2|X|)|U_j^{\delta,n-1}| \leq (1+2|X|)\delta^\theta.\]
Therefore $\delta \leq |U_{j,k}^{\delta,n-1}| \leq \delta^\theta$ and $|U_{j,k}^{\delta,n-1}| \leq |U_j^{\delta,n-1}|$ for all $j,k$.
 
 Let $C_{\delta} \coloneqq \{ \, w \in I^{n-1} : |X_w| \leq |X|\delta \, \}$. If $w \in C_\delta$ then there exists $p_w \in S_w(X) \cap P_{n-1}$, and there exists $j$ such that $p_w \in U_j^{\delta,n-1}$, so the neighbourhood set $\mathcal{S}_{|X|\delta} (U_j^{\delta,n-1})$ covers $S_w(X)$. Thus 
 \begin{equation}\label{changelevelinclusion} P_n \cap S_w(X) \subseteq S_w(X) \subseteq \mathcal{S}_{|X|\delta} (U_j^{\delta,n-1}) \subseteq \bigcup_{k=1}^{A_{d,1+2|X|}} U_{j,k}^{\delta,n-1}. 
 \end{equation}
 If, on the other hand, $w \in I^{n-1}\setminus C_\delta$, then $|X|\delta < |X_w| \leq |X|R_w$ and $\delta/R_w < 1$. Consider the cover $\{U_l^{\delta/R_w,1}\}_l$ of $P_1$ whose existence is guaranteed by the base case $n=1$. 
  For each $l$, let $W_{w,l}$ be a set with diameter $|W_{w,l}| = \max\{|S_w(U_l^{\delta/R_w,1} \cap X)|,\delta\}$ such that $S_w(U_l^{\delta/R_w,1} \cap X) \subseteq W_{w,l}$. Since $P_n \coloneqq \{ \, S_w(x) : w \in I^n \, \}$, the sets $\{S_w(U_l^{\delta/R_w,1} \cap X)\}_l$ cover $P_n \cap S_w(X)$, so $\{W_{w,l}\}_l$ covers $P_n \cap S_w(X)$. 
  By the definition of $R_w$, for all $l$ we have $|S_w(U_l^{\delta/R_w,1} \cap X)| \leq R_w |U_l^{\delta/R_w,1} \cap X| \leq R_w |U_l^{\delta/R_w,1}|$, %
  and also $\delta = R_w\delta/R_w \leq R_w |U_l^{\delta/R_w,1}|$, so 
 \begin{equation}\label{intlemmakeybound}
 \delta \leq |W_{w,l}| \leq R_w |U_l^{\delta/R_w,1}| \leq R_w (\delta/R_w)^\theta \leq \delta^\theta. 
 \end{equation}
 Now, $\{U_{j,k}^{\delta,n-1}\} \cup \{W_{w,l}\}$ covers $P_n$ and the diameter of each element of this cover lies in the interval $[\delta,\delta^\theta]$. Moreover, since $|U_{j,k}^{\delta,n-1}| \leq |U_j^{\delta,n-1}|$ for all $j,k$, and by~\eqref{intlemmakeybound},  
 \begin{align*}
  \sum_j \sum_{k=1}^{A_{d,1+2|X|}}|U_{j,k}^{\delta,n-1}|^s &+ \sum_{w \in I^n \setminus C_\delta} \sum_l |W_{w,l}|^s \\
  &\leq A_{d,1+2|X|} \sum_j |U_j^{\delta,n-1}|^s + \sum_{w \in I^n \setminus C_\delta} R_w \sum_l |U_l^{\delta/R_w,1}|^s \\
 &\leq A_{d,1+2|X|} B_{n-1} + B_1 \phi_n(s),
 \end{align*}
 recalling the definition of $\phi_n$ from~\eqref{e:definephin}. 
 Therefore letting $B_n \coloneqq A_{d,1+2|X|} B_{n-1} + B_1 \phi_n(s)$, since $s>\theta_S$, $\phi_n(s) < \infty$, so $B_n < \infty$, and the claim holds by induction. 
 
 \ref{changelevelupd}. By Lemma~\ref{phiinvertible} we may assume that $\Phi$ is invertible. Then~\ref{changelevelupd} holds by the same proof as~\ref{changeleveluid}, with $\uid$, $\delta^\theta$ and $(\delta/R_w)^\theta$ replaced by $\upd$, $\Phi^{-1}(\delta)$ and $\Phi^{-1}(\delta/R_w)$ respectively. In place of~\eqref{intlemmakeybound}, $R_w \Phi^{-1}(\delta/R_w) \leq \Phi^{-1}(\delta)$ holds since $\Phi(\delta)/\delta \searrow 0$ monotonically as $\delta \to 0^+$ by assumption. 
 \end{proof}

Mauldin and Urbański~\cite[Theorem~3.15]{Mauldin1996iifs} show that the Hausdorff dimension of the limit set $F$ of a CIFS is $h$. 
In fact, this is true even if the cone condition~\ref{cone} is not assumed (see~\cite[Theorem~19.6.4]{Urbanski2022book}), but in this chapter we do use the cone condition in the proof of Lemma~\ref{coneimplication} (and hence Lemmas~\ref{fitin} and~\ref{samewithinlevel}). 
We now use the fact that $h = \dim_{\mathrm H} F$, together with the upper bounds in Theorem~\ref{inttypeub} and Lemmas~\ref{changelevel} and~\ref{samewithinlevel}, to prove the main result of this chapter, Theorem~\ref{mainint}. This gives the following simple formulae for other dimensions of the limit set as the maximum of the Hausdorff dimension of the limit set and the corresponding dimension of any set $P$ satisfying certain conditions. 
From Lemma~\ref{fitin} we see that examples for the set $P$ include $\{ \, S_i(x) : i \in I \, \}$ for any given $x \in X$, as in~\cite{Mauldin1999ctdfrac}, or the set of fixed points in $X$ of the contractions $S_i$. 
There is a typo in the definition of $P$ in the published version of~\cite{Banaji2021infinite}. 

\begin{thm}\label{mainint}
For all CIFSs with limit set $F$ and notation as above, for all subsets $P \subseteq \cup_{i \in I} S_i(X)$ which satisfy $0 < \inf_{i \in I} \# (P \cap S_i(X))  \leq  \sup_{i \in I} \# (P \cap S_i(X)) < \infty$, 
\begin{enumerate}[label=(\roman*)]
\item\label{boxeq} $\ubd F = \dim_\mathrm{P} F = \max\{h,\ubd P\}$ (very similar to Mauldin and Urbański \cite[Theorem~2.11]{Mauldin1999ctdfrac} but with a more general condition on $P$)
\item\label{inteq} $\uid F = \max\{h,\uid P\}$ for all $\theta \in [0,1]$
\item\label{phieq} $\upd F = \max\{h,\upd P\}$ if $\Phi$ is monotonically admissible
\end{enumerate}
\end{thm}

\begin{proof}

\ref{boxeq}. follows from the case $\theta=1$ of~\ref{inteq} and the fact that $\ubd F = \dim_\mathrm{P} F$ by~\cite[Theorem~3.1]{Mauldin1996iifs}. 

\ref{inteq}. For each $n \in \N$ let $P_n \coloneqq \{ \, x \in X : x = S_w(x) \mbox{ for some } w \in I^n \, \}$, so $P_n \subseteq F \subseteq \cup_{w \in I^n} S_w(X)$. Then by Lemmas~\ref{samewithinlevel} and~\ref{changelevel}~\ref{changeleveluid}, $\uid P_n \leq \max\{h,\uid P\}$ for all $n \in \N$. Therefore by Theorem~\ref{inttypeub}~\ref{intub}, 
\[ \uid F \leq \max\{ h, \lim_{n \to \infty} \uid P_n \} \leq \max\{ h, \max\{h,\uid P \} \} = \max\{h,\uid P\}.\]
 But $P_n \subseteq F$ so since $\uid$ is monotonic for subsets, by Lemma~\ref{samewithinlevel}, $\uid P = \uid P_1 \leq \uid F$, and by~\cite[Theorem~3.15]{Mauldin1996iifs}, $h = \dim_\mathrm{H} F \leq \uid F$, so $\max\{h,\uid P\} \leq \uid F$. Therefore $\uid F = \max\{h,\uid P\}$, as required. 

\ref{phieq} is similar to~\ref{inteq}. 
\end{proof}

A consequence is the following bounds for the lower versions of the dimensions in Theorem~\ref{mainint}. 

\begin{cor}\label{lowerintcor}
For all CIFSs with limit set $F$ and notation as above, for all subsets $P \subseteq \cup_{i \in I} S_i(X)$ which satisfy $0 < \inf_{i \in I} \# (P \cap S_i(X))  \leq  \sup_{i \in I} \# (P \cap S_i(X)) < \infty$,  
\begin{enumerate}[label=(\roman*)]
\item\label{lowerbox} $\max\{h,\lbd P\} \leq \lbd F \leq \max\{h,\ubd P\}$
\item\label{lowerint} $\max\{h,\lid P\} \leq \lid F \leq \max\{h,\uid P\}$ for all $\theta \in [0,1]$ 
\item\label{lowerphi} $\max\{h,\lpd P\} \leq \lpd F \leq \max\{h,\upd P\}$ if $\Phi$ is monotonically admissible. 
\end{enumerate}
\end{cor} 

\begin{proof}
We prove~\ref{lowerint};~\ref{lowerbox} and~\ref{lowerphi} are similar. By Theorem~\ref{mainint}~\ref{inteq} $\lid F \leq \uid F = \max\{h,\uid P\}$. If $P_1$ is the set of fixed points in $X$ of the maps $\{S_i\}_{i \in I}$ then by Lemma~\ref{samewithinlevel}, $\lid P = \lid P_1 \leq \lid F$, and by~\cite[Theorem~3.15]{Mauldin1996iifs}, $h = \dim_\mathrm{H} F \leq \lid F$, so $\max\{h,\lid P\} \leq \lid F$, as required. 
\end{proof} 

\begin{question}\label{q:sharp}
Are the bounds in Corollary~\ref{lowerintcor} sharp or can they be improved in general? 
\end{question}

Very recently, the author and Rutar~\cite{BanajiPreprintinfinitelowerbox} have calculated a formula for $\lbd F$, which interestingly depends on more refined properties of the covering function $r \mapsto N_r(P)$ than merely $\lbd P$ and $\ubd P$. The box dimension of $F$ exists if and only if the bounds in part~\ref{lowerbox} of Question~\ref{q:sharp} coincide, and if this is not the case then the sharp upper bound for $\lbd F$ in terms of $h$, $\lbd P$ and $\ubd P$ is strictly smaller than $\max\{h,\ubd P\}$. 

If the fixed points are arranged to be at $\frac{1}{\log n}$ then by Falconer, Fraser and Kempton \cite[Example~1]{Falconer2020firstintermediate}, the intermediate dimensions of the set of fixed points will be~$1$ for all $\theta \in (0,1]$. However, the contraction ratios can be made small enough (and tending to~$0$ rapidly enough) so that $h < 1$. In this case, the limit set and its closure will have the same Hausdorff dimension as they differ by a countable set (namely the images of the point 0 under maps corresponding to finite words) by Mauldin and Urbański \cite[Lemma~2.1]{Mauldin1996iifs}. Therefore by Theorem~\ref{recoverinterpolation} from page~\pageref{recoverinterpolation}, the $\Phi$-intermediate dimensions can be used to `recover the interpolation' between the Hausdorff and box dimensions of the limit set $F$ in the sense that for all $s \in [h,1]$ there exists an admissible function $\Phi_s$ with $\dim^{\Phi_s} F = s$. Theorem~\ref{mainint}~\ref{phieq} can help to find these functions. 
It is possible for the intermediate dimensions to be discontinuous at $\theta = 0$ even when the box dimension is less than 1. Indeed, consider a countable compact subset $P \subset \mathbb{R}$ with box and Assouad dimension equal and strictly between 0 and 1, so by~\cite[Proposition~2.4]{Falconer2020firstintermediate}, $\dim_{\theta} P = \dim_\mathrm{B} P$ for all $\theta \in (0,1)$. We may then choose a set of similarity maps whose fixed points form the set $P$ and whose contraction ratios are small enough that the system forms a CIFS with the Hausdorff dimension of the limit set being smaller than $\dim_\mathrm{B} P$. Then the intermediate dimensions of the limit set will be discontinuous at $\theta = 0$ by Theorem~\ref{mainint}.

Our results are also relevant to the dimension theory of infinite parabolic iterated function systems. 
In such a system, each map $S \colon X \to X$ still satisfies $||S(x) - S(y)|| < ||x-y||$ for all $x,y \in X$, but finitely many of the maps may contain a \emph{parabolic fixed point} $p \in X$, meaning that $S(p) = p$ but the derivative of $S$ (or an extension of $S$) at $p$ has norm 1. The other countably many maps must be uniformly contracting, and are called \emph{hyperbolic}. 
Manneville--Pomeau maps of the form ${x \mapsto x - x^q}$ for fixed $q>1$ are examples of functions with a parabolic fixed point at $0$. 
The theory of parabolic IFSs has been developed by Mauldin and Urbański in~\cite{Mauldin2000parabolic}, and they have also been studied in~\cite{Mauldin2002parabolic,
Banaji2022assouad}, \cite[Section~9.2]{Fraser2020book}, and many other works. 
Given an infinite parabolic IFS as defined in~\cite[Section~2]{Mauldin2000parabolic}, one can associate an `induced' uniformly contracting infinite CIFS (see~\cite[Theorem~5.2]{Mauldin2000parabolic}). It is clear that if $F$ is the limit set of the parabolic IFS and $F^*$ is the limit set of the induced CIFS then $F^* \subseteq F$ with $F \setminus F^*$ countable, and $F$ and $F^*$ have the same closure. Therefore if $\dim$ is Hausdorff/box/intermediate/Assouad dimension, then $\dim F = \dim F^*$. In particular, Theorem~\ref{mainint} can be applied directly to the induced system to give information about the corresponding dimension of $F$. 
This is only relevant for systems consisting of infinitely many hyperbolic maps (and finitely many parabolic maps), because the limit sets of finite parabolic IFSs have equal Hausdorff and upper box dimensions (see \cite[Remark~6.6]{Urbanski1996paraboliccantor} and \cite{Mauldin2002parabolic}). 
In~\cite[Section~6]{Banaji2022assouad}, however, we use inducing to calculate the Assouad spectrum of a class of `parabolic Cantor sets' (see~\cite{Urbanski1996paraboliccantor}), which are generated by finite parabolic IFSs. 

\subsection{An example and a first application}

We use Proposition~\ref{p:lattice} from page~\pageref{p:lattice} and a result of Burrell, Falconer and Fraser to give an application of Theorem~\ref{mainint} to orthogonal projections. Dimension theory of orthogonal projections has a long history in fractal geometry, see~\cite{Falconer2015proj:FractalsV,
Shmerkin2015proj:FractalsV}. 
There has been particular interest in orthogonal projections of dynamically defined sets, where one can often obtain more precise information than is provided by the general projection theorems, see~\cite{Hochman2012entropy,Shmerkin2015proj:FractalsV}. The following example falls into this category. 
Recall the definition 
\[ G_{p,d} \coloneqq \{ \, x/||x||^2 : x \in \{ 1^p,2^p,3^p,\dotsc \}^d \, \}. \]

\begin{example}\label{proj}
Let $p>0$ and consider a set of contracting similarity maps on $\mathbb{R}^2$ with fixed points lying in the set $G_{p,2}$ from Proposition~\ref{p:lattice}, with no two maps having the same fixed point. Assume the contraction ratios are small enough that the system forms a CIFS, with limit set $F$, say, and small enough that $\dim_\mathrm{H} F < 1$. Then by Theorem~\ref{mainint} and Corollary~\ref{lowerintcor}, 
\[ \dim_{\theta} F = \max\left\{\dim_\mathrm{H} F, \frac{2\theta}{p+\theta} \right\},\]
which is continuous at $\theta=0$. Therefore by Burrell, Falconer and Fraser's Theorem~\ref{t:burrellproj} from page~\pageref{t:burrellproj} there exists $c<1$ such that $\ubd \pi(F) \leq c$ for every orthogonal projection $\pi \colon \mathbb{R}^2 \to \mathbb{R}$, and $\ubd \pi(F) = c$ for almost every orthogonal projection $\pi$ (with respect to the natural measure on projective space). This conclusion is perhaps most interesting when $p$ is very close to 0 (and so $\bd F$ is very close to 2) and $\dim_\mathrm{H} F$ is very close to 1. 
\end{example}
More generally, if $1 \leq k < d$ are integers and the contraction ratios lie on $G_{p,d}$ and $\dim_\mathrm{H} F < k$, then $\ubd \pi(F) < k$ for every orthogonal projection $\pi \colon \Rd \to \R^k$.

\section{Continued fraction sets}\label{ctdfracsect}

\subsection{Real continued fractions}

In this section we apply Theorem~\ref{mainint} to give information about sets of irrational numbers whose continued fractions have restricted entries, as in the following definition. 

\begin{defn}\label{ctdfracdefn}
For a non-empty, proper subset $I \subset \N$, define 
\[ F_I \coloneqq \left\{ \, z \in (0,1) \setminus \mathbb{Q} : z = \frac{1}{b_1 + \frac{1}{b_2 + \frac{1}{\ddots}}}, b_n \in I \mbox{ for all } n \in \N \, \right\}. \]
\end{defn}

The following lemma shows why our general results can be applied in this setting. 

\begin{lma}
Working in $\mathbb{R}$, letting $X \coloneqq [0,1]$ and $V \coloneqq (-1/8,9/8)$, 
\begin{enumerate}[label=(\roman*)]
\item\label{ctdfracnot1} If $1 \notin I$ then $\{ \, S_b(x) \coloneqq 1/(b+x) : b \in I \, \}$ is a CIFS with limit set $F_I$. 
\item\label{ctdfrac1} If $1 \in I$ then $\{ \, S_b(x) \coloneqq 1/(b+x) : b \in I, b \neq 1 \, \} \cup \left\{ \, S_{1b}(x) \coloneqq \frac{1}{b+\frac{1}{1+x}} : b \in I \, \right\}$ is a CIFS with limit set $F_I$. 
\end{enumerate}
\end{lma}
\begin{proof}
\ref{ctdfracnot1} is verified in~\cite[page~4997]{Mauldin1999ctdfrac}, and~\ref{ctdfrac1} can be verified similarly, noting that when $1 \in I$ the different CIFS is needed to ensure that the system is uniformly contractive, because the derivative of $x \to 1/(1+x)$ at $x=0$ is $-1$. 
\end{proof}

It follows from~\cite[Theorem~3.15]{Mauldin1996iifs} that $\dim_\mathrm{H} F_I = h$; the Hausdorff dimension of such limit sets has been studied in~\cite{Mauldin1999ctdfrac,Kessebohmer2006iifs,Heinemann2002cifs,
Ingebretson2020ctdfrac,Chousionis2020ctdfrac} %
 and other works. It follows from~\cite[Theorem~2.11]{Mauldin1999ctdfrac} that $\ubd F_I = \max\{h,\ubd\{ \, 1/b : b \in I \, \}\}$. In Theorem~\ref{ctdfracintthm} we apply Theorem~\ref{mainint} to give information about the intermediate dimensions of $F_I$. 

\begin{thm}\label{ctdfracintthm}
Using the notation in Definition~\ref{ctdfracdefn}, for all non-empty proper $I \subset \N$, 
\begin{enumerate}[label=(\roman*)]
\item\label{ctdfracint} For all $\theta \in [0,1]$,  
\begin{gather*}
 \uid F_I = \max\{h,\uid\{ \, 1/b : b \in I \, \}\}; \\*
 \max\{h,\lid\{ \, 1/b : b \in I \, \}\} \leq \lid F_I \leq \max\{h,\uid\{ \, 1/b : b \in I \, \}\}.
 \end{gather*}
\item\label{ctdfraccts} The maps $\theta \mapsto \uid F_I$ and $\theta \mapsto \lid F_I$ are continuous at $\theta = 0$.
\end{enumerate}
\end{thm} 

\begin{proof}

\ref{ctdfracint} \textbf{Case 1:} Assume $1 \notin A$. Then the result follows from Theorem~\ref{mainint}~\ref{inteq} and Corollary~\ref{lowerintcor}~\ref{lowerint} if we take $P=\{ \, S_b(0) : b \in I \, \} = \{ \, 1/b : b \in I \, \}$. 

\textbf{Case 2:} Assume $1 \in A$. Since the map $x \mapsto 1/(1+x)$ is bi-Lipschitz on $[0,1]$ and $\uid$ is stable under bi-Lipschitz maps, 
\[\uid \left\{ \, \frac{1}{1+\frac{1}{b}} \, \right\} = \uid \{ \, 1/b : b \in I \, \}.\]
 Since the removal of finitely many points from a set does not change its dimension, \[\uid \{ \, 1/b : b \in I, b \neq 1 \, \} = \uid \{ \, 1/b : b \in I \, \}.\] It is clear from the definition that $\uid$ is finitely stable, so the equality for $\uid F_I$ follows from Theorem~\ref{mainint}~\ref{inteq} if we take 
\[ P \coloneqq \left\{ \, \frac{1}{1+\frac{1}{b}} \, \right\} \cup \{ \, 1/b : b \in I, b \neq 1 \, \}.\] 
The lower bound holds since 
\[ \max\{h,\lid\{ \, 1/b : b \in I \, \}\} = \max\{ \dim_\mathrm{H} F_I, \lid \{ \, 1/b : b \in I, b \neq 1 \, \} \} \leq \lid F_I \]
by~\cite[Theorem~3.15]{Mauldin1996iifs}. 

\ref{ctdfraccts} For all $\theta \in (0,1]$, 
\begin{align*}
 \uid F_I &= \max\{h,\uid\{ \, 1/b : b \in I \, \}\} 
 &&\text{by~\ref{ctdfracint}} \\
 &\leq \max\{ h,\uid\{ \, 1/b : b \in \N \, \}\} \\
  &= \max\{ h, \theta/(1+\theta) \} &&\text{by \cite[Proposition~3.1]{Falconer2020firstintermediate}} \\ 
 &\xrightarrow[\theta \to 0^+]{} \max\{ h, 0 \} = h = \dim_\mathrm{H} F_I = \overline{\dim}_{0} F_I &&\text{by \cite[Theorem~3.15]{Mauldin1996iifs}},
 \end{align*}
 so $\theta \mapsto \uid F_I$ is continuous at $\theta = 0$. Since $\dim_\mathrm{H} F_I \leq \lid F_I \leq \uid F_I$ for all $\theta \in [0,1]$ it follows that $\theta \mapsto \lid F_I$ is also continuous at $\theta = 0$. 
\end{proof}

The following example is similar to~\cite[Theorem~6.2]{Mauldin1999ctdfrac}; we consider a nice family of subsets $I$ which result in the upper and lower intermediate dimensions coinciding. 

\begin{corollary}\label{c:ctdfracex}
Fix $p>1$ and for $l \in \N$, $l \geq 2$ define 
\[ I_{p,l} \coloneqq \{ \, \lfloor n^p \rfloor : n \geq l \}. \]
Then the intermediate dimensions of the continued fraction set exist and are given by  
  \begin{equation}\label{e:realctdfracex}
   \dim_{\theta} F_{I_{p,l}} = \max\left\{\dim_{\mathrm{H}} F_{I_{p,l}},\frac{\theta}{p+\theta} \right\}. 
   \end{equation}
  Moreover, there exists $q \in \N$ such that for all $l \geq q$ we have 
   \[ \dim_{\mathrm{H}} F_{I_{p,l}} < \dim_{\mathrm{B}} F_{I_{p,l}} = \frac{1}{p+1}. \] 
  \end{corollary}
  \begin{proof}
 Since $I_{p,l}$ is bi-Lipschitz equivalent to a cofinite subset of $\{ \, i^{-p} : i \in \N \, \}$, \cite[Proposition~3.1]{Falconer2020firstintermediate} gives 
 \[ \dim_{\theta} \{ \, 1/b : b \in I_{p,l} \, \} = \dim_{\theta} \{ \, i^{-p} : i \in \N \, \} = \frac{\theta}{p+\theta}\]
 for $\theta \in [0,1]$. 
Therefore the bounds in Theorem~\ref{ctdfracintthm}~\ref{ctdfracint} coincide and~\eqref{e:realctdfracex} holds. 
It was shown in~\cite[Section~3]{Mauldin1996iifs} that $\theta_S = \inf\{ \, t > 0 : \psi_1(t) < \infty \, \}$. 
But there exists $C \geq 1$ such that $1/(C b^2) \leq ||S_b'|| \leq C/b^2$ for all $b \in \N$. 
Therefore, since  
  \[ \sum_{n=1}^\infty ((n^p)^{-2})^{1/(2p)} = \sum_{n=1}^\infty n^{-1} = \infty;\] 
  \[ \sum_{n=1}^\infty ((n^p)^{-2})^s = \sum_{n=1}^\infty n^{-2p/s} < \infty \qquad \mbox{for all } s>1/(2p), \]
   it follows that the finiteness parameter $\theta_{I_{p,l}} = 1/(2p)$ and so $\dim_{\mathrm{H}} F_{I_{p,l}} > 1/(2p)$. But in~\cite[Theorem~1.5]{Mauldin1999ctdfrac} Mauldin and Urbański showed that $\theta_{I_{p,l}}$ is the infimum of the Hausdorff dimension of cofinite subsystems, so $\dim_{\mathrm{H}} F_{I_{p,l}} \to 1/(2p)$ as $l \to \infty$. Since $1/(2p) < 1/(p+1) = \dim_{\mathrm{B}} \{ \, 1/b : b \in I_{p,l} \, \}$, for all sufficiently large $l$ we have $\dim_{\mathrm{H}} F_{I_{p,l}} < \dim_{\mathrm{B}} F_{I_{p,l}} = 1/(p+1)$. 
   \end{proof}
The graph of the intermediate dimensions of the continued fraction set from Corollary~\ref{c:ctdfracex} in the case $p=2$ and $l$ large enough that $1/4 < \dim_{\mathrm{H}} F_{I_{p,l}} < \dim_{\mathrm{B}} F_{I_{p,l}} = 1/3$ is the black curve in Figure~\ref{fig:holder}. 
Note that the intermediate dimensions of the graph of the popcorn function in Figure~\ref{f:holder} on page~\pageref{f:holder} has a similar form. 

Recall the discussion in Section~\ref{s:holderintro}.  Example~\ref{holderint} shows that the intermediate dimensions can give better information about H{\"o}lder exponents than either the Hausdorff or box dimensions. 

\begin{example}\label{holderint}
Let $p,q$ be such that $1<p<q < 2p-1 < \infty$. As in Corollary~\ref{c:ctdfracex} there exists $l \in \N$ large enough so that if $I_{p,l} \coloneqq \{ \, \lfloor n^p \rfloor : n \geq l \}$ then 
\[ 1/(2p) < h_p < 1/(q+1) < 1/(p+1)\]
 where $h_p$ is the Hausdorff dimension of the continued fraction set $\dim_\mathrm{H} F_{I_{p,l}}$, and then $\dim_{\theta} F_{I_{p,l}} = \max \left\{ h_p, \frac{\theta}{p+\theta}\right\}$. Similarly, if $I_q$ is a subset of $\N$ whose symmetric difference with $\{ \, \lfloor n^q \rfloor : n \in \N \}$ is finite then $\dim_{\theta} F_{I_q} = \max \left\{ h_q, \frac{\theta}{q+\theta} \right\}$, where $h_q \coloneqq \dim_\mathrm{H} F_{I_q}$. If $I_q$ is also such that $h_q \in \left( \frac{ph_p}{q-qh_p+ph_p},\frac{1}{q+1} \right)$ and $f \colon F_{I_q} \to \mathbb{R}$ is an $\alpha$-H{\"o}lder map such that $f(F_{I_q}) \supseteq F_{I_{p,l}}$ then~\eqref{generalholderint} gives the best upper bound for $\alpha$ when $\theta = \frac{qh_q}{1-h_q}$, when 
\[ \alpha^{-1} h_q = \alpha^{-1} \dim_{\theta} F_{I_q} \geq  \uid f(F_{I_q}) \geq \dim_{\theta} F_{I_{p,l}} = \frac{\theta}{p+\theta} = \frac{qh_q}{p-ph_q + qh_q},\]
and so \[ \alpha \leq \frac{p-ph_q + qh_q}{q}.\] 
Using the Hausdorff dimension merely gives that $\alpha \leq h_q/h_p$, and the box dimension merely gives $\alpha \leq \frac{p+1}{q+1}$. The intermediate dimensions of the two sets and the upper bound with a certain choice of parameters are plotted in Figure~\ref{fig:holder}. 
\end{example}
\begin{figure}[ht]
\center{\includegraphics[width=0.7\textwidth]
        {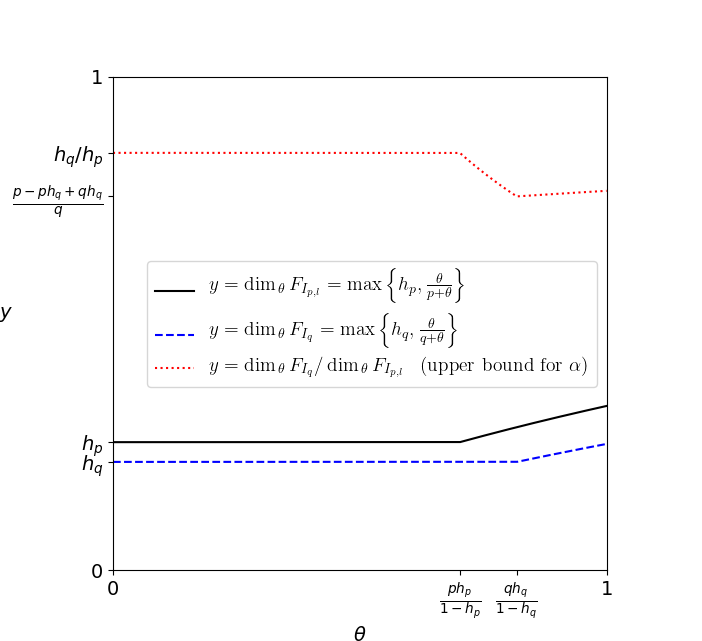}}
        \caption{\label{fig:holder}
        Graph of the intermediate dimensions of the real continued fraction sets in Example~\ref{holderint} and the upper bound for $\alpha$ against $\theta$ in the case $p=2$, $q=2.9$, $h_p \approx 0.26$, $h_q \approx 0.22$. 
 }
\end{figure}

In the following corollary of Theorem~\ref{ctdfracintthm} we apply results of Burrell~\cite{Burrell2022brownian} to give some consequences of the continuity of the intermediate dimensions of continued fraction sets for dimensions of images of $F_I$ under index-$\alpha$ fractional Brownian motion. 
Perhaps the most interesting part of Corollary~\ref{brownian} is the sufficiency of the condition $\alpha > h$ for the upper box dimension of the image to be strictly less than 1; this is an example of how the intermediate dimensions can be used to obtain information about the box dimension of sets. 

\begin{cor}\label{brownian}
Let $\alpha \in (0,1)$ and let $B_\alpha \colon \mathbb{R} \to \mathbb{R}$ denote index-$\alpha$ fractional Brownian motion. Then for all non-empty, proper subsets $I \subset \N$, recalling that $h = \dim_{\mathrm H} F_I$, 
\begin{enumerate}[label=(\roman*)]
\item\label{browncts} The maps $\theta \mapsto \uid B_\alpha(F_I)$ and $\theta \mapsto \lid B_\alpha(F_I)$ are almost surely continuous at $\theta = 0$. %
\item\label{brownlarge} If $\alpha > h$ then almost surely 
\[ h/\alpha = \dim_\mathrm{H} B_\alpha(F_I) \leq \ubd B_\alpha(F_I) < 1.\]  
\item\label{brownsmall} If $\alpha \leq h$ then almost surely 
\[ \dim_\mathrm{H} B_\alpha(F_I) = \dim_{\mathrm{B}} B_\alpha(F_I) = 1.\] 
\end{enumerate}
\end{cor}

\begin{proof}

\ref{browncts}. follows immediately from Theorem~\ref{ctdfracintthm}~\ref{ctdfraccts} and Burrell~\cite[Corollary~3.5]{Burrell2022brownian}. 

\ref{brownlarge}. The set $F_I$ is the limit set of a CIFS so it is Borel (see the discussion after Lemma~\ref{fitin}), so Kahane's general results~\cite[Chapter~18]{Kahane1985fractbrown} %
give $h/\alpha = \dim_\mathrm{H} B_\alpha(F_I)$ almost surely. The middle inequality is a general property of the dimensions, and $\ubd B_\alpha(F_I) < 1$ almost surely by Theorem~\ref{ctdfracintthm}~\ref{ctdfraccts} and Burrell~\cite[Corollary~3.7]{Burrell2022brownian}. 

\ref{brownsmall} follows from Kahane~\cite[Chapter~18]{Kahane1985fractbrown} since $F_I$ is Borel. 
\end{proof}

Note that since $\dim_\mathrm{P} B_\alpha(F_I), \lbd B_\alpha(F_I) \in (\dim_\mathrm{H} B_\alpha(F_I),\ubd B_\alpha(F_I))$, if $\alpha > h$ then almost surely $\dim_\mathrm{P} B_\alpha(F_I) < 1$ and $\lbd B_\alpha(F_I) < 1$. On the other hand, if $\alpha \leq h$ then almost surely $\dim_\mathrm{P} B_\alpha(F_I) = \lbd B_\alpha(F_I) = 1$. 

\subsection{Complex continued fractions}\label{compsect}

In this section we study sets of complex numbers which have a complex continued fraction expansion with restricted entries. For a non-empty $I \subseteq \{ \, m + n i : m \in \N, n \in \mathbb{Z} \, \}$, define 
\[ F_I \coloneqq \left\{ \, z \in \mathbb{C} : z = \frac{1}{b_1 + \frac{1}{b_2 + \frac{1}{\ddots}}}, b_n \in I \mbox{ for all } n \in \N \, \right\}. \]
If $1 \notin I$ then it can be verified, as in~\cite[Section~6]{Mauldin1996iifs}, that if $1 \notin I$ then $\{ \, S_b(z) \coloneqq 1/(b+z) : b \in I \, \}$ is a CIFS with limit set $F_I$, with $X \subset \mathbb{C}$ being the closed disc centred at 1/2 with radius 1/2, and $V = B(1/2,3/4)$. If $1 \in I$ then $S_1$ is not uniformly contracting but it is straightforward to verify that 
\[ \{ \, S_b(z) \coloneqq {1/(b+z)} : {b \in I,} {b \neq 1} \, \} \cup \left\{ \, S_{1b}(z) \coloneqq {\frac{1}{b+\frac{1}{1+z}}} : {b \in I} \, \right\}\]
 is a CIFS with the same limit set. By~\cite[Theorem~3.15]{Mauldin1996iifs}, the Hausdorff dimension can be determined by the topological pressure function, and has been studied in~\cite{Hanus1998complexctd} with estimates given in~\cite[Section~6]{Mauldin1996iifs} and~\cite{Gardner1983complexctd,Priyadarshi2016complexctd,
Falk2018complexctd,Ingebretson2020ctdfrac}. 
\begin{thm}\label{compmain}
Using the notation above, for all $\theta \in [0,1]$, 
\begin{enumerate}[label=(\roman*)]
\item\label{compctdfracint} For all $\theta \in [0,1]$, 
\begin{gather*}
 \uid F_I = \max\{h,\uid\{ \, 1/b : b \in I \, \}\}; \\*
 \max\{h,\lid\{ \, 1/b : b \in I \, \}\} \leq \lid F_I \leq \max\{h,\uid\{ \, 1/b : b \in I \, \}\}.
 \end{gather*}
\item\label{compctdfraccts} The maps $\theta \mapsto \uid F_I$ and $\theta \mapsto \lid F_I$ are continuous at $\theta = 0$.
\end{enumerate}
\end{thm}

\begin{proof}

We sketch the proof as it is similar to the proof of Theorem~\ref{ctdfracintthm}. 

\ref{compctdfracint}. follows from Theorem~\ref{mainint}. 

\ref{compctdfraccts}. follows from~\ref{compctdfracint} since $\{ \, 1/b : b \in I \, \}$ is the disjoint union of bi-Lipschitz copies of two subsets of $G_{1,2}$, whose intermediate dimensions are continuous by Proposition~\ref{p:lattice}. 
\end{proof}

\begin{corollary}\label{c:complexexample}
For $p \in (1,\infty)$ and $R \in [0,\infty)$ let 
\[ I_{p,R} \coloneqq \{ \, \lfloor m^p \rfloor + \lfloor n^p \rfloor i : {n,m \in \N} \, \} \setminus {B(0,R)}.\]
Then 
\begin{equation}\label{e:complexctdfracexample} \dim_{\theta} F_{I_{p,R}} = \max\left\{ \dim_\mathrm{H} F_{I_{p,R}} , \frac{2\theta}{p+\theta} \right\}, 
\end{equation}
and for all $R$ sufficiently large, $\dim_\mathrm{H} F_{I_{p,R}} < \dim_\mathrm{B} F_{I_{p,R}} = 2/(p+1)$. 
\end{corollary}
\begin{proof}
The set $\{ \, 1/b : {b \in I} \, \}$ is bi-Lipschitz equivalent to a cofinite subset of the set~$G_{p,2}$ from Proposition~\ref{p:lattice}, so 
\[ \dim_{\theta} \{ \, 1/b : b \in I \, \} = \dim_{\theta} G_{p,2} = \frac{2\theta}{p+\theta}.\] 
Therefore the bounds in Theorem~\ref{compmain}~\ref{compctdfracint} coincide and~\eqref{e:complexctdfracexample} holds. 
For all $t \geq 0$, writing $\simeq$ to mean up to multiplication by a positive, finite function of $t$, $p$ and $R$ and/or addition by a real-valued function of $t$, $p$ and $R$, and using the convention that $a \infty = \infty + c = \infty$ for $a \in (0,\infty)$ and $c \in \mathbb{R}$, we have
 \begin{align*}
  \psi_1(t) &= \sum_{b \in I_{p,R}} ||S_b||^t \\
  &\simeq \sum_{b \in I_{p,R}} |b|^{-2t} &\text{(Koebe distortion theorem)} \\
  &= \sum_{n = 0}^\infty \sum_{\substack{b \in I_{p,R} \\ 2^n \leq |b| < 2^{n+1}}} |b|^{-2t} \\
  &\simeq \sum_{n = 0}^\infty \# \{ \, b \in I_{p,R} : 2^n \leq |b| < 2^{n+1} \, \} (2^n)^{-2t} \\
  &\simeq \sum_{n=0}^\infty (2^{n/p})^2 (2^n)^{-2t} \\
  &= \sum_{n=0}^\infty 4^{n(p^{-1} - t)}. 
  \end{align*}
  Therefore the finiteness parameter $\theta_{I_{p,R}} = 1/p$. By~\cite[Theorem~1.5]{Mauldin1999ctdfrac}, 
  \begin{equation*}
  \dim_\mathrm{H} F_{I_{p,R}} \xrightarrow[R\to \infty]{} \frac{1}{p} < \frac{2}{p+1}, 
  \end{equation*}
  as required. 
\end{proof}

Again there are consequences for fractional Brownian images. 

\begin{cor}
Let $\alpha \in (0,1)$ and let $B_\alpha \colon \mathbb{C} \to \mathbb{C}$ denote index-$\alpha$ fractional Brownian motion (identifying $\mathbb{C}$ with $\mathbb{R}^2$). Then for all non-empty $I \subseteq \{ \, m + n i : m \in \N, n \in \mathbb{Z} \, \}$, 
\begin{enumerate}[label=(\roman*)]
\item The maps $\theta \mapsto \uid B_\alpha(F_I)$ and $\theta \mapsto \lid B_\alpha(F_I)$ are almost surely continuous at $\theta = 0$. %
\item If $\alpha > (\dim_\mathrm{H} F_I)/2$ then almost surely \[(\dim_\mathrm{H} F_I)/\alpha = \dim_\mathrm{H} B_\alpha(F_I) \leq \ubd B_\alpha(F_I) < 2.\]
\item If $\alpha \leq (\dim_\mathrm{H} F_I)/2$ then almost surely $\dim_\mathrm{H} B_\alpha(F_I) = \dim_{\mathrm{B}} B_\alpha(F_I) = 2$.
\end{enumerate}
\end{cor}

\begin{proof}
This follows from Theorem~\ref{compmain} in a similar way to how Corollary~\ref{brownian} follows from Theorem~\ref{ctdfracintthm}. 
\end{proof}

Since $\dim_\mathrm{P} B_\alpha(F_I), \lbd B_\alpha(F_I) \in (\dim_\mathrm{H} B_\alpha(F_I),\ubd B_\alpha(F_I))$, if $\alpha > (\dim_\mathrm{H} F_I)/2$ then almost surely $\dim_\mathrm{P} B_\alpha(F_I) < 2$ and $\lbd B_\alpha(F_I) < 2$, whereas if $\alpha \leq (\dim_\mathrm{H} F_I)/2$ then almost surely $\dim_\mathrm{P} B_\alpha(F_I) = \lbd B_\alpha(F_I) = 2$.

\section{Generic attractors}\label{genericsect}

\subsection{Background and motivation}

Often, and especially in non-conformal settings and in the presence of overlaps, it is difficult to compute the dimension of a particular IFS attractor.  In a seminal paper from 1988 ~\cite{Falconer1988generic} Falconer introduced the idea of studying the generic dimension of a (finitely generated) self-affine set by fixing a set of matrices and then randomising the translations in a  suitable way. It turns out that, for Lebesgue almost every choice of translations, the Hausdorff and box dimension of the associated self-affine set are given by the affinity dimension: a dimension formula depending only on the matrices. 

K\"aenm\"aki and  Reeve~\cite{Kaenmaki2014infiniteaffine} extended Falconer's  approach to the theory of infinitely generated self-affine sets.  Here one needs to randomise infinitely many translations, and do so in a manner which outputs a bounded set.  As such, the natural space to draw from is $V^\mathbb{N}$ for some bounded set $V \subseteq \Rd$ with positive $d$-dimensional Lebesgue measure $\mathcal{L}_d(V)>0$. 
For convenience from now on we assume $V= [0,1)^d$.
The space $V^{\N}$ carries a natural infinite product probability measure
\[
\mu = \prod_{i \in \mathbb{N}} \mathcal{L}_d \vert_V.
\]
 K\"aenm\"aki and  Reeve proved that if one fixes an infinite collection of strictly contracting matrices on $\Rd$ and randomises the translations according to $\mu$, then  the Hausdorff dimension of the associated attractor is almost surely given by the natural extension of the affinity dimension to the infinite case.  In comparison with Falconer's result, this notably omits the box dimension.  There is good reason for this since the box dimension of an infinitely generated attractor depends much more sensitively on the translations themselves, as we have seen above. 

We show here that almost surely the box dimension is $d$, that is, the ambient spatial dimension.  In fact we show more.  We show that for an arbitrary IIFS (not necessarily consisting of affine maps), the associated attractor is generically somewhere dense, and therefore the box and intermediate dimensions are all generically equal to $d$. 
 Moreover, \emph{generically} can refer to either $\mu$-almost surely, or for a comeagre set of translates (topologically generic). 
 If one equips $V^\mathbb{N}$ with a topological group structure as the infinite product of the group $V$ under addition mod 1 with the product topology, then being \emph{prevalent} in the sense of~\cite{Elekes2020prevalence,
Ott2005prevalence,Christensen1973prevalence,
Hunt1992prevalence} is the same as holding $\mu$-almost surely.

\subsection{Results}

Fix an IIFS $\{S_i\}_{i \in \mathbb{N}}$ defined on $[0,2]^d$ with the property that $S_i([0,2]^d) \subseteq [0,1]^d$ for all $i \in \mathbb{N}$. For $t = (t_1, t_2, \dots) \in V^\mathbb{N}$ let  $F_t$ denote the attractor of the IIFS $\{S_i+t_i\}_{i \in \mathbb{N}}$ with $S_i+t_i$ defined on $[0,2]^d$ by $(S_i+t_i)(x) = S_i(x)+t_i$. We assume throughout that the contraction ratios of the maps $S_i$ only accumulate at zero. The assumption that $[0,2]^d$ maps into $[0,1]^d$ is to ensure the maps can be composed with translations in a well-defined way. 

Write $\fix(g)$ to denote the unique fixed point of a contraction $g$ on $[0,2]^d$.  We use the following simple lemma to relate random translations to random fixed points. 
Fixed points are useful because they necessarily belong to the attractor.

\begin{lma}\label{genericlma}
Let $u \in V$, $g$  be a contraction on $[0,2]^d$ with $g([0,2]^d) \subseteq [0,1]^d$, $q \in [0,2]^d$, and $\delta>0$.  Then
\[
\fix(g+u) \in B(q,\delta) \Leftrightarrow u \in B(q-g(\fix(g+u)), \delta).
\]
\end{lma}

\begin{proof}
By the definition of a fixed point, 
\begin{align*}
\fix(g+u) \in B(q,\delta) &\Leftrightarrow ||\fix(g+u) - q || < \delta \\
&\Leftrightarrow ||g(\fix(g+u))+u - q || < \delta \\
&\Leftrightarrow u \in B(q-g(\fix(g+u)), \delta),
\end{align*}
as required.
\end{proof}

\begin{thm}\label{measure}
For $\mu$-almost all $t \in V^\mathbb{N}$, the attractor $F_t$ is somewhere dense in $[0,2]^d$, so for all $\theta \in (0,1]$,
\[
\dim_{\theta} F_t = \dim_\mathrm{B} F_t = d.
\]
\end{thm}

\begin{proof}
Let $z \in [0,1]^d$ be an accumulation point of the set $\{\fix(S_i)\}_i$. We prove that $F_t$ is almost surely dense in the explicit (unit) square $V + z$. The idea is that for infinitely many $i \in \N$ the fixed point of $S_i$ will be close to $z$ and the contraction ratio of $S_i$ will be small, so if the fixed point of $S_i + t_i$ is far away from a given point $q \in V + z$ then $t_i$ must be far away from $q-z$, which we use to bound the measure. We have 
\begin{align*}
&\left\{ \, t \in V^\mathbb{N} \ : \  F_t \text{ is nowhere dense} \, \right\}  \subseteq    \left\{ \, t \in V^\mathbb{N} \ : \  F_t \text{ is not dense in $V+z$} \, \right\} \\
& \subseteq \bigcup_{q \in (V+z) \cap \mathbb{Q}^d} \bigcup_{\delta  \in \mathbb{Q}^+} \left\{ \, t \in V^\mathbb{N} \ : \  \forall i \in \mathbb{N} , \, \fix(S_i+t_i) \notin  B(q,\delta)  \, \right\} \\
& =  \bigcup_{q \in (V+z) \cap \mathbb{Q}^d} \bigcup_{\delta  \in \mathbb{Q}^+}   \left\{ \, t \in V^\mathbb{N} \ : \   \forall i \in \mathbb{N} , \,  t_i  \notin  B(q-S_i(\fix(S_i +t_i)),\delta)  \, \right\} \quad \text{(Lemma~\ref{genericlma})} \\
&=  \bigcup_{q \in (V+z) \cap \mathbb{Q}^d} \bigcup_{\delta  \in \mathbb{Q}^+} \bigcap_{N=1}^\infty T_{q,\delta, N}
\end{align*}
where for $N \in \mathbb{N}$,
\[
T_{q,\delta, N} \coloneqq \left(\prod_{i=1}^N \{ \, t \in V : || q-S_i(\fix(S_i +t)) - t || \geq \delta \, \} \right) \times \left(\prod_{j=N+1}^\infty V \right).
\]

For each $i \in \N$, $\{ \, t \in V : || q-S_i(\fix(S_i +t)) - t || \geq \delta \, \}$ is a Borel set as it is the preimage of $[\delta,\infty)$ by a continuous function, so each $T_{q,\delta, N}$ is $\mu$-measurable. 
By definition of $z$ we can find infinitely many $i \in \mathbb{N}$ such that the maximum of $||\fix(S_i) - z||$ and the contraction ratio of $S_i$ is less than $\delta/10$. 
For all such $i \in \N$, if $t \in B(q-z,\delta/2) \cap V$ then 
\begin{align*}
 || q-S_i(\fix(S_i +t)) - t || &\leq ||q-z-t|| + ||z-\fix(S_i)|| + ||\fix(S_i) - S_i(\fix(S_i + t))|| \\*
 &< \frac{\delta}{2} + \frac{\delta}{10} + 2\frac{\delta}{10} \\
 &< \delta.
 \end{align*}
Therefore for infinitely many $i \in \mathbb{N}$,   
\[
\mathcal{L}_d \Big( \{ t \in V : || q-S_i(\fix(S_i +t)) - t || \geq \delta \} \Big) \leq 1-2^{-d} \left(\frac{\min\{\delta,1/2\}}{2}\right)^d k_d < 1, 
\] %
where $k_d$ is the $d$-dimensional Lebesgue measure of a ball in $\Rd$ of unit radius.
This uniform bound away from 1 is independent of $i$, so for all $q$ and $\delta$, $\mu(T_{q,\delta, N}) \to 0$ as $N \to \infty$, so $\bigcap_{i=1}^\infty T_{q,\delta, N}$ is $\mu$-measurable (as the countable intersection of $\mu$-measurable sets) with $\mu\left(\bigcap_{i=1}^\infty T_{q,\delta, N}\right) = 0$. 
Therefore 
\[ \bigcup_{q \in (V+z) \cap \mathbb{Q}^d} \bigcup_{\delta  \in \mathbb{Q}^+} \bigcap_{i=1}^\infty T_{q,\delta, N}\]
 is a countable union of $\mu$-measurable sets with $\mu$-measure 0, so it is itself $\mu$-measurable with $\mu$-measure 0, which proves the result. 
\end{proof}

Next, we establish a topological result. 
We endow $V^\mathbb{N}$ with the Hilbert cube metric
\[
d(t, s) = \left(\sum_{i=1}^\infty \frac{||t_i-s_i||^2}{i^2}\right)^{1/2},
\]
noting that this generates the product topology on $V^\mathbb{N}$. 
Recall that a subset of $V^{\N}$ is called \emph{residual} or \emph{comeagre} (topologically generic) if it contains a countable intersection of open dense sets. 
Note that $V$ is homeomorphic to the separable complete metric space $[0,\infty)$, so $V$ is a Polish space, hence the countable product $V^{\N}$ is also a Polish space. The Baire Category Theorem therefore implies that $V^{\N}$ is a Baire space, meaning that residual subsets are dense. 

\begin{thm}\label{baire}
For a residual set of $t \in V^\mathbb{N}$, the attractor $F_t$ is somewhere dense in $[0,2]^d$, so for all $\theta \in (0,1]$, 
\[
\dim_{\theta} F_t = \dim_\mathrm{B} F_t = d.
\]
\end{thm}

\begin{proof}
Let $z \in [0,1]^d$ be as in the proof of Theorem~\ref{measure}.  Then
\begin{align*}
& \left\{ \, t \in V^\mathbb{N} \ : \  F_t \text{ is somewhere dense} \, \right\}   \\ 
& \supseteq   \left\{ \, t \in V^\mathbb{N} \ : \  F_t \text{ is  dense in $V+z$} \, \right\} \\ 
& =  \bigcap_{q \in (V+z) \cap \mathbb{Q}^d} \bigcap_{\delta  \in \mathbb{Q}^+}  \left\{ \, t \in V^\mathbb{N} \ : \  \exists i \in \mathbb{N}, \, \fix(S_i+t_i) \in  B(q,\delta) \, \right\} \\ 
& =  \bigcap_{q \in (V+z) \cap \mathbb{Q}^d} \bigcap_{\delta  \in \mathbb{Q}^+}  \left\{ \, t \in V^\mathbb{N} \ : \  \exists i \in \mathbb{N}, \, t_i  \in  B(q-S_i(\fix(S_i+t_i)) ,\delta)  \, \right\} &\text{(Lemma~\ref{genericlma})}\\ 
& \eqqcolon  \bigcap_{q \in (V+z) \cap \mathbb{Q}^d} \bigcap_{\delta  \in \mathbb{Q}^+} T_{q,\delta}.
\end{align*}
Fix $q \in (V+z) \cap \mathbb{Q}^d $ and $ \delta  \in \mathbb{Q}^+$. The set $T_{q,\delta}$ is immediately seen to be open since $\fix(S_i+t_i)$ is continuous in $t_i$ and we use open balls. Moreover, it is also dense since an element $t \in V^{\mathbb{N}}$ may be approximated arbitrarily well in the metric $d$ within the set $T_{q,\delta}$ by replacing $t_i$ with $q-z$ for sufficiently large $i$.
\end{proof}

\begin{rem}\label{modifycontractions}
In the above setting, all the contraction ratios were bounded above by 1/2, but we can easily avoid this. Indeed, fix any $c>0$ and fix any IIFS $\{S_i\}_{i \in \N}$ of contractions defined on $[0,1+c]^d$ satisfying $S_i([0,1+c]^d) \subseteq [0,1]^d$ for all $i \in \N$. Then the contraction ratios are bounded above by $\frac{1}{1-c}$ and we assume they accumulate only at 0. We can let $V \coloneqq [0,c)^d$, and again $V^\N$ can be equipped with a natural probability measure by taking the infinite product of the Lebesgue measure on $V$ and then normalising. Moreover, $V^\N$ can be equipped with a topological group structure by taking the infinite product (with the product topology) of the group $V$ under addition \emph{mod $c$} on each of the $d$ coordinates. 
If $t = (t_1,t_2,\dotsc) \in V^\N$ then $\{S_i + t_i\}_{i \in \N}$ is an IIFS of contractions on $[0,1+c]^d$, with limit set $F_t$, say. Similar proofs show that under the assumptions of either Theorem~\ref{measure} or Theorem~\ref{baire}, $F_t$ is somewhere dense in $[0,1+c]^d$, so the same conclusions about dimensions hold. 
\end{rem}

Under the assumptions of either Theorem~\ref{measure} or Theorem~\ref{baire}, either in the setting of those theorems or in the more general setting of Remark~\ref{modifycontractions}, the attractor $F_t$ is somewhere dense. This means that if $\dim$ is any notion of dimension which is stable under closure, for example any of the $\Phi$-intermediate or Assouad type dimensions, then $\dim F_t = d$. 
If all of the $S_i$ are bi-Lipschitz, then by a result of Mauldin and Urbański~\cite[Theorem~3.1]{Mauldin1996iifs} the upper box and packing dimensions of the attractor coincide, so under the assumptions of Theorem~\ref{measure} or~\ref{baire}, it holds that $\dim_\mathrm{P} F_t = d$.

\chapter{Bedford--McMullen carpets}\label{s:bm}

\section{Introduction}\label{sec:01}

\subsection{Self-affine carpets}

In this chapter, which is based on~\cite{Banaji2021bedford} (joint with I.~Kolossv\'ary), we consider \emph{self-affine sets}, which are the attractors of IFSs consisting of affine contractions. The dimension theory of self-affine sets (surveyed in~\cite{Falconer2013affine}) is broadly divided into two strands of research. The former is concerned with verifying that in many `generic' situations the Hausdorff and box dimensions coincide with Falconer's affinity dimension~\cite{Barany2019HochmanRapaport,
Falconer1988generic}. This chapter, however, is concerned with the exceptional theory, where the dimensions can take different values, and the intermediate dimensions are therefore relevant. 
In particular, self-affine sets in the plane for which the matrices of the defining contractions are diagonal are often called `carpets,' and various models with different levels of generality have been studied~\cite{Bedford1984phd,McMullen1984carpet,
Lalley1992gatzouras,Baranski2007carpet,Feng2005lq,
Kenyon1996bmexactdim}. 
The three-dimensional versions are often called sponges, and surprisingly a class of such sets appear in the paper~\cite{Das2017dimgap} as the first examples of expanding repellers whose Hausdorff dimension is not attained as the Hausdorff dimension of any ergodic invariant measure; it is not known whether a repeller with this property exists in the plane. There are many interesting open problems about self-affine sets, such as whether the box dimension of every self-affine set exists. 

In this chapter, we work with the simplest model of self-affine carpets, which were originally introduced independently by~Bedford~\cite{Bedford1984phd} and McMullen~\cite{McMullen1984carpet}. %
A Bedford--McMullen carpet $\Lambda$ is a subset of the plane, but can also be viewed as an invariant subset of the 2-torus $[0,1)^2$ under the toral endomorphism $(x,y)\mapsto (mx \mod 1, ny \mod 1)$.  
Figure~\ref{f:carpetIFS} shows a simple example of a Bedford--McMullen carpet with distinct Hausdorff and box dimension. The three shaded rectangles show the image of $[0,1]^2$ under the three maps in the IFS, and the attractor is also shown. For this carpet, $\dim_{\mathrm H}\Lambda\approx1.34968<1.36907\approx\dim_{\mathrm B}\Lambda$. 
These carpets can be realised as cross-sections of invariant sets of discrete dynamical systems given by iterating three-dimensional horseshoe maps, showing that dynamically invariant sets can have distinct Hausdorff and box dimensions. 
We refer the interested reader to the survey~\cite{Fraser2021bedford} for background on the dimension theory of these carpets. 
Our main objective is to determine a formula for the intermediate dimensions of all Bedford--McMullen carpets. In the process, we uncover new interesting features about the form of the intermediate dimensions and make an unexpected connection to multifractal analysis and bi-Lipschitz equivalence of these carpets. 
\begin{figure}[ht]
\centering
\includegraphics[width=.6\textwidth]{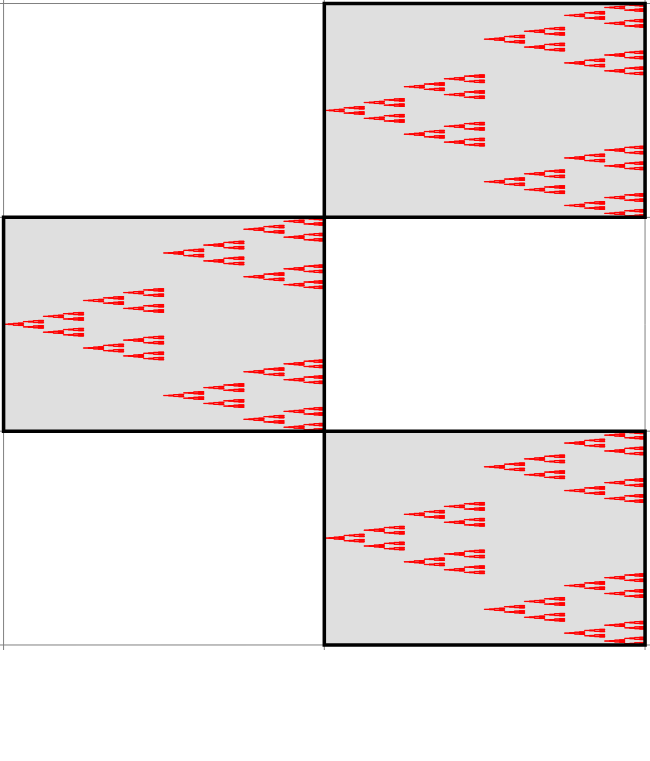}
\caption{A Bedford--McMullen carpet with non-uniform fibres. The images of $[0,1]^2$ under the iterated function system generating the carpet are shaded.}\label{f:carpetIFS}
\end{figure}
\begin{figure}[ht]
\centering
\subfigure{\includegraphics[width=.9\textwidth]{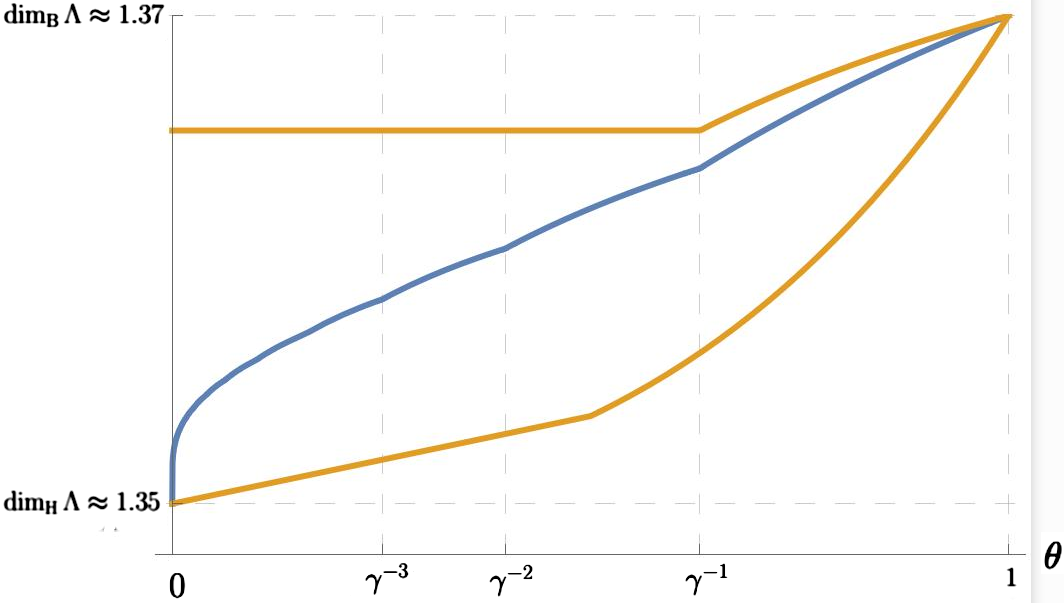}}
\caption{The graph of the intermediate dimensions of the Bedford--McMullen carpet from Figure~\ref{f:carpetIFS} is plotted in \textcolor{blue}{blue}, using the formula obtained in this chapter. Note that there are countably many phase transitions (one at each negative integer power of $\gamma$, the first three of which are labelled). 
Certain bounds which were obtained previously are plotted in \textcolor{orange}{orange}.}\label{f:bmintdimgraph}
\end{figure}

These carpets are constructed by splitting $[0,1]^2$ into $m$ columns of equal width and $n$ rows of equal height for some integers $n>m\geq 2$ and considering maps of the form
\begin{equation*}
	f_{(i,j)}(\underline{x})\coloneqq \begin{pmatrix} 1/m & 0 \\ 0 & 1/n \end{pmatrix} \begin{pmatrix} x \\ y \end{pmatrix} + \begin{pmatrix} (i-1)/m \\ (j-1)/n
	\end{pmatrix}
\end{equation*}
for the index set $(i,j)\in \mathcal{A}\subseteq \{1,\dotsc,m\}\times\{1,\dotsc,n\}$. 
The attractor 
\begin{equation*}
	\Lambda = \bigcup_{(i,j)\in \mathcal{A}} f_{(i,j)} ( \Lambda)
\end{equation*}
of the IFS $\mathcal{F}=\{f_{(i,j)}\}_{(i,j)\in \mathcal{A}}$ is called a \emph{Bedford--McMullen carpet}. 
Fix notation 
\[\gamma \coloneqq \frac{\log n}{\log m}>1.\]
For the remainder of the chapter, we index the maps of $\mathcal{F}$ by $i\in\{1,\dotsc,N\}$. Let $1<M\leq m$ denote the number of non-empty columns and $\mathbf{N}\coloneqq (N_1,\dotsc, N_M)$, where $N_{\jh}$ is the number of maps $f_i$ that map $[0,1]^2$ to the $\jh$-th non-empty column. Observe that $N=N_1+\dotsb+N_M$. 

Let $\mathcal{P}_M$ denote the set of probability vectors on $\{1,\dotsc,M\}$. 
The \emph{entropy} of $\mathbf{q}\in\mathcal{P}_M$ is
\begin{equation*}
	H(\mathbf{q}) = -\sum_{\jh=1}^M q_{\jh}\log q_{\jh},
\end{equation*}
where we use the convention that $0 \log 0 = 0$. 
We introduce
\begin{equation*}
	\mathbf{P}=(P_1,\dotsc,P_M) \coloneqq \left(\frac{N_1}{N},\dotsc,\frac{N_M}{N} \right)  \quad \mbox{ and } \quad   \mathbf{Q} \coloneqq \left( \frac{1}{M}, \dotsc, \frac{1}{M} \right) . 
\end{equation*}
For $\mathbf{q}\in\mathcal{P}_M$, it holds that $H(\mathbf{q})\leq \log M$ with equality if and only if $\mathbf{q} = \mathbf{Q}$. 
For the entire chapter, we also introduce  
\begin{equation}\label{e:underlineoverlinedef}
\underline{t}\coloneqq \frac{1}{M} \sum_{\jh=1}^{M} \log N_{\jh} \quad \mbox{ and } \quad \overline{t}\coloneqq \log N-H(\mathbf{P}).
\end{equation}
We say that $\Lambda$ has \emph{uniform (vertical) fibres} if and only if $\mathbf{P}=\mathbf{Q}$, in other words each non-empty column has the same number of maps. 

Bedford and McMullen showed that the Hausdorff and box dimensions of $\Lambda$ are equal to 
\begin{align}
	\dim_{\mathrm H}\Lambda &= \frac{1}{\log m}\log \Bigg( \sum_{\jh=1}^MN_{\jh}^{\log m / \log n} \Bigg) \label{eq:102}, \\*          %
	\dim_{\mathrm B}\Lambda &= \frac{\log N}{\log n} + \left(1-\frac{\log m}{\log n}\right)\frac{\log M}{\log m}. \label{eq:103} 
\end{align}
In particular, $\dim_{\mathrm H}\Lambda = \dim_{\mathrm B}\Lambda$ if and only if $\Lambda$ has uniform fibres. 
In this case, $\theta\mapsto \dim_{\theta}\Lambda$ is just a constant function. Therefore, we assume throughout that the carpet has non-uniform fibres, in which case the appropriate dimensional Hausdorff measure of $\Lambda$ is infinite~\cite{Peres1994infinitemeasure}. Using the concavity of the logarithm function, an immediate consequence of non-uniform fibres is that $\underline{t} < \log(N/M)$. 
Carpets with uniform and non-uniform fibres are shown in \cite[Figure~15.1]{Fraser2021bedford}. 

\subsection{Previous results}

Previous papers on the topic~\cite{Falconer2020firstintermediate,
Fraser2021bedford,Kolossvary2022bm} have established crude bounds for the intermediate dimensions and speculated about the possible form. The question of determining the intermediate dimensions of all Bedford--McMullen carpets was explicitly asked in \cite{Falconer2021intdimsurvey,
Falconer2020firstintermediate,Fraser2021bedford,  Kolossvary2022bm}. Loosely speaking, the results of Falconer, Fraser and Kempton~\cite{Falconer2020firstintermediate} concentrate on the behaviour of $\dim_{\theta}\Lambda$ for $\theta$ close to 0, while the results of Kolossv\'ary~\cite{Kolossvary2022bm} concentrate on the behaviour for $\theta\geq \gamma^{-1}$. 

A linear lower bound for the intermediate dimensions of Bedford--McMullen carpets with non-uniform fibres was obtained in~\cite{Falconer2020firstintermediate} which shows that $\underline{\dim}_{\theta}\Lambda>\dim_{\mathrm H}\Lambda$ for every $\theta\in(0,1]$. %
For many carpets, but not all, a lower bound in~\cite{Kolossvary2022bm} performs better than the linear bound and general bounds such as Corollary~\ref{c:lower-bound-general} from page~\pageref{c:lower-bound-general} for large values of $\theta$. The lower bound depicted in Figure~\ref{f:bmintdimgraph} is the best combination of these results. 
Falconer, Fraser and Kempton~\cite[Proposition~4.1]{Falconer2020firstintermediate} show an upper bound of the form $\overline{\dim}_{\theta}\Lambda \leq \dim_{\mathrm H}\Lambda + c/(-\log \theta)$ for an explicit $c>0$ and $\theta$ sufficiently small. In particular, this implies that $\theta\mapsto\underline{\dim}_{\theta}\Lambda$ and $\theta\mapsto\overline{\dim}_{\theta}\Lambda$ are continuous also at $\theta=0$. Hence, the results of Burrell, Falconer and Fraser~\cite[Section~6]{Burrell2021projections} and Burrell~\cite[Section~3]{Burrell2022brownian} can be applied. For example, if $\dim_{\mathrm H} \Lambda < 1$ then $\overline{\dim}_{\mathrm B} \pi(\Lambda) < 1$ for every orthogonal projection $\pi$ from $\mathbb{R}^2$ onto a 1-dimensional subspace, regardless of the value of $\dim_{\mathrm B} \Lambda$. For almost every projection, $\underline{\dim}_{\theta} \pi(\Lambda)$ and $\overline{\dim}_{\theta}\pi(\Lambda)$ are continuous at $\theta = 0$, and if $\gamma \notin \mathbb{Q}$ then this holds for \emph{every} orthogonal projection. 
Furthermore, if $B_h \colon \mathbb{R}^2 \to \mathbb{R}^2$ is index-$h$ fractional Brownian motion, then $\theta\mapsto\underline{\dim}_{\theta}B_h(\Lambda)$ and $\theta\mapsto\overline{\dim}_{\theta}B_h(\Lambda)$ are almost surely continuous, and if $h > (\dim_{\mathrm H} \Lambda)/2$ then almost surely $\overline{\dim}_{\mathrm B} B_h (\Lambda) < 2$. 

A cover of $\Lambda$ is constructed in~\cite{Kolossvary2022bm} using just the two extreme scales to obtain an explicit upper bound of the form $\overline{\dim}_{\theta}\Lambda\leq\dim_{\mathrm B}\Lambda -\Delta(\theta)$ for $\theta\geq \gamma^{-1}$, where $\Delta(\theta)\searrow 0$ as $\theta\to 1$ and has a strictly positive derivative at $\theta=1$. This bound was used to show that $\overline{\dim}_{\theta}\Lambda$ is not concave for the whole range of $\theta$ in general, already hinting at richer behaviour than previously witnessed in other examples. Figure~\ref{f:bmintdimgraph} shows this upper bound. 

\subsection{Summary of results}
The formal statements are presented in Section~\ref{sec:CompResults}. In Theorem~\ref{thm:main} we state an explicit formula for $\underline{\dim}_{\theta}\Lambda=\overline{\dim}_{\theta}\Lambda$ for all $\theta$, thus fully resolving the problem of calculating the intermediate dimensions of all Bedford--McMullen carpets $\Lambda$. For illustration see Figure~\ref{f:bmintdimgraph}, where $\theta\mapsto\dim_{\theta}\Lambda$ is plotted for the carpet from Figure~\ref{f:carpetIFS}. 
Central to the formula is a large deviations rate function for which we give three additional equivalent characterisations in Proposition~\ref{prop:1}: one in terms of a pressure-like function, another as a certain probability vector with an entropy maximising property, and finally a relationship to the multifractal spectrum of the uniform self-affine measure on $\Lambda$. In the proof of Theorem~\ref{thm:main}, we construct a cover of $\Lambda$ that uses an increasing number of different scales in the permissible range as $\theta\to 0$. We also show in Corollary~\ref{cor:twoscales} that using more than two scales is necessary.

In Corollary~\ref{cor:allprop}, we prove all the features suggested by the plot in Figure~\ref{f:bmintdimgraph} about the form of the intermediate dimensions for all carpets. Namely, $\theta\mapsto\dim_{\theta}\Lambda$ is strictly increasing and has phase transitions at all negative integer powers of $\gamma$. Between consecutive phase transitions the intermediate dimensions are analytic and strictly concave. Moreover, for $\theta$ small enough $\dim_{\theta}\Lambda$ behaves like $\dim_{\mathrm H}\Lambda+c(\log\theta)^{-2}$. In particular, the derivative tends to $+\infty$ as $\theta\to0$. No previous family of sets has shown such rich and complex behaviour. Some illustrative examples are presented in Section~\ref{subsec:ex}.

We show in Theorem~\ref{thm:multifractal} that two different carpets with non-uniform fibres have equal intermediate dimensions for every $\theta\in[0,1]$ if and only if the multifractal spectra of the uniform Bernoulli measure on the two carpets are equal. If, in addition, it is assumed that the two carpets are defined on the same grid, then Theorem~\ref{thm:multifractal} provides further equivalent conditions for their intermediate dimensions to be the same: a certain condition on the rate functions appearing in the formula or certain relationships between the parameters of the carpets, or the equality of the intermediate dimensions on any one open interval of $[\gamma^{-1},1]$. 

Our main application relates to bi-Lipschitz equivalence. It is known~\cite[Corollary~1.1]{RaoPreprintlipschitz} that the equality of these multifractal spectra is necessary for two carpets to be bi-Lipschitz equivalent if it is assumed that the two carpets are defined on the same grid and are totally disconnected. Since bi-Lipschitz maps preserve intermediate dimensions, Theorem~\ref{thm:multifractal} implies that both of these assumptions can be dropped, see Corollary~\ref{cor:biLip}. In Proposition~\ref{p:biLip} we construct two carpets which are not bi-Lipschitz equivalent by Corollary~\ref{cor:biLip}, but where this does not follow from~\cite{RaoPreprintlipschitz}. This is the first instance where intermediate dimensions are used to show that two sets are not bi-Lipschitz equivalent, but where this fact does not follow from any other notion of dimension or existing result. For this example, we also use the intermediate dimensions to give estimates on the H\"older distortion of the two carpets. 

For comparison, we mention that the calculation of the Assouad spectrum of Bedford--McMullen carpets~\cite{Fraser2018secondassouad} is not as involved as the proof of Theorem~\ref{thm:main} for the intermediate dimensions. Indeed, the intermediate dimensions are more subtle in that they depend on all the $N_i$ individually, as does the Hausdorff dimension. This is in contrast to the Assouad spectrum (and indeed the lower spectrum and the box, packing, Assouad, quasi-Assouad, lower and quasi-lower dimensions) which depend only on $m,n,N,M,\max_{1 \leq \ih \leq M} N_{\ih}$ and $\min_{1 \leq \ih \leq M} N_{\ih}$, see~\cite{Fraser2021bedford}. The Assouad spectrum also has just one phase transition at $\theta=\gamma^{-1}$, which occurs when the spectrum reaches the Assouad dimension and thus is constant for $\theta\in(\gamma^{-1},1)$.

\section{Results and examples}\label{sec:CompResults}

\subsection{Main result: formula for intermediate dimensions}

Recalling~\eqref{e:underlineoverlinedef}, let $t\in(\underline{t},\overline{t})$. 
This is a non-empty interval because the fact that there are non-uniform fibres implies that $\underline{t} < \log (N/M) < \overline{t}$. 
Let $X_1,X_2,\dotsc,X_J,\dotsc $ be a sequence of independent and identically distributed random variables taking values in the set $\{\log N_1,\dotsc,\log N_M\}$, with 
\begin{equation}\label{e:definerandomvar} 
\mathds{P}(X_1 = \log N_{\ih}) = \frac{1}{M} \cdot \# \{ \, \jh \in \{1,\dotsc,M\} : N_{\jh} = N_{\ih} \, \} .
\end{equation}
 Then $\underline{t}$ is the expectation of $X_1$. 
The large deviations rate function of the average $\frac{1}{J}\sum_{i=1}^J X_i$ is
\begin{equation}\label{eq:22}
I(t)=\sup_{\lambda\in \R} \bigg\{\lambda t - \log\bigg( \frac{1}{M} \sum_{\jh=1}^{M} N_{\jh}^{\lambda}\bigg)\bigg\},
\end{equation} 
noting that $\frac{1}{M} \sum_{\jh=1}^{M} N_{\jh}^{\lambda}$ is the expectation of $e^{\lambda X_1}$. 
For $t \in [\underline{t},\max_{1 \leq \ih \leq M} \log N_{\ih})$, differentiating shows that the supremum in the definition of $I(t)$ is attained at the unique $\lambda \geq 0$ satisfying $t = \sum_{\ih=1}^M \frac{ N_{\ih}^\lambda }{\sum_{\jh=1}^M N_{\jh}^\lambda}\log N_{\ih}$. This allows $I(t)$ to be calculated numerically. On this interval, $I(t)$ is real analytic (as the Legendre transform of an analytic function). %
The derivative $I'(t)>0$ is the value of $\lambda$ at which the supremum is attained and $I(t)$ is strictly increasing for $t\in[\underline{t},\max \log N_{\ih}]$. %
Moreover, $I''(t) > 0$, so the function is strictly convex on this interval. Some particular values of interest are 
\begin{alignat*}{2}
I(\underline{t}) &= 0, \qquad \qquad &&I'(\underline{t}) = 0, \\ 
I(\overline{t}) &= \log M - H(\mathbf{P}), \qquad \qquad &&I'(\overline{t}) = 1, 
\end{alignat*}
see~\cite[Lemma~2.2.5]{Dembo2010largedeviations}. 
Moreover, 
\[ I(\max_{1 \leq \ih \leq M} \log N_{\ih}) = \log M - \log \# \{ \, \jh \in \{1,\dotsc,M\} : N_{\jh} = \max_{1 \leq \kh \leq M} N_{\kh} \, \}, \]
and 
\[ I'(t) \to \infty\quad \mbox{ as } t \to (\max_{1 \leq \ih \leq M} \log N_{\ih})^- . \] 

For $s \in \mathbb{R}$, %
we define the function $T_s \colon \mathbb{R} \to \mathbb{R}$ by 
\begin{equation}\label{eq:defineiteratingfunction}
T_s(t) \coloneqq \left(s-\frac{\log M}{\log m}\right)\log n +\gamma I(t). 
\end{equation}
For $\ell \in \mathbb{N}$ we denote the composition by $T_s^{\ell} \coloneqq \underbrace{T_s \circ \dotsb \circ T_s }_{\ell \mbox{ times}}$, and $T_s^0$ denotes the identity function. 
We use the sequences $(t_{\ell})_{\ell=1}^\infty = (t_{\ell} (s))_{\ell=1}^\infty$ defined by 
\begin{equation}\label{eq:definetsequence}
t_{\ell}(s) \coloneqq T_s^{\ell-1}\left(\left(s-\frac{\log M}{\log m}\right)\log n \right).
\end{equation}
Note that these depend only on $s$ and the carpet, but not on $\theta$. Observe that $T_s(\underline{t}) = t_1(s)$ for all $s \in \mathbb{R}$. 
We are now ready to state the main result of this chapter (see Section~\ref{sec:proofofmain} for the rather involved proof). 
\begin{theorem}\label{thm:main}
Let $\Lambda$ be a Bedford--McMullen carpet with non-uniform fibres. For all $\theta \in (0,1)$, $\dim_{\theta} \Lambda$ exists %
and is given in the following way. For fixed $\theta \in (0,1)$ let $L = L(\theta) \coloneqq 1 + \lfloor \frac{-\log \theta}{\log \gamma} \rfloor$, so $\gamma^{-L} < \theta \leq \gamma^{-(L-1)}$. %
Then there exists a unique solution $s = s(\theta) \in (\dim_{\mathrm H} \Lambda, \dim_{\mathrm B} \Lambda)$ to the equation 
\begin{equation}\label{eq:mainformula}
\gamma^L\theta \log N - (\gamma^L\theta - 1) t_L(s) + \gamma(1-\gamma^{L-1}\theta)(\log M - I(t_L(s))) - s\log n = 0, 
\end{equation}
and $s(\theta) = \dim_{\theta} \Lambda$. 
\end{theorem}
In the case $L=1$ the formula~\eqref{eq:mainformula} simplifies to 
\[ \dim_{\mathrm B} \Lambda - \frac{1}{\log n} \left( \frac{1}{\theta} - 1\right)I(t_1(s)) - s = 0.\] 
If $\theta = \gamma^{-(L-1)}$ for some $L \in \mathbb{N}$ with $L \geq 2$, then it becomes 
\[ \dim_{\mathrm B}\Lambda - \frac{1}{\log m}\left(1-\frac{1}{\gamma}\right)I(t_{L-1}(s)) - s = 0.\] %

Theorem~\ref{thm:main} and Corollary~\ref{cor:allprop} below fully resolve \cite[Problem~15.8.1]{Fraser2021bedford} and the questions about Bedford--McMullen carpets in Falconer's survey paper~\cite[Section~14.8]{Falconer2021intdimsurvey}, and indeed provide more information. In particular, this is the first time it has been shown that the intermediate dimensions of Bedford--McMullen carpets exist for $\theta \in (0,1)$. %
Tools used in the proof include the method of types (see~\cite{Kolossvary2022lqtypes}) and a variant of a mass distribution principle for the intermediate dimensions, see Proposition~\ref{prop:mdp} from page~\pageref{prop:mdp}. 
The proof of the upper bound involves the construction of an explicit cover using scales $\delta, \delta^{\gamma}, \delta^{\gamma^2}, \dotsc, \delta^{\gamma^{L-1}}$ and $\delta^{1/\theta},\delta^{1/(\gamma\theta)}, \dotsc, \delta^{1/(\gamma^{L-1}\theta)}$. This cover consists of \emph{approximate squares}, which we define in Section~\ref{subsec:approxsquare}. We decide which parts of each approximate square to cover at which scale depending on how the different parts of the symbolic representation of the approximate square relate to each other. 
The proof simplifies when $\theta \geq 1/\gamma$ (where we just use the smallest and largest permissible scales), and when $\theta = \gamma^{-k}$ for $k \in \mathbb{N}$ (where we use scales $\delta,\delta^{\gamma},\dotsc,\delta^{\gamma^k}$ due to scales `lining up'). 
Note that the cover jumps from using $2k$ scales when $\theta \in (\gamma^{-k},\gamma^{-(k-1)})$ to using $2k + 2$ scales when $\theta \in (\gamma^{-(k+1)},\gamma^{-k})$ (and uses only $k$ scales when $\theta = \delta^{-k}$), which gives an indication of why one might expect a phase transition at $\theta = \gamma^{-k}$.

For sets whose intermediate dimensions have previously been calculated such as spirals~\cite{Burrell2022spirals}, sequences~\cite[Section~3.1]{Falconer2020firstintermediate}, and concentric spheres and topologist's sine curves~\cite{Tan2020}, only the two extreme scales were used in the cover. Several of the results in this thesis, however, use many different scales to obtain the upper bound. In particular, we see from~\eqref{intkeybound} on page~\pageref{intkeybound} that many different scales in the interval $[\delta,\delta^\theta]$ are generally used in the construction of the cover for infinitely generated self-similar sets. Moreover, it is clear from the construction in Theorem~\ref{t:gen-h-form} on page~\pageref{t:gen-h-form} that there are inhomogeneous Moran sets for which every cover approximating the intermediate dimensions arbitrarily closely would require an unbounded number of scales as $\delta$ tends to zero. This answers a question of Falconer~\cite[Section~14.8]{Falconer2021intdimsurvey}. Corollary~\ref{cor:twoscales}, which we prove in Section~\ref{sec:proofcorollaries}, %
shows that for Bedford--McMullen carpets with non-uniform fibres, more than two scales are needed when $\theta$ is small. 
\begin{corollary}\label{cor:twoscales}
Let $\Lambda$ be a Bedford--McMullen carpet with non-uniform fibres. There exist $\theta_0, \epsilon, \delta_0 > 0$ such that for all $\theta \in (0,\theta_0)$ and all covers $\{U_i\}$ of $\Lambda$ that uses at most two scales, both of which are less than $\delta_0$, we have $\sum_i |U_i|^{\dim_{\theta} \Lambda + \epsilon} \geq 1$. 
\end{corollary}%

We now continue with corollaries of Theorem~\ref{thm:main} about the form of the graph of the function $\theta \mapsto \dim_{\theta} \Lambda$ that do not follow from the general theory, give a rather unexpected connection to multifractal analysis and bi-Lipschitz equivalence of two carpets, and provide equivalent formulations of the rate function $I(t)$.

\subsection{Form of the intermediate dimensions}

We assume that the carpet has non-uniform fibres, otherwise, $\theta \mapsto \dim_{\theta} \Lambda$ is just a constant function. %
We denote the left and right derivatives at $\theta$ by 
\begin{equation*}
\partial_- \dim_{\theta} \Lambda \coloneqq \lim_{h \to 0^+} \frac{\dim_{\theta} \Lambda - \dim_{\theta - h} \Lambda}{h} \;\text{ and }\; \partial_+ \dim_{\theta} \Lambda \coloneqq \lim_{h \to 0^+} \frac{\dim_{\theta + h} \Lambda - \dim_{\theta} \Lambda}{h}. 
\end{equation*}

\begin{corollary}\label{cor:allprop}
Let $\Lambda$ be a Bedford--McMullen carpet with non-uniform fibres. Then the function $\theta \mapsto \dim_{\theta} \Lambda$ has the following properties:
\begin{enumerate}[label=(\roman*)]
\item it is real analytic on the interval $(\gamma^{-L},\gamma^{-(L-1)})$ for all $L \in \mathbb{N}$; \label{itemi}
\item $\partial_- \dim_{\theta} \Lambda$ exists at every $\theta \in (0,1]$ and $\partial_+ \dim_{\theta} \Lambda$ exists at every $\theta \in (0,1)$; \label{itemii}
\item it is strictly increasing and has phase transitions at every negative integer power of $\gamma$. More precisely, there exists $C_0>0$ depending only on $\Lambda$ such that for all $\theta \in (0,1)$,
\begin{equation*}
C_0 < \partial_- \dim_{\theta} \Lambda  \leq  \partial_+ \dim_{\theta} \Lambda
\end{equation*}
with equality if and only if for all $L \in \mathbb{N}$ we have $\theta\neq \gamma^{-L}$. Moreover, $\frac{\partial_+ \dim_{\gamma^{-L}} \Lambda}{\partial_- \dim_{\gamma^{-L}} \Lambda}$ converges to a constant in $(1,\infty)$ as $L \to \infty$; \label{itemiii}
\item there exist $C \in [1,\infty)$ and $\theta_0 > 0$ depending only on $\Lambda$ such that for all $\theta \in (0,\theta_0]$,
\[ \dim_{\mathrm H} \Lambda + \frac{C^{-1}}{(\log \theta)^2} \leq \dim_{\theta} \Lambda \leq \dim_{\mathrm H} \Lambda + \frac{C}{(\log \theta)^2}; \]
\label{itemiv}
\item it is strictly concave on the interval $[\gamma^{-L},\gamma^{-(L-1)}]$ for all $L \in \mathbb{N}$. \label{itemv}
\end{enumerate}
\end{corollary}
In~\cite[Proposition~4.1]{Falconer2020firstintermediate} for small enough $\theta$ the upper bound $\overline{\dim}_{\theta}\Lambda\leq \dim_{\mathrm H}\Lambda+c(-\log \theta)^{-1}$ was proved for a constant $c$ depending only on $\Lambda$. Corollary~\ref{cor:allprop}~\ref{itemiv} shows that although this bound is not sharp, we do indeed have that $\frac{\dim_{\theta} \Lambda - \dim_{\mathrm H} \Lambda}{\theta} \to \infty$ as $\theta \to 0^+$. 

\subsection{Multifractal analysis and bi-Lipschitz equivalence}\label{subsec:multifractal}

In this section, and in Section~\ref{s:multifractalproof} where we prove the results in this section, it is convenient to change notation. Here, the parameters $(M_0,N_1,\dotsc,N_{M_0}, R_1,\dotsc, R_{M_0})$ will define a Bedford--McMullen carpet. 
Now $M_0$ denotes the number of different values that the number of maps in a non-empty column can take. 
Note that $M_0\geq 2$, since $M_0=1$ corresponds to the uniform fibre case. 
We write $N_1,\dotsc,N_{M_0}$ for the actual values that the number of maps in a non-empty column can take, and we order them as $N_1>N_2>\dotsb>N_{M_0}$. 
For each $\ih \in \{1,\dotsc,M_0\}$, we write $R_{\ih}$ for the number of columns containing exactly $N_{\ih}$ maps. 
As with the previous notation, we write $M = \sum_{\ih = 1}^{M_0} R_{\ih}$ for the number of non-empty columns, and $N = \sum_{\ih = 1}^{M_0} R_{\ih} N_{\ih}$ for the total number of maps. 
For example, for the carpet in Figure~\ref{f:carpetIFS}, the number of maps in a non-empty column is either~$1$ or~$2$, so $N_1 = 2$ and $N_2 = 1$; each corresponds to just one column, so $R_1 = R_2 = 1$, and $M_0 = \# \{1,2\} = 2$.

A central problem in multifractal analysis is to examine the way a Borel measure $\mu$ is spread over its support $\supp\, \mu$. For a survey of this topic, we refer the reader to~\cite[Chapter~17]{Falconer2014main}. 
The \emph{local dimension of $\mu$ at $x$} is
\begin{equation*}
\dim_{\mathrm{loc}}(\mu,x) = \lim_{r\to 0} \frac{\log \mu(B(x,r))}{\log r}
\end{equation*}
if the limit exists, which approximately measures the rate of decay of $\mu(B(x,r))$ as a power law $r^{\alpha}$. The measure $\mu$ is \emph{exact dimensional} if $\dim_{\mathrm{loc}}(\mu,x)$ is equal to a specific $\alpha$ for $\mu$-almost all $x$. However, $\dim_{\mathrm{loc}}(\mu,x)$ can still potentially take up a whole spectrum of different $\alpha$. This motivates the definition of the \emph{fine} or \emph{Hausdorff multifractal spectra}
\begin{equation*}
f_{\mu}(\alpha) \coloneqq \dim_{\mathrm H} \{ \, x\in\supp\, \mu : \dim_{\mathrm{loc}}(\mu,x)=\alpha \, \}.
\end{equation*}

Concentrating on the self-affine setting, given a self-affine iterated function system $\mathcal{S}=\{S_1,\dotsc,S_N\}$, meaning that all $S_i\colon \Rd\to\Rd$ are contracting affine maps, and a probability vector $\mathbf{p}$ with strictly positive entries, the \emph{self-affine measure} $\mu_{\mathbf{p}}$ is the unique probability measure supported on the attractor of $\mathcal{S}$ satisfying
\begin{equation*}
\mu_{\mathbf{p}}(A) = \sum_{i=1}^N p_i \mu_{\mathbf{p}}(S_i^{-1}A) \;\text{  for all Borel sets } A\subset \Rd.
\end{equation*} 
It is known that all self-affine measures are exact dimensional~\cite{Barany2017exactdim,
Feng2023exactdim}; this was resolved earlier in~\cite{Kenyon1996bmexactdim} for those supported on Bedford--McMullen carpets. Moreover, the dimension satisfies a Ledrappier--Young type formula, as per the strand of research initiated in~\cite{Ledrappier1985youngfirst,Ledrappier1985youngsecond}. 
The fine multifractal spectrum of self-affine measures on Bedford--McMullen carpets is also known~\cite{Barral2007bmmultifractal,
Jordan2011bm,King1995bmmultifractal} and (under the separation condition assumed in~\cite{King1995bmmultifractal}) has been generalised to higher dimensions in~\cite{Olsen1998sponge}. 
When $\mathbf{p}=(1/N,\dotsc,1/N)$ is the uniform vector, we simply write $\nu=\mu_{\mathbf{p}}$ and call it the \emph{uniform self-affine measure}. 
In this case, define the function $\beta_{\nu}(\xi)$ for $\xi \geq 0$ by \begin{equation}\label{eq:definebeta}
m^{-\beta_{\nu}(\xi)}N^{-\xi} \sum_{\ih=1}^{M_0} R_{\ih} N_{\ih}^{\gamma^{-1} + (1-\gamma^{-1})\xi} = 1.
\end{equation}
 Note that because of the minus sign before $\beta_{\nu}(\xi)$ (which in some papers is erroneously omitted), $\beta_{\nu}(\xi)$ is a convex function. 
 Define
 \begin{equation*}
\alpha_{\text{min}} \coloneqq \frac{\log N}{\log m} - \frac{1 - \gamma^{-1}}{\log m} \log N_{1}; \qquad \alpha_{\text{max}} \coloneqq \frac{\log N}{\log m} - \frac{1 - \gamma^{-1}}{\log m} \log N_{M_0}. 
\end{equation*}%
 Then by \cite[Theorem~1]{Jordan2011bm}, the multifractal spectrum is 
\begin{equation}\label{eq:uniformmultifractal} f_{\nu}(\alpha) = \inf_{\xi} (\alpha \xi + \beta_{\nu}(\xi)) = -\beta_{\nu}^*(-\alpha) \qquad \mbox{for all} \qquad \alpha \in (\alpha_{\text{min}},\alpha_{\text{max}}), 
\end{equation}
where $\beta_{\nu}^*(\alpha') \coloneqq \sup_{\xi'} (\alpha' \xi' - \beta_{\nu}(\xi'))$ is the Legendre transform of $\beta_{\nu}$ (defined in the same way as for the rate function in~\eqref{eq:22}). 

Another quantity that is closely related to the multifractal spectrum is the $L^q$ spectrum of a measure $\mu$ (see~\cite[Chapter~11]{Falconer1997techniques}). It is a function $T_\mu \colon \R \to \R$ which quantifies the global fluctuation of $\mu$, and its value at $q=0$ describes the box dimension of the support of the measure. %
The $L^q$ spectrum can be defined by 
\begin{equation}\label{e:lqdef}
T_\mu(q) \coloneqq \lim_{\delta \to 0^+} \frac{\log T_{\delta}(\mu,q)}{-\log \delta} 
\end{equation}
if the limit exists, where 
\[ T_{\delta}(\mu,q) \coloneqq \sup \left\{ \, \sum_i (\mu(B_i))^q : B_i \mbox{ disjoint balls of radius } \delta \mbox{ centred in } \supp(\mu) \, \right\}. \]

The $L^q$ spectrum of self-similar measures has been studied in~\cite{Cawley1992multifractal}, \cite[Chapter~11]{Falconer1997techniques}, and more recently (under the exponential separation condition) in Shmerkin's groundbreaking paper~\cite{Shmerkin2019furstenberg}. 
For self-similar measures satisfying the open set condition, the limit~\eqref{e:lqdef} exists and the multifractal formalism is satisfied, meaning that the Legendre transform of the $L^q$ spectrum equals the multifractal spectrum~\cite{Olsen1995multifractalformalism}. 
Now let $\nu$ be the uniform self-affine measure on a Bedford--McMullen carpet with non-uniform fibres, and let $\nu_x$ be the measure obtained by projecting $\nu$ orthogonally onto the $x$-axis. 
Note that $\nu_x$ is a homogeneous self-similar measure with all contraction ratios equal to $1/m$, satisfying the open set condition. For $1 \leq \ih \leq M_0$, the weight $N_{\ih}/N$ occurs with multiplicity $R_{\ih}$. 
Therefore for $q$ in an open neighbourhood of $1$, $T_{\nu_x}(q)$ satisfies  
\[ \sum_{\ih = 1}^{M_0} R_{\ih} \left(\frac{N_{\ih}}{N}\right)^q \left(\frac{1}{m}\right)^{T_{\nu_x}(q)} = 1. 
\]%
A direct calculation shows that 
\[ 
T_{\nu_x}(q) = -\frac{\log N}{\log m} q + \frac{\log \sum_{\ih = 0}^{M_0} R_{\ih} (N_{\ih})^q}{\log m}.
\] 
The $L^q$ spectra of self-affine measures have been studied in~\cite{Kolossvary2022lqtypes,Feng2005lq,
Falconer1999lq}. 
For self-affine measures supported on Bedford--McMullen carpets, the limit in~\eqref{e:lqdef} still exists but the multifractal formalism does not generally hold. 
Applying a result of Feng and Wang \cite[Theorem~2]{Feng2005lq} shows that $T_\nu(q)$ satisfies 
\[ 
N \cdot N^{-q} m^{-T_{\nu_x}(q)} n^{-(T_\nu(q) - T_{\nu_x}(q))} = 1.
\]
A direct manipulation shows that for $q$ in an open neighbourhood of $1$, 
\begin{align}\label{e:lqformula}
\begin{split}
T_\nu(q) &= \frac{\log N}{\log n} - q \frac{\log N}{\log m} + \left( 1 - \frac{\log m}{\log n}\right) \frac{\log \sum_{\ih=1}^{M_0} R_{\ih} (N_{\ih})^q}{\log m} \\*
&= \left( 1-\frac{\log m}{\log n}\right) \beta_{\nu}\left( \frac{\log (n^{q}/m)}{\log(n/m)} \right). 
\end{split}
\end{align}

Before describing the connection between the multifractal spectra and the intermediate dimensions, we observe that the grid on which a carpet can be defined is not unique, and in fact by iterating the IFS one can see that every carpet can be defined on infinitely many grids. For example, by iterating the IFS, the carpet from Figure~\ref{f:carpetIFS} can be defined on a $2 \times 3$ grid or on a $4 \times 9$ grid (though of course many carpets on a $4 \times 9$ grid cannot be realised on a $2 \times 3$ grid). 
Theorem~\ref{t:grid} gives information about the grids on which a carpet can be defined, and is proved in Section~\ref{s:multifractalproof}. %
It demonstrates for example that since $2$ and $3$ are multiplicatively independent, a carpet with non-uniform fibres defined on a $2 \times 4$ grid cannot be defined on a $3 \times 9$ grid (even though $\log 4 / \log 2 = \log 9 / \log 3$). 

\begin{thm}\label{t:grid}
\begin{enumerate}[label=(\roman*)]
\item\label{i:gridonecarpet} 
If a Bedford--McMullen carpet $\Lambda$ with non-uniform fibres can be defined on both a $m \times n$ grid and on a $m' \times n'$ grid, then $\log n / \log m = \log n' / \log m'$ and $\log n / \log n' \in \Q$. %
\item\label{i:gridtwocarpets} 
Consider two carpets $\Lambda_1$ and $\Lambda_2$ with non-uniform fibres which are defined by IFSs $\mathcal{S}_1$ and $\mathcal{S}_2$ on grids of size $m_1 \times n_1$ and $m_2 \times n_2$ respectively. 
Then they can be realised on the same grid if and only if 
\[ \frac{\log n_1}{\log n_2} = \frac{\log m_1}{\log m_2} \in \Q. \]  
\end{enumerate}
\end{thm}
In fact, we will see below that if two carpets with non-uniform fibres are bi-Lipschitz equivalent then they can be defined on the same grid. The same is true even if we merely assume the carpets have the same intermediate dimensions, or support uniform Bernoulli measures with equal multifractal spectra. 

We make some remarks about part~\ref{i:gridtwocarpets}. The reverse implication is immediate. Indeed, if $\log n_1 / \log n_2 = \log m_1 / \log m_2  = a/b \in \Q$, then the $b$-th iterate of $\mathcal{S}_1$ and the $a$-th iterate of $\mathcal{S}_2$ are both defined on the same grid of size $m_1^b \times n_1^b$. 
It is straightforward to see that the rate function $I(t)$ of $\mathcal{S}_1$ and the rate function $I^{(b)}(t)$ of the $b$-th iterate of $\mathcal{S}_1$ are related by $I^{(b)}(bt) =b I(t)$. 
For the forward implication, if both carpets can be realised on the same grid, then the fact that $\log n_1 / \log m_1 = \log n_2 / \log m_2$ was noted by Fraser and Yu using the Assouad spectrum in \cite[Theorem~3.3]{Fraser2018secondassouad}, and also follows from the intermediate dimension formula, noting that geometric quantities such as dimensions of course do not change by taking an iterate of the system. 
The fact that $n$ and $n'$ must be multiplicatively dependent (as must $m$ and $m'$) follows from Lemma~\ref{lem:multifractallemma} on page~\pageref{lem:multifractallemma}, and is related to work of Meiri and Peres~\cite[Theorem~1.2]{Meiri1999furstenberg}. 
This is in turn related to Furstenberg's $\times 2, \times 3$ principle (which suggests that expansions in multiplicatively independent bases should have no common structure). 
Other work along these lines includes~\cite{Furstenberg1967principle,Wu2019furstenberg,
Shmerkin2019furstenberg}, but there are many challenging open problems, such as whether there exists a non-atomic measure on the torus that is $\times 2$- and $\times 3$-invariant but not Lebesgue.

Our next result, which we prove in Section~\ref{s:multifractalproof} using Theorem~\ref{thm:main}, gives a direct connection between the intermediate dimensions and the multifractal and $L^q$ spectra of the uniform self-affine measure. 

\begin{theorem}\label{thm:multifractal}
Let $\Lambda$ and $\Lambda'$ be two Bedford--McMullen carpets with non-uniform fibres, and denote the corresponding uniform self-affine measures by $\nu$ and $\nu'$. Then the following are equivalent:
\begin{enumerate}[label=(\roman*)]
\item $\dim_{\theta}\Lambda = \dim_{\theta}\Lambda'$ for every $\theta\in[0,1]$; \label{item1}
\item $f_{\nu}(\alpha)=f_{\nu'}(\alpha)$ for all $\alpha \in (\alpha_{\text{min}},\alpha_{\text{max}})$. \label{itemnow2}
\end{enumerate}
Moreover, if~\ref{item1},~\ref{itemnow2} hold, then both carpets can be defined on the same grid. %

Now assume that $\Lambda$ and $\Lambda'$ are defined on the same $m\times n$ grid to begin with, with parameters $\{M_0,N_1,\dotsc,N_{M_0}, R_1,\dotsc, R_{M_0}\}$ and $\{M'_0,N'_1,\dotsc,N'_{M'_0}, R'_1,\dotsc, R'_{M'_0}\}$, respectively. Denote the corresponding rate functions defined in~\eqref{eq:22} by $I(t)$ and $I'(t)$. Let $\underline{t}$ and $\overline{t}$ be as defined previously, for the carpet $\Lambda$. Let $(a,b) \subset (\gamma^{-1},1)$ be a (non-empty) open interval. Then each of~\ref{item1},~\ref{itemnow2} is equivalent to each of the following: 
\begin{enumerate}[label=(\roman*)]
\setItemnumber{3}\item $\dim_{\theta}\Lambda = \dim_{\theta}\Lambda'$ for every $\theta \in (a,b)$; \label{itemnow3}
\setItemnumber{4}\item $T_{\nu}(q) = T_{\nu'}(q)$ for all $q \in \R$. \label{itemlq}
	\setItemnumber{5}\item $I(t) = I'(t-\gamma \log(M'/M))$ for all $t \in (\underline{t},\overline{t})$; \label{item4}
	\setItemnumber{6}\item $M_0=M'_0$, furthermore, $N_{\ih}/N'_{\ih}=(R'_{\ih}/R_{\ih})^{\gamma}= (M'/M)^{\gamma}$ for all $\ih=1,\dotsc,M_0$. \label{itemnow5}
\end{enumerate}   
\end{theorem}

We make several comments about Theorem~\ref{thm:multifractal}. 
\begin{rem}\label{r:multifraccomments}
\begin{enumerate}%

\item\label{i:citerao}
For carpets defined on the same grid, the equivalence of~\ref{itemnow2} and the explicit condition~\ref{itemnow5} was proved by Rao, Yang and Zhang in~\cite[Theorem~1.2]{RaoPreprintlipschitz}, using~\cite{Jordan2011bm}. 

\item\label{i:multifractalratelink}
In Step~4 of the proof of Proposition~\ref{prop:1} in Section~\ref{sec:proofProp}, we use scaling properties of Legendre transforms to establish a direct link between the multifractal spectrum of the uniform Bernoulli measure and the rate function $I(t)$. 
Since $I(t)$ appears in the intermediate dimension formula, this indicates why the link between the intermediate dimensions and multifractal spectrum in Theorem~\ref{thm:multifractal} is to be expected. 

\item\label{i:raosamem}
For carpets defined on the same grid with $M=M'$,~\ref{itemnow5} is simply saying that the column sequence of one carpet is a permutation of the column sequence of the other.

\item\label{i:coarsespec}
Equality of the $L^q$ dimensions and the coarse multifractal spectra can be added to the above equivalences in Theorem~\ref{thm:multifractal} if the carpets are defined on the same grid, since these quantities can be obtained from the $L^q$ spectrum by dividing by $1-q$ or taking the Legendre transform respectively. 

\item\label{i:otherdimmeas}
If~\ref{itemlq} holds then other notions of dimension of $\nu$ and $\nu'$ which can be deduced from their $L^q$ spectra must be equal, such as exact (Hausdorff/packing/entropy) dimension, correlation dimension (R\'enyi entropy), Frostman dimension, and box dimension (in the sense of~\cite{Falconer2022boxmeas}). 

\item\label{i:otherdimsets}
The formulae in~\cite{Fraser2021bedford} and~\ref{itemnow5} can be used to show that equality of intermediate dimensions implies equality of other notions of dimensions of sets such as packing, Assouad, quasi-Assouad, lower, quasi-lower or modified lower dimensions, or the Assouad spectrum or lower spectrum for any fixed $\theta \in (0,1)$. 

\item\label{i:analytic}
Since the multifractal spectrum is analytic (as the Legendre transform of an analytic function), if $I \subseteq (\alpha_{\text{min}},\alpha_{\text{max}})$ is an open interval, then~\ref{itemnow2} holds for all $\alpha \in I$ if and only if it holds for all $\alpha \in (\alpha_{\text{min}},\alpha_{\text{max}})$. 
Similarly, if $J \subseteq (\underline{t},\overline{t})$ is an open interval then~\ref{item4} holds for all $t \in J$ if and only if it holds for all $t \in (\underline{t},\overline{t})$. 
\end{enumerate}
\end{rem}

\begin{question}
In the statement of Theorem~\ref{thm:multifractal}, can $(a,b)$ be taken to be an arbitrary non-empty open subinterval of $(0,1)$? 
\end{question}

Since the proof strategy of Lemma~\ref{lem:multifractallemma} on page~\pageref{lem:multifractallemma} does not seem to work under the assumption that the $L^q$ spectra are equal, we ask the following question. 

\begin{question}
Do there exist two Bedford--McMullen carpets with non-uniform fibres which cannot be realised on the same grid but whose uniform Bernoulli measures have the same $L^q$ spectra? 
\end{question}

Turning now to bi-Lipschitz equivalence, recall that two metric spaces $(X,d_X)$ and $(Y,d_Y)$ are \emph{bi-Lipschitz equivalent} if there is a bi-Lipschitz map $f\colon X\to Y$. In our setting $X$ and $Y$ are two Bedford--McMullen carpets with the Euclidean distance. The following open problem seems challenging: 
\begin{question}\label{ques:bilip}
Find an explicit necessary and sufficient condition that determines, given two iterated function systems each generating a Bedford--McMullen carpet, whether or not the two carpets are bi-Lipschitz equivalent. 
\end{question}%
Partial progress towards Question~\ref{ques:bilip} has been made in~\cite{Li2013lipschitz,RaoPreprintlipschitz,
Yang2020lipschitz}, all of which assume some disconnectivity property. Fraser and Yu~\cite{Fraser2018secondassouad} used the Assouad spectrum to show that $\gamma$ is a bi-Lipschitz invariant within the class of Bedford--McMullen carpets, a fact which is also evident from observing the form of the intermediate dimensions. Moreover, the \emph{gap sequence} of a set is a topological quantity which has been shown to be bi-Lipschitz invariant \cite{Rao2008lipschitz}, and which is known for Bedford--McMullen carpets \cite{Miao2017gapseq,Liang2022gapseq}. Using the fact that the intermediate dimensions are stable under bi-Lipschitz maps, we obtain the following necessary condition for bi-Lipschitz equivalence as an immediate corollary of Theorem~\ref{thm:multifractal}. 

\begin{corollary}\label{cor:biLip}
Let $\Lambda$ and $\Lambda'$ be two Bedford--McMullen carpets with non-uniform fibres which are bi-Lipschitz equivalent, and let $\nu$ and $\nu'$ be the corresponding uniform Bernoulli measures. Then $f_{\nu}(\alpha)=f_{\nu'}(\alpha)$ for all $\alpha \in (\alpha_{\text{min}},\alpha_{\text{max}})$ and $T_{\nu}(q) = T_{\nu'}(q)$ for all $q \in \R$, and both carpets can be defined on the same $m \times n$ grid, on which condition~\ref{itemnow5} above holds. 
\end{corollary} 
This strengthens~\cite[Corollary~1.1]{RaoPreprintlipschitz}, where it is assumed that $\Lambda$ and $\Lambda'$ are totally disconnected and defined on the same grid. 
In Proposition~\ref{p:biLip}, we construct two carpets which we know are not bi-Lipschitz equivalent by Corollary~\ref{cor:biLip}, but where \cite[Corollary~1.1]{RaoPreprintlipschitz} does not apply. 
Corollary~\ref{cor:biLip} also shows in particular that if two carpets defined on the same grid with the same number of non-empty columns are bi-Lipschitz equivalent then the column sequence of one must be a permutation of the column sequence of the other (though we are not able to draw this conclusion if the number of non-empty columns is different, see Example~\ref{ex:rao}). 

One natural question would be to investigate the intermediate dimensions of self-affine carpets of Lalley and Gatzouras~\cite{Lalley1992gatzouras} or Bara\'nski~\cite{Baranski2007carpet}, or higher-dimensional self-affine sponges. 
Indeed, in light of the recent paper~\cite{BanajiPreprintgl} proving that the Assouad spectrum of Gatzouras--Lalley carpets (unlike Bedford--McMullen carpets) can be a differentiable function of $\theta$, it is natural to ask whether the intermediate dimensions of Gatzouras--Lalley carpets can also be differentiable on $(0,1)$. 
We expect calculating a formula for the intermediate dimensions of such self-affine sets to be challenging, not least because there is no clear single analogue of the important quantity~$\gamma$, and this is not something which we will explore in this thesis. 
We remark, however, that Huang, Rao, Wen and Xu~\cite{HuangPreprintboxcounting} have introduced so-called box-counting measures of metric spaces and shown that Bedford--McMullen and generalised Gatzouras--Lalley type sponges and Bara\'nski carpets admit such measures. Indeed, for Bedford--McMullen carpets, these are simply the uniform Bernoulli measures. 
After the paper~\cite{Banaji2021bedford} on which this chapter is based appeared on arXiv, Huang \emph{et al.} proved without any connectivity assumption that the multifractal spectrum of box-counting measures is a bi-Lipschitz invariant; this was proved directly, without using the intermediate dimensions. %
Their result therefore generalises both the result from~\cite{RaoPreprintlipschitz} and our result that the multifractal spectrum of uniform Bernoulli measures on Bedford--McMullen carpets with non-uniform fibres is bi-Lipschitz invariant. 
Huang \emph{et al.} also ask in \cite[Open problem~1]{HuangPreprintboxcounting} whether two generalised Gatzouras--Lalley or Bara\'nski sponges have the same intermediate dimensions if and only if their corresponding box-counting measures have the same multifractal spectra.

\subsection{Equivalent forms of the rate function}\label{subsec:eqformsofrate}%

In this section we provide equivalent formulations of the rate function $I(t)$ in terms of a pressure-like function, a certain probability vector with an entropy maximising property, and the multifractal spectra $f_{\nu}(\alpha)$ defined in~\eqref{eq:uniformmultifractal}. As a result, our main formula~\eqref{eq:mainformula} for $\dim_{\theta}\Lambda$ can be expressed with any of these quantities. 

We begin by defining the pressure-like function. For $\iiv=(i_1,\dotsc,i_J)\in\{1,\dotsc,M\}^J$ and $k\in\{0,1,\dotsc,J\}$, we introduce
\begin{equation}\label{eq:20}
	\psi_{\iiv|k}(s) \coloneqq M^{k \gamma}\cdot n^{-sk}\cdot \prod_{\ell=1}^k N_{i_\ell}.
\end{equation}
In particular, for $k=0$, $\psi_{\iiv|0}(s)\equiv 1$. The interpretation of $\psi_{\iiv|k}(s)$ later is that it gives the $s$-cost of a set in the cover with diameter related to $k$, see Remark~\ref{rem:pressureexp}. Moreover, we define the sum
\begin{equation*}
	\Psi_J(s)\coloneqq \sum_{\iiv\in\{1,\dotsc,M\}^J}\min_{k\in\{0,1,\dotsc,J\}} \psi_{\iiv|k}(s).
\end{equation*}
This is connected to the total $s$-cost of the optimal cover, see Remark~\ref{rem:pressureexp} for additional explanation. To determine the critical exponent it is natural to define a pressure-like quantity as the exponential growth rate of $\Psi_J(s)$, more formally,
\begin{equation}\label{eq:21}
	\underline{P}(s)\coloneqq \liminf_{J\to\infty}\frac{1}{J}\log \Psi_J(s) \;\text{ and }\; \overline{P}(s)\coloneqq \limsup_{J\to\infty}\frac{1}{J}\log \Psi_J(s).
\end{equation} 

The probability vector $\mathbf{Q}^*_{t}\in\mathcal{P}_M$ is defined by
\begin{equation}\label{eq:104}
	H(\mathbf{Q}^*_{t}) = \sup\Big\{ \, H(\mathbf{p}): \mathbf{p}\in\mathcal{P}_M \;\text{ such that }\; \sum_{\jh=1}^M p_{\jh}\log N_{\jh} = t \, \Big\}.
\end{equation}
It is well defined, see Lemma~\ref{lem:32}. Moreover, $H(\mathbf{Q}^*_{t})<\log M$ since $t>\underline{t}$.

We regularly relate the arguments $s$ and $t$ to each other via the transformation
\begin{equation}\label{eq:23}
	t= t_1(s) = \left(s-\frac{\log M}{\log m}\right) \log n, \;\text{ or equivalently } s=\frac{t}{\log n} + \frac{\log M}{\log m} .
\end{equation}
We do so to ensure that 
\begin{equation}\label{eq:309}
\psi_{\iiv|J}(s)\leq 1 \;\Longleftrightarrow\; \frac{1}{J} \sum_{\ell=1}^{J} \log N_{i_\ell} \leq \Big(s-\frac{\log M}{\log m}\Big)\cdot \log n =t,
\end{equation}
which now follows from~\eqref{eq:20} and straightforward algebraic manipulations. 
Now, using~\eqref{eq:102} and~\eqref{eq:103}, $\underline{t}$ maps to
\begin{equation}\label{e:sunderlinedef}
	\underline{s}\coloneqq \dim_{\mathrm H}\Lambda - \frac{1}{\log n} \left( \frac{\log n}{\log m} \log \Bigg(  \frac{1}{M}\sum_{\jh=1}^{M} N_{\jh}^{\frac{\log m}{\log n}} \Bigg) - \underline{t} \right),
\end{equation}
while $\overline{t}$ maps to
\begin{equation}\label{e:soverlinedef}
	\overline{s} \coloneqq \dim_{\mathrm B}\Lambda + \frac{\log M - H(\mathbf{P})}{\log n}.
\end{equation}
Observe that by Jensen's inequality and non-uniform fibres, $\underline{s}<\dim_{\mathrm H}\Lambda<\dim_{\mathrm B}\Lambda<\overline{s}$. %
In Proposition~\ref{prop:1}, the key technical result of Section~\ref{subsec:eqformsofrate}, we make a clear connection between~\eqref{eq:21}, \eqref{eq:104}, \eqref{eq:22} and~\eqref{eq:uniformmultifractal} for pairs of $(s,t)$ related by~\eqref{eq:23}. 
The proof is non-trivial and is given in Section~\ref{sec:proofProp}. 

\begin{prop}\label{prop:1}
	Fix $s\in(\underline{s},\overline{s})$. Then $\underline{P}(s)=\overline{P}(s)$; let $P(s)$ denote this common value. Furthermore, for every pair $(s,t)$ related by~\eqref{eq:23},
	\begin{equation*}
		\log M - I(t) = P(s) = H(\mathbf{Q}^*_{t}) =  (\log m ) f_{\nu} \left(\frac{\log N}{\log m} - \left(\frac{1}{\log m} - \frac{1}{\log n}\right) t   \right)  - \frac{t}{\gamma}. 
	\end{equation*}
\end{prop}
Note also that by~\eqref{e:lqformula} and standard properties of Legendre transforms, $f_{\nu}$ and $I$ can be written in terms of the Legendre transform of $T_{\nu}$. 

\subsection{Illustrative examples}\label{subsec:ex}%

A simple example of a carpet with non-uniform fibres is shown in Figure~\ref{f:carpetIFS}. The examples in this section show additional interesting behaviour. 

\begin{rem}
All figures of the graphs in this chapter were created using \emph{Wolfram Mathematica 12.3}, keeping simple implementation in mind rather than efficiency. 
For a fixed $\theta\in(0,1)$, the value of $\dim_{\theta}\Lambda$ was approximated by taking $2^{25}$ equally spaced points in the interval $(\dim_{\mathrm{H}}\Lambda,\dim_{\mathrm B}\Lambda)$ and choosing the point $s(\theta)$ for which the expression in~\eqref{eq:mainformula} was closest to~$0$.
\end{rem}

\begin{rem}
It was first observed in~\cite{Kolossvary2022bm} that the graph $\theta\mapsto \dim_{\theta}\Lambda$ can approach $\dim_{\mathrm B}\Lambda$ from below the straight line $\ell(\theta)=\dim_{\mathrm H}\Lambda +\theta(\dim_{\mathrm B}\Lambda-\dim_{\mathrm H}\Lambda)$, indicating that it is possible for $\dim_{\theta}\Lambda$ not to be concave on the whole range of $\theta$. 
From Corollary~\ref{cor:allprop} it follows that in this case the graph $\theta\mapsto \dim_{\theta}\Lambda$ must intersect $\ell(\theta)$. 
In fact, there are even carpets where the graph intersects $\ell(\theta)$ twice, as shown on the left of Figure~\ref{fig:exSeries}. 
For the carpets in this figure, all parameters remain the same except for $m$, which causes different behaviour for larger values of $\theta$ as it changes. 
For $m\leq 25$, the graph stays above $\ell(\theta)$ for all $\theta$. 
\end{rem}

\begin{figure}[ht]
	\centering
	\includegraphics[width=0.99\textwidth]{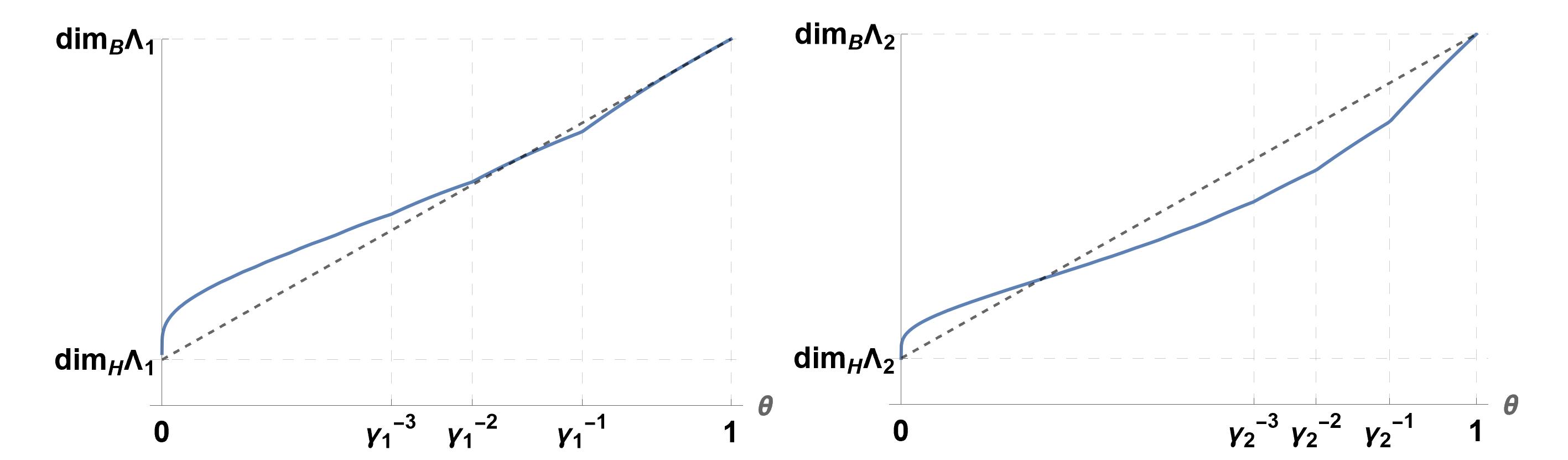}
	\caption{Parameters $n=100$ and $\mathbf{N}=(51,50,50,50,50,50)$ are the same in each example; only $m$ varies from $30$ on the left to $50$ on the right.}
	\label{fig:exSeries}
\end{figure}

Fraser and Yu \cite[Proposition~3.4]{Fraser2018secondassouad} proved a similar result to Proposition~\ref{p:biLip} for the Assouad spectrum. 
\begin{prop}\label{p:biLip}
Consider the two Bedford--McMullen carpets %
$\Lambda$ and $\Lambda'$ with $m=M=32$ and $n=243$ and the following parameters: 
\begin{align*}
&\Lambda \colon \qquad M_0=3, \quad \{N_1,N_2,N_3\} = \{27,3,1\} \text{ and } \{R_1,R_2,R_3\} = \{2,11,19\}, \\*
&\Lambda' \colon  \qquad M'_0=3, \quad \{N'_1,N'_2,N'_3\} = \{27,9,1\} \text{ and } \{R'_1,R'_2,R'_3\} = \{1,6,25\}.
\end{align*}
There exists $\theta \in (0,1)$ with $\dim_{\theta} \Lambda \neq \dim_{\theta} \Lambda'$, so $\Lambda$ and $\Lambda'$ are not bi-Lipschitz equivalent. 
However, $\dim \Lambda = \dim \Lambda'$ where $\dim$ can be Hausdorff or box dimension or any of the notions of dimensions mentioned in part~\ref{i:otherdimsets} of Remark~\ref{r:multifraccomments}. 
\end{prop}
\begin{proof}
Since all $R_{\ih}/R'_{\ih}$ are different, it follows from part~\ref{itemnow5} of Theorem~\ref{thm:multifractal} that there exists $\theta \in (0,1)$ with $\dim_{\theta} \Lambda \neq \dim_{\theta} \Lambda'$. 
Since $N = 106$, $\max_{1 \leq i \leq 3} N_i =\max_{1 \leq i \leq 3} N'_i= 27$, and $\min_{1\leq i \leq M} N_i = \min_{1\leq i \leq M} N'_i = 1$, we can use~\eqref{eq:102} and \cite[Corollary~15.5.3]{Fraser2021bedford} to show that the Hausdorff and modified lower dimensions are equal, and the formulae in~\cite{Fraser2021bedford} to show that the other dimensions are equal. 
\end{proof}
Figure~\ref{fig:exRatio} shows the plots of $\dim_{\theta}\Lambda$ and $\dim_{\theta}\Lambda'$ from Proposition~\ref{p:biLip} side-by-side on the left, and the ratio $\dim_{\theta}\Lambda'/\dim_{\theta}\Lambda$ on the right. 
Note that the fact that $\Lambda$ and $\Lambda'$ from are not bi-Lipschitz equivalent is revealed only by the intermediate dimensions, not by any of the other dimensions mentioned above. 
If all the rectangles are chosen in a specific row, then neither $\Lambda$ nor $\Lambda'$ is totally disconnected, so~\cite[Corollary~1.1]{RaoPreprintlipschitz} does not apply. 
We can use H\"older distortion to obtain a quantitative improvement of the assertion that $\Lambda$ and $\Lambda'$ are not bi-Lipschitz equivalent. Indeed, assume $f \colon \Lambda' \to \mathbb{R}^2$ is $\alpha$-H\"older with $f(\Lambda') \supseteq \Lambda$. Then the optimal value of $\theta$ to consider is $\theta = \gamma^{-2} = \left(\frac{\log 2}{\log 3}\right)^2 \approx 0.40$. 
By~\eqref{generalholderint}, 
\[ \alpha \leq \frac{\dim_{\gamma^{-2}} \Lambda'}{\dim_{\gamma^{-2}} f(\Lambda')} \leq \frac{\dim_{\gamma^{-2}} \Lambda'}{\dim_{\gamma^{-2}} \Lambda} < 0.9995, \]
with the last inequality computed numerically using Theorem~\ref{thm:main}. 
\begin{figure}[ht]
	\centering
	\includegraphics[width=0.99\textwidth]{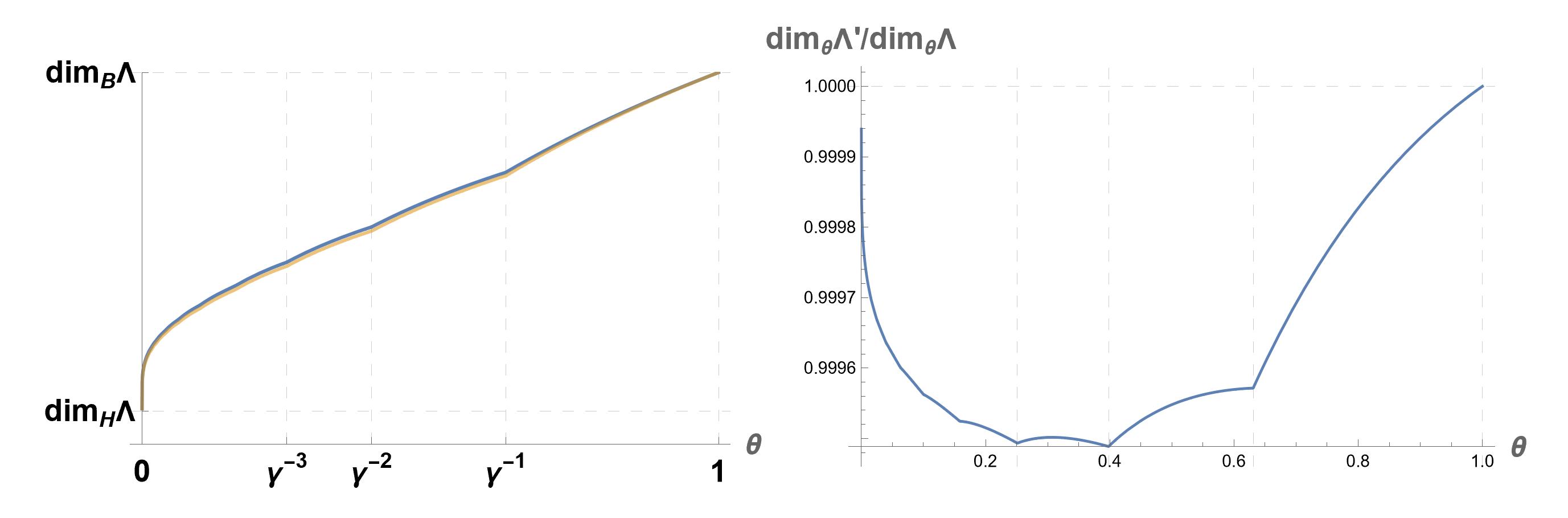}
	\caption{Left: plot of $\dim_{\theta}\Lambda$ (\textcolor{blue}{blue}) and $\dim_{\theta}\Lambda'$ (\textcolor{orange}{orange}) from Proposition~\ref{p:biLip}. Right: ratio of $\dim_{\theta}\Lambda'/\dim_{\theta}\Lambda$ for $\theta\geq \gamma^{-35}$.}
	\label{fig:exRatio}
\end{figure}

\begin{prop}
For any carpet with just two column types, meaning that $M_0=2$ using notation from Section~\ref{subsec:multifractal}, the rate function can be given explicitly by $I(t)=\log M -H(\mathbf{Q}^*_t)$, where 
\[ H(\mathbf{Q}^*_t) = \frac{-1}{\log (N_1/N_2)}\left(\!\! (t-\log N_2)\log \frac{t-\log N_2}{R_1\log (N_1/N_2)}+ (\log N_1-t)\log \frac{\log N_1-t}{R_2\log (N_1/N_2)}  \!\right). \]
\end{prop}
\begin{proof}
Let $\mathbf{q}=(q_1,\dotsc,q_{M_0})\in\mathcal{P}_{M_0}$.
It is straightforward to see that the supremum in~\eqref{eq:104} will not change if we restrict to vectors $\mathbf{p}\in\mathcal{P}_M$ of the form 
\begin{equation}\label{eq:105}
p_{\ih} = q_j/R_j, \;\text{ if the } \ih\text{-th non-empty column has } N_j \text{ maps},
\end{equation}
in other words measure is distributed uniformly amongst columns with the same number of maps. 
As a result, the linear constraints in~\eqref{eq:104} can be rewritten as 
\begin{equation}\label{eq:106}
1 = \sum_{j=1}^{M_0} q_j  \,\;\text{ and }\;\, t =  \sum_{j=1}^{M_0} q_j \log N_j.
\end{equation}
In particular, since $M_0=2$, there is a single vector $(q^*_1,q^*_2)$ which satisfies~\eqref{eq:106}, namely 
\begin{equation*}
q^*_1= \frac{t-\log N_2}{\log( N_1/N_2)} \;\;\text{ and }\;\; q^*_2 =\frac{\log N_1 -t}{\log (N_1/N_2)},
\end{equation*}
recalling $N_1>N_2$. 
Using~\eqref{eq:105}, we can calculate the entropy of the entropy-maximising vector, 
\[ H(\mathbf{Q}^*_t) = - \sum_{j=1}^{M_0} q^*_j \log (q^*_j/R_j), \]
and conclude from Proposition~\ref{prop:1} that $I(t)=\log M -H(\mathbf{Q}^*_t)$, as required. 
\end{proof}

\begin{example}[using the parameters from Example~1.2 of Rao, Yang and Zhang~\cite{RaoPreprintlipschitz}]\label{ex:rao}%
Consider two Bedford--McMullen carpets defined on the same grid with $m=8$, $n=27$ and the following parameters: 
\begin{align*}
 \Lambda: \qquad &M_0 = 2, \quad \{N_1,N_2\} = \{6,3\} \mbox{ and } \{R_1,R_2\} = \{1,1\}, \\*
\Lambda': \qquad &M_0' = 2, \quad \{N_1',N_2'\} = \{2,1\} \mbox{ and } \{R_1',R_2'\} = \{2,2\}. 
\end{align*}
Then condition~\ref{itemnow5} from Theorem~\ref{thm:multifractal} holds, so $\dim_{\theta} \Lambda = \dim_{\theta} \Lambda'$ for all $\theta \in [0,1]$, despite the fact that the carpets are defined on the same grid with different parameters. This is only possible because the number of non-empty columns is different. 
\end{example}
It is shown in~\cite{RaoPreprintlipschitz} that the carpets in Example~\ref{ex:rao} are not bi-Lipschitz equivalent. Therefore equality of the intermediate dimensions is not a sufficient condition for two carpets with non-uniform fibres to be bi-Lipschitz equivalent, even if they are assumed to be defined on the same grid and totally disconnected. This raises the following question.  

\begin{question}
Suppose two Bedford--McMullen carpets both have non-uniform fibres, are defined on the same grid, and are bi-Lipschitz equivalent. Does it follow that both carpets must have identical parameters $(M_0,N_1,\dotsc,N_{M_0}, R_1,\dotsc, R_{M_0})$? 
\end{question}

\begin{example}\label{ex:Nice}%
Consider the two carpets $\Lambda$ and $\Lambda'$ with parameters 
\begin{align*}
\Lambda : \qquad &n=36, \quad m=6, \quad M = M_0 = 2,\quad \{ N_1,N_2\} = \{9,6\} \mbox{ and }\{R_1,R_2\} = \{1,1\} \\
\Lambda' : \qquad &n=36, \quad m=4,\quad M = M_0 = 2, \quad \{ N_1,N_2\} = \{6,4\} \mbox{ and }\{R_1,R_2\} = \{1,1\}.
\end{align*}
Then it can be checked from Theorem~\ref{thm:main} that $\dim_{\theta} \Lambda = \dim_{\theta} \Lambda'$ for all $\theta \in [1/2,1]$, but not for the whole range of $\theta$; by Theorem~\ref{thm:multifractal} this is only possible because the carpets are defined on different grids. By Corollary~\ref{cor:allprop}, the graph of $\dim_{\theta} \Lambda$ has a phase transition at $\theta = 1/2$ but the graph of $\dim_{\theta} \Lambda'$ does not, see Figure~\ref{fig:exNice}.
\end{example}
\begin{figure}[ht]
	\centering
	\includegraphics[width=0.5\textwidth]{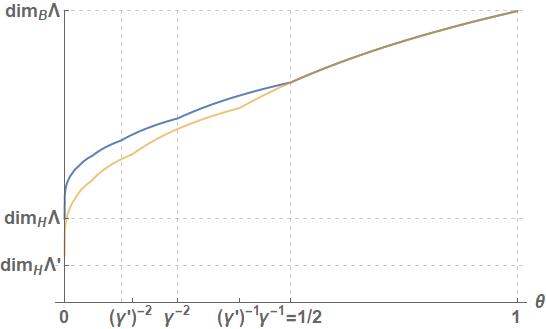}
	\caption{Plots of the intermediate dimensions of carpets in Example~\ref{ex:Nice}.}
	\label{fig:exNice}
\end{figure}

\section{Proof of equivalent forms of the rate function}\label{sec:proofProp}

In this section we will prove Proposition~\ref{prop:1}. 
Recall notation from Section~\ref{sec:CompResults}, and assume that $s\in(\underline{s},\overline{s})$ and that $(s,t)$ are related by~\eqref{eq:23}. We will prove Proposition~\ref{prop:1} in four steps. 
\begin{description}
\item[Step 1] $H(\mathbf{Q}^*_{t}) = \log M - I(t)$ 
\item[Step 2] $\overline{P}(s) \leq H(\mathbf{Q}^*_{t})$ 
\item[Step 3] $H(\mathbf{Q}^*_{t}) \leq \underline{P}(s)$ 
\item[Step 4] $I(t) = -(\log m ) f_{\nu} \left(\frac{\log N}{\log m} - \left(\frac{1}{\log m} - \frac{1}{\log n}\right) t   \right)  + \frac{t}{\gamma} + \log M$ 
\end{description}
We will prove the steps in separate subsections, which will also contain some auxiliary results which are important in their own right in the proof of Theorem~\ref{thm:main}. 

\subsection{Preliminaries} 
First, we need some preliminaries, in particular to describe the probability vectors $\mathbf{P}^*_t$ and $\mathbf{Q}^*_t$ which have certain optimising properties, and to recall some facts from the method of types. 
If $k,k_1,k_2 \in \mathbb{N}$ satisfy $0 \leq k_1 < k_2 \leq k$, and $\iih \in \{1,\dotsc,M\}^k$, then we define the average 
\begin{equation}\label{eq:tauaverage} \tau(\iih,k_1,k_2) \coloneqq \max\Bigg\{ \underline{t} , \frac{1}{k_2 - k_1} \sum_{j=k_1 + 1}^{k_2} \log N_{\ih_j} \Bigg\}. 
\end{equation}%
Recall, $\mathcal{P}_M$ denotes the set of probability vectors on $[M]=\{1,\dotsc,M\}$ and we introduced two distinguished probability vectors $\mathbf{P}\coloneqq (N_1/N,\dotsc,N_M/N)$ and $\mathbf{Q} \coloneqq (1/M,\dotsc,1/M)$.
Recall that the \emph{Kullback--Leibler divergence}, also known as the \emph{relative entropy} of $\mathbf{p}\in\mathcal{P}_M$ with respect to $\mathbf{q}\in\mathcal{P}_M$ is
\begin{equation*}
H(\mathbf{p}\| \mathbf{q}) \coloneqq \sum_{\ih=1}^M p_{\ih} \log \left( \frac{p_{\ih}}{ q_{\ih}}\right) = - H(\mathbf{p}) - \sum_{\ih=1}^M p_{\ih} \log q_{\ih} ,
\end{equation*}
where we set $0 \log 0 = 0$ and $0 \log (0/q_{\ih}) = 0$ regardless of the value of $q_{\ih}$. 
It is asymmetric and $H(\mathbf{p}\| \mathbf{q})\geq 0$ with equality if and only if $\mathbf{p}=\mathbf{q}$. In particular, 
\begin{equation}\label{eq:32}
H(\mathbf{p}\| \mathbf{P}) = \log N -H(\mathbf{p}) -\sum_{\ih=1}^M p_{\ih}\log N_{\ih} \quad \text{ and } \quad H(\mathbf{p}\| \mathbf{Q}) = \log M - H(\mathbf{p}).
\end{equation}

Recall that $\underline{t}\coloneqq \frac{1}{M} \sum_{\jh=1}^{M} \log N_{\jh}$ and $\overline{t}\coloneqq \log N-H(\mathbf{P})$, and let $t\in(\underline{t},\overline{t})$. We divide the set $\mathcal{P}_M$ into two parts: 
\begin{equation}\label{eq:300}
\mathcal{G}_t \coloneqq \bigg\{ \, \mathbf{p}\in\mathcal{P}_M: \sum_{\ih=1}^M p_{\ih}\log N_{\ih} \leq t \, \bigg\} \;\text{ and }\; \mathcal{F}_t \coloneqq \bigg\{ \, \mathbf{p}\in\mathcal{P}_M: \sum_{\ih=1}^M p_{\ih}\log N_{\ih} \geq t \, \bigg\},
\end{equation}
so that $\mathcal{E}_t\coloneqq \mathcal{G}_t \cap \mathcal{F}_t = \big\{ \, \mathbf{p}\in\mathcal{P}_M: \sum_{\ih} p_{\ih}\log N_{\ih} = t \, \big\}$. 
The reason for doing this will become clear in~\eqref{eq:301} in the proof of Step~2. 
Since $t>\underline{t}$, it follows that $\mathbf{Q}\in\mathcal{G}_t\setminus\mathcal{E}_t$, whereas $\mathbf{P}\in\mathcal{F}_t\setminus\mathcal{E}_t$ because $t<\overline{t} = \sum_{\ih=1}^M P_{\ih}\log N_{\ih}$.

\begin{lemma}\label{lem:31}
Let $t\in(\underline{t},\overline{t})$ and $\mathbf{p}\in\mathcal{E}_t$. Then 
\begin{equation*}
H(\mathbf{p}\| \mathbf{P}) = H(\mathbf{p}\| \mathbf{Q})+\log(N/M) -t.
\end{equation*}
\end{lemma}

\begin{proof}
Let $\mathbf{p}\in\mathcal{E}_t$. Let $N'_1<N'_2< \dotsb <N'_{M_0}$ denote the different values that the set $\{N_1,\dotsc,N_M\}$ takes. %
For $1\leq j\leq M_0$ let $I_j\coloneqq \{\, \ih\in[M] :  N_{\ih} = N'_j \, \}$ and  $q_j\coloneqq \sum_{\ih\in I_j}p_{\ih}$. Then,
\begin{equation*}
1=\sum_{\ih=1}^M p_{\ih} = \sum_{j=1}^{M_0} q_j  \,\;\text{ and }\;\, t = \sum_{\ih=1}^M p_{\ih}\log N_{\ih} =  \sum_{j=1}^{M_0} q_j \log N'_j.
\end{equation*}
This is a linear system of equations for $\{q_j\}_{j=1}^{M_0}$. Straightforward Gaussian elimination yields
\begin{equation*}
\left[\begin{array}{@{}ccc|c@{}}
1  & \dots & 1 & 1 \\
\log N'_1 & \dots & \log N'_{M_0} & t
\end{array}\right] %
\;\sim\;
\left[\begin{array}{@{}ccccc|c@{}}
1 & 1 & 1 & \dots & 1 & 1 \\
0 & 1 & \frac{\log(N'_3/N'_1)}{\log(N'_2/N'_1)} & \dots & \frac{\log(N'_{M_0}/N'_1)}{\log(N'_2/N'_1)} & \frac{t- \log N'_1}{\log(N'_2/N'_1)} 
\end{array}\right].
\end{equation*}
Thus, for every solution $(q_1,\dotsc,q_{M_0})$, we see that $q_3,\dotsc,q_{M_0}$ are free variables, and moreover
\begin{equation*}
q_2 = \frac{1}{\log (N'_2 / N'_1)} \bigg( t - \sum_{j=3}^{M_0} q_j \log N'_j - \bigg(1-\sum_{j=3}^{M_0} q_j\bigg) \log N'_1\bigg)
\end{equation*}
and $q_1=1-q_2-\sum_{j=3}^{M_0} q_j$. It now follows by a straightforward calculation that $\sum_{j=1}^{M_0} q_j \log N'_j=t$. By~\eqref{eq:32}, the result follows. 
\end{proof}

We introduce $\mathbf{P}^*_t\in\mathcal{G}_t$, $\mathbf{Q}^*_t\in\mathcal{F}_t$ defined by
\begin{equation}\label{eq:defineqstar}
H(\mathbf{P}^*_t\| \mathbf{P}) = \inf_{\mathbf{p}\in\mathcal{G}_t} H(\mathbf{p}\| \mathbf{P}) \qquad \text{ and }\qquad H(\mathbf{Q}^*_t\| \mathbf{Q}) = \inf_{\mathbf{q}\in\mathcal{F}_t} H(\mathbf{q}\| \mathbf{Q}).
\end{equation}
Due to~\eqref{eq:32} and Lemma~\ref{lem:32}, %
 this definition of $\mathbf{Q}^*_t$ is equivalent to~\eqref{eq:104}.
\begin{lemma}\label{lem:32}
Let $t\in(\underline{t},\overline{t})$. Then both $\mathbf{P}^*_t$ and $\mathbf{Q}^*_t$ are well defined and unique, with $\mathbf{P}^*_t,\mathbf{Q}^*_t\in\mathcal{E}_t$. Moreover, 
\begin{equation}\label{eq:33}
H(\mathbf{P}^*_t\| \mathbf{P}) = H(\mathbf{Q}^*_t\| \mathbf{Q})+\log(N/M)-t.
\end{equation}
\end{lemma}
\begin{proof}
Both $H(\cdot\,\| \mathbf{P})$ and $H(\cdot\,\| \mathbf{Q})$ are continuous on the domain $\mathcal{P}_M$, and $\mathcal{G}_t$ and $\mathcal{F}_t$ are compact, so both infima are attained. 

We proceed to differentiate the relative entropy (with respect to a fixed vector $\mathbf{q}$ in the interior of $\mathcal{P}_M$) along straight lines in $\mathcal{P}_M$. 
Fix $\mathbf{p},\mathbf{q} \in \mathcal{P}_M$, and assume that all entries of $\mathbf{q}$ are positive. Fix $\mathbf{v} \in \R^M \setminus \{0\}$ satisfying $\sum_{\ih=1}^M v_{\ih} = 0$. 
Let $[-t_-,t_+]$ denote the maximal interval (containing the origin) such that $H(\mathbf{p} + t \mathbf{v} \| \mathbf{q} ) \in \mathcal{P}_M$ for all $t \in [-t_-,t_+]$. 
Then for all $t \in (-t_-,t_+)$, %
a direct computation gives 
\[ \frac{d}{dt} H(\mathbf{p} + t \mathbf{v} \| \mathbf{q} )  = \sum v_{\ih} \log \left( \frac{p_{\ih} + t v_{\ih}}{q_{\ih}} \right) ; \qquad \frac{d^2}{dt^2} H(\mathbf{p} + t \mathbf{v} \| \mathbf{q} ) = \sum \frac{v_{\ih}^2}{p_{\ih} + t v_{\ih}} > 0, \]
where the sums are taken over all indices $1 \leq \ih \leq M$ for which $p_{\ih} + t v_{\ih} > 0$ for all $t \in (-t_-,t_+)$. %
Therefore $t \mapsto H(\mathbf{p} + t \mathbf{v} \| \mathbf{q} )$ is strictly convex, and has at most one minimum in $[-t_-,t_+]$. 
In particular, the uniqueness of $\mathbf{P}^*_t$ and $\mathbf{Q}^*_t$ follows from the convexity of $\mathcal{G}_t$ and $\mathcal{F}_t$ respectively. 
Note also that for all $t \in (0,t_+)$,
\[ \frac{d}{dt} H(\mathbf{q} + t \mathbf{v} \| \mathbf{q} )  = \sum_{\ih=1}^M v_{\ih} \log \left(1 + \frac{v_{\ih}}{q_{\ih}}t \right) > 0,\] %
so $t \mapsto H(\mathbf{q} + t \mathbf{v} \| \mathbf{q} )$ is strictly increasing on $(0,t_+)$. 
Applying this with $\mathbf{q}$ taken to be $\mathbf{P}$ or $\mathbf{Q}$ respectively gives that $\mathbf{P}^*_t,\mathbf{Q}^*_t\in\mathcal{E}_t$.  

To conclude, Lemma~\ref{lem:31} gives that
\begin{align*}
H(\mathbf{Q}^*_t\| \mathbf{P}) = H(\mathbf{Q}^*_t\| \mathbf{Q}) +\log(N/M)-t \leq H(\mathbf{P}^*_t\| \mathbf{Q}) +\log(N/M)-t &= H(\mathbf{P}^*_t\| \mathbf{P}) \\
&\leq H(\mathbf{Q}^*_t\| \mathbf{P}),
\end{align*}
so there is equality throughout. 
\end{proof}
The importance of choosing $t$ to lie in the interval $(\underline{t},\overline{t})$ (or equivalently $s\in(\underline{s},\overline{s})$) is that in this case the hyperplane $\mathcal{E}_t$ separates $\mathbf{P}$ and $\mathbf{Q}$. Otherwise, either $H(\mathbf{P}^*_t\| \mathbf{P})=0$ or $H(\mathbf{Q}^*_t\| \mathbf{Q})=0$, and~\eqref{eq:33} does not necessarily hold.

\subsection{Method of types} 
The method of types is an elementary tool developed to study discrete memoryless systems in information theory. It has since found applications in hypothesis testing, combinatorics and large deviations (see~\cite{Csiszar1998methodoftypes} for some background). Kolossv\'ary used it to calculate the box dimension of Gatzouras--Lalley and Bara\'nski sponges in $\Rd$~\cite{Kolossvary2022lqtypes}. %

Let $[M]=\{1,\dotsc,M\}$ denote a finite alphabet and assume $\iiv=(i_1,\dotsc,i_J)\in[M]^J$. For $J'\leq J$, the \emph{type of $\iiv$ at level $J'$} is the empirical probability vector
\begin{equation*}%
\boldsymbol{\tau}_{J'}^{J}(\iiv) = \frac{1}{J'}\big( \#\{ \, 1\leq\ell\leq J': i_{\ell}=j \, \} \big)_{j\in[M]} \in [0,1]^M.
\end{equation*}
When $J'=J$, we simply write $\boldsymbol{\tau}_J(\iiv) \coloneqq \boldsymbol{\tau}_J^{J}(\iiv)$. The set of all possible types of $[M]^J$ is
\begin{equation*}
\mathcal{T}_J=\big\{ \, \mathbf{p}\in\mathcal{P}_M: \text{ there exists } \iiv\in[M]^J \text{ such that } \mathbf{p}=\boldsymbol{\tau}_J(\iiv) \, \big\},
\end{equation*}
and for $J'\leq J$, the \emph{type class of} $\mathbf{p}\in \mathcal{T}_{J'}$ amongst $[M]^J$ is the set
\begin{equation*}
T_{J'}^{J}(\mathbf{p}) = \big\{ \, \iiv\in[M]^J: \boldsymbol{\tau}_{J'}^J(\iiv)=\mathbf{p} \, \big\}.
\end{equation*}
Similarly, $T_{J}(\mathbf{p}) \coloneqq T_{J}^{J}(\mathbf{p})$. 

We use the following two simple facts:
\begin{equation}\label{eq:205}
\# \mathcal{T}_J \leq (J+1)^M
\end{equation}
and for every type class
\begin{equation}\label{eq:206}
\big(J+1\big)^{-M} e^{J\cdot H(\mathbf{p})} 
\leq \#T_{J}(\mathbf{p}) \leq e^{J\cdot H(\mathbf{p})},
\end{equation}
see~\cite[Lemmas~2.1.2 and~2.1.8]{Dembo2010largedeviations}. 
The importance of~\eqref{eq:205} is that $\# \mathcal{T}_J$ grows only polynomially in $J$; on the other hand, the exponential terms in~\eqref{eq:206} are the same in both the lower and upper bounds. Since we are looking for critical exponents, sub-exponential multiplicative terms do not influence our calculations. To simplify notation, we write $f(J)\asymp g(J)$ if the exponential rates of growth of $f(J)>0$ and $g(J)>0$ exist and are equal to each other, so 
\begin{equation*}
f(J)\asymp g(J) \;\Longleftrightarrow\; \lim_{J\to\infty}\frac{1}{J}\log f(J) = \lim_{J\to\infty}\frac{1}{J}\log g(J). 
\end{equation*}  
In particular, if $f(J)$ is sub-exponential in $J$ (for example when $f(J) = \# \mathcal{T}_J$), then $f(J)\asymp 1$. %

The set $\mathcal{T}_J$ is a discrete set with polynomially many points which becomes dense in $\mathcal{P}_M$ as $J\to\infty$. For $\mathbf{p}\in\mathcal{P}_M$ let $\mathbf{p}_{J}$ denote the `best approximation' of $\mathbf{p}$ in $\mathcal{T}_J$, in the sense that $\|\mathbf{p}-\mathbf{p}_{J}\| = \min_{\mathbf{q}\in\mathcal{T}_J} \|\mathbf{p}-\mathbf{q}\|$, where we can take any norm. If there are many such $\mathbf{p}_{J}$ then we can choose the one with smallest coordinates when ordered lexicographically. %
For large enough $J$, $\|\mathbf{p}-\mathbf{p}_{J}\|$ is arbitrarily small. In particular, property~\eqref{eq:206} and the continuity of the entropy imply that $\#T_{J}(\mathbf{p}_J) \asymp e^{J\cdot H(\mathbf{p})}$. 

\subsection{Proof of Step 1}

This is a standard argument in large deviations theory, and in fact holds for all $t \in (\underline{t},\max_{1\leq i \leq M} \log N_i)$. We include a sketch of it for the convenience of the reader. An alternative approach would be to use Lagrange multipliers.

Let $\mathbf{I}=I_1,I_2,\dotsc$ be an infinite sequence of independent and identically distributed random variables on the set $\{1,\dotsc,M\}$ according to $\mathbf{q}\in\mathcal{P}_M$. Let $\mathds{P}_{\mathbf{q}}\coloneqq \mathbf{q}^{\mathds{N}}$ denote the product measure corresponding to the distribution of the sequence $\mathbf{I}$. Then $\boldsymbol{\tau}_J(\mathbf{I})$, the type of $(I_1,\dotsc,I_J)$, is a vector-valued random variable. For all $\mathbf{p}\in\mathcal{T}_J$,
\begin{equation*}
(J+1)^{-M} e^{-J\cdot H(\mathbf{p}\|\mathbf{q})} 
\leq \mathds{P}_{\mathbf{q}}(\boldsymbol{\tau}_J(\mathbf{I})=\mathbf{p}) \leq e^{-J\cdot H(\mathbf{p}\|\mathbf{q})},
\end{equation*}
see~\cite[Lemma~2.1.9]{Dembo2010largedeviations}. Sanov's theorem~\cite[Theorem~2.1.10]{Dembo2010largedeviations} shows that the family of laws $\mathds{P}_{\mathbf{q}} (\boldsymbol{\tau}_J(\mathbf{I})\in \cdot )$ satisfies a large deviations principle with the rate function $H(\cdot\|\mathbf{q})$. In particular, for $\mathbf{q}=\mathbf{Q}=(1/M,\dotsc,1/M)$ and the subset $\mathcal{F}_{t}$:
\begin{equation*}
\mathds{P}_{\mathbf{Q}} (\boldsymbol{\tau}_J(\mathbf{I})\in\mathcal{F}_{t} ) \asymp e^{-J \inf_{\mathbf{q}\in\mathcal{F}_t} H(\mathbf{q}\|\mathbf{Q})} = e^{-J H(\mathbf{Q}^*_{t}\|\mathbf{Q})}. 
\end{equation*}
Now define the random variable $X_\ell\coloneqq \log N_{I_\ell}$ and the averages $Y_J\coloneqq \frac{1}{J}\sum_{\ell=1}^JX_{\ell}$. Then
\begin{equation*}
Y_J = \sum_{\ih=1}^M \tau_{J,\ih}(\mathbf{I}) \cdot \log N_{\ih}
\end{equation*}
is a continuous function of $\boldsymbol{\tau}_{J}(\mathbf{I})$. Hence, by the `contraction principle'~\cite[Section~4.2.1]{Dembo2010largedeviations}, the rate function $I(t)$ of $\mathds{P}_{\mathbf{Q}} (Y_J\in \cdot )$ is equal to
\begin{equation*}
I(t) = \inf \Big\{ \, H(\mathbf{q}\|\mathbf{Q}): \sum_{\ih=1}^M q_{\ih} \log N_{\ih} = t \, \Big\}.
\end{equation*}
In particular, Lemma~\ref{lem:32} implies that
\begin{equation*}
I(t) = \inf_{\mathbf{q}\in\mathcal{E}_t} H(\mathbf{q}\|\mathbf{Q}) \stackrel{\eqref{eq:defineqstar}}{=} H(\mathbf{Q}^*_{t}\|\mathbf{Q}) \stackrel{\eqref{eq:32}}{=} \log M - H(\mathbf{Q}^*_{t}).
\end{equation*}

\subsection{Proof of Step 2}

Since $\min_{k\in\{0,1,\dotsc,J\}} \psi_{\iiv|k}(s)\leq \min\{1,\psi_{\iiv|J}(s)\}$, recall the definition of $\psi_{\iiv|k}(s)$ from~\eqref{eq:20},
\begin{equation}\label{eq:302}
\Psi_J(s) \leq \#\big\{ \, \iiv\in[M]^J: 1<\psi_{\iiv|J}(s) \, \big\} + \sum_{\iiv\in[M]^J:\, \psi_{\iiv|J}(s)\leq 1} \psi_{\iiv|J}(s).
\end{equation} 

\begin{lemma}\label{lem:33-first}
Assume $t \in (\underline{t},\max_{1\leq i \leq N} \log N_i)$ and that $(s,t)$ are related by~\eqref{eq:23}. Then
\begin{equation*}
\#\big\{ \, \iiv\in[M]^J: 1<\psi_{\iiv|J}(s) \, \big\} \asymp e^{J\cdot H(\mathbf{Q}^*_{t})}.
\end{equation*}
\end{lemma}
\begin{proof}
For any given word $\iiv\in[M]^J$, the average of the $\log N_{i_\ell}$ does not depend on the particular order of the symbols in $\iiv$, just on the relative frequency of each symbol. In other words, only the type $\boldsymbol{\tau}_J(\iiv)=(\tau_{J,1}(\iiv),\dotsc,\tau_{J,M}(\iiv))$ of $\iiv$ matters, and (recalling~\eqref{eq:309})
\begin{equation*}
\psi_{\iiv|J}(s)\leq 1 \;\Longleftrightarrow\;  \sum_{\ih=1}^{M} \tau_{J,\ih}(\iiv) \log N_{\ih} \leq t.
\end{equation*}
This reduces the problem back to a condition on probability vectors $\mathbf{p}\in\mathcal{P}_M$. This is the reason why we introduced $\mathcal{G}_t$ and $\mathcal{F}_t$ the way we did in~\eqref{eq:300}; we now see that
\begin{equation}\label{eq:301}
\psi_{\iiv|J}(s)\leq 1 \;\Longleftrightarrow\; \boldsymbol{\tau}_J(\iiv)\in \mathcal{G}_t.
\end{equation}

We are now ready to determine the exponential rate of growth of the two terms in~\eqref{eq:302} separately by grouping together words according to type class.
Let $\mathbf{Q}^*_{t,J}\in(\mathcal{F}_{t}\cap \mathcal{T}_J)$ be the type for which $H(\mathbf{Q}^*_{t,J}) = \max_{\mathbf{q}\in(\mathcal{F}_{t}\cap \mathcal{T}_J)} H(\mathbf{q})$ (if there is more than one such vector then we can choose the smallest lexicographically). Then 
\begin{equation*}
(J+1)^{-M}e^{J\cdot H(\mathbf{Q}^*_{t,J})} \leq \#\big\{ \, \iiv\in[M]^J: 1<\psi_{\iiv|J}(s) \, \big\} = \sum_{\mathbf{q}\in(\mathcal{F}_{t}\cap \mathcal{T}_J)} \# T_{J}(\mathbf{q})  \leq \#\mathcal{T}_J\cdot e^{J\cdot H(\mathbf{Q}^*_{t,J})}, 
\end{equation*}
where we used~\eqref{eq:206} for the two inequalities. As $J\to\infty$, the set $\mathcal{T}_J$ becomes dense in $\mathcal{P}_M$, and as a result $\|\mathbf{Q}^*_{t,J}-\mathbf{Q}^*_{t}\|\to 0$ so $H(\mathbf{Q}^*_{t,J})\to H(\mathbf{Q}^*_{t})$. 
Hence, it follows from~\eqref{eq:205} that $\#\big\{ \, \iiv\in[M]^J: 1<\psi_{\iiv|J}(s) \, \big\} \asymp e^{J\cdot H(\mathbf{Q}^*_{t})}$. 

An alternative way to see this would be to let $X_1,X_2,\dotsc,X_J,\dotsc $ be a sequence of i.i.d. random variables defined by~\eqref{e:definerandomvar}. Then 
\[ \#\big\{ \, \iiv\in[M]^J: 1<\psi_{\iiv|J}(s) \, \big\} = M^J \mathds{P}\left( \sum_{i = 1}^J X_i > tJ \right) \asymp M^J e^{-J \cdot I(t)} = e^{J\cdot H(\mathbf{Q}^*_{t})}, \]
where we used~\eqref{eq:309} for the first equality, Cram\'er's theorem from large deviations theory for the asymptotic equality, and Step~1 for the final equality. 
\end{proof}

\begin{lemma}\label{lem:33-second}
Assume $s\in(\underline{s},\overline{s})$ and that $(s,t)$ are related by~\eqref{eq:23}. Then
\begin{equation*}
 \sum_{\iiv\in[M]^J:\, \psi_{\iiv|J}(s)\leq 1} \psi_{\iiv|J}(s) \asymp e^{J\cdot H(\mathbf{Q}^*_{t})}.
\end{equation*}
\end{lemma}

\begin{proof}
If $\iiv\in T_J(\mathbf{p})$, then
\begin{equation}\label{eq:308}
\psi_{\iiv|J}(s) = M^{J \log_m n} n^{-sJ}\cdot \prod_{\ih=1}^M N_{\ih}^{p_{\ih}J} = e^{J\left( \left(\frac{\log M}{\log m}-s\right)\cdot \log n  + \sum_{\ih} p_{\ih}\log N_{\ih} \right)} = e^{J( -t+ \sum_{\ih} p_{\ih}\log N_{\ih} )}.
\end{equation}
Using~\eqref{eq:206} and~\eqref{eq:32}, 
\begin{equation*}
\#T_J(\mathbf{p})\cdot \psi_{\iiv|J}(s) \asymp e^{J( -t + \sum_{\ih} p_{\ih}\log N_{\ih}+H(\mathbf{p})  )} 
= e^{J\left( -t + \log N - H(\mathbf{p}\| \mathbf{P})\right)}.
\end{equation*} 
Similarly to $\mathbf{Q}^*_{t,J}$, let $\mathbf{P}^*_{t,J}\in(\mathcal{G}_{t}\cap \mathcal{T}_J)$ satisfy $H(\mathbf{P}^*_{t,J}\| \mathbf{P}) = \min_{\mathbf{p}\in(\mathcal{G}_{t}\cap \mathcal{T}_J)} H(\mathbf{p}\|\mathbf{P})$. We have $H(\mathbf{P}^*_{t,J}\| \mathbf{P})\to H(\mathbf{P}^*_{t}\| \mathbf{P})$ as $J\to\infty$. Then
\begin{equation*}
\sum_{\iiv\in[M]^J:\, \psi_{\iiv|J}(s)\leq 1}  \psi_{\iiv|J}(s) = 
\sum_{\mathbf{p}\in(\mathcal{G}_{t}\cap \mathcal{T}_J)} \sum_{\iiv\in T_J(\mathbf{p})} \psi_{\iiv|J}(s) 
\asymp e^{J\left( -t + \log N - H(\mathbf{P}^*_{t,J}\| \mathbf{P})\right)}.
\end{equation*}
Using~\eqref{eq:33} and~\eqref{eq:32} in the exponent, $-t + \log N - H(\mathbf{P}^*_{t}\| \mathbf{P}) = \log M-H(\mathbf{Q}^*_{t}\| \mathbf{Q}) = H(\mathbf{Q}^*_{t})$, and the assertion follows. 
\end{proof}

Note the importance of the assumption $t < \overline{t}$ in the proof of Lemma~\ref{lem:33-second}. 
Lemmas~\ref{lem:33-first} and~\ref{lem:33-second} and~\eqref{eq:302} immediately imply that $\overline{P}(s) \leq H(\mathbf{Q}^*_{t})$.

\subsection{Proof of Step 3}

This in fact holds for all $t \in (\underline{t},\max_{1\leq i \leq N} \log N_i)$. Assume $(s,t)$ are related by~\eqref{eq:23}. Fix $R \in \mathbb{N}$. For $J \in \mathbb{N}$, if $l \in \{0,1,\dotsc,R-1\}$, let $J_{l,R} \coloneqq  \lfloor (l+1) J / R \rfloor - \lfloor l J / R \rfloor$. We introduce 
\begin{align}\label{eq:definestjr}
\begin{split}
 S_{t,J,R} \coloneqq \big\{\, \iiv = (i_1,\dotsc,i_J) \in [M]^J : &(i_{\lfloor l J / R \rfloor + 1} , \dotsc, i_{\lfloor (l+1) J / R \rfloor }) \in T_{J_{l,R}}(\mathbf{Q}^*_{t,J_{l,R}}) \\*
&\mbox{ for all } l \in \{0,1,\dotsc,R-1\} \, \big\}.
\end{split}
 \end{align}
Then 
\begin{equation}\label{eq:lowerboundcard}
 \#  S_{t,J,R} = \prod_{l=0}^{R-1} \# T_{J_{l,R}}(\mathbf{Q}^*_{t,J_{l,R}}) \stackrel{\eqref{eq:206}}{\asymp} \prod_{l=0}^{R-1} e^{J_{l,R} \cdot H(\mathbf{Q}^*_{t,J_{l,R}})} \asymp e^{J \cdot H(\mathbf{Q}^*_{t})} = e^{J(\log M - I(t))}
 \end{equation}
as $J \to \infty$, using Step~1 of Proposition~\ref{prop:1} in the last step. %

Suppose $(i_1,\dotsc,i_J) \in S_{t,J,R}$. For all $l \in \{0,1,\dotsc,R-1\}$, $\mathbf{Q}^*_{t,J_{l,R}} \to \mathbf{Q}^*_t \in \mathcal{E}_t$, so $\psi_{(i_{\lfloor l J / R \rfloor + 1} , \dotsc, i_{\lfloor (l+1) J / R \rfloor }) | J_{l,R}}(s) \asymp 1$ as $J \to \infty$ by~\eqref{eq:309}. Let $J' \in \mathbb{N}$ be large enough that for all $J \geq J'$ and $l \in \{0,1,\dotsc,R-1\}$, $\psi_{(i_{\lfloor l J / R \rfloor + 1} , \dotsc, i_{\lfloor (l+1) J / R \rfloor }) | J_{l,R}}(s) \geq e^{-\overline{t}J_{l,R}/R}$. %
Assume $J \geq J'$. For each $k \in \{1,\dotsc,J\}$ let $\mathbf{p}(k)$ denote the type class of $(i_1,\dotsc,i_k)$ and let $l \in \{0,1,\dotsc,R-1\}$ be such that $\lfloor l J / R \rfloor < k \leq \lfloor (l+1) J / R \rfloor$. 
Then 
\begin{align*}
\psi_{(i_1,\dotsc,i_k) | k}(s) \stackrel{\eqref{eq:308}}{=} e^{k( -t+ \sum_{\ih} p_{\ih}(k)\log N_{\ih})} &\geq e^{\lfloor lJ/R\rfloor (-t  + \sum_{\ih} p_{\ih}(\lfloor lJ/R\rfloor)\log N_{\ih})} e^{(\lfloor lJ/R\rfloor - k)t} \\
&\geq \psi_{(i_1,\dotsc,i_{\lfloor l J / R \rfloor }) |  \lfloor l J / R \rfloor}(s) e^{-2\overline{t}J/R} \\*
&\geq e^{-3\overline{t}J/R},\stepcounter{equation}\tag{\theequation}\label{eq:lowerboundcombinatorialfudge}
\end{align*}
where in the last step we use that by~\eqref{eq:20}, 
\begin{equation*} \psi_{(i_1,\dotsc,i_A,i_{A+1},\dotsc,i_{A+A'}) | A+A'}(s) = \psi_{(i_1,\dotsc,i_A) | A}(s) \psi_{(i_{A+1},\dotsc,i_{A+A'}) | A'} (s).
\end{equation*} Therefore  
\begin{equation*}
\underline{P}(s) = \liminf_{J\to\infty} \frac{1}{J} \log  \Psi_J(s) \geq \liminf_{J\to\infty} \frac{1}{J} \log  \sum_{\iiv\in S_{t,J,R}} \min_{k\in\{0,1,\dotsc,J\}} \psi_{\iiv|k}(s) \geq
H(\mathbf{Q}^*_{t})-3\overline{t}/R. 
\end{equation*}
Since $R \in \mathbb{N}$ was arbitrary, the assertion $\underline{P}(s) \geq H(\mathbf{Q}^*_{t})$ follows.

	\subsection{Proof of Step 4}
This is an application of~\cite[Theorem~1]{Jordan2011bm}, and in fact holds for all 
\[ t \in \left(\min_{1 \leq i \leq M} \log N_i, \max_{1 \leq i \leq M} \log N_i\right).\]
 Indeed, by~\eqref{eq:definebeta}, for all $\lambda$, 
\begin{gather*}
(\log m) \left(  \beta_{\nu}\left( \frac{\lambda - \gamma^{-1}}{(\log m) (\frac{1}{\log m} - \frac{1}{\log n})}  \right)  + \frac{(\lambda - \gamma^{-1})\frac{\log N}{\log m}}{(\log m) (\frac{1}{\log m} - \frac{1}{\log n})}  \right)  -  \log M \\*
\qquad \qquad = \log \left( \frac{1}{M}\sum_{\ih=1}^M N_i^{\lambda} \right). 
\end{gather*}
Taking the Legendre transform of each side and using standard properties of Legendre transforms and~\eqref{eq:22} and~\eqref{eq:uniformmultifractal} proves Step~4. 
This completes the proof of Proposition~\ref{prop:1}. 

\section{Proof of the intermediate dimensions formula}\label{sec:proofofmain}

In this section, we prove Theorem~\ref{thm:main}. 
We begin by collecting some notation and facts used in the proof. Throughout the section, $\Lambda$ is a Bedford--McMullen carpet with non-uniform fibres.

\subsection{Approximate squares}\label{subsec:approxsquare}

Let $\mathcal{F}=\{f_i\}_{i=1}^N$ be an IFS generating a Bedford--McMullen carpet $\Lambda$ with $M$ non-empty columns. Recall, $[N]=\{1,2,\dotsc,N\}$ and $[M]=\{1,2,\dotsc,M\}$. To keep track of which column $f_i$ maps to, we introduce the function
\begin{equation}
\phi\colon [N]\to[M],\;\; \phi(i)\coloneqq \ih\; \text{ if } f_i \text{ maps to column } \ih.
\end{equation}\label{eq:40}
We define the symbolic spaces
$\Sigma=[N]^\mathbb N$ and $\Sigma_{\mathcal{H}}=[M]^\mathbb N$
with elements $\ii=i_1i_2\dotsb\in\Sigma$ and $\iih=\ih_1\ih_2\dotsb\in\Sigma_{\mathcal{H}}$. We use the convention that indices $i$ corresponding to maps have a `dot' while the indices $\ih$ corresponding to columns have a `hat' on top. 
To truncate~$\ii$ (respectively~$\iih$) at the first $n$ symbols we write $\ii|n=i_1i_2\dotsb i_n$ (respectively $\iih|n$). The longest common prefix of $\ii$ and $\jj$ is denoted by $\ii\wedge\jj$: its length is $|\ii\wedge\jj|=\min \{\, k : i_k\neq j_k \, \}-1$. The function $\phi$ naturally induces the map $\Phi\colon  \Sigma\to\Sigma_{\mathcal{H}}$ defined by
\begin{equation*}
\Phi(\ii)\coloneqq \iih = \phi(i_1)\phi(i_2)\dotsb.
\end{equation*}
Slightly abusing notation, $\Phi$ is also defined on finite words: $\Phi(i_1\dotsb i_n)=\phi(i_1)\dotsb\phi(i_n)$.

For compositions of maps, we use the standard notation $f_{i_1\dotsb i_n}\coloneqq f_{i_1}\circ f_{i_2}\circ\dotsb \circ f_{i_n}$.
The $n$-th level cylinder corresponding to $\ii$ is
$ C_{n}(\ii)\coloneqq f_{\ii|n}([0,1]^2)$.
The sets $\{C_{n}(\ii)\}_{n=1}^\infty$ form a nested sequence of compact sets with diameter tending to zero, hence their intersection is a unique point $x\in\Lambda$. This defines the natural projection $\Pi\colon \Sigma\to\Lambda$,
\begin{equation*}
\Pi(\ii)\coloneqq \lim_{n\to\infty} \bigcap_{n=1}^\infty C_{n}(\ii)=\lim_{n\to\infty} f_{\ii|n}(\underline 0).
\end{equation*}
The coding of a point $x\in\Lambda$ is not necessarily unique, but $\Pi$ is finite-to-one.

It is not efficient to cover cylinder sets separately; instead they are grouped together to form `approximate squares' which play the role of balls in a cover of the attractor. 
Recall, $\gamma=\log_m n$ and for $\ell,K\in\mathbb{N}$ let
\begin{equation*}
\gamma^{\ell}(K)\coloneqq \lfloor \gamma^{\ell} \cdot K \rfloor.
\end{equation*}
In particular, we write $\gamma(K) = \gamma^1(K)$, and $n^{-K}\leq m^{-\gamma(K)} < n^{-(K-1)}$. A level-$K$ \emph{approximate square} is
\begin{equation*}
B_K(\ii) \coloneqq \big\{ \, C_{\gamma(K)}(\jjv): \jjv\in[N]^{\gamma(K)},\, \jjv|K=\ii|K,\, \Phi(\jjv) = \Phi(\ii|\gamma(K)) \, \big\}.
\end{equation*}
It is a collection of level-$\gamma(K)$ cylinder sets that lie in the same level-$\gamma(K)$ column of a specific level-$K$ cylinder set. In other words, $\Pi(\jj)\in B_K(\ii)$ if and only if $|\ii\wedge\jj|\geq K$ and $|\Phi(\ii)\wedge\Phi(\jj)|\geq \gamma(K)$. Hence, abusing notation slightly, we identify $B_K(\ii)$ with the single sequence
\begin{equation*}%
	B_K(\ii) = (i_1,\dotsc,i_{K}, \phi(i_{K+1}), \dotsc, \phi(i_{\gamma(K)})) = (i_1,\dotsc,i_{K}, \ih_{K+1}, \dotsc, \ih_{\gamma(K)}).
\end{equation*}
The choice of $\gamma(K)$ implies that there exists $C\geq 1$ independent of $K$ and $\ii$ such that $C^{-1}n^{-K}\leq |B_K(\ii)|\leq C n^{-K}$. The constant $C$ does not influence the behaviour of the $s$-cost of any cover with approximate squares. It is easy to see that two  approximate squares are either disjoint, completely agree or intersect just on the boundary. Hence, the set of all level-$K$ approximate squares, denoted by $\mathcal{B}_K$, gives an efficient $n^{-K}$-cover of $\Lambda$ with cardinality
\begin{equation*}
\# \mathcal{B}_K = N^{K}\cdot M^{\gamma(K)-K} \stackrel{\eqref{eq:103}}{\asymp} n^{K\dim_{\mathrm B}\Lambda}.
\end{equation*}

\subsection{Two lemmas}

Recall that $T_s(t)=\big(s-\frac{\log M}{\log m}\big)\log n +\gamma I(t)$. Since $I(t)$ is strictly convex, there exists a unique~$t'$ such that $I'(t')=\gamma^{-1}$. 

\begin{lemma}\label{lem:41}
For each fixed $s \in \R$, the function $T_s(t)$ is strictly convex with a minimum at $\underline{t}$, and satisfies $T_s'(t') = 1$. 
Moreover, for all $s\geq \dim_{\mathrm H}\Lambda$ and $t\in\mathbb{R}$ we have $T_s(t) \geq t$ with equality if and only if $s=\dim_{\mathrm H}\Lambda$ and $t=t'$.
\end{lemma}

\begin{proof}
Since $I(t)$ is strictly convex with a minimum at $\underline{t}$, the same is true of $T_s(t)$ for each fixed $s \in \R$, and the definition of $t'$ implies that $T'_s(t')=1$. 
Using the formula~\eqref{eq:102} for $\dim_{\mathrm H}\Lambda$ and then that 
\[I(t')=\gamma^{-1}t'-\log\bigg(\frac{1}{M}\sum_{\ih=1}^M N_{\ih}^{\gamma^{-1}}\bigg),\]
one gets $T_{\dim_{\mathrm H}\Lambda}(t')=t'$ after simplifications. 
Note that $T_s(t') > T_{\dim_{\mathrm H}\Lambda}(t')=t'$ for all $s > \dim_{\mathrm H}\Lambda$, which is enough to complete the proof. 
\end{proof}

 Since $I(t)$ is strictly increasing, let $t^*$ denote the unique solution to the equation
\begin{equation}\label{eq:44}
\dim_{\mathrm H}\Lambda = \dim_{\mathrm B}\Lambda - \big(1-\gamma^{-1}\big)\frac{I(t^*)}{\log m}. 
\end{equation}
Recall the notation for $t_\ell(s)$ from~\eqref{eq:definetsequence}. 

\begin{lemma}\label{lem:tprimelesststar}
We have $t_1(\dim_{\mathrm H} \Lambda) < t' < t^*$. 
\end{lemma}
\begin{proof}
Since $I(\underline{t}) = 0$, we have $t_1(\dim_{\mathrm H} \Lambda) = T_{\dim_{\mathrm H} \Lambda} (\underline{t})$. 
Also $I'(\underline{t}) < \gamma^{-1} = I'(t')$ and $I$ is strictly convex, so $\underline{t} < t'$. 
But $T_{\dim_{\mathrm H} \Lambda}$ is strictly increasing, so 
\[ t_1(\dim_{\mathrm H} \Lambda) = T_{\dim_{\mathrm H} \Lambda} (\underline{t}) < T_{\dim_{\mathrm H} \Lambda} (t') = t',\] 
where the last equality is by Lemma~\ref{lem:41}. 

To prove $t' < t^*$, for $z \in \mathbb{R}$ define 
\[ f(z) \coloneqq \log n \log M - \log n \log \sum_{\ih=1}^M N_{\ih}^{\frac{z}{\log n}}  + \log n \log (N/M) - (\log n - z)\frac{\sum_{\hat k=1}^M N_{\hat k}^{\frac{z}{\log n}}\log N_{\hat k}}{\sum_{\jh=1}^M N_{\jh}^{\frac{z}{\log n}}}. \]
Then after some algebraic manipulations, 
\begin{equation*}%
f'(z) = -\left(1 - \frac{z}{\log n}\right) \left( \sum_{\ih} \frac{N_{\ih}^{\frac{z}{\log n}}}{\sum_{\jh} N_{\jh}^{\frac{z}{\log n}}} (\log N_{\ih})^2  -  \Bigg( \sum_{\hat k} \frac{N_{\hat k}^{\frac{z}{\log n}}}{\sum_{\hat\ell} N_{\hat\ell}^{\frac{z}{\log n}}} \log N_{\hat k} \Bigg)^2  \right)  < 0 
\end{equation*}
for all $z \in [\log m,\log n)$ by Jensen's inequality, using that $\Lambda$ has non-uniform fibres. Moreover, $f$ is continuous on $\mathbb{R}$, so $f(\log m) > f(\log n) = 0$, so  
\begin{align}
0 &< \frac{f(\log m)}{\log (n/m)} \nonumber \\*
&= \frac{\mathrm{d}}{\mathrm{d}\lambda} \Bigg(   \lambda \Big( \Big( \dim_{\mathrm H} \Lambda - \frac{\log M}{\log m} \Big) \log n  +  \gamma (\dim_{\mathrm B} \Lambda - \dim_{\mathrm H} \Lambda)  \frac{\log m}{1-\gamma^{-1}} \Big)  \nonumber \\*
&\phantom{----} -  \log\Bigg(\frac{1}{M}\sum_{\jh = 1}^M N_{\jh}^\lambda \Bigg)  \Bigg) \Big|_{\lambda = \gamma^{-1}} \label{eq:tprimelesststar2} \\
&= \frac{\mathrm{d}}{\mathrm{d}\lambda} \left( \lambda T_{\dim_{\mathrm H} \Lambda}(t^*) - \log \Bigg(\frac{1}{M}\sum_{\jh = 1}^M N_{\jh}^\lambda \Bigg) \right) \Big|_{\lambda = \gamma^{-1}} \label{eq:tprimelesststar3}, 
\end{align}
where~\eqref{eq:tprimelesststar2} is by~\eqref{eq:102} and~\eqref{eq:103}, and~\eqref{eq:tprimelesststar3} is by~\eqref{eq:44} and~\eqref{eq:defineiteratingfunction}. 
This means that the value of $\lambda$ at which the supremum in the definition of $I(T_{\dim_{\mathrm H} \Lambda}(t^*))$ in~\eqref{eq:22} is attained is greater than $\gamma^{-1}$. Equivalently, $I'(T_{\dim_{\mathrm H} \Lambda}(t^*)) > \gamma^{-1}$. By the definition of $t'$, this means that $T_{\dim_{\mathrm H} \Lambda}(t^*) > t'$. By Lemma~\ref{lem:41}, it follows that $t' < t^*$. 
\end{proof}

\subsection{Upper bound \texorpdfstring{for $\theta = \gamma^{-(L-1)}$}{when θ is a power of 1/γ}}\label{sec:upperinteger}

In Lemma~\ref{lem:50} we construct a cover in the case when $\theta = \gamma^{-(L-1)}$ in order to establish certain relations which are crucial in bounding the $s$-cost of the more complicated general cover in Section~\ref{sec:mainupper}. 
In Section~\ref{sec:prooflower} we will establish the matching lower bound. 
Recall, for $\theta\in(0,1)$ we defined $L \coloneqq 1 + \lfloor \frac{-\log \theta}{\log \gamma}\rfloor$ so that $\gamma^{-L} < \theta \leq \gamma^{-(L-1)}$, and for $\delta>0$ we define $K \coloneqq \lfloor\frac{-\log \delta}{\log n}\rfloor$. 
Figure~\ref{f:insideapproxsquare} is a diagram representing some approximate squares of levels $K$, $\gamma(K)$ and $\gamma^2(K)$. 
The proof strategy is to keep a level-$K$ approximate square at level $K$ if and only if $\tau(\iih,K,\gamma(K))$ exceeds a constant $\tau_1$ which remains unspecified for now (recalling notation from~\eqref{eq:tauaverage}). Of those which we subdivide, we keep them at level $\gamma(K)$ if and only if $\tau(\iih,\gamma(K),\gamma^2(K)) \geq \tau_2$. Continuing this process gives a cover of $\Lambda$ using approximate squares at levels $K, \gamma(K),\dotsc,\gamma^{L-1}(K) = \lfloor K/\theta\rfloor$. This means that, up to some constant, all covering sets have diameter in the correct range $[\delta,\delta^{1/\theta}]$. In Lemma~\ref{lem:50} we calculate the $s$-cost of this cover for an arbitrary tuple, which will allow us to prove in Lemma~\ref{lem:51} that the relevant $t_i(s)$ are bounded above by $\overline{t}$, so results from Section~\ref{sec:proofProp} will apply. At the end we will optimise the thresholds so that the exponential rate of growth of each part is the same. The unique $s$ for which this can be done gives us the upper bound for $\overline{\dim}_{\theta}\Lambda$. 
\begin{figure}[ht]
\center{\includegraphics[width=.35\textwidth]
        {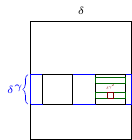}}
        \caption{\label{f:insideapproxsquare}
        Approximate squares of size $\delta^\gamma$ and $\delta^{\gamma^2}$ inside an approximate square of size $\delta$. 
 }
\end{figure}

Lemma~\ref{lem:simpleindcomb} will be used to calculate the cost of the cover in Lemma~\ref{lem:50}. 
Lemma~\ref{lem:simpleindcomb} is rather similar to Lemma~\ref{lem:33-second}, and is also proved using the method of types. 
For a probability vector $\mathbf{p}$ we write $t_{\mathbf{p}} \coloneqq \sum_{\ih} p_{\ih} \log N_{\ih}$. For a word $(i_1,\dotsc,i_J) \in [N]^J$, we write $\iih=(\ih_1,\dotsc,\ih_J)=(\phi(i_1),\dotsc,\phi(i_J))\in [M]^J$, recall~\eqref{eq:40}.
\begin{lemma}\label{lem:simpleindcomb}
For all $t \in [\underline{t}, \max_{1\leq \ih \leq M} \log N_{\ih})$, 
\[ \# \{ \, (i_1,\dotsc,i_J) \in [N]^J : \tau(\iih,0,J) \leq t \, \} \asymp e^{(\min\{t,\overline{t}\} + \log M - I(\min\{t,\overline{t}\} ))J}. \]
If, on the other hand, $t \in (\min_{1\leq \ih \leq M} \log N_{\ih}, \underline{t})$, then 
\[ \limsup_{J \to \infty} \frac{1}{J} \log \# \Big\{ \, (i_1,\dotsc,i_J) \in [N]^J : \sum_{j=1}^J \log N_{\ih_j} \leq t J \, \Big\} \leq t + \log M < \underline{t} + \log M. \]
\end{lemma}

\begin{proof}
The second part of the statement holds simply by the fact that there are $M^J$ strings of length $J$ on the alphabet $[M]$, so we only prove the first part of the statement. 
The strategy for the upper bound is to work with an arbitrary type class and then use the fact that there are only polynomially many type classes. Note that if $\mathbf{p} \in \mathcal{T}_J$ and $\jjv$ is any representative of the type class $T_J(\mathbf{p})$ then 
\[ \# \{ \, \mathbf{i} \in [N]^J : \iih = \jjv \, \} = \prod_{\kh = 1}^M N_{\kh}^{p_{\kh} J}.\]  Therefore 
\begin{align*}
\# &\{ \, (i_1,\dotsc,i_J) \in [N]^J : \tau(\iih,0,J) \leq t \, \} = \sum_{\mathbf{p} \in \mathcal{T}_J : t_{\mathbf{p}} \leq t}  \# T_J(\mathbf{p}) \cdot \prod_{\kh = 1}^M N_{\kh}^{p_{\kh} J}  \\
&\leq \sum_{\mathbf{p} \in \mathcal{T}_J : t_{\mathbf{p}} \leq t} \prod_{\kh = 1}^M N_{\kh}^{p_{\kh} J} \cdot \# \Big\{ \, \iih \in [M]^J : \frac{1}{J} \sum_{\ell=1}^{J} \log N_{\ih_\ell} \geq t_{\mathbf{p}} \, \Big\}  \\
&\asymp \sum_{\mathbf{p} \in \mathcal{T}_J : t_{\mathbf{p}} \leq t} e^{t_{\mathbf{p}} J} e^{J\cdot H(\mathbf{Q}^*_{t_{\mathbf{p}}})} \qquad \text{by Lemma~\ref{lem:33-first}} \\
&\asymp \sum_{\mathbf{p} \in \mathcal{T}_J : t_{\mathbf{p}} \leq t} e^{J( t_{\mathbf{p}} + \log M - I(t_{\mathbf{p}}))} \qquad \text{by Step~1 of Proposition~\ref{prop:1}} \\
& \leq \# \mathcal{T}_J \cdot e^{(\min\{t,\overline{t}\} + \log M - I(\min\{t,\overline{t}\} ))J} \quad \text{since } \substack{ 0 < I'(\tau) < 1 \text{ if } \tau \in (\underline{t},\overline{t}], \\ I'(\tau) > 1 \text{ if } \tau \in (\overline{t},\max_{1\leq \ih \leq M} \log N_{\ih})} \\
&\asymp e^{(\min\{t,\overline{t}\} + \log M - I(\min\{t,\overline{t}\} ))J} \qquad \text{by the upper bound of~\eqref{eq:206}}. 
\end{align*}

For the lower bound, if $\mathbf{q}$ is the closest approximation in $\mathcal{T}_J$ to $\mathbf{Q}^*_{\min\{t,\overline{t}\}}$ for which $t_{\mathbf{q}} \leq t$, then 
\begin{align*}
 &\# \{ \, (i_1,\dotsc,i_J) \in [N]^J : \tau(\iih,0,J) \leq t \, \} \geq  \# T_J(\mathbf{q}) \cdot \prod_{\kh = 1}^M N_{\kh}^{q_{\kh} J} \\
 &\geq e^{t_{\mathbf{q}} J} (J+1)^{-M} e^{J\cdot H(\mathbf{q})} \qquad \text{by the lower bound of~\eqref{eq:206}}\\
 &\asymp e^{(\min\{t,\overline{t}\} + H(\mathbf{Q}^*_{\min\{t,\overline{t}\}}))J} \qquad  \substack{\text{since } \mathbf{q} \to \mathbf{Q}^*_{\min\{t,\overline{t}\}} \in \mathcal{E}_{\min\{t,\overline{t}\}} \text{ by Lemma~\ref{lem:32}}\\
 \text{and since } H \text{ is continuous}}\\
 &\asymp e^{(\min\{t,\overline{t}\} + \log M - I(\min\{t,\overline{t}\}))J} \qquad \text{by Step~1 of Proposition~\ref{prop:1}}.
 \end{align*}
Therefore the first part of the statement of Lemma~\ref{lem:simpleindcomb} holds. 
\end{proof}

In order to calculate the $s$-cost of the cover we construct in the proof of Lemma~\ref{lem:50}, for $\btau = (\tau_1,\dotsc,\tau_{L-1}) \in (\underline{t},\overline{t})^{L-1}$ and $s \in [\dim_{\mathrm H} \Lambda,\dim_{\mathrm B} \Lambda]$ we introduce 
\begin{align*}
G_{1}^{\btau}(s) \coloneqq \frac{\log N}{\log n} &+ \gamma^{L-1}\big(1-\gamma^{-1}\big)\frac{\log M}{\log m} - \gamma^{L-1}s \\*
&+ \frac{\gamma-1}{\log n} \sum_{i=0}^{L-2}\gamma^{i}(\tau_{L-1-i}+\log M -I(\tau_{L-1-i})) 
\end{align*}
and 
\begin{align*}%
 G_{\ell}^{\btau}(s) \coloneqq \frac{\log N}{\log n} &+\frac{\gamma-1}{\log n} \Big( \gamma^{L-\ell}(\log M - I(\tau_{\ell - 1}))+\sum_{i=0}^{L-1-\ell} \gamma^{i}(\tau_{L-1-i}+\log M -I(\tau_{L-1-i})) \Big) \\*
 &-\gamma^{L-\ell}s 
\end{align*}
for $\ell = 2,3,\dotsc,L$. In particular, when $\ell=L$ the sum is empty and 
\[ G_{L}^{\btau}(s) = \dim_{\mathrm B} \Lambda - (1-\gamma^{-1})\frac{I(\tau_{L-1})}{\log m} - s\]
(note the similarity with~\eqref{eq:44}). 
Recall from~\eqref{eq:41} on page~\pageref{eq:41} that 
\begin{equation*}
S_{\delta, \theta}^{s}(\Lambda) = \inf \Big\{ \, \sum_{i}\left|U_{i}\right|^{s}:\left\{U_{i}\right\}_{i} \text { is a cover of } \Lambda \text { such that } \delta^{1/\theta} \leq\left|U_{i}\right| \leq \delta \text { for all } i \, \Big\}.
\end{equation*}

\begin{lemma}\label{lem:50}
For all $L \geq 2$, all tuples $\btau = (\tau_1,\dotsc,\tau_{L-1}) \in (\underline{t},\overline{t})^{L-1}$, and all $s \in [\dim_{\mathrm H} \Lambda,\dim_{\mathrm B} \Lambda]$, 
\begin{equation*}
\limsup_{\delta\searrow 0} \frac{\log S_{\delta, \gamma^{-(L-1)}}^{s}(\Lambda)}{-\log \delta} \leq \max_{1 \leq \ell \leq L} G_{\ell}^{\btau}(s).
\end{equation*}
\end{lemma}

\begin{proof}
We construct a cover of $\Lambda$ as follows. For $\ell \in \mathbb{N}$, let $J_{\ell}\coloneqq \gamma^{\ell}(K)-\gamma^{\ell-1}(K)$. %
Define 
\begin{equation*}
\mathcal{U}_{1}^{\btau} \coloneqq \left\{ B_{\gamma^{L-1}(K)}(\mathbf{i}) \in \mathcal{B}_{\gamma^{L-1}(K)} : \tau(\iih,\gamma^k(K),\gamma^{k+1}(K)) \leq \tau_{L-1-k} \;\text{ for } k=0,\dotsc,L-2 \right\}.
\end{equation*}

For $\ell \in \{2,3,\dotsc,L-1 \}$, define $\mathcal{U}_{\ell}^{\btau}$ to be the set of level-$\gamma^{L-\ell}(K)$ approximate squares $B_{\gamma^{L-\ell}(K)}(\mathbf{i})$ for which
\begin{align*}%
&\tau(\iih,\gamma^k(K),\gamma^{k+1}(K)) \leq \tau_{L-k-1} \;\text{ for } k=0,\dotsc,L-\ell-1, \\*
\text{ and } &\tau(\iih,\gamma^{L-\ell}(K),\gamma^{L-\ell+1}(K)) > \tau_{\ell - 1}. 
\end{align*}
Define $\mathcal{U}_{L}^{\btau} \coloneqq \big\{ \,  B_{K}(\ii)\in \mathcal{B}_{K} : \tau(\iih,K,\gamma(K)) > \tau_{L-1} \, \big\}$. 
By construction this is a cover: 
\[ \Lambda \subseteq \bigcup_{\ell=1}^{L}\mathcal{U}_{\ell}^{\btau}. \]
Observe that for all $0\leq k<\ell\leq L-1$, if $B\in\mathcal{U}_k$ and $B'\in\mathcal{U}_{\ell}$ then they are either disjoint or intersect on their boundary, but it can never happen that $B\cap B'=B'$. 

For $\ell \in \{2,3,\dotsc, L\}$, the symbolic representation of a level-$\gamma^{L-\ell}(K)$ approximate square $B_{\gamma^{L-\ell}(K)}(\mathbf{i}) \in \mathcal{U}_{\ell}^{\btau}$ is 
\[ (\underbrace{ i_1,\dotsc, i_{K} }_{\in  [N] \mbox{ freely}} , \underbrace{ i_{K+1},\dotsc, i_{\gamma(K)} }_{\tau(\iih,K,\gamma(K)) \leq \tau_{L-1}} , \dotsb , \underbrace{ i_{\gamma^{L-\ell-1}(K)+1},\dotsc,i_{\gamma^{L-\ell}(K)} }_{\tau(\iih,\gamma^{L-\ell-1}(K),\gamma^{L-\ell}(K)) \leq \tau_{\ell}}, \underbrace{ \ih_{\gamma^{L-\ell}(K) + 1},\dotsc,\ih_{\gamma^{L-l+1}(K)} }_{\tau(\iih,\gamma^{L-\ell}(K),\gamma^{L-l+1}(K)) > \tau_{\ell - 1}}). \] 
Therefore the $s$-cost of $\mathcal{U}_{\ell}^{\btau}$ is 
\begin{align*}
\sum_{U \in \mathcal{U}_{\ell}^{\btau}} |U|^s &\asymp \# \mathcal{U}_{\ell}^{\btau} \cdot n^{-\gamma^{L-\ell}Ks} \\
&\asymp N^K \prod_{k=0}^{L-\ell-1} \# \{ \, \mathbf{i} \in [N]^{J_{k+1}} : \tau(\iih,0,J_{k+1}) \leq \tau_{L-k-1} \, \} \\*
&\phantom{\leq}\times \# \{ \,  \jjh \in [M]^{J_{L-\ell+1}} : \tau(\jjh,0,J_{L-\ell+1}) > \tau_{\ell - 1} \,\} \cdot n^{-\gamma^{L-\ell}Ks} \\
&\asymp N^K \prod_{k=0}^{L-\ell-1} e^{(\tau_{L-k-1} + \log M - I(\tau_{L-k-1}))J_{k+1}} \cdot e^{(\log M - I(\tau_{\ell - 1}))J_{L-l+1}} n^{-\gamma^{L-\ell}Ks} \\
&\asymp n^{K\cdot G_{\ell}^{\btau}(s)},\stepcounter{equation}\tag{\theequation}\label{eq:intcost1}
\end{align*}
by Lemmas~\ref{lem:simpleindcomb} and~\ref{lem:33-first} and algebraic manipulations. In the case $l=L$ we used the convention that the empty product equals 1. 

The symbolic representation of $B_{\gamma^{L-1}(K)}(\mathbf{i}) \in \mathcal{U}_1^{\btau}$ is 
\[ (\underbrace{ i_1,\dotsc, i_{K} }_{\in  [N] \mbox{ freely}} , \underbrace{ i_{K+1},\dotsc, i_{\gamma(K)} }_{\tau(\iih,K,\gamma(K)) \leq \tau_{L-1}} , \dotsb , \underbrace{ i_{\gamma^{L-2}(K)+1},\dotsc,i_{\gamma^{L-1}(K)} }_{\tau(\iih,\gamma^{L-2}(K),\gamma^{L-1}(K)) \leq \tau_1} , \underbrace{ \ih_{\gamma^{L-1}(K) + 1},\dotsc,\ih_{\gamma^{L}(K)} }_{\in [M] \mbox{ freely}} ). \]
Therefore, as in~\eqref{eq:intcost1}, 
\begin{equation}\label{eq:intcost2}
\sum_{U \in \mathcal{U}_1^{\btau}} |U|^s \asymp N^K \prod_{k=0}^{L-2} e^{(\tau_{L-k-1} + \log M - I(\tau_{L-k-1}))J_{k+1}} \cdot M^{\gamma^L K - \gamma^{L-1}K} n^{-\gamma^{L-1}Ks} \asymp n^{K \cdot G_{1}^{\btau}(s)}. 
\end{equation}
We have bounded the $s$-cost of each part of the cover, so the proof is complete. 
\end{proof}
 In~\eqref{eq:intcost1} and~\eqref{eq:intcost2} it was crucial that each $\tau_i \in (\underline{t},\overline{t})$ when using Lemma~\ref{lem:simpleindcomb}. 
Lemma~\ref{lem:50} tells us the exponential growth rate of the cover for all tuples $\btau$, but motivated by~\eqref{eq:45}, of particular interest is the case when 
\[  G_{1}^{\btau}(s) = \dotsb = G_{L}^{\btau}(s) = 0.\]
A tuple $\btau = (\tau_1,\dotsc,\tau_{L-1}) \in (\underline{t},\overline{t})^{L-1}$ satisfies $G_{1}^{\btau}(s) = \dotsb = G_{L}^{\btau}(s)$ if and only if $\tau_1 = t_1(s)$ (from $G_{1}^{\btau}(s) = G_{2}^{\btau}(s)$) and $\tau_{k} = T_s(\tau_{k-1})$ for $k = 2,3,\dotsc , L$ (from $G_{k}^{\btau}(s) = G_{k+1}^{\btau}(s)$). Equivalently, $\tau_k = t_k(s)$ for $k = 1,2,\dotsc,L-1$. 
The next lemma ensures that each of $t_1(s), \dotsc,t_L(s)$ lies in the correct range $(\underline{t},\overline{t})$.  In particular, writing $\mathbf{t}\coloneqq (t_1(s), \dotsc,t_{L-1}(s))$, we can then apply Lemma~\ref{lem:50} to obtain the upper bound 
\[ \limsup_{\delta\searrow 0} \frac{\log S_{\delta, \gamma^{-(L-1)}}^{s}(\Lambda)}{-\log \delta} \leq G_{L}^{\mathbf{t}}(s) = \dim_{\mathrm B} \Lambda - (1-\gamma^{-1})\frac{I(t_{L-1}(s))}{\log m} - s. \] 

\begin{lemma}\label{lem:51}
Let $\theta\in(0,1)$, $L = L(\theta) \coloneqq 1 + \lfloor \frac{-\log \theta}{\log \gamma} \rfloor$ and $s\in(\dim_{\mathrm H}\Lambda,\overline{\dim}_{\gamma^{-(L-1)}}\,\Lambda]$. Then, using notation from~\eqref{eq:definetsequence} and~\eqref{eq:44},  
\[ \underline{t} < t_1(\dim_{\mathrm H} \Lambda) < t_1(s) < t_2(s) < \dotsb < t_L(s) \leq t_L(\overline{\dim}_{\gamma^{-(L-1)}} \Lambda) < T_{\dim_{\mathrm H} \Lambda}(t^*) < \overline{t}. \]%
\end{lemma}

\begin{proof}
 Recall from~\eqref{e:sunderlinedef} and~\eqref{e:soverlinedef} that $\underline{s}<\dim_{\mathrm H}\Lambda<s$. It is therefore immediate that
\begin{equation*}
\underline{t} = t_1(\underline{s}) < t_1(\dim_{\mathrm H} \Lambda) < t_1(s). 
\end{equation*}
It follows from Lemma~\ref{lem:41} that $t_1(s) < t_2(s) < \dotsb < t_L(s)$ for all $s>\dim_{\mathrm H}\Lambda$. 
Since $s \leq \overline{\dim}_{\gamma^{-(L-1)}} \Lambda$ we have $t_L(s) \leq t_L(\overline{\dim}_{\gamma^{-(L-1)}} \Lambda)$. 

We now prove $T_{\dim_{\mathrm H} \Lambda}(t^*) < \overline{t}$. To do so, we define, for every fixed $\mathbf{p} \in \mathcal{P}_M$, the function $f_{\mathbf{p}} \colon (0,\infty) \to \mathbb{R}$ by 
\[ f_{\mathbf{p}}(z) \coloneqq \log n \cdot \log \sum_{\ih=1}^M p_{\ih}^{z/\log n}   -   (\log n -z)H(\mathbf{p}) . \]
Recall that $\mathbf{Q} = (1/M,\dotsc,1/M)$ and $\mathbf{P} = (N_1/N,\dotsc,N_M/N)$. 
Clearly, $f_{\mathbf{p}}(\log n) = 0$ for all $\mathbf{p}$ and $f_{\mathbf{Q}}(z) = 0$ for every $z$. The derivative of $f_{\mathbf{p}}(z)$ with respect to $z$ is
\begin{equation*}
f'_{\mathbf{p}}(z) = H(\mathbf{p}) + \sum_{\ih=1}^M \frac{p_{\ih}^{z/\log n}}{\sum_{\jh=1}^M p_{\jh}^{z/\log n}} \log p_{\ih}, 
\end{equation*}
while after some algebraic manipulations, we obtain that
\begin{equation*}
f''_{\mathbf{p}}(z) =\frac{1}{\log n} \left(
\sum_{\ih=1}^M \frac{p_{\ih}^{z/\log n}}{\sum_{\jh=1}^M p_{\jh}^{z/\log n}} (\log p_{\ih})^2 -
\Bigg( \sum_{\ih=1}^M \frac{p_{\ih}^{z/\log n}}{\sum_{\jh=1}^M p_{\jh}^{z/\log n}} \log p_{\ih} \Bigg)^2 
\right) \geq 0
\end{equation*}
by Jensen's inequality with equality if and only if $\mathbf{p}=\mathbf{Q}$ (here we use that $\mathbf{p}$ has strictly positive entries). Hence, for $\mathbf{p}\neq\mathbf{Q}$, $f_{\mathbf{p}}(z)$ is a strictly convex function for all $z>0$, and also $f'_{\mathbf{p}}(\log n)=0$, so $f_{\mathbf{p}}(z)$ has a global minimum at $z=\log n$. In particular, since $m \neq n$ and $\mathbf{P} \neq \mathbf{Q}$ as we assume the carpet has non-uniform fibres, $f_{\mathbf{P}}(\log m) > 0$. 
Using formulae~\eqref{eq:102} and $\eqref{eq:103}$ for $\dim_{\mathrm H}\Lambda$ and $\dim_{\mathrm B}\Lambda$, algebraic manipulations show that $f_{\mathbf{P}}(\log m) > 0$ is equivalent to 
\[ \left( \dim_{\mathrm H} \Lambda - \frac{\log M}{\log m} \right) \log n  +  \gamma \frac{\log m}{1-\gamma^{-1}} (\dim_{\mathrm B} \Lambda - \dim_{\mathrm H} \Lambda)   <   \log N  - H(\mathbf{P}). \]
But we can express $I(t^*)$ from~\eqref{eq:44} and use the definition $\overline{t} \coloneqq \log N  - H(\mathbf{P})$ to show that this is equivalent to the assertion $T_{\dim_{\mathrm H} \Lambda}(t^*) < \overline{t}$, as required. 

It remains to prove $t_L(\overline{\dim}_{\gamma^{-(L-1)}}\Lambda) < T_{\dim_{\mathrm H} \Lambda}(t^*)$. To do so, we first prove the weaker claim $t_{L-1}(\overline{\dim}_{\gamma^{-(L-1)}} \Lambda)  <  t^*$ using the fact that Lemma~\ref{lem:50} holds for an arbitrary tuple $\btau$. Assume for a contradiction that $t_{L-1}(\overline{\dim}_{\gamma^{-(L-1)}} \Lambda)  \geq t^*$. We define a tuple $\btau = (\tau_1,\dotsc,\tau_{L-1}) \in (\underline{t},\overline{t})^{L-1}$ as follows. %
If $t_1(\overline{\dim}_{\gamma^{-(L-1)}} \Lambda) \geq t^*$, then define $\tau_l \coloneqq t^*$ for all $l \in \{ 1,2,\dotsc,L-1\}$, noting that $\underline{t} < t^* \leq T_{\dim_{\mathrm H} \Lambda}(t^*) < \overline{t}$. If, on the other hand, $t_1(\overline{\dim}_{\gamma^{-(L-1)}} \Lambda) < t^*$, then define 
\[ l_* \coloneqq \max \{ \, l \in \{ 1,2,\dotsc,L-2 \} : t_{l}(\overline{\dim}_{\gamma^{-(L-1)}} \Lambda) < t^* \, \} .\] %
For $l \in \{1,2,\dotsc,l_*\}$ let $\tau_l \coloneqq t_l(\overline{\dim}_{\gamma^{-(L-1)}} \Lambda)$. For $l \in \{l_*+1,l_* + 2,\dotsc,L-1\}$ let $\tau_l \coloneqq t_{l_*} + \frac{l - l_*}{L-1-l_*}(t^* - t_{l_*}(s))$, %
so 
\[ \underline{t} < t_1(\dim_{\mathrm H} \Lambda) < t_1(\overline{\dim}_{\gamma^{-(L-1)}} \Lambda) = \tau_1 < \tau_2 < \dotsb < \tau_{L-2} < \tau_{L-1} = t^* \leq T_{\dim_{\mathrm H} \Lambda}(t^*) < \overline{t}.\] 
In either case, $\tau_1 \leq t_1(\overline{\dim}_{\gamma^{-(L-1)}} \Lambda)$ and $\tau_{l+1} \leq T_{\overline{\dim}_{\gamma^{-(L-1)}} \Lambda}(\tau_{l})$ for all $l \in \{1,2,\dotsc,L-2\}$, so 
\[ G_{1}^{\btau}(\overline{\dim}_{\gamma^{-(L-1)}} \Lambda) \leq G_{2}^{\btau}(\overline{\dim}_{\gamma^{-(L-1)}} \Lambda) \leq \dotsb \leq G_{L}^{\btau}(\overline{\dim}_{\gamma^{-(L-1)}} \Lambda).\]
 Therefore by Lemma~\ref{lem:50}, 
\begin{align*}
0 = \limsup_{\delta\searrow 0} \frac{\log S_{\delta, \gamma^{-(L-1)}}^{\overline{\dim}_{\gamma^{-(L-1)}} \Lambda}(\Lambda)}{-\log \delta} &\leq \max_{1 \leq \ell \leq L} G_{\ell}^{\btau}(\overline{\dim}_{\gamma^{-(L-1)}} \Lambda) 
 = G_{L}^{\btau}(\overline{\dim}_{\gamma^{-(L-1)}} \Lambda) < 0 
\end{align*} 
(the last inequality holds since $\tau_{L-1} = t^*$ and $\overline{\dim}_{\gamma^{-(L-1)}} \Lambda > \dim_{\mathrm H} \Lambda$, see \cite[Section~4]{Falconer2020firstintermediate}), a contradiction. Thus $t_{L-1}(\overline{\dim}_{\gamma^{-(L-1)}} \Lambda)  <  t^* \leq T_{\dim_{\mathrm H} \Lambda}(t^*) < \overline{t}$ (using Lemma~\ref{lem:41}). 

To complete the proof that $t_L(\overline{\dim}_{\gamma^{-(L-1)}} \Lambda) < T_{\dim_{\mathrm H} \Lambda}(t^*)$, we apply Lemma~\ref{lem:50} again but this time with the optimal tuple $\mathbf{t} \coloneqq (t_1(\overline{\dim}_{\gamma^{-(L-1)}} \Lambda),\dotsc,t_{L-1}(\overline{\dim}_{\gamma^{-(L-1)}} \Lambda))$ which we now know lies in the correct range: 
\begin{align*}
 0 &= \limsup_{\delta\searrow 0} \frac{\log S_{\delta, \gamma^{-(L-1)}}^{\overline{\dim}_{\gamma^{-(L-1)}} \Lambda}(\Lambda)}{-\log \delta} \\
 &\leq \max_{1 \leq \ell \leq L} G_{\ell}^{\mathbf{t}}(\overline{\dim}_{\gamma^{-(L-1)}} \Lambda) \\
 &= G_{L}^{\mathbf{t}}(\overline{\dim}_{\gamma^{-(L-1)}} \Lambda) \\
 &= \dim_{\mathrm B}\Lambda - (1-\gamma^{-1})\frac{I\big( t_{L-1}(\overline{\dim}_{\gamma^{-(L-1)}} \Lambda)\big)}{\log m} - \overline{\dim}_{\gamma^{-(L-1)}} \Lambda,\stepcounter{equation}\tag{\theequation}\label{eq:combrightrangeupperbound}
 \end{align*}%
noting that all terms in the maximum are in fact equal by the definition of $\mathbf{t}$. Therefore 
\begin{align}
t_L(\overline{\dim}_{\gamma^{-(L-1)}} \Lambda) 
&=  T_{\overline{\dim}_{\gamma^{-(L-1)}} \Lambda}(t_{L-1}(\overline{\dim}_{\gamma^{-(L-1)}} \Lambda)) \label{eq:rightrangelast1}\\
&\leq \left(\overline{\dim}_{\gamma^{-(L-1)}} \Lambda  -  \frac{\log M}{\log m}\right) \log n  +   \gamma\frac{\log m}{1-\gamma^{-1}}  (  \dim_{\mathrm B} \Lambda - \overline{\dim}_{\gamma^{-(L-1)}} \Lambda )  \label{eq:rightrangelast2} \\
&<   \left( \dim_{\mathrm H} \Lambda -  \frac{\log M}{\log m}\right) \log n  +   \gamma\frac{\log m}{1-\gamma^{-1}}  (  \dim_{\mathrm B} \Lambda -  \dim_{\mathrm H} \Lambda)  \label{eq:rightrangelast3}\\
&= T_{\dim_{\mathrm H} \Lambda}(t^*), \label{eq:rightrangelast4}
\end{align}
where~\eqref{eq:rightrangelast1} is by~\eqref{eq:definetsequence};~\eqref{eq:rightrangelast2} is by~\eqref{eq:defineiteratingfunction} and~\eqref{eq:combrightrangeupperbound};~\eqref{eq:rightrangelast3} is since $\overline{\dim}_{\gamma^{-(L-1)}} \Lambda > \dim_{\mathrm H} \Lambda$; and~\eqref{eq:rightrangelast4} is by~\eqref{eq:defineiteratingfunction} and~\eqref{eq:44}. 
This completes the proof. 
\end{proof}

\subsection{Lower bound}\label{sec:prooflower}%
For $\theta \in (0,1)$, $s \in (\dim_{\mathrm{H}} \Lambda,\overline{\dim}_{\gamma^{-(L-1)}} \Lambda]$, sufficiently small $\delta$, and $R \in \mathbb{N}$, we define a measure $\mu = \mu_{\delta,s,\theta,R}$ which we will use to apply the mass distribution principle. Recall that $K = K(\delta) = \lfloor\frac{-\log \delta}{\log n}\rfloor$. The measure will be defined by putting point masses on some carefully chosen level-$\lfloor K/\theta \rfloor$ approximate squares (corresponding to the very smallest scale $\delta^{1/\theta}$).

If $k \in \mathbb{N}$ and $B_k$ is a level-$k$ approximate square in $\mathcal{B}_k$, we can choose a point $\Lambda_{B_k} \in \Lambda$ in the interior of $B_k$. We can make this choice explicitly by choosing the image of a distinguished point in $\Lambda$ inside the top-most (in the plane) level-$\gamma(k)$ cylinder within $B_k$. 
Let $\delta_{\Lambda_{B_k}}$ denote a unit point mass at $\Lambda_{B_k}$. 
For $l \in \mathbb{N}$ define 
\begin{equation}\label{eq:definejnumbers}
J_l \coloneqq \gamma^l(K)  -   \lfloor K/(\gamma^{L-l}\theta) \rfloor;  \qquad J_l' \coloneqq \lfloor K/(\gamma^{L-l}\theta) \rfloor - \gamma^{l-1}(K). 
\end{equation} %
Using notation from~\eqref{eq:definestjr}, define $C_{K,s,\theta,R}$ to be the set of level-$\lfloor K/\theta\rfloor$ approximate squares $(i_1,\dotsc,i_{\lfloor K/\theta\rfloor}, \ih_{\lfloor K/\theta\rfloor+1}, \dotsc, \ih_{\gamma(\lfloor K/\theta\rfloor)}) \in \mathcal{B}_{\lfloor K/\theta\rfloor}$ for which 
\begin{equation}\label{eq:definemass1}
(\ih_{\lfloor K/(\gamma^{L-l}\theta) \rfloor + 1},\dotsc,\ih_{\gamma^l(K)}) \in S_{t_{L-l}(s),J_l,R}  \quad \mbox{for} \quad l \in \{1,2,\dotsc,L-1\}.
\end{equation}
and
\begin{equation}\label{eq:definemass2}
(\ih_{\gamma^{l-1}(K) + 1},\dotsc,\ih_{\lfloor K/(\gamma^{L-l}\theta) \rfloor}) \in        S_{t_{L-l+1}(s),J_l',R} \quad \mbox{for} \quad l \in \{1,2,\dotsc,L\}.
\end{equation}
Note that when $\theta = \gamma^{-(L-1)}$ we do not impose the condition~\eqref{eq:definemass2}. 
Now we define the measure 
\begin{equation}\label{eq:definemu}
\mu = \mu_{\delta,s,\theta,R} \coloneqq \sum_{B_{\lfloor K/\theta \rfloor} \in C_{K,s,\theta,R}} n^{-Ks/\theta} \delta_{\Lambda_{B_{\lfloor K/\theta \rfloor}}}.
\end{equation}
This is clearly supported on $\Lambda$. 

For the remainder of this section, for 
\[ k \in \{ K, {K+1},\dotsc,{\lfloor K/\theta \rfloor} \}, \]
$B_k$ will denote an arbitrary level-$k$ approximate square in $\mathcal{B}_k \cap \mathrm{supp}(\mu)$. By the definition of the $S_{t,J,R}$ sets, $\mu(B_k)$ depends on $k,\delta,s,\theta,R$ and the carpet, but if $k = \gamma^l(K)$ or $k = \lfloor K/\gamma^j \theta \rfloor$ for some $j \in \{0,\dotsc,L-1\}$, then $\mu(B_k)=\mu(B'_k)$ for all $B_k,B'_k\in\mathcal{B}(k) \cap \mathrm{supp}(\mu)$. 
\begin{lemma}\label{lem:lowernicescales}
Fix $\theta \in (0,1)$, $s \in (\dim_{\mathrm{H}} \Lambda,\overline{\dim}_{\theta} \Lambda]$, and $R \in \mathbb{N}$ as above. %
For all $l = 0,1,\dotsc,L-1$, as $K \to \infty$, the following two asymptotic equalities hold: 
\begin{equation}\label{eq:nicescalesfirst} \mu_{\delta,s,\theta,R}(B_{\gamma^{l}(K)}) \asymp n^{-\gamma^l K s},
\end{equation}
\begin{equation}\label{eq:nicescaleslast}
\mu_{\delta,s,\theta,R}\left(B_{\left\lfloor \frac{K}{\theta \cdot \gamma^{L-l-1}} \right\rfloor} \right) \asymp n^{- \frac{Ks}{\theta \cdot \gamma^{L-l-1}}}.
\end{equation}
\end{lemma}

\begin{proof}%
The proof is an induction argument, starting with the smaller scales. Note that~\eqref{eq:nicescaleslast} holds for $l=L-1$ by the definition of $\mu$. We first use this to prove~\eqref{eq:nicescalesfirst} for $l=L-1$. 
Indeed, consider an approximate square $B_{\gamma^{L-1}(K)}(\mathbf{i}) \in \mathcal{B}_{\gamma^{L-1}(K)}$ and assume it intersects $\mathrm{supp}(\mu)$. 
Because of the way $\mu$ is defined, the mass of all such squares will be the same. 
To calculate this mass, we need to count up the number of level-$\lfloor K/\theta \rfloor$ approximate squares $B_{\lfloor K/\theta \rfloor}(\mathbf{j})$ which lie inside the level-$\gamma^{L-1}(K)$ square (so $B_{\lfloor K/\theta \rfloor}(\mathbf{j}) \subset B_{\gamma^{L-1}(K)}(\mathbf{i})$), and which also carry mass. 

To count the number of such smaller squares, it is helpful to compare the symbolic representation of any such square $B_{\lfloor K/\theta \rfloor}(\mathbf{j})$ with that of $B_{\gamma^{L-1}(K)}(\mathbf{i})$: 
{\small
\begin{align*}
B_{\gamma^{L-1}(K)}(\mathbf{i}):  ( i_1,\dotsc,i_{\gamma^{L-1}(K)}, \overbrace{  \ih_{\gamma^{L-1}(K) + 1},\dotsc,\ih_{\lfloor K/\theta \rfloor}  }^{ \in S_{t_1(s),J_L',R} }     &,\ih_{\lfloor K/\theta \rfloor + 1},\dotsc, \ih_{\gamma^L (K)}  ) \\*
B_{\lfloor K/\theta \rfloor}(\mathbf{j}): ( \underbrace{ j_1, \dotsc, j_{\gamma^{L-1}(K)} }_{\mbox{equal}}, \underbrace{  j_{\gamma^{L-1}(K) + 1},\dotsc,j_{\lfloor K/\theta \rfloor}}_{\mbox{same column}} &, \underbrace{ \jh_{\lfloor K/\theta \rfloor + 1},\dotsc, \jh_{\gamma^L(K)} }_{\mbox{equal}},\underbrace{  \jh_{\gamma^L(K) + 1},\dotsc, \jh_{\gamma(\lfloor K/\theta \rfloor)} }_{\in [M] \mbox{ freely}}) 
\end{align*}
}
For $t \in (\underline{t},\overline{t})$, $J \in \mathbb{N}$, $R \in \mathbb{N}$, let $\mathbf{p}(t,J,R) = (p_1(t,J,R),\dotsc,p_M(t,J,R))$ be the type class of every element of $S_{t,J,R}$. 
 Then $\mathbf{p}(t,J,R) \xrightarrow[J \to \infty]{} \mathbf{Q}^*_{t} \in \mathcal{E}_{t}$ by Lemma~\ref{lem:32}. Therefore 
 \begin{align*}
 \mu(B_{\gamma^{L-1}(K)}(\mathbf{i})) &\asymp  \prod_{\ih=1}^M N_{\ih}^{p_{\ih}(t_1(s),J_L',R) \cdot J_L'} M^{\gamma K/\theta - \gamma^L K} n^{-Ks/\theta} \\
 &= e^{t_1(s) J_L'  + K((\gamma/\theta - \gamma^L)\log M - (s\log n)/\theta)} \\
 &\asymp n^{-\gamma^{L-1} K s},
\end{align*}
using~\eqref{eq:definejnumbers} and the definition of $t_1(s)$ in~\eqref{eq:definetsequence}. Therefore~\eqref{eq:nicescalesfirst} holds for $l=L-1$. 

We now use this to prove~\eqref{eq:nicescaleslast} for $l=L-2$. If %
$B_{\gamma^{L-1}(K)}(\mathbf{j}) \subset B_{\lfloor K/(\gamma \theta)\rfloor}(\mathbf{i})$, %
then 
\begin{align*}
&(  i_1,\dotsc,i_{\lfloor K/(\gamma \theta)\rfloor}  , \overbrace{ \ih_{\lfloor K/(\gamma \theta)\rfloor + 1}, \dotsc, \ih_{\gamma^{L-1}(K)}  }^{\in S_{t_1(s),J_{L-1},R} } , \ih_{\gamma^{L-1}(K) + 1},\dotsc, \ih_{\lfloor K/\theta \rfloor}  )  \\*
&(  \underbrace{ j_1,\dotsc,j_{\lfloor K/(\gamma \theta)\rfloor} }_{\mbox{equal}}, \underbrace{ j_{\lfloor K/(\gamma \theta)\rfloor +1} ,\dotsc,  j_{\gamma^{L-1}(K)}   }_{\mbox{same column}}  , \underbrace{ \jh_{\gamma^{L-1}(K) + 1}, \dotsc, \jh_{\lfloor K/\theta \rfloor} }_{\mbox{equal}}, \underbrace{ \jh_{\lfloor K/\theta \rfloor + 1},\dotsc, \jh_{\gamma^L (K)} }_{\in [M] \mbox{ freely}} ).
\end{align*}
Therefore by Lemma~\ref{lem:32}, case $l=L-1$ of~\eqref{eq:nicescalesfirst},~\eqref{eq:definejnumbers}, and~\eqref{eq:definetsequence}, 
\begin{equation*}
\mu(B_{\lfloor K/(\gamma \theta)\rfloor}(\mathbf{i}))  \asymp e^{t_1(s)J_{L-1}} M^{\gamma^L K - K/\theta} n^{-\gamma^{L-1}Ks}  \asymp n^{-Ks/(\gamma \theta)}, 
\end{equation*}
so~\eqref{eq:nicescaleslast} holds for $l=L-2$. 

Now fix any $l \in \{0,1,\dotsc, L-2\}$ and assume that~\eqref{eq:nicescaleslast} holds for $l$. We show that this implies that~\eqref{eq:nicescalesfirst} holds for $l$. Indeed, if %
$B_{\lfloor \frac{K}{\theta \cdot \gamma^{L-l-1}} \rfloor}(\mathbf{j}) \subset B_{\gamma^l (K)}(\mathbf{i})$, %
then 
{\small
\begin{align*}
&(  i_1, \dotsc, i_{\gamma^l (K)} , \overbrace{  \ih_{\gamma^l (K) + 1}, \dotsc,  \ih_{\left\lfloor \frac{K}{\gamma^{L-l-1} \cdot \theta} \right\rfloor}  }^{ \in S_{t_{L-l}(s),J_{l+1}',R}} ,  \ih_{\left\lfloor \frac{K}{\gamma^{L-l-1}\cdot \theta} \right\rfloor + 1}, \dotsc, \ih_{\gamma^{l+1} (K)}  )  \\*
&( \underbrace{ j_1, \dotsc, j_{\gamma^l (K)} }_{\mbox{equal}}  ,  \underbrace{ j_{\gamma^l(K) + 1} , \dotsc, j_{\left\lfloor \frac{K}{\gamma^{L-l-1}\cdot \theta} \right\rfloor} }_{\mbox{same column}}  ,  \underbrace{ \jh_{\left\lfloor \frac{K}{\gamma^{L-l-1}\cdot \theta} \right\rfloor + 1} , \dotsc, \jh_{\gamma^{l+1}(K)} }_{\mbox{equal}}  , \underbrace{  \jh_{\gamma^{l+1}(K) + 1} , \dotsc,      \jh_{\left\lfloor \frac{K}{\gamma^{L-l-2} \cdot \theta} \right\rfloor}  }_{ \in S_{t_{L-l-1}(s),J_{l+2}',R}} ). 
\end{align*}
}
Therefore %
\begin{align*}
\mu(B_{\gamma^l(K)}(\mathbf{i})) &\asymp e^{t_{L-l}(s) J_{l+1}'} \cdot \#  S_{t_{L-l-1}(s),J_{l+2}',R} \cdot n^{-K s \gamma^{l+1-L}\theta^{-1}} &\substack{\text{by Lemma~\ref{lem:32}}\\
\text{and case } l \text{ of~\eqref{eq:nicescaleslast}}} \\
&\asymp e^{t_{L-l}(s) J_{l+1}'} \cdot e^{(\log M - I(t_{L-l-1}(s))) J_{l+2}'} n^{-K s \gamma^{l+1-L}\theta^{-1}}  &\text{by~\eqref{eq:lowerboundcard}}  \\*
&\asymp  n^{-\gamma^l K s} &\text{by~\eqref{eq:definejnumbers} and~\eqref{eq:definetsequence}}, 
\end{align*}
so indeed~\eqref{eq:nicescalesfirst} holds for $l$. 

Finally, fix any $l \in \{0,1,\dotsc, L-3\}$ %
and assume~\eqref{eq:nicescalesfirst} holds for $l+1$. We now show that this implies that~\eqref{eq:nicescaleslast} holds for $l$. Indeed, if %
$B_{\gamma^{l+1}(K)}(\mathbf{j}) \subset B_{\left\lfloor \frac{K}{\gamma^{L-l-1} \theta} \right\rfloor}(\mathbf{i})$, %
then 
{\footnotesize
\begin{align*}
&(  i_1,\dotsc, i_{\left\lfloor \frac{K}{\gamma^{L-l-1} \theta} \right\rfloor} , \overbrace{ \ih_{\left\lfloor \frac{K}{\gamma^{L-l-1} \theta} \right\rfloor + 1} , \dotsc, \ih_{\gamma^{l+1}(K)} }^{\in S_{t_{L-l-1}(s),J_{l+1},R}} , \ih_{\gamma^{l+1}(K) + 1} , \dotsc, \ih_{\left\lfloor \frac{K}{\gamma^{L-l-2} \theta} \right\rfloor }  )  \\*
&( \underbrace{ j_1,\dotsc, j_{\left\lfloor \frac{K}{\gamma^{L-l-1} \theta} \right\rfloor} }_{\mbox{equal}} , \underbrace{ j_{\left\lfloor \frac{K}{\gamma^{L-l-1} \theta} \right\rfloor + 1} , \dotsc, j_{\gamma^{l+1}(K)} }_{\mbox{same column}} , \underbrace{ \jh_{\gamma^{l+1}(K) + 1} , \dotsc, \jh_{\left\lfloor \frac{K}{\gamma^{L-l-2} \theta} \right\rfloor } }_{\mbox{equal}} ,\underbrace{ \jh_{\left\lfloor \frac{K}{\gamma^{L-l-2} \theta} \right\rfloor  + 1},\dotsc,  \jh_{\gamma^{l+2}(K) } }_{\in S_{t_{L-l-2}(s),J_{l+2},R}}  ). 
\end{align*}
}
As above, 
\begin{equation*}
\mu\left(B_{\left\lfloor \frac{K}{\gamma^{L-l-1} \theta} \right\rfloor}(\mathbf{i})\right) \asymp e^{t_{L-l-1}(s) J_{l+1}} \cdot e^{(\log M - I(t_{L-l-2}(s))) J_{l+2}} n^{-\gamma^{l+1} K s} \asymp n^{- \frac{Ks}{\gamma^{L-l-1} \theta}}, 
\end{equation*}
so indeed~\eqref{eq:nicescaleslast} holds for $l$. 
The proof is complete by induction. 
\end{proof}

In Lemma~\ref{lem:lowerallscales} we prove that if we make $R$ large enough then the mass is sufficiently evenly distributed for us to apply the mass distribution principle Proposition~\ref{prop:mdp} from page~\pageref{prop:mdp} in Section~\ref{sec:proofconclusion}. 
\begin{lemma}\label{lem:lowerallscales}
Let $\theta \in (0,1)$ and $s \in (\dim_{\mathrm{H}} \Lambda,\overline{\dim}_{\theta} \Lambda]$. For all $s'<s$ there exists $\delta_0 \in (0,1)$ and $R \in \mathbb{N}$ depending on $s,s',\theta$ and the carpet such that for all $\delta \in (0,\delta_0)$ and $k \in \{K,K+1,\dotsc,\lfloor K/\theta \rfloor \}$, if $b_k$ is any level-$k$ approximate square then $\mu_{\delta,s,\theta,R}(b_k) < n^{-ks'}$. 
\end{lemma}
\begin{proof}
Fix $\theta \in (0,1)$, $s \in (\dim_{\mathrm{H}} \Lambda,\overline{\dim}_{\theta} \Lambda]$ and $R \in \mathbb{N}$. 
The idea is that for each scale $k$, we will choose from the finitely many scales considered in Lemma~\ref{lem:lowernicescales} the one which corresponds to the largest size that is smaller than $n^{-k}$. We will then bound the number of approximate squares of this level which are contained in each level-$k$ approximate square which carries mass, and use Lemma~\ref{lem:lowernicescales} to bound the mass of the level-$k$ approximate square. 

Let $J' \in \mathbb{N}$ be large enough that for each $t \in \{t_1(s),\dotsc,t_{L-1}(s)\}$ and $(s_t,t)$ related by~\eqref{eq:23},~\eqref{eq:lowerboundcombinatorialfudge} holds for all $J \geq J'$ and $k' \in \{1,\dotsc,J\}$. 
By Lemma~\ref{lem:32}, we may increase $J'$ to assume further that if $\iiv \in S_{t,J,R}$ then $\prod_{j=1}^J N_{\ih_j} \leq e^{(t+1/R)J}$. Then 
\begin{align}\label{eq:lastpartproduct}
\begin{split}
\prod_{j=k'}^{J} N_{\ih_j} = \frac{\prod_{j=1}^{J} N_{\ih_j}}{\prod_{j'=1}^{k'-1} N_{\ih_{j'}}} \leq \frac{e^{(t+1/R)J}}{\psi_{(\ih_1,\dotsc,\ih_{k'}) | k'}(s_t) n^{k's_t}M^{-k'\gamma}} &\leq e^{(1+3\overline{t})J/R} e^{tJ}n^{-k's_t} M^{\gamma k'} \\*
&= e^{(1+3\overline{t})J/R  +    (J-k')t}. 
\end{split}
\end{align}
We may increase $J'$ to ensure that for all $J \geq J'$ and $k' \in \{1,\dotsc,J\}$, letting $l' \in \{0,1,\dotsc,R-1\}$ be such that $\lfloor l'J/R \rfloor < k' \leq \lfloor (l'+1)J/R \rfloor $, if $(\jh_1,\dotsc,\jh_{k'}) \in [M]^{k'}$ then 
\begin{align}\label{eq:boundrightpart}
\begin{split}
 \# \{ \, (\ih_1,\dotsc,\ih_J) \in S_{t,J,R} : \ih_p = \jh_p \mbox{ for } p \in \{1,\dotsc,k'\} \, \} &\leq e^{(J-\lfloor l'J/R \rfloor) (H(\mathbf{Q}^*_t) + 1/R)} \\*
&\leq e^{(J-k') (\log M - I(t)) + 3J (1+H(\mathbf{Q}^*_t))/R}.
\end{split}
\end{align}
Let $\delta_0>0$ be small enough that for all $\delta \in (0,\delta_0)$, $J' < J_1 \leq J_2 \leq J_3 \leq \dotsb$ and, if $\theta \neq \gamma^{-(L-1)}$, $J' < J_1' \leq J_2' \leq J_3' \leq \dotsb$. By decreasing $\delta_0$ further we may assume by Lemma~\ref{lem:lowernicescales} that for all $l \in \{0,1,\dotsc,L-1\}$, 
\begin{equation}\label{eq:massonesquare}
\mu_{\delta,s,\theta,R}(B_{\gamma^{l}(K)}) \leq n^{\gamma^l K (-s+1/R)} \qquad \mbox{and} \qquad \mu_{\delta,s,\theta,R}\left(B_{\left\lfloor \frac{K}{\gamma^{L-l-1} \theta} \right\rfloor} \right) \leq n^{ \frac{K}{\gamma^{L-l-1} \theta}(-s+1/R)}.
\end{equation}

We now consider symbolic representations of approximate squares in a similar way to Lemma~\ref{lem:lowernicescales}. 

\emph{Case 1:} 
Suppose $l \in \{0,1,\dotsc,L-2\}$ and 
\[ k \in \{{\gamma^l(K)+1},{\gamma^{l}(K)+2},\dotsc, {\lfloor K \gamma^{l+1-L}\theta^{-1} \rfloor - 1}\}.\] 
The symbolic representations of approximate squares $B_{\lfloor K \gamma^{l+1-L}\theta^{-1} \rfloor} (\mathbf{j}) \subset B_{k}(\mathbf{i})$ are as follows (broken onto two lines because they do not fit onto one line): 
\begin{align*}
&(i_1,\dotsc, i_{\gamma^l(K)}, \overbrace{ i_{\gamma^l(K)+1},\dotsc,i_{k}  , \ih_{k+1},\dotsc,\ih_{\left\lfloor \frac{K}{\gamma^{L-l-1}\theta} \right\rfloor } }^{\in S_{t_{L-l}(s),J_{l+1}',R}} , \\*
&(\underbrace{ j_1,\dotsc,j_{\gamma^l(K)},j_{\gamma^l(K)+1},\dotsc,j_k}_{\mbox{equal}}, \underbrace{ j_{k+1},\dotsc,j_{\left\lfloor \frac{K}{\gamma^{L-l-1}\theta} \right\rfloor}}_{\mbox{same column}} ,
\end{align*}
and continuing
\begin{align*}
&\ih_{\left\lfloor \frac{K }{\gamma^{L-l-1}\theta} \right\rfloor + 1},\dotsc,\ih_{\gamma^{l+1}(K)},\ih_{\gamma^{l+1}(K)+1},\dotsc, \ih_{\gamma(k)}  ) \\*
& \rlap{\ensuremath{\underbrace{ \phantom{\jh_{\left\lfloor \frac{K }{\gamma^{L-l-1}\theta} \right\rfloor + 1} ,\dotsc, \jh_{\gamma^{l+1}(K)},  \jh_{\gamma^{l+1}(K)+1},\dotsc, \jh_{\gamma(k)} }}_{\mbox{equal}}}}  \jh_{\left\lfloor \frac{K }{\gamma^{L-l-1}\theta} \right\rfloor + 1} ,\dotsc,  \jh_{\gamma^{l+1}(K)}, \underbrace{ \vphantom{\frac{a}{\frac{a}{\frac{b}{c}}}}  \jh_{\gamma^{l+1}(K)+1},\dotsc, \jh_{\gamma(k)} , \jh_{\gamma(k)+1} ,\dotsc, \jh_{\left\lfloor \frac{K}{\gamma^{L-l-2}\theta} \right\rfloor}}_{ \in S_{t_{L-l-1}(s),J_{l+2}',R}} )  .
\end{align*}
Therefore we can bound the mass 
\begin{align}
\mu(B_k(\mathbf{i})) &\leq C \prod_{y=k+1}^{\left\lfloor \frac{K}{\gamma^{L-l-1}\theta} \right\rfloor} N_{\ih_y}  e^{ \left(\frac{K}{\gamma^{L-l-2}\theta} - \gamma k \right)(\log M - I(t_{L-l-1}(s)))  +  3J_{l+2}'(1+H(\mathbf{Q}^*_{t_{L-l-1}(s)}))/R} \nonumber \\*
&\phantom{\leq}\times n^{\frac{K}{\gamma^{L-l-1}\theta}(-s+1/R)} \label{eq:allscales1} \\
&\leq C e^{\left(\frac{K}{\gamma^{L-l-1}\theta} - k\right)t_{L-l}(s) + (1+3\overline{t})J_{l+1}'/R} e^{ \left(\frac{K}{\gamma^{L-l-2}\theta} - \gamma k \right)(\log M - I(t_{L-l-1}(s)))} \nonumber \\*
 &\phantom{\leq}\times e^{3J_{l+2}'(1+H(\mathbf{Q}^*_{t_{L-l-1}(s)}))/R}  n^{\frac{K}{\gamma^{L-l-1}\theta}(-s+1/R)} \label{eq:allscales2} \\
&\leq C n^{-ks + \big((1+3\overline{t})J_{l+1}'/(\log n) +3J_{l+2}'(1+H(\mathbf{Q}^*_{t_{L-l-1}(s)}))/(\log n) + \frac{K}{\gamma^{L-l-1}\theta}  \big)/R } \label{eq:allscales3} \\
&\leq C n^{\Big(-s + \frac{(1+3\overline{t}) + 3(1+H(\mathbf{Q}^*_{t_{L-l-1}(s)})) + \log n}{R\theta \log n}     \Big)k}, \label{eq:allscales4}
\end{align}
where $C$ is a constant depending only on the carpet,  %
~\eqref{eq:allscales1} is by~\eqref{eq:boundrightpart} and~\eqref{eq:massonesquare}; \eqref{eq:allscales2} is by~\eqref{eq:lastpartproduct}; \eqref{eq:allscales3} is by~\eqref{eq:definetsequence} and algebraic manipulations; and~\eqref{eq:allscales4} is since $\max\left\{ J_{l+1}' ,J_{l+2}',\frac{K}{\gamma^{L-l-1}\theta} \right\} \leq K/\theta < k/\theta$. 

\emph{Case 2:} Suppose $l \in \{0,1,\dotsc,L-3\}$ and $k \in \{\lfloor K \gamma^{l+1-L}\theta^{-1} \rfloor + 1,\dotsc,\gamma^{l+1}(K) - 1\}$. If $B_{\gamma^{l+1}(K)}(\mathbf{j}) \subset B_k(\mathbf{i})$, then 
\begin{align*}
&(i_1,\dotsc,i_{\left\lfloor \frac{K }{\gamma^{L-l-1}\theta} \right\rfloor }, \overbrace{ i_{\left\lfloor \frac{K }{\gamma^{L-l-1}\theta} \right\rfloor + 1},\dotsc, i_{k},\ih_{k+1},\dotsc,\ih_{\gamma^{l+1}(K)} }^{\in S_{t_{L-l-1}(s),J_{l+1},R}} , \\*
&(\underbrace{ j_1,\dotsc,j_{\left\lfloor \frac{K }{\gamma^{L-l-1}\theta} \right\rfloor },j_{\left\lfloor \frac{K }{\gamma^{L-l-1}\theta} \right\rfloor + 1},\dotsc, j_{k} }_{\mbox{equal}},\underbrace{ j_{k+1},\dotsc,j_{\gamma^{l+1}(K)} }_{\mbox{same column}} , 
\end{align*}
continuing 
\begin{align*}
&\ih_{\gamma^{l+1}(K) + 1},\dotsc,\ih_{\left\lfloor \frac{K }{\gamma^{L-l-2}\theta} \right\rfloor },\ih_{\left\lfloor \frac{K }{\gamma^{L-l-2}\theta} \right\rfloor  + 1},\dotsc,\ih_{\gamma(k)} ) \\*
& \rlap{\ensuremath{\underbrace{ \phantom{ \jh_{\gamma^{l+1}(K) + 1},\dotsc,\jh_{\left\lfloor \frac{K }{\gamma^{L-l-2}\theta} \right\rfloor },\jh_{\left\lfloor \frac{K }{\gamma^{L-l-2}\theta} \right\rfloor  + 1},\dotsc,\jh_{\gamma(k)} }}_{\mbox{equal}}}} \jh_{\gamma^{l+1}(K) + 1},\dotsc,\jh_{\left\lfloor \frac{K }{\gamma^{L-l-2}\theta} \right\rfloor },\underbrace{\vphantom{\frac{a}{\frac{a}{\frac{b}{c}}}}  \jh_{\left\lfloor \frac{K }{\gamma^{L-l-2}\theta} \right\rfloor  + 1},\dotsc,\jh_{\gamma(k)},\jh_{\gamma(k) + 1},\dotsc,\jh_{\gamma^{l+2}(K)} }_{\in S_{t_{L-l-2}(s),J_{l+2},R}}). 
\end{align*}
Therefore there is a constant $C>0$ such that 
\begin{align*}
\mu(B_k(\mathbf{i})) &\leq C e^{(\gamma^{l+1}K - k) t_{L-l-1}(s) + (1+3\overline{t})J_{l+1}/R }  e^{(\gamma^{l+2}K - \gamma k)(\log M - I(t_{L-l-2}(s)))} \\*
&\phantom{\leq}\times e^{3 J_{l+2} (1 + H(\mathbf{Q}^*_{t_{L-l-2}(s)}))/R } n^{\gamma^{l+1}K(-s+1/R)} \\
&\leq C n^{\Big(-s + \frac{(1+3\overline{t}) + 3(1+H(\mathbf{Q}^*_{t_{L-l-2}(s)})) + \log n}{R\theta \log n}     \Big)k}.  
\end{align*}

\emph{Case 3:} If $k \in \{ \lfloor K/(\gamma\theta) \rfloor + 1,\dotsc, \gamma^{L-1}(K) - 1\}$ and $B_{\gamma^{L-1}(K)}(\mathbf{j}) \subset B_{k}(\mathbf{i})$, then 
\begin{align*}
&( i_1,\dotsc,i_{\left\lfloor \frac{K}{\gamma\theta}\right\rfloor}, \overbrace{ i_{\left\lfloor \frac{K}{\gamma\theta}\right\rfloor + 1} , \dotsc, i_k, \ih_{k+1},\dotsc,\ih_{\gamma^{L-1}(K)} }^{\in S_{t_1(s),J_{L-1},R}}, \\*
&( \underbrace{ j_1,\dotsc,j_{\left\lfloor \frac{K}{\gamma\theta}\right\rfloor},j_{\left\lfloor \frac{K}{\gamma\theta}\right\rfloor + 1} , \dotsc, j_k }_{\mbox{equal}},\underbrace{ j_{k+1},\dotsc,j_{\gamma^{L-1}(K)} }_{\mbox{same column}},
\end{align*}
continuing
\begin{align*}
&\ih_{\gamma^{L-1}(K)+1},\dotsc,\ih_{\lfloor K/\theta\rfloor},\ih_{\lfloor K/\theta\rfloor + 1},\dotsc,\ih_{\gamma(k)} ) \\*
&\underbrace{ \jh_{\gamma^{L-1}(K)+1},\dotsc,\jh_{\lfloor K/\theta\rfloor},\jh_{\lfloor K/\theta\rfloor + 1},\dotsc,\jh_{\gamma(k)}}_{\mbox{equal}},\underbrace{ \jh_{\gamma(k)+1},\dotsc,\jh_{\gamma^L(K)} }_{\in [M] \mbox{ freely}} ). 
\end{align*}
Therefore by the definition of $t_1(s)$ in~\eqref{eq:definetsequence} there is a constant $C>0$ such that 
\begin{equation*}
\mu(B_k(\mathbf{i})) \leq C e^{(\gamma^{L-1}K-k)t_1(s) + (1+3\overline{t})J_{L-1}/R} M^{\gamma^L K - \gamma k} n^{\gamma^{L-1} K(-s+1/R)} \leq C n^{\left( -s + \frac{1+3\overline{t} + \log n}{R\theta \log n} \right) k}. 
\end{equation*}

\emph{Case 4:} Finally, if $k \in \gamma^{L-1}(K) + 1,\dotsc , \lfloor K/\theta \rfloor - 1$ and $B_{\lfloor K/\theta \rfloor}(\mathbf{j}) \subset B_k(\mathbf{i})$ then 
\begin{align*}
&(i_1,\dotsc,i_{\gamma^{L-1}(K)}, \overbrace{ i_{\gamma^{L-1}(K)+1},\dotsc,i_k,\ih_{k+1},\dotsc,\ih_{\lfloor K/\theta \rfloor} }^{\in S_{t_1(s),J_L',R}} \\*
&(\underbrace{ j_1,\dotsc,j_{\gamma^{L-1}(K)},j_{\gamma^{L-1}(K)+1},\dotsc,j_k }_{\mbox{equal}} ,\underbrace{ j_{k+1} ,\dotsc , j_{\lfloor K/\theta \rfloor} }_{\mbox{same column}},
\end{align*}
continuing 
\begin{align*}
&\ih_{\lfloor K/\theta \rfloor + 1},\dotsc,\ih_{\gamma^L(K)},\ih_{\gamma^L(K) + 1},\dotsc, \ih_{\gamma(k)} ) \\*
&\underbrace{ \jh_{\lfloor K/\theta \rfloor + 1},\dotsc,\jh_{\gamma^L(K)},\jh_{\gamma^L(K) + 1},\dotsc, \jh_{\gamma(k)} }_{\mbox{equal}},\underbrace{ \jh_{\gamma(k)+1},\dotsc,\jh_{\gamma(K/\theta)} }_{\in [M] \mbox{ freely}} ). 
\end{align*}
Therefore by the definition of $t_1(s)$ there exists a constant $C>0$ such that 
\begin{equation*}
\mu(B_k(\mathbf{i})) \leq C e^{(K/\theta - k)t_1(s) + (1+3\overline{t})J_L'/R} M^{\gamma K/\theta - \gamma k} n^{-Ks/\theta} \leq C n^{( -s+ (1+3\overline{t})/R ) k}.
\end{equation*}
Therefore the result follows (using Lemma~\ref{lem:lowernicescales} if $k \in \{K, \gamma(K),\dotsc,\gamma^{L-1}(K)\}$) if we take $R$ large enough depending on $s,s',\theta,C$ and the carpet. 
\end{proof} 

We write 
\begin{equation}\label{eq:definemainquantity} 
G(\theta,s) \coloneqq \gamma^L\theta \log N - (\gamma^L\theta - 1) t_L(s) + \gamma(1-\gamma^{L-1}\theta)(\log M - I(t_L(s))) - s\log n.
\end{equation}
\begin{lemma}\label{lem:lowertotalmass}
Fix $\theta \in (0,1)$, $s \in (\dim_{\mathrm{H}} \Lambda,\overline{\dim}_{\gamma^{-(L-1)}} \Lambda]$, and $R \in \mathbb{N}$. The total mass 
\[ \mu_{\delta,s,\theta,R}(\Lambda) \asymp e^{K \cdot G(\theta,s) / (\gamma^L \theta)} \mbox{ as } K \to \infty. \]
\end{lemma}
\begin{proof}
The symbolic representation of a level-$K$ approximate square $B_K(\mathbf{i}) \in \mathcal{B}_{K} \cap \mathrm{supp}(\mu)$ is 
\begin{equation*}
( \underbrace{ i_1,\dotsc, i_K }_{\in [N] \mbox{ freely}}, \underbrace{ \ih_{K+1},\dotsc, \ih_{\lfloor K \gamma^{-(L-1)}\theta^{-1}\rfloor} }_{\in S_{t_L(s),J_1',R}} , \underbrace{ \ih_{\lfloor K \gamma^{-(L-1)}\theta^{-1}\rfloor + 1},\dotsc, \ih_{\gamma(K)} }_{\in S_{t_{L-1}(s),J_1,R}} ).
\end{equation*}
Therefore 
\begin{align*}
 \mu(\Lambda) &\asymp \# (\mathcal{B}_{K} \cap \mathrm{supp}(\mu)) \cdot n^{-Ks} &\text{by case } l=0 \text{ of~\eqref{eq:nicescalesfirst}} \\
 &\asymp N^K e^{(\log M - I(t_L(s)))J_1'} e^{(\log M - I(t_{L-1}(s)))J_1} n^{-Ks} &\text{by~\eqref{eq:lowerboundcard}} \\ 
 &\asymp e^{K \cdot G(\theta,s) / (\gamma^L \theta)} &\text{by~\eqref{eq:definejnumbers},\eqref{eq:definetsequence},\eqref{eq:definemainquantity}},
 \end{align*}
 completing the proof. 
\end{proof}

We have now proved enough to give Theorem~\ref{thm:main} in the case when~$\theta$ is a negative integer power of~$\gamma$, see Section~\ref{sec:proofconclusion}.

 \subsection{Upper bound for general $\theta$}\label{sec:mainupper}%
Suppose $L \in \mathbb{N}$, $\theta \in (\gamma^{-L},\gamma^{-(L-1)})$, $s \in (\dim_{\mathrm H} \Lambda,\overline{\dim}_{\gamma^{-(L-1)}} \Lambda]$ and $0 < \delta \ll 1$. We define a cover $\{V_j\}_j$ of $\Lambda$ (depending on $\theta$, $s$ and $\delta$) as follows. Every level-$K$ cylinder will be covered in the same way, and the cover will consist of approximate squares $(i_1,\dotsc,i_k,\ih_{k+1},\dotsc,\ih_{\gamma(k)})$ of different levels $k \in \{ K, K+1,\dotsc, \lfloor K/\theta \rfloor\}$. This means that the diameter of each element of the cover will, up to an irrelevant multiplicative constant depending only on the carpet, lie in the interval $[\delta^{1/\theta},\delta]$. In fact, we will use only the scales $\gamma^l(K)$ and $\lfloor K/(\gamma^l \theta)\rfloor$ for $l \in \{0,1,2,\dotsc,L-1\}$. 
Figure~\ref{fig:CoverHelp} provides a diagram which may help the reader follow the construction of the cover. 
\begin{figure}[th]
\center{\includegraphics[width=.99\textwidth]
        {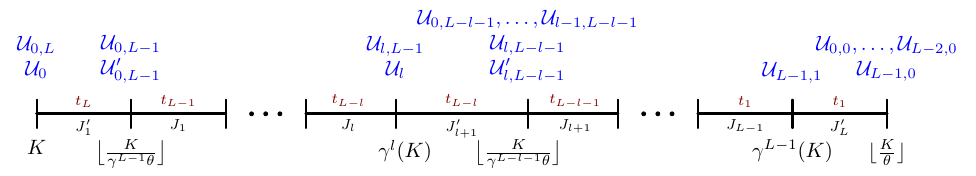}}
        \caption{Visualising the cover in~\eqref{eq:definerealcover} for $L \geq 3$. 
        Here, $l$ denotes an arbitrary number in $\{1,2,\dotsc,L-2\}$. 
The indices of the symbolic representation and the lengths of the different parts are in black. Above the scales explicitly written out are the sets (in \textcolor{blue}{blue}) which make up the part of the cover consisting of approximate squares of the corresponding level. The `critical' thresholds for the averages of the different parts of the symbolic representation are in red. Recall that the $t_i$ depend on $s$, and the sets that make up the cover depend on $s$ and $\theta$.}
\label{fig:CoverHelp}
\end{figure}

Recall that $L \coloneqq 1 + \lfloor \frac{-\log \theta}{\log \gamma} \rfloor$, and $\iih = (\ih_1,\dotsc,\ih_k,\ih_{k+1},\dotsc,\ih_{\gamma(k)})$, and we use the notation from~\eqref{eq:tauaverage}. 
We define $\mathcal{U}_{L-1,1}$ to be the set of level-$\gamma^{L-1}(K)$ approximate squares for which~\eqref{eq:secondthreshold} and~\eqref{eq:firstthreshold} below hold for all $j \in \{ 1,2,\dotsc,L-1\}$, %
and~\eqref{eq:lastthreshold} holds: 
\begin{align}
	\tau(\iih,\lfloor K/(\gamma^{L-j}\theta)\rfloor,\gamma^j(K))   &<    t_{L-j}(s); \label{eq:secondthreshold} \\
	\tau(\iih,\gamma^{j-1}(K),\lfloor K/(\gamma^{L-j}\theta)\rfloor )  &<   \frac{\gamma^j  -  (\gamma^{L-j}\theta)^{-1}}{(\gamma^{L-j}\theta)^{-1}  -  \gamma^{j-1}}  \big( t_{L-j}(s) - \tau(\iih, \lfloor K/(\gamma^{L-j}\theta)\rfloor , \gamma^j(K)) \big) \nonumber \\*
	&\phantom{-}+ t_{L-j+1}(s) ; \label{eq:firstthreshold} \\
	\tau(\iih,\gamma^{L-1}(K),\lfloor K/\theta \rfloor)    &\geq   t_1(s).  \label{eq:lastthreshold}
\end{align}
Note that when defining $\mathcal{U}_{L-1,1}$ we imposed no restriction on $\ih_{\lfloor K/\theta \rfloor + 1},\dotsc,\ih_{\gamma^L(K)}$, or on $i_1,\dotsc,i_{K-1}$. 
Define $\mathcal{U}_{L-1,0}$ to be the set of level-$\lfloor K/\theta \rfloor$ approximate squares for which~\eqref{eq:secondthreshold} and~\eqref{eq:firstthreshold} hold for all $j \in \{ 1,2,\dotsc,L-1 \}$, and~\eqref{eq:lastthreshold} does \emph{not} hold (no restriction on $\ih_{\lfloor K/\theta \rfloor + 1},\dotsc,\ih_{\gamma(\lfloor K/\theta \rfloor)}$). 
If $L=1$ then our cover is simply $\mathcal{U}_{0,0} \cup \mathcal{U}_{0,1}$, so for the rest of the construction of the cover we assume that $L > 1$.

For $l = 0,1,\dotsc,L-2$ we define $\mathcal{U}_l$ to be the set of level-$\gamma^l(K)$ approximate squares which satisfy condition~\eqref{eq:secondthreshold} for all $j \in \{1,2,\dotsc,l+1\}$, and which satisfy~\eqref{eq:firstthreshold} for all $j \in \{1,2,\dotsc, l\}$ %
 but do \emph{not} satisfy~\eqref{eq:firstthreshold} for $j=l+1$. 
For $l = 0,1,\dotsc,L-2$ we define $\mathcal{U}_{l,L-l}$ to be the set of level-$\gamma^l(K)$ approximate squares for which 
\eqref{eq:secondthreshold} holds for all $j \in \{ 1,2,\dotsc,l\}$ %
but not for $j=l+1$, and~\eqref{eq:firstthreshold} holds for all $j \in \{ 1,2,\dotsc,l\}$, %
and~\eqref{eq:tsaverage} holds: 
\begin{align}\label{eq:tsaverage}
\begin{split}
\tau(\iih,\gamma^l(K),&\lfloor K/(\gamma^{L-l-1}\theta) \rfloor) \geq \\*
&T_s \Big(  \max\Big\{ \underline{t},  \frac{1}{(\gamma^{L-l-2}\theta)^{-1}  -  \gamma^{l+1}} \Big( \Big(  \frac{1}{\gamma^{L-l-2}\theta}  - \frac{1}{\gamma^{L-l-1} \theta} \Big)  t_{L-l-1}(s)  \\*
&\phantom{T_s \;}-   \big(\gamma^{l+1}  -  (\gamma^{L-l-1}\theta)^{-1}  \big) \tau(\iih,\lfloor K/(\gamma^{L-l-1}\theta) \rfloor  , \gamma^{l+1}(K))   \Big)    \Big\} \Big).
\end{split}
\end{align}%

For $l = 0,1,\dotsc,L-2$ define $\mathcal{U}_{l,L-l-1}'$ to be the set of level-$\lfloor K/(\gamma^{L-l-1}\theta) \rfloor$ approximate squares for which~\eqref{eq:secondthreshold} and~\eqref{eq:firstthreshold} hold for all $j \in \{1,2,\dotsc,l\}$, %
 and~\eqref{eq:tsaverage} does not hold, and~\eqref{eq:secondpartverybig} holds: 
 \begin{align}\label{eq:secondpartverybig}
 \begin{split}
\tau(\iih,&\lfloor K/(\gamma^{L-l-1}\theta) \rfloor  , \gamma^{l+1}(K)) \geq \\* 
&\frac{1}{\gamma^{l+1}  -  (\gamma^{L-l-1}\theta)^{-1}  }   \Big(  \Big(  \frac{1}{\gamma^{L-l-2}\theta}  - \frac{1}{\gamma^{L-l-1} \theta} \Big)  t_{L-l-1}(s)  
-   ((\gamma^{L-l-2}\theta)^{-1}  -  \gamma^{l+1}) \underline{t}  \Big) .
\end{split}
 \end{align}
 Note that~\eqref{eq:secondpartverybig} means that~\eqref{eq:secondthreshold} does not hold for $j=l+1$, and that the maximum in~\eqref{eq:tsaverage} equals $\underline{t}$ (since $t_{L-l-1}(s) > \underline{t}$). Note also that we imposed no restriction on $(\ih_{\gamma^{l+1}(K) + 1},\dotsc,\ih_{\lfloor K/(\gamma^{L-l-2}\theta))\rfloor})$. %
For $l = 0,1,\dotsc,L-2$ define $\mathcal{U}_{l,L-l-1}$ to be the set of level-$\lfloor K/(\gamma^{L-l-1}\theta) \rfloor$ approximate squares for which~\eqref{eq:secondthreshold} holds for all $j \in \{1,2,\dotsc,l\}$ %
 but not for $j=l+1$, and~\eqref{eq:firstthreshold} holds for all $j \in \{1,2,\dotsc,l\}$, %
 and~\eqref{eq:tsaverage} does not hold, and~\eqref{eq:secondpartverybig} does not hold, and~\eqref{eq:thirdthreshold} holds: 
 \begin{align}\label{eq:thirdthreshold}
\begin{split} 
 \tau(\iih ,    \gamma^{l+1}(K) , \lfloor K/(\gamma^{L-l-2}\theta) \rfloor ) &\geq  \frac{1}{(\gamma^{L-l-2}\theta)^{-1}  -  \gamma^{l+1}} \Big( \Big(  \frac{1}{\gamma^{L-l-2}\theta}  - \frac{1}{\gamma^{L-l-1} \theta} \Big)  t_{L-l-1}(s)  \\*
&\phantom{\geq}-   (\gamma^{l+1}  -  (\gamma^{L-l-1}\theta)^{-1}  ) \tau(\iih,\lfloor K/(\gamma^{L-l-1}\theta) \rfloor  , \gamma^{l+1}(K))   \Big) .
 \end{split}
\end{align}

For $l = 0,1,\dotsc,L-2$ define $\mathcal{U}_{l,0}$ to be the set of level-$\lfloor K/\theta \rfloor$ approximate squares for which~\eqref{eq:secondthreshold} holds for all $j \in \{1,2,\dotsc,l\}$ %
but not for $j=l+1$, and~\eqref{eq:firstthreshold} holds for all $j \in \{1,2,\dotsc,l\}$, %
 and~\eqref{eq:tsaverage} does not hold, and~\eqref{eq:secondpartverybig} does not hold, and~\eqref{eq:thirdthreshold} does not hold, and~\eqref{eq:lowerthresholds} holds for all $j \in \{ 1,2,\dotsc,L-l-2\}$: 
\begin{equation}\label{eq:lowerthresholds}
 \tau(\iih,  \lfloor K/(\gamma^{j}\theta) \rfloor, \lfloor K/(\gamma^{j-1}\theta) \rfloor )   <    t_j(s).
\end{equation}
Note that we imposed no restriction on $\ih_{\lfloor K/\theta \rfloor + 1},\dotsc,\ih_{\gamma(\lfloor K/\theta \rfloor)}$, and in the case $l=L-2$ we did not require the extra condition~\eqref{eq:lowerthresholds}. %
If $L=2$ then we have constructed the cover 
\[ \Lambda \subseteq \mathcal{U}_0 \cup \mathcal{U}_{0,0} \cup \mathcal{U}_{0,1}\cup \mathcal{U}_{0,1}'  \cup  \mathcal{U}_{0,2}\cup \mathcal{U}_{1,0} \cup \mathcal{U}_{1,1}.\] 

If $L >2$, then for $l = 0,1,\dotsc,L-3$ and $k = 1,2,\dotsc,L-l-2$
define $\mathcal{U}_{l,k}$ to be the set of level-$\lfloor K/(\gamma^k \theta) \rfloor$ approximate squares for which~\eqref{eq:secondthreshold} holds for all $j \in \{1,2,\dotsc,l\}$ %
 but not for $j=l+1$, and~\eqref{eq:firstthreshold} holds for all $j \in \{1,2,\dotsc,l\}$, %
 and~\eqref{eq:tsaverage} does not hold, and~\eqref{eq:secondpartverybig} does not hold, and~\eqref{eq:thirdthreshold} does not hold, and~\eqref{eq:lowerthresholds} holds for all $j \in \{ k+1,k+2,\dotsc, L-l-2 \}$ %
 but not for $j=k$. 
We have finally constructed a cover of~$\Lambda$: 
\begin{equation}\label{eq:definerealcover}
\Lambda \subseteq \mathcal{U}_{L-1,0} \cup \mathcal{U}_{L-1,1}  \cup   \bigcup_{l = 0}^{L-2} \bigg(  \mathcal{U}_l \cup  \mathcal{U}_{l,L-l-1}' \cup \bigcup_{k=0}^{L-l} \mathcal{U}_{l,k}  \bigg). 
\end{equation}
 For simplicity, we denote the cover by $\{ V_j\}_j$. Observe that any two elements of this cover are either disjoint or intersect on their boundary; it can never happen that one is contained within the other. 
 Figure~\ref{fig:CoverHelp} depicts the different parts of the cover in the most complicated case, namely when $\gamma^{-L} < \theta < \gamma^{-(L-1)}$ for some natural number $L \geq 3$. 

We will bound the $s$-cost of this cover in Lemma~\ref{lem:mainuppercost}. For this, we need Lemmas~\ref{lem:simpleindcomb},~\ref{lem:inductioncombinatorics} and~\ref{lem:tripleindcomb}, which we prove using the method of types. 
The inequalities in Lemma~\ref{lem:inductioncombinatorics} mimic~\eqref{eq:secondthreshold} and~\eqref{eq:firstthreshold}. 

\begin{lemma}\label{lem:inductioncombinatorics}
Suppose $c \in (0,1)$, $\overline{t} > t_1 > t_2 > \underline{t}$ %
and $J \in \mathbb{N}$. Then as $J \to \infty$, 
\begin{enumerate}[label=(\roman*)]

\item\label{eq:inductioncombinatorics1}\mbox{}\vspace{-\baselineskip}%
\begin{align*}
\# &\{ \, \iih \in [M]^J : \tau(\iih,\lfloor cJ \rfloor ,J) \leq t_2, \, \tau(\iih,0,\lfloor cJ \rfloor) \geq t_1 + ((1-c)/c) ( t_2 - \tau(\iih,\lfloor cJ \rfloor,J) ) \, \} \\*
&\asymp e^{J(c(\log M - I(t_1)) + (1-c)(\log M - I(t_2)))};
\end{align*}

\item\label{eq:inductioncombinatorics2}\mbox{}\vspace{-\baselineskip}
\begin{align*}
\# &\{ \, \mathbf{i} \in [N]^{J} : \tau(\iih,\lfloor cJ \rfloor ,J) \leq t_2, \, \tau(\iih,0,\lfloor cJ \rfloor) \leq t_1 + ((1-c)/c) ( t_2 - \tau(\iih,\lfloor cJ \rfloor,J) ) \, \} \\*
&\asymp e^{J (c(t_1 + \log M - I(t_1)) + (1-c)(t_2 + \log M - I(t_2)))}. 
\end{align*}

\end{enumerate}
\end{lemma}

\begin{proof}

The lower bounds for the asymptotic growth follow from considering those strings for which $\ih_{\lfloor cJ \rfloor + 1}, \dotsc, \ih_J$ and $\ih_{1},\dotsc,\ih_{\lfloor cJ \rfloor}$ are the best approximations to $\mathbf{Q}^*_{t}$ and $\mathbf{Q}^*_{T_s(t)}$ respectively in $\mathcal{T}_{J-\lfloor cJ \rfloor}$ and $\mathcal{T}_{\lfloor cJ \rfloor}$ for which the required inequalities hold. 
The strategy for the upper bounds is to fix arbitrary type classes for the different parts of the string which satisfy the desired inequalities and then use the fact that there are only polynomially many type classes. 

\ref{eq:inductioncombinatorics1}. Fix $\mathbf{p} \in \mathcal{T}_{\lfloor cJ \rfloor}$ and $\mathbf{q} \in \mathcal{T}_{J - \lfloor cJ \rfloor}$ such that $t_{\mathbf{q}} \leq t_2$ and $t_{\mathbf{p}} \geq t_1 + ((1-c)/c) ( t_2 - t_{\mathbf{q}} )$, recalling that $t_{\mathbf{p}} = \sum_{\ih} p_{\ih} \log N_{\ih}$. Then 
\begin{align}
\# T_{\lfloor cJ \rfloor}(\mathbf{p}) \cdot \# T_{J - \lfloor cJ \rfloor}(\mathbf{q}) &\leq e^{J( c(\log M - I(t_{\mathbf{p}}))    +   (1-c) (\log M - I(\max\{t_{\mathbf{q}},\underline{t}\}))      )} \label{eq:alignindcomb1} \\
&\leq e^{J( c(\log M - I(t_1 + ((1-c)/c) ( t_2 - \max\{t_{\mathbf{q}},\underline{t}\} )))    +   (1-c) (\log M - I(\max\{t_{\mathbf{q}},\underline{t}\}))      )} \label{eq:alignindcomb2} \\ 
&\leq e^{J(c(\log M - I(t_1)) + (1-c)(\log M - I(t_2)))}. \label{eq:alignindcomb3}
\end{align}
In~\eqref{eq:alignindcomb1} we used~\eqref{eq:206} and Step~1 of Proposition~\ref{prop:1}. 
In~\eqref{eq:alignindcomb2} we used that the rate function is increasing. 
In~\eqref{eq:alignindcomb3} we used the convexity of the rate function. 
Therefore using~\eqref{eq:205} we can bound the cardinality of the set in the statement of~\ref{eq:inductioncombinatorics1} from above by 
\[(\lfloor cJ \rfloor + 1 )^M (J-\lfloor cJ \rfloor +1)^M e^{J(c(\log M - I(t_1)) + (1-c)(\log M - I(t_2)))}. \]

\ref{eq:inductioncombinatorics2}. Similarly, fix $\mathbf{p} \in \mathcal{T}_{\lfloor cJ \rfloor}$ and $\mathbf{q} \in \mathcal{T}_{J - \lfloor cJ \rfloor}$ such that $t_{\mathbf{q}} \leq t_2$ and $t_{\mathbf{p}} \leq t_1 + {((1-c)/c) ( t_2 - t_{\mathbf{q}} )}$. 
Then 
\begin{align} 
&\# \{ \, \mathbf{i} \in [N]^{\lfloor cJ \rfloor} : \iih \in T_{\lfloor cJ \rfloor}(\mathbf{p}) \, \} \cdot \# \{ \,  \mathbf{j} \in [N]^{J- \lfloor cJ \rfloor} : \jjh \in T_{J- \lfloor cJ \rfloor}(\mathbf{q}) \, \} \nonumber \\
&\leq e^{J(c(\max\{t_{\mathbf{p}},\underline{t}\} + \log M - I(\max\{t_{\mathbf{p}},\underline{t}\})) + (1-c)(\max\{t_{\mathbf{q}},\underline{t}\} + \log M - I(\max\{t_{\mathbf{q}},\underline{t}\})))} \nonumber \\
&\leq e^{ J c(\min\{t_1 + \frac{1-c}{c} ( t_2 - \max\{t_{\mathbf{q}},\underline{t}\} ),\overline{t} \} + \log M -  I(\min\{t_1 + \frac{1-c}{c} ( t_2 - \max\{t_{\mathbf{q}},\underline{t}\} ),\overline{t} \}) )} \nonumber \\*
&\phantom{\leq}\times e^{J (1-c)( \max\{t_{\mathbf{q}},\underline{t}, t_2 - \frac{c}{1-c} (\overline{t} - t_1)\} + \log M - I( \max\{t_{\mathbf{q}},\underline{t}, t_2 - \frac{c}{1-c} (\overline{t} - t_1) \} )  )  } \label{eq:bothlower2} \\
&\leq e^{J (c(t_1 + \log M - I(t_1)) + (1-c)(t_2 + \log M - I(t_2)))}. \label{eq:bothlower3}
\end{align}
In~\eqref{eq:bothlower2} we used the fact that $I'(t) < 1$ if $t \in (\underline{t},\overline{t})$ and $I'(t) > 1$ if $t \in (\overline{t},\max_{1\leq i \leq M} \log N_i)$. In~\eqref{eq:bothlower3} we used the convexity of the rate function. 
In light of~\eqref{eq:205}, the result follows. 
\end{proof}

The inequalities in Lemma~\ref{lem:tripleindcomb} mimic~\eqref{eq:secondthreshold},~\eqref{eq:secondpartverybig},~\eqref{eq:tsaverage} and~\eqref{eq:thirdthreshold}. 

\begin{lemma}\label{lem:tripleindcomb}
Suppose $s \in (\dim_{\mathrm H} \Lambda,\dim_{\mathrm B} \Lambda)$, $c \in (0,1)$, $\underline{t} < t < T_s(t) < \overline{t}$ and $J \in \mathbb{N}$. Then as $J \to \infty$, 
\begin{enumerate}[label=(\roman*)]
\item\label{eq:tripleindcomb1}\mbox{}\vspace{-\baselineskip}
\begin{align*}
\# &\{ \, \iih \in [M]^J  :     \tau(\iih,\lfloor cJ \rfloor, J )  \geq  t \\
&\mbox{ and }  \tau(\iih,0,\lfloor cJ \rfloor)    \geq  T_s  (  \max\{ \underline{t},     ( 1 - c + \gamma c ) t / (\gamma c)    -    (1-c) \tau(\iih,\lfloor cJ \rfloor, J ) / (\gamma c)    \}         )  \, \}   \\
&\asymp   e^{ J( c( \log M - I(T_s(t)) )   +   (1-c) ( \log M - I(t) )     ) }    ;
\end{align*}

\item\label{eq:tripleindcomb2}\mbox{}\vspace{-\baselineskip}
\begin{align*}
\# \{ \, &(i_1,\dotsc,i_{\lfloor cJ \rfloor} , \ih_{\lfloor cJ \rfloor + 1},\dotsc, \ih_J, \ih_{J+1},\dotsc, \ih_{\lfloor (1+ \gamma c) J \rfloor}  \in [N]^{\lfloor cJ \rfloor}  \times [M]^{\lfloor (1 + \gamma c) J \rfloor - \lfloor cJ \rfloor} ) : \\
&t \leq \tau(\iih,\lfloor cJ \rfloor, J  ) \leq ( (1-c+\gamma c) t  -   \gamma c \underline{t} )/(1-c) \\
&\mbox{ and } \tau(\iih,0,\lfloor cJ \rfloor)  \leq  T_s( (  ( 1 - c + \gamma c ) t   -    (1-c) \tau(\iih,\lfloor cJ \rfloor, J )) / (\gamma c) ) \\
&\mbox{ and }   \tau(\iih,J, \lfloor (1 + \gamma c) J \rfloor  ) \geq  (( 1 - c + \gamma c ) t    -    (1-c) \tau(\iih,\lfloor cJ \rfloor, J ) )/ (\gamma c)      \, \} \\
&\asymp   e^{J(c(T_s(t) + \log M - I(T_s(t))) + (1-c + \gamma c) (\log M - I(t))   )};
 \end{align*}

\item\label{eq:tripleindcomb3}\mbox{}\vspace{-\baselineskip}
\begin{align*}
\# &\{ \, \mathbf{i}  \in [N]^{\lfloor (1 + \gamma c) J\rfloor }  : t \leq \tau(\iih,\lfloor cJ \rfloor, J  ) \leq ( (1-c+\gamma c) t  -   \gamma c \underline{t}  )/(1-c) \\
&\mbox{ and } \tau(\iih,0,\lfloor cJ \rfloor)  \leq  T_s( ( ( 1 - c + \gamma c ) t   -    (1-c) \tau(\iih,\lfloor cJ \rfloor, J )) / (\gamma c) ) \\
&\mbox{ and }   \tau(\iih,J, \lfloor (\gamma c + 1)J \rfloor  ) \leq  (( 1 - c + \gamma c ) t  -    (1-c) \tau(\iih,\lfloor cJ \rfloor, J )) / (\gamma c)      \, \} \\
&\asymp   e^{J(c(T_s(t) + \log M - I(T_s(t))) + (1-c + \gamma c) (t + \log M - I(t))   )}.
\end{align*}

\end{enumerate}
\end{lemma}

\begin{proof}
The proof strategy is rather similar to that of Lemma~\ref{lem:inductioncombinatorics}. The lower bounds follow from considering those strings for which $\ih_{\lfloor cJ \rfloor + 1}, \dotsc, \ih_J$ and $\ih_{1},\dotsc,\ih_{\lfloor cJ \rfloor}$ are the best approximations to $\mathbf{Q}^*_{t}$ and $\mathbf{Q}^*_{T_s(t)}$ respectively in $\mathcal{T}_{J-\lfloor cJ \rfloor}$ and $\mathcal{T}_{\lfloor cJ \rfloor}$, and (for~\ref{eq:tripleindcomb2} and~\ref{eq:tripleindcomb3}) $\ih_{J+1},\dotsc,\ih_{\lfloor \gamma c J \rfloor}$ is the best approximation to $\mathbf{Q}^*_{t}$ in $\mathcal{T}_{\lfloor (1 + \gamma c) J \rfloor - J}$, for which the required inequalities hold. 
The upper bounds follow from the following estimates and~\eqref{eq:205}. 

 \ref{eq:tripleindcomb1}. Fix $\mathbf{p} \in \mathcal{T}_{\lfloor cJ \rfloor}$ and $\mathbf{q} \in \mathcal{T}_{J - \lfloor cJ \rfloor}$ such that $t_{\mathbf{q}} \geq t$ and 
 \[ t_{\mathbf{p}} \geq T_s  (  \max\{ \underline{t},     ( 1 - c + \gamma c ) t / (\gamma c)    -    (1-c) t_{\mathbf{q}} / (\gamma c).\]
  Then 
 \begin{align*} 
\# T_{\lfloor cJ \rfloor}(\mathbf{p}) \cdot \# T_{J- \lfloor cJ \rfloor}(\mathbf{q}) &\leq e^{Jc(\log M -  I( T_s  (  \max\{ \underline{t},     ( 1 - c + \gamma c ) t / (\gamma c)  -  (1-c) t_{\mathbf{q}} / (\gamma c) \}  )) )}  \\*
&\phantom{\leq}\times e^{J(1-c) (\log M -  I(\min\{t_{\mathbf{q}},( (1-c+\gamma c) t  -   \gamma c \underline{t} )/(1-c)\})    )  )} \\
&\leq e^{ J( c( \log M - I(T_s(t)) )   +   (1-c) ( \log M - I(t) )     ) }.  
\end{align*}
The final step holds since 
\begin{align*}
 ( (1-c+\gamma c) t  -   \gamma c \underline{t} )/(1-c)  &\leq t \leq t_{\mathbf{q}}; \\*
 T_s\big(( (1-c+\gamma c) t  -   \gamma c \underline{t} )/(1-c)\big)  &\leq  T_s(t) < \overline{t}, 
\end{align*}
so using standard properties of the rate function, the derivative of the exponent with respect to $t_{\mathbf{q}}$ is negative.

\ref{eq:tripleindcomb2}. Now fix $\mathbf{p} \in \mathcal{T}_{\lfloor cJ \rfloor}$, $\mathbf{q} \in \mathcal{T}_{J - \lfloor cJ \rfloor}$ and $\mathbf{r} \in \mathcal{T}_{\lfloor (1 + \gamma c) J \rfloor - J}$ such that $t \leq t_{\mathbf{q}} \leq ( (1-c+\gamma c) t  -  \gamma c \underline{t} )/(1-c)$, $t_{\mathbf{p}} \leq  T_s( (  ( 1 - c + \gamma c ) t   -    (1-c) t_{\mathbf{q}}) / (\gamma c) )$ and $t_{\mathbf{r}} \geq (( 1 - c + \gamma c ) t    -    (1-c) t_{\mathbf{q}} )/ (\gamma c)$. Then 
\begin{align*}
&\# \{ \, \mathbf{i} \in [N]^{\lfloor cJ \rfloor} : \iih \in T_{\lfloor cJ \rfloor}(\mathbf{p}) \, \} \cdot \# T_{J- \lfloor cJ \rfloor}(\mathbf{q}) \cdot \# T_{\lfloor (1 + \gamma c) J \rfloor - J}(\mathbf{r}) \\
&\leq  e^{J( c(T_s( (  ( 1 - c + \gamma c ) t   -    (1-c) t_{\mathbf{q}}) / (\gamma c) ) + \log M - I(T_s( (  ( 1 - c + \gamma c ) t   -    (1-c) t_{\mathbf{q}}) / (\gamma c) )))                    +   (1-c) (\log M - I(t_{\mathbf{q}})) )} \\*
&\phantom{\leq}\times e^{J(  \gamma c (\log M  -  I((( 1 - c + \gamma c ) t    -    (1-c) t_{\mathbf{q}} )/ (\gamma c)))        )}   \\
&\leq e^{J(c(T_s(t) + \log M - I(T_s(t))) + (1-c + \gamma c) (\log M - I(t))   )}
\end{align*}
since the derivative of the exponent with respect to $t_{\mathbf{q}}$ is negative. %

\ref{eq:tripleindcomb3}. Now fix $\mathbf{p} \in \mathcal{T}_{\lfloor cJ \rfloor}$, $\mathbf{q} \in \mathcal{T}_{J - \lfloor cJ \rfloor}$ and $\mathbf{r} \in \mathcal{T}_{\lfloor (1 + \gamma c) J \rfloor - J}$ such that $t \leq t_{\mathbf{q}} \leq ( (1-c+\gamma c) t  -  \gamma c \underline{t} )/(1-c)$, $t_{\mathbf{p}} \leq  T_s( (  ( 1 - c + \gamma c ) t   -    (1-c) t_{\mathbf{q}}) / (\gamma c) )$ and $t_{\mathbf{r}} \leq (( 1 - c + \gamma c ) t    -    (1-c) t_{\mathbf{q}} )/ (\gamma c)$. Then 
\begin{align}
 &\# \{ \, \mathbf{i} \in [N]^{\lfloor cJ \rfloor} : \iih \in T_{\lfloor cJ \rfloor}(\mathbf{p}) \, \} \cdot \# \{ \, \mathbf{j} \in [N]^{J - \lfloor cJ \rfloor} : \jjh \in T_{J - \lfloor cJ \rfloor}(\mathbf{q}) \, \} \nonumber \\*
 &\phantom{\leq}\times \# \{ \, \mathbf{k} \in [N]^{\lfloor (1 + \gamma c) J \rfloor - J} : \kkh \in T_{\lfloor (1 + \gamma c) J \rfloor - J}(\mathbf{r}) \, \} \nonumber \\
 &\leq e^{J c(T_s( (  ( 1 - c + \gamma c ) t   -    (1-c) t_{\mathbf{q}}) / (\gamma c) ) + \log M - I(T_s( (  ( 1 - c + \gamma c ) t   -    (1-c) t_{\mathbf{q}}) / (\gamma c) ))) } \\*
 &\phantom{\leq}\times e^{J(1-c) (\min\{t_{\mathbf{q}},\overline{t}\} + \log M - I(\min\{t_{\mathbf{q}},\overline{t}\})) } \nonumber \\*
 &\phantom{\leq}\times e^{J  \gamma c  ((( 1 - c + \gamma c ) t    -    (1-c) t_{\mathbf{q}} )/ (\gamma c) + \log M  -  I((( 1 - c + \gamma c ) t    -    (1-c) t_{\mathbf{q}} )/ (\gamma c))        )}  \nonumber \\
 &\leq e^{J   c(T_s( (  ( 1 - c + \gamma c ) t   -    (1-c) \min\{t_{\mathbf{q}},\overline{t}\}) / (\gamma c) ) + \log M - I(T_s( (  ( 1 - c + \gamma c ) t   -    (1-c) \min\{t_{\mathbf{q}},\overline{t}\}) / (\gamma c) )))                    }\nonumber \\*
 &\phantom{\leq}\times e^{J   (1-c) (\min\{t_{\mathbf{q}},\overline{t}\} + \log M - I(\min\{t_{\mathbf{q}},\overline{t}\})) } \nonumber \\*
 &\phantom{\leq}\times e^{J  \gamma c  ((( 1 - c + \gamma c ) t    -    (1-c) \min\{t_{\mathbf{q}},\overline{t}\} )/ (\gamma c) + \log M  -  I((( 1 - c + \gamma c ) t    -    (1-c) \min\{t_{\mathbf{q}},\overline{t}\} )/ (\gamma c))        )} \label{eq:aligntripleindcomb1} \\
 &\leq e^{J(c(T_s(t) + \log M - I(T_s(t))) + (1-c + \gamma c) (t + \log M - I(t))   )}. \label{eq:aligntripleindcomb2}
\end{align}
We have~\eqref{eq:aligntripleindcomb1} because $t \leq t_{\mathbf{q}} \leq ( (1-c+\gamma c) t  -  \gamma c \underline{t} )/(1-c)$, so 
\[ \underline{t} \leq (( 1 - c + \gamma c ) t  -  (1-c) t_{\mathbf{q}} )/ (\gamma c) \leq T_s\big((( 1 - c + \gamma c ) t    -    (1-c) t_{\mathbf{q}} )/ (\gamma c)\big)  \leq T_s(t) < \overline{t},\] 
so the derivative of the rate function here is between 0 and~1. 
We have~\eqref{eq:aligntripleindcomb2} since the derivative of the exponent in~\eqref{eq:aligntripleindcomb1} with respect to $t_{\mathbf{q}}$ is negative. %
In light of~\eqref{eq:205}, this completes the proof. 
\end{proof}

We are now ready to prove an upper bound for the $s$-cost of the cover that we have constructed. Recall that $G(\theta,s)$ is defined in~\eqref{eq:definemainquantity}. 
\begin{lemma}\label{lem:mainuppercost}
For $L \in \mathbb{N}$, $\theta \in (\gamma^{-L},\gamma^{-(L-1)})$, $s \in (\dim_{\mathrm H} \Lambda,\dim_{\gamma^{-(L-1)}} \Lambda]$ and $0 < \delta \ll 1$, let $\{V_j\}_j$ be the cover of $\Lambda$ defined in~\eqref{eq:definerealcover}. Then as $K = K(\delta) \to \infty$, 
\[ \sum_j |V_j|^s \asymp e^{K \cdot G(\theta,s)/(\gamma^L \theta)}. \]
\end{lemma}
\begin{proof}
The strategy is to bound the $s$-costs of the different parts of the cover separately using Lemmas~\ref{lem:simpleindcomb},~\ref{lem:inductioncombinatorics} and~\ref{lem:tripleindcomb}, which can be applied since the $t_i(s)$ lie in the right range by Lemma~\ref{lem:51}. 
We use the convention that an empty product equals 1. 
In the following, as above, $R \in \mathbb{N}$ will be fixed and arbitrary. 
We first consider the $s$-cost of $\mathcal{U}_{l}$ for $l \in \{0,1,\dotsc,L-2\}$: %
\begin{align*}
\sum_{U \in \mathcal{U}_{l}} |U|^s &\asymp \# \mathcal{U}_{l} \cdot  n^{-\gamma^{l}K s} \\
  &\asymp N^K \prod_{j = 1}^{l} e^{K(t_{L-j}(s) + \log M - I(t_{L-j}(s)))(\gamma^j - \gamma^{-(L-j)}\theta^{-1})} \\*
  &\phantom{\leq}\times e^{K(t_{L-j+1}(s) + \log M - I(t_{L-j+1}(s)) ) (\gamma^{-(L-j)}\theta^{-1} - \gamma^{j-1}) ) }  \\*
&\phantom{\leq}\times e^{K((\log M - I(t_{L-l-1}(s)))(\gamma^{l+1} - \gamma^{-(L-l-1)}\theta^{-1})  +  (\log M - I(t_{L-l}))  (\gamma^{-(L-l-1)}\theta^{-1} - \gamma^{l})   )}  \\*
&\phantom{\leq}\times n^{-\gamma^{l}K s}  \qquad \qquad \text{by Lemma~\ref{lem:inductioncombinatorics}~\ref{eq:inductioncombinatorics1} and~\ref{eq:inductioncombinatorics2}} \\
&\asymp \# (\mathcal{B}_{\gamma^{l}(K)} \cap \mathrm{supp}(\mu_{\delta,s,\theta,R})) \cdot n^{-\gamma^{l}K s} \qquad \substack{\text{ by~\eqref{eq:lowerboundcard}} \\ \text{and Proposition~\ref{prop:1}, Step~1}} \\ 
&\asymp \mu_{\delta,s,\theta,R} (\Lambda) \qquad  \qquad \text{by~\eqref{eq:definemu} and Lemma~\ref{lem:lowernicescales}}  \\
&\asymp e^{K \cdot G(\theta,s)/(\gamma^L \theta)} \qquad \text{by Lemma~\ref{lem:lowertotalmass}}. \stepcounter{equation}\tag{\theequation}\label{eq:mainupperfirstcost}
\end{align*}

The $s$-costs of $\mathcal{U}_{l,0}$ are equal for all $l \in \{0,1,\dotsc,L-1\}$: 
\begin{align*}
\sum_{U \in \mathcal{U}_{l,0}} |U|^s &\asymp \# \mathcal{U}_{L-1,0} \cdot n^{-K s/\theta}  \\
&\asymp  \prod_{j = 1}^{L-1} e^{K(t_{L-j}(s) + \log M - I(t_{L-j}(s)))(\gamma^j - \gamma^{-(L-j)}\theta^{-1})} \\*
&\phantom{\leq}\times e^{K(t_{L-j+1}(s) + \log M - I(t_{L-j+1}(s)) ) (\gamma^{-(L-j)}\theta^{-1} - \gamma^{j-1})  )   }  \\*
&\phantom{\leq}\times N^K e^{(t_1(s) + \log M - I(t_1(s)) )(1/\theta - \gamma^{L-1}) K }  M^{(\gamma - 1)K/\theta}   \\*
&\phantom{\leq}\times n^{-K s/\theta} \qquad \substack{\text{by Lemmas~\ref{lem:simpleindcomb},~\ref{lem:inductioncombinatorics}~\ref{eq:inductioncombinatorics2}} \\ \text{and (when }l<L-1)\text{ Lemma~\ref{lem:tripleindcomb}~\ref{eq:tripleindcomb3}} } \\
&\asymp \# (\mathcal{B}_{\lfloor K/\theta\rfloor} \cap \mathrm{supp}(\mu_{\delta,s,\theta,R})) \cdot n^{-Ks/\theta} \asymp \mu_{\delta,s,\theta,R} (\Lambda) \asymp e^{K \cdot G(\theta,s)/(\gamma^L \theta)}. 
\end{align*}

To bound the $s$-cost of $\mathcal{U}_{l,L-l}$ when $l \in \{0,1,\dotsc,L-2\}$, note that by Lemma~\ref{lem:33-first} and Step~1 of Proposition~\ref{prop:1}, 
\begin{align*}
\# &\Big\{ \, \iih \in [M]^{\gamma^{l+1}(K) - \gamma^l(K)}  :  \tau(\iih,\lfloor K/(\gamma^{L-l-1}\theta) \rfloor - \gamma^l(K),\gamma^{l+1}(K) -   \gamma^l(K) ) \\
&\phantom{--}\geq \frac{1}{(\gamma^{l+1}  -  (\gamma^{L-l-1}\theta)^{-1}  )}   \Big(  \Big(  \frac{1}{\gamma^{L-l-2}\theta}  - \frac{1}{\gamma^{L-l-1} \theta} \Big)  t_{L-l-1}(s)  -   ((\gamma^{L-l-2}\theta)^{-1}  -  \gamma^{l+1}) \underline{t}  \Big)  \\
&\phantom{--}\mbox{ and } \tau( \iih, 0,   \lfloor K/(\gamma^{L-l-1}\theta) \rfloor - \gamma^l(K)  )  \geq  T_s(\underline{t}) \, \Big\} \\
&\asymp  e^{-K (\gamma^{l+1}  -  \gamma^{l+1-L}\theta ) I\left(\frac{1}{(\gamma^{l+1}  -  (\gamma^{L-l-1}\theta)^{-1}    )}   \big(  \big(  \frac{1}{\gamma^{L-l-2}\theta}  - \frac{1}{\gamma^{L-l-1} \theta} \big)  t_{L-l-1}(s) - ((\gamma^{L-l-2}\theta)^{-1}  -  \gamma^{l+1}) \underline{t}  \big)\right)  }\\*
&\phantom{\leq}\times e^{K \gamma^{l+1}  -  \gamma^{l+1-L}\theta ) \log M }\cdot e^{K \left(\gamma^{l+1-L}\theta   -  \gamma^l)  (\log M - I(T_s(\underline{t})))   \right)} \\
&\leq  e^{K\left(   (\gamma^{l+1}  -  \gamma^{l+1-L}\theta^{-1}  )\left( \log M - I(t_{L-l-1}(s))    \right)   +  (\gamma^{l+1-L}\theta^{-1}  -  \gamma^l)  (\log M - I(t_{L-l}(s)))   \right)},\stepcounter{equation}\tag{\theequation}\label{eq:combwithincostproof}
\end{align*}
where the last step follows from the case $t_{\mathbf{q}} = ( (1-c+\gamma c) t  -   \gamma c \underline{t} )/(1-c)$ of Lemma~\ref{lem:tripleindcomb}~\ref{eq:tripleindcomb1}. 

Now, for $l \in \{0,1,\dotsc,L-1\}$, %
\begin{align*}
\sum_{U \in \mathcal{U}_{l,L-l}} |U|^s &\asymp  N^K \prod_{j = 1}^{l} e^{K(t_{L-j}(s) + \log M - I(t_{L-j}(s)))(\gamma^j - \gamma^{-(L-j)}\theta^{-1})} \\*
&\phantom{\leq}\times e^{K(t_{L-j+1}(s) + \log M - I(t_{L-j+1}(s)) ) (\gamma^{-(L-j)}\theta^{-1} - \gamma^{j-1})  )   }  \\*
&\phantom{\leq}\times e^{K((\log M - I(t_{L-l-1}(s)))(\gamma^{l+1} - \gamma^{-(L-l-1)}\theta^{-1})  +  (\log M - I(t_{L-l}))  (\gamma^{-(L-l-1)}\theta^{-1} - \gamma^{l})   )}  \\*
&\phantom{\leq}\times n^{-\gamma^{l}K s}  \\
&\asymp \# (\mathcal{B}_{\gamma^{l}(K)} \cap \mathrm{supp}(\mu_{\delta,s,\theta,R})) \cdot n^{-\gamma^{l}K s} \asymp \mu_{\delta,s,\theta,R} (\Lambda) \asymp e^{K \cdot G(\theta,s)/(\gamma^L \theta)}. 
\end{align*}
In the case $l=L-1$ we used Lemma~\ref{lem:inductioncombinatorics}~\ref{eq:inductioncombinatorics2} and Lemma~\ref{lem:33-first} and Step~1 of Proposition~\ref{prop:1}. 
In the case $l<L-1$ we used Lemma~\ref{lem:inductioncombinatorics}~\ref{eq:inductioncombinatorics2} and (in the case when the maximum in~\eqref{eq:tsaverage} does not take the value $\underline{t}$, or equivalently when~\eqref{eq:secondpartverybig} does not hold) Lemma~\ref{lem:tripleindcomb}~\ref{eq:tripleindcomb1}, and (in the case when the maximum does take the value $\underline{t}$) we used~\eqref{eq:combwithincostproof}. 

Now, for $l \in \{0,1,\dotsc,L-2\}$, 
 \begin{align*}
\sum_{U \in \mathcal{U}_{l,L-l-1}'} |U|^s &\asymp N^K \prod_{j = 1}^{l} e^{K(t_{L-j}(s) + \log M - I(t_{L-j}(s)))(\gamma^j - \gamma^{-(L-j)}\theta^{-1})} \\*
&\phantom{\leq}\times e^{K(t_{L-j+1}(s) + \log M - I(t_{L-j+1}(s)) ) (\gamma^{-(L-j)}\theta^{-1} - \gamma^{j-1})   } \cdot e^{K(\gamma^{l+1} - \gamma^{l+1-L}\theta^{-1} )}  \\*
&\hspace{-2.5cm}\phantom{\leq}\times e^{-K (\gamma^{l+1} - \gamma^{l+1-L}\theta^{-1} ) I\left(  \frac{1}{\gamma^{l+1}  -  (\gamma^{L-l-1}\theta)^{-1}  }   \big(  \big(  \frac{1}{\gamma^{L-l-2}\theta}  - \frac{1}{\gamma^{L-l-1} \theta} \big)  t_{L-l-1}(s)    -   ((\gamma^{L-l-2}\theta)^{-1}  -  \gamma^{l+1}) \underline{t}  \big)  \right)   } \\*
&\phantom{\leq}\times e^{K(\gamma^{l+1-L}\theta^{-1}  -  \gamma^l) ( T_s(\underline{t})  +  \log M - I(T_s(\underline{t})) ) } \cdot  M^{K(\gamma^{l+2-L}\theta^{-1} - \gamma^{l+1})}   \cdot n^{K/(\gamma^{L-l-1}\theta)} \\
&\leq N^K \prod_{j = 1}^{l} e^{K(t_{L-j}(s) + \log M - I(t_{L-j}(s)))(\gamma^j - \gamma^{-(L-j)}\theta^{-1})} \\*
&\phantom{\leq}\times e^{K(t_{L-j+1}(s) + \log M - I(t_{L-j+1}(s)) ) (\gamma^{-(L-j)}\theta^{-1} - \gamma^{j-1})     }   \\*
&\phantom{\leq}\times e^{ K  (\gamma^{l+2-L}\theta^{-1} - \gamma^{l+1-L}\theta^{-1} ) (\log M - I(t_{L-l-1}(s))} \\*
&\phantom{\leq}\times e^{K(\gamma^{l+1-L}\theta^{-1}  -  \gamma^l) ( t_{L-l}(s)  +  \log M - I(t_{L-l}(s))) )  } \cdot  n^{K/(\gamma^{L-l-1}\theta)} \\
&\asymp  \# (\mathcal{B}_{\lfloor K/(\gamma^{L-l-1}\theta) \rfloor}  \cap \mathrm{supp}(\mu_{\delta,s,\theta,R})) \cdot n^{K/(\gamma^{L-l-1}\theta)} \asymp \mu_{\delta,s,\theta,R} (\Lambda) \\
&\asymp e^{K \cdot G(\theta,s)/(\gamma^L \theta)}, 
\end{align*}
where the inequality follows from the case $t_{\mathbf{q}} = ( (1-c+\gamma c) t  -   \gamma c \underline{t} )/(1-c)$ of the proof of Lemma~\ref{lem:tripleindcomb}~\ref{eq:tripleindcomb2}. 

For $l \in \{0,1,\dotsc,L-2\}$, using Lemma~\ref{lem:tripleindcomb}~\ref{eq:tripleindcomb2} (and Lemma~\ref{lem:inductioncombinatorics}~\ref{eq:inductioncombinatorics2}), 
\begin{align*}
\sum_{U \in \mathcal{U}_{l,L-l-1}} |U|^s 
&\asymp N^K \prod_{j = 1}^{l} e^{K(t_{L-j}(s) + \log M - I(t_{L-j}(s)))(\gamma^j - \gamma^{-(L-j)}\theta^{-1})} \\*
&\phantom{\leq}\times e^{K(t_{L-j+1}(s) + \log M - I(t_{L-j+1}(s)) ) (\gamma^{-(L-j)}\theta^{-1} - \gamma^{j-1})    }  \\*
&\phantom{\leq}\times e^{ K (\gamma^{l+2-L}\theta^{-1} - \gamma^{l+1-L}\theta^{-1} ) (\log M - I(t_{L-l-1}(s))} \\*
&\phantom{\leq}\times e^{K(\gamma^{l+1-L}\theta^{-1}  -  \gamma^l) ( t_{L-l}(s)  +  \log M - I(t_{L-l}(s))) )  } n^{K/(\gamma^{L-l-1}\theta)} \\
&\asymp  \# (\mathcal{B}_{\lfloor K/(\gamma^{L-l-1}\theta) \rfloor}  \cap \mathrm{supp}(\mu_{\delta,s,\theta,R})) \cdot n^{K/(\gamma^{L-l-1}\theta)} \\
&\asymp \mu_{\delta,s,\theta,R} (\Lambda) \asymp e^{K \cdot G(\theta,s)/(\gamma^L \theta)}.
\end{align*}

Finally, if $L \geq 3$, $l \in \{0,1,\dotsc,L-3\}$ and $k \in \{1,2,\dotsc, L-l-2\}$, 
\begin{align*}
\sum_{U \in \mathcal{U}_{l,k}} |U|^s 
&\asymp N^K \prod_{j = 1}^{L-k-1} e^{K(t_{L-j}(s) + \log M - I(t_{L-j}(s)))(\gamma^j - \gamma^{-(L-j)}\theta^{-1})}\\*
&\phantom{\leq}\times e^{K(t_{L-j+1}(s) + \log M - I(t_{L-j+1}(s)) ) (\gamma^{-(L-j)}\theta^{-1} - \gamma^{j-1})  }   \\*
&\phantom{\leq}\times e^{K(\gamma^{-k}\theta^{-1}  -  \gamma^{L-k-1})(t_{L-k-1}(s)  +  \log M - I(t_{L-k-1})) } \\*
&\phantom{\leq}\times e^{K(\gamma^{-(k+1)}\theta^{-1} - \gamma^{-k}\theta^{-1})(\log M - I(t_{L-k-2}(s))) } \cdot n^{-Ks/(\gamma^k \theta)} \\
&\asymp \# \mathcal{B}_{\lfloor K/(\gamma^k \theta)\rfloor} \cap \mathrm{supp}(\mu_{\delta,s,\theta,R}) \cdot n^{-Ks/(\gamma^k \theta)} \asymp \mu_{\delta,s,\theta,R} (\Lambda) \asymp e^{K \cdot G(\theta,s)/(\gamma^L \theta)},
\end{align*}
using Lemma~\ref{lem:tripleindcomb}~\ref{eq:tripleindcomb3} (and Lemma~\ref{lem:inductioncombinatorics}~\ref{eq:inductioncombinatorics2} and Lemmas~\ref{lem:simpleindcomb},~\ref{lem:33-first} and~\ref{lem:lowertotalmass}). 
We have now bounded the $s$-cost of each part of the cover, so the proof is complete.  
\end{proof}

Note that when we applied Lemma~\ref{lem:inductioncombinatorics}, in the proof of Lemma~\ref{lem:mainuppercost} (for example in~\eqref{eq:mainupperfirstcost}), we used that $t_L(s) < \overline{t}$ for all $s \in [\dim_{\gamma^{-L}} \Lambda,\dim_{\gamma^{-(L-1)}} \Lambda]$, see Lemma~\ref{lem:51}. 
We needed Section~\ref{sec:upperinteger} to establish the inequality $t_L(s) < \overline{t}$ before Section~\ref{sec:mainupper}.

\subsection{Conclusion and discussion of the proof}\label{sec:proofconclusion} 

We now conclude the proof of Theorem~\ref{thm:main} by combining the upper bounds in Sections~\ref{sec:upperinteger} and~\ref{sec:mainupper} and a lower bound from the mass distribution principle Proposition~\ref{prop:mdp} on page~\pageref{prop:mdp}, which can be applied by virtue of the results in Section~\ref{sec:prooflower}.

\begin{proof}[Proof of Theorem~\ref{thm:main}]%
Fix $\theta \in (0,1)$ and $s \in (\dim_{\mathrm{H}} \Lambda,\overline{\dim}_{\gamma^{-(L-1)}} \Lambda]$. Then 
\[ \limsup_{\delta\searrow 0} \frac{\log S_{\delta, \theta}^{s}(\Lambda)}{-\log \delta} \leq \frac{G(\theta,s)}{\gamma^L \theta \log n}.\]
This follows from Lemma~\ref{lem:51} and Lemma~\ref{lem:50} with $\btau = (t_1(s),\dotsc, t_{L-1}(s))$ in the case $\theta = \gamma^{-(L-1)}$, and from Lemma~\ref{lem:mainuppercost} in the case $\gamma^{-L} < \theta < \gamma^{-(L-1)}$. 
If $U \subset \mathbb{R}^2$ is Borel with $|U| <1$ then the number of approximate squares of level $\lceil \frac{-\log |U|}{\log n} \rceil$ which $U$ intersects is at most an absolute constant depending only on the carpet. Therefore by Lemma~\ref{lem:lowerallscales}, for all $s' < s$ there exists $\delta_0 > 0$ and $R \in \mathbb{N}$ such that for all $\delta \in (0,\delta_0)$, if $\delta^{1/\theta} \leq |U| \leq \delta$ then $\mu_{\delta,s,\theta,R}(U) < |U|^{s'}$.  This means that we can use Lemma~\ref{lem:lowertotalmass} and apply the mass distribution principle Proposition~\ref{prop:mdp} and deduce that 
\[ \liminf_{\delta\searrow 0} \frac{\log S_{\delta, \theta}^{s'}(\Lambda)}{-\log \delta} \geq \frac{G(\theta,s)}{\gamma^L \theta \log n}.\] %
Since $s' < s$ was arbitrary and $\liminf_{\delta\searrow 0} \frac{\log S_{\delta, \theta}^{s'}(\Lambda)}{-\log \delta}$ is continuous in $s'$ by \cite[Lemma~2.1]{Burrell2021projections}, $\liminf_{\delta\searrow 0} \frac{\log S_{\delta, \theta}^{s}(\Lambda)}{-\log \delta} \geq G(\theta,s)/(\gamma^L \theta \log n)$. Thus 
\[ \frac{\log S_{\delta, \theta}^{s}(\Lambda)}{-\log \delta} \to \frac{G(\theta,s)}{\gamma^L \theta \log n} \mbox{ as } \delta \searrow 0. \]
Therefore by~\eqref{eq:45}, $\dim_{\theta} \Lambda$ exists and $G(\theta,\dim_{\theta} \Lambda) = 0$. 
By induction $t_1(s),t_2(s),\dotsc,t_L(s)$ are strictly increasing functions of $s$. Thus for fixed $\theta$, $G(\theta,s)$ is strictly decreasing in $s$, so $s=\dim_{\theta} \Lambda$ is the only solution $s \in [\dim_{\mathrm{H}} \Lambda, \dim_{\mathrm{B}} \Lambda]$ to the equation $G(\theta,s) = 0$. 
\end{proof}

\begin{rem}\label{rem:pressureexp}
The significance of the pressure function can be illustrated by the simple case $\theta \geq \gamma^{-1}$. Indeed, in this case the optimal cover which gives the smallest possible $s$-cost (up to absolute multiplicative constants depending only on the carpet) involves subdividing a level-$K$ approximate square to a level $k \in \{K,K+1,\dotsc,\lfloor K/\theta \rfloor \}$ which minimises $\psi_{(\ih_{K+1},\dotsc,\ih_{\lfloor K/\theta \rfloor}) | k}(s)$. By considering the symbolic representation of the approximate squares in such a cover, the $s$-cost is, up to multiplicative constants, 
\begin{align*}
 N^K \Psi_{\lfloor K/\theta \rfloor - K}(s) M^{\gamma (K) -\lfloor K/\theta \rfloor} n^{-Ks} &\asymp e^{K( \log N + (\theta^{-1} - 1)P(s) - \theta^{-1} \log M - s\log n)} \\*
 &\asymp e^{K( (\dim_{\mathrm B} \Lambda) \log n - (\theta^{-1} - 1)I(t) - s\log n )}
 \end{align*}
by Proposition~\ref{prop:1}, where $(s,t)$ are related by~\eqref{eq:23}. 
Therefore the exponential growth rate of this $s$-cost is in fact the same as that of the cover constructed in Section~\ref{sec:mainupper} using just the two extreme scales $K$ and $\lfloor K/\theta \rfloor$. 
\end{rem}

\section{Proof of corollaries and applications}\label{sec:proofcorollaries}

In this section we prove the corollaries and consequences of Theorem~\ref{thm:main}. 

\subsection{Proof of Corollary \ref{cor:allprop}}\label{subsec:proofform}

\begin{proof}[Proof of part~\ref{itemi}]
For 
\[ (\theta, s) \in (\gamma^{-L},\gamma^{-(L-1)}) \times (\dim_{\mathrm H} \Lambda, \dim_{\mathrm B} \Lambda) \coloneqq D \subset \mathbb{R}^2,\]
define $G(\theta,s)$ by~\eqref{eq:definemainquantity}. Then $G(\theta,s)$ has continuous partial derivatives of all orders, so is $C^\infty$, on $D$. Moreover, the rate function $I$ is analytic (as the Legendre transform of an analytic function) and compositions of analytic functions are analytic. It follows that for all $(\theta,s) \in D$ and $(\theta_1,s_1) \in \mathbb{R}^2$ there exists $r = r(\theta,s,\theta_1,s_1) > 0$ such that the function $\lambda \mapsto G(\theta+\lambda \theta_1,s+\lambda s_1)$ is real analytic for $\lambda \in (-r,r)$. Therefore by a result of Siciak \cite[Theorem~1]{Siciak1970analytic}, $G(\theta,s)$ is jointly analytic in $(\theta,s) \in D$. Thus by the analytic implicit function theorem, the function $\theta \mapsto \dim_{\theta} \Lambda$ (describing the zero set of $G(\theta,s)$) is analytic for $\theta \in (\gamma^{-L},\gamma^{-(L-1)})$. 
\end{proof}

The next lemma gives a formula for the derivative. Recall, if $\theta\in(\gamma^{-L},\gamma^{-(L-1)}]$ then the formula for the intermediate dimension $s(\theta) = \dim_{\theta} \Lambda$ is
\begin{equation}\label{eq:61}
\gamma^L\theta \log N - (\gamma^L\theta - 1) t_L(s(\theta)) + \gamma(1-\gamma^{L-1}\theta)\big(\log M - I(t_L(s(\theta)))\big) - s(\theta)\log n = 0.
\end{equation}

\begin{lemma}\label{lem:60}
For all $L\in\mathbb{N}$ and $\theta\in(\gamma^{-L},\gamma^{-(L-1)})$, we have $\partial_- \dim_{\theta} \Lambda  =  \partial_+ \dim_{\theta} \Lambda = s'(\theta)$, where
\begin{equation*}
s'(\theta) = \frac{\gamma^L}{\log n} \cdot \frac{\log N - t_L(s(\theta))-\log M + I\big(t_L(s(\theta))\big) }{ 1+\big(\gamma^L \theta -1 + \gamma(1-\gamma^{L-1}\theta)I'\big(t_L(s(\theta))\big)\big) \cdot A_L(\theta) },
\end{equation*}
where $I'$ denotes the derivative of the rate function $I$, and 
\begin{equation*}
A_L(\theta)\coloneqq \sum_{\ell=0}^{L-1}\gamma^{\ell}\prod_{m=1}^{\ell}I'\big(t_{L-m}(s(\theta))\big), 
\end{equation*}
with the empty product defined to be 1. It follows that for $\theta=\gamma^{-L}$
\begin{align}
\partial_- \dim_{\gamma^{-L}} \Lambda &= \frac{\gamma^{L+1}}{\log n} \cdot \frac{\log N - t_{L+1}(s(\gamma^{-L}))-\log M + I\big(t_{L+1}(s(\gamma^{-L}))\big) }{ 1+(\gamma-1) \cdot A_{L+1}(\gamma^{-L}) }, \label{eq:62} \\*
\partial_+ \dim_{\gamma^{-L}} \Lambda &= \frac{\gamma^L}{\log n} \cdot \frac{\log N - t_L(s(\gamma^{-L}))-\log M + I\big(t_L(s(\gamma^{-L}))\big) }{ 1+(\gamma-1)I'\big(t_L(s(\gamma^{-L}))\big) \cdot A_L(\gamma^{-L}) }. \label{eq:63}
\end{align}
\end{lemma}

\begin{proof}
We first show by induction on $L$ that
\begin{equation}\label{eq:60}
\frac{\mathrm{d}}{\mathrm{d}\theta}t_L(s(\theta)) = s'(\theta)\cdot \log n \cdot A_L(\theta).
\end{equation}
For $L=1$, we have $t_1(s(\theta))=\big(s(\theta)-\frac{\log M}{\log m}\big)\log n$ and $A_1(\theta)=1$, so the claim holds. Assuming~\eqref{eq:60} for $L-1$, we now prove for $L$:
\begin{align*}
\frac{\mathrm{d}}{\mathrm{d}\theta}t_L(s(\theta)) &= \frac{\mathrm{d}}{\mathrm{d}\theta}T_{s(\theta)}\big(t_{L-1}(s(\theta))\big) = s'(\theta)\cdot \log n + \gamma I'\big( t_{L-1}(s(\theta)) \big)\cdot \frac{\mathrm{d}}{\mathrm{d}\theta}t_{L-1}(s(\theta)) \\
&= s'(\theta)\cdot \log n + \gamma I'\big( t_{L-1}(s(\theta)) \big)\cdot s'(\theta)\cdot \log n \cdot A_{L-1}(\theta) \\
&= s'(\theta)\cdot \log n \cdot \bigg( 1+ \sum_{\ell=0}^{L-2}\gamma^{\ell+1}\prod_{m=0}^{\ell}I'\big(t_{L-1-m}(s(\theta))\big) \bigg) \\
&= s'(\theta)\cdot \log n \cdot A_L(\theta),
\end{align*}
completing the proof of~\eqref{eq:60}.

Differentiating~\eqref{eq:61} with respect to $\theta$,
\begin{multline*}
s'(\theta)\cdot \log n = \gamma^L \log N - \gamma^L t_L(s(\theta)) -(\gamma^L \theta-1)\frac{\mathrm{d}}{\mathrm{d}\theta}t_L(s(\theta)) \\*
-\gamma^L\big( \log M - I\big(t_L(s(\theta)) \big) \big)  -\gamma(1-\gamma^{L-1}\theta) \frac{\mathrm{d}}{\mathrm{d}\theta}I\big(t_L(s(\theta))\big).
\end{multline*}
Using~\eqref{eq:60}, after rearranging we obtain the formula for $s'(\theta)$. The claims for $\theta=\gamma^{-L}$ follow by analyticity.
\end{proof}

\begin{proof}[Proof of part~\ref{itemii}]
Follows directly from Lemma~\ref{lem:60}.
\end{proof}

The following lemma describes the behaviour of $t_{L(\theta)} (\dim_{\theta} \Lambda)$ for small $\theta$. 
\begin{lemma}\label{lem:tlconvergence}
We have \begin{itemize} \item $\lim_{L \to \infty} t_{L-1}(\dim_{\gamma^{-(L-1)}} \Lambda) = \liminf_{\theta \to 0^+} t_{L(\theta)} (\dim_{\theta} \Lambda) = t^*$
\item $\lim_{L \to \infty} t_L(\dim_{\gamma^{-(L-1)}} \Lambda) = \limsup_{\theta \to 0^+} t_{L(\theta)} (\dim_{\theta} \Lambda) = T_{\dim_{\mathrm H} \Lambda}(t^*)$
\end{itemize}
\end{lemma}
\begin{proof}
By~\eqref{eq:44}, the equation that $\dim_{\gamma^{-(L-1)}} \Lambda$ satisfies from Theorem~\ref{thm:main}, and the fact that the intermediate dimensions and rate function are continuous, we have $t_{L-1}(\dim_{\gamma^{-(L-1)}} \Lambda) \to t^*$ as $L \to \infty$. It follows that $t_L(\dim_{\gamma^{-(L-1)}} \Lambda) \to T_{\dim_{\mathrm H} \Lambda}(t^*)$, and that $\limsup_{\theta \to 0^+} t_{L(\theta)} (\dim_{\theta} \Lambda) \geq T_{\dim_{\mathrm H} \Lambda}(t^*)$. By considering $\theta > \gamma^{-L}$ very close to $\gamma^{-L}$, we see that $\liminf_{\theta \to 0^+} t_{L(\theta)} \dim_{\theta} \Lambda \leq t^*$. If $\gamma^{-L} < \theta \leq \gamma^{-(L-1)}$ then $t_L(\dim_{\gamma^{-L}} \Lambda) < t_L(\dim_{\theta} \Lambda) \leq t_L(\dim_{\gamma^{-(L-1)}} \Lambda)$. Therefore 
\begin{align*}
 t^* = \lim_{L \to \infty} t_L(\dim_{\theta} \Lambda) \leq \liminf_{\theta \to 0^+} t_{L(\theta)} (\dim_{\theta} \Lambda) &\leq \limsup_{\theta \to 0^+} t_{L(\theta)} (\dim_{\theta} \Lambda) \\
 &\leq \lim_{L \to \infty} t_L(\dim_{\gamma^{-(L-1)}} \Lambda) \\*
 &= T_{\dim_{\mathrm H} \Lambda}(t^*),
 \end{align*}
completing the proof. 
\end{proof}

\begin{proof}[Proof of part~\ref{itemiii}]
For brevity, let us write 
\begin{equation}\label{eq:64}
s'(\theta) = \frac{\gamma^L}{\log n}\cdot \frac{f_L(\theta)}{1+g_L(\theta)A_L(\theta)}.
\end{equation}
Lemma~\ref{lem:51} ensures that $\underline{t}<t_1(\dim_{\mathrm{H}}\Lambda)<t_{\ell}(s(\theta))<T_{\dim_{\mathrm{H}}\Lambda}(t^{\ast})<\overline{t}$ for all $1\leq \ell\leq L$. Using that $I$ is strictly increasing and convex, there exist constants $c_1,c_1',c_2,c_2'>0$ independent of $\theta$ such that for all $t_1(\dim_{\mathrm{H}}\Lambda)\leq t\leq T_{\dim_{\mathrm{H}}\Lambda}(t^{\ast})$,
\begin{equation*}
0<c_1< I(t) <c_1'<\log M - H(\mathbf{P}) \;\text{ and }\; 0<c_2< I'(t) <c_2'<1.
\end{equation*}
Hence, recalling that $\overline{t} \coloneqq \log N - H(\mathbf{P})$ and $I(\overline{t}) = \log M - H(\mathbf{P})$, there exists $c_3>0$ such that the numerator
\begin{equation*}
f_L(\theta)= \overline{t} - t_L(s(\theta)) - \big( I(\overline{t}) - I\big(t_L(s(\theta))\big) \big) \geq c_3 >0.
\end{equation*}
Furthermore, there also exists $c_4>0$ such that $0<c_4\leq g_L(\theta)\leq c_4^{-1}<\infty$. Therefore,
\begin{equation*}
s'(\theta)\geq \frac{c_3}{\log n} \cdot \frac{\gamma^L}{1+c_4^{-1}A_L(\theta)} \geq \frac{c_3}{\log n} \cdot \frac{\gamma^L}{1+c_4^{-1}\cdot \frac{\gamma^L-1}{\gamma-1}} \geq \frac{c_3}{\log n \big(\gamma^{-1}+\frac{c_4^{-1}}{\gamma-1}\big)} =:C_0 > 0. 
\end{equation*}%

Next we show that $\partial_- \dim_{\gamma^{-L}} \Lambda < \partial_+ \dim_{\gamma^{-L}} \Lambda$ for all $L\in\mathbb{N}$ from~\eqref{eq:62} and~\eqref{eq:63}. We can divide both~\eqref{eq:62} and~\eqref{eq:63} by $\gamma^L/\log n$. We claim that
\begin{equation*}
\gamma^{-1}\cdot  A_{L+1}(\gamma^{-L}) - I'\big(t_L(s(\gamma^{-L}))\big) \cdot A_L(\gamma^{-L}) = \gamma^{-1},
\end{equation*} 
implying $\gamma^{-1}\cdot (1+g_{L+1}(\gamma^{-L})\cdot A_{L+1}(\gamma^{-L})) = 1+g_{L}(\gamma^{-L})\cdot A_{L}(\gamma^{-L})$. Indeed, applying the definition of $A_{L+1}(\gamma^{-L})$ and $A_{L}(\gamma^{-L})$, observe that it is a telescopic sum with only $\gamma^{-1}$ not cancelling out. Since we also have 
\begin{equation*}
t_{L+1}(s(\gamma^{-L})) - t_{L}(s(\gamma^{-L})) -\big( I\big(t_{L+1}(s(\gamma^{-L}))\big) - I\big(t_{L}(s(\gamma^{-L}))\big) \big)>0 
\end{equation*}
for the same reason that $f_L(\theta)>0$, it follows that $\partial_- \dim_{\gamma^{-L}} \Lambda < \partial_+ \dim_{\gamma^{-L}} \Lambda$. Finally, by Lemma~\ref{lem:tlconvergence}, 
\[ \frac{\partial_+ \dim_{\gamma^{-L}} \Lambda}{\partial_- \dim_{\gamma^{-L}} \Lambda} \xrightarrow[(L \to \infty)]{} \frac{\log N - t^* - \log M + I(t^*)}{\log N - T_{\dim_{\mathrm H} \Lambda}(t^*) - \log M + I(T_{\dim_{\mathrm H} \Lambda}(t^*))} \in (1,\infty). \qedhere \]%
\end{proof}

\begin{proof}[Proof of part~\ref{itemiv}]
The idea is to estimate $L(\theta)$ (which is the number of the $t_i(s(\theta))$) in terms of $s(\theta) \coloneqq \dim_{\theta} \Lambda$. We do this by using the fact that $t'$ is a neutral fixed point of the function $T_{\dim_{\mathrm H} \Lambda}$, and 
\[ t_1(s(\theta)) \xrightarrow[\theta \to 0^+]{} \left(\dim_{\mathrm H} \Lambda - \frac{\log M}{\log m}\right) \log n < t' \] 
and $\liminf_{\theta \to 0^+} t_{L(\theta)}(s(\theta)) > t'$ by Lemmas~\ref{lem:tlconvergence} and~\ref{lem:tprimelesststar}, so most of the $t_i(s(\theta))$ lie close to $t'$. 

By Lemma~\ref{lem:41} and Taylor's theorem, since $T_s'(t') = 1$ and $T_s''(t') > 0$, there exists $c \geq 1$ such that for all $t \in (\underline{t},\overline{t})$, 
\begin{align}\label{eq:taylor}
\begin{split}
 T_s(t') + T_s'(t')(t-t') + c^{-1}(t-t')^2 &\leq T_s(t) \leq T_s(t') + T_s'(t')(t-t') + c(t-t')^2; \\*
t +  c^{-1}(t-t')^2 &\leq T_s(t) \leq t + c(t-t')^2.
\end{split}
\end{align}
If $L$ is large enough and $k_L \coloneqq \big\lfloor \max\left\{ L/10 , \frac{2}{\log 2} \log \left( \frac{\overline{t} - \underline{t}}{L(s(\theta)-\dim_{\mathrm H}\Lambda)\log n}\right) \right\} \big\rfloor$ then 
\begin{equation*}
t'-2^{k_L}\frac{L(s(\theta)-\dim_{\mathrm H}\Lambda)\log n}{4} < \underline{t} \;\text{ and }\; t' + 2^{k_L}\frac{L(s(\theta)-\dim_{\mathrm H}\Lambda)\log n}{4} > \overline{t}. 
\end{equation*}

Suppose $k \in \{1,2,\dotsc,k_L\}$. Then by~\eqref{eq:taylor}, 
\begin{align*}
 \# \Big\{ i \in \{1,2,\dotsc,L(\theta) \} : t' - 2^k\frac{L(s(\theta)-\dim_{\mathrm H}\Lambda)\log n}{4} &< t_i(s(\theta)) \\*
 &\phantom{\leq}\leq 2^{k-1} \frac{L(s(\theta)-\dim_{\mathrm H}\Lambda)\log n}{4} \Big\} \\
 &\leq 1 + \frac{16c}{2^k L(s(\theta)-\dim_{\mathrm H}\Lambda)\log n}. 
 \end{align*}
 Summing up, it follows that 
 \[ \# \left\{  i : t_i(s(\theta)) \leq t' - \frac{L(s(\theta)-\dim_{\mathrm H}\Lambda)\log n}{4} \right\} \leq k_L +  \frac{16c}{L(s(\theta)-\dim_{\mathrm H}\Lambda)\log n}, \]
 and we similarly obtain the same bound for the number of $t_i(s)$ which are greater than $t' + \frac{L(s(\theta) + \dim_{\mathrm H}\Lambda)\log n}{4}$. 
 But 
 \[ \# \left\{ i : t' - \frac{L(s(\theta)-\dim_{\mathrm H}\Lambda)\log n}{4} < t_i(s(\theta)) \leq t' + \frac{L(s(\theta)-\dim_{\mathrm H}\Lambda)\log n}{4} \right\} \leq 1 + L/2. \]
 Therefore for $L$ sufficiently large, 
 \[ L \leq  2 k_L +  \frac{32c}{L(s(\theta)-\dim_{\mathrm H}\Lambda)\log n} + 1 + \frac{L}{2} \leq 0.9 L + \frac{40c}{ L(s(\theta)-\dim_{\mathrm H}\Lambda)\log n}. \]
 Decreasing $\theta_0$ further if required, this tells us that for all $\theta < \theta_0$, 
 \[ s(\theta) \leq  \dim_{\mathrm H} \Lambda +  \frac{400c}{L^2 \log n}  \leq  \dim_{\mathrm H} \Lambda + \frac{500c\cdot (\log \gamma)^2}{(\log \theta)^2 \log n} , \]
 proving the upper bound. This shows in particular that $L(s(\theta)-\dim_{\mathrm H}\Lambda)\log n \to 0$ as $\theta \to 0^+$. 
 
 For the lower bound, we may decrease $\theta_0$ further to assume that for all $\theta<\theta_0$, $t_1(\dim_{\mathrm H} \Lambda) < t_1(s(\theta)) < t' - 3L(s(\theta)-\dim_{\mathrm H}\Lambda)\log n$ and $t_L(s(\theta)) > t'$ by Lemmas~\ref{lem:tlconvergence} and~\ref{lem:tprimelesststar}. Then by~\eqref{eq:taylor}, for large enough $L$, 
 \begin{align*}
  L &\geq \# \left\{ i : t' - 3L(s(\theta)-\dim_{\mathrm H}\Lambda)\log n \leq t_i(s(\theta)) \leq t' - L(s(\theta)-\dim_{\mathrm H}\Lambda)\log n \right\} \\
  &\geq \frac{2L(s(\theta)-\dim_{\mathrm H}\Lambda)\log n}{1.1((s(\theta)-\dim_{\mathrm H}\Lambda)\log n + c( 3L(s(\theta)-\dim_{\mathrm H}\Lambda)\log n)^2)}. 
  \end{align*}%
 Rearranging, for large enough $L$,  
 \[ s(\theta) \geq \dim_{\mathrm H}\Lambda + \frac{0.9}{L^2 \cdot 9.9 c \log n} \geq  \dim_{\mathrm H}\Lambda + \frac{(\log \gamma)^2}{20c \log n (\log \theta)^2 }. \qedhere \]
\end{proof} 

\begin{proof}[Proof of part~\ref{itemv}]
One can differentiate~\eqref{eq:64} to obtain
\begin{equation*}
s''(\theta) = \frac{\gamma^L\cdot \big( f'_L(\theta)(1+g_L(\theta)A_L(\theta)) - f_L(\theta)(g'_L(\theta)A_L(\theta)+g_L(\theta)A'_L(\theta)) \big)}{(1+g_L(\theta)A_L(\theta))^2\log n}.
\end{equation*}
The sign of $s''(\theta)$ is determined by the sign of the term in parentheses in the numerator. We know that there exists $c>0$ such that $f_L(\theta), g_L(\theta), A_L(\theta)\geq c$. There also exists $c'>0$ such that
\begin{equation*}
f'_L(\theta) = \underbrace{\big( I'\big(t_L(s(\theta))\big)-1 \big)}_{\leq -c'<0}\cdot \underbrace{\frac{\mathrm{d}}{\mathrm{d}\theta}t_L(s(\theta))}_{\geq c'>0 \text{ by~\eqref{eq:60}}} \leq -(c')^2<0
\end{equation*}
and $A'_L(\theta)>c'$ because all $I'\big(t_{\ell}(s(\theta))\big)$ and $I''\big(t_{\ell}(s(\theta))\big)$ are uniformly positive. Finally,
\begin{equation*}
g'_L(\theta) = \gamma^L \underbrace{\big( 1-I'\big(t_L(s(\theta))\big) \big)}_{\geq c'>0} + \gamma \underbrace{( 1-\gamma^{L-1}\theta)}_{\geq 0} \underbrace{I''\big( t_L(s(\theta)) \big)}_{>0}\cdot \underbrace{\frac{\mathrm{d}}{\mathrm{d}\theta}t_L(s(\theta))}_{\geq c'>0} \geq \gamma^L c'>0.
\end{equation*}
Hence, $s''(\theta)\leq C<0$ for some uniform constant $C>0$, implying that $s(\theta)$ is strictly concave on every interval $[\gamma^{-L},\gamma^{-(L-1)}]$.
\end{proof}
Note that $s''(\theta)\leq C<0$ holds for a constant $C$ that is independent of both $L$ and $\theta$.

\subsection{Proof of Corollary \ref{cor:twoscales}}
We now use the fact that $\dim_{\theta} \Lambda$ is strictly increasing in $\theta$ to prove Corollary~\ref{cor:twoscales}. 

\begin{proof}[Proof of Corollary~\ref{cor:twoscales}]
Fix $\eta>0$ small enough that $\dim_{\mathrm H} \Lambda + \eta < \dim_{\mathrm B} \Lambda$. Since $\dim_{\theta} \Lambda$ exists and is continuous in $\theta$ including at $\theta=0$ (see Theorem~\ref{thm:main} and \cite[Proposition~4.1]{Falconer2020firstintermediate}), we can fix $\theta_1 < 1/\gamma$ small enough that $\dim_{\theta_1} \Lambda + \eta < \dim_{\mathrm B} \Lambda$. 
Let $s \in (\dim_{\mathrm H} \Lambda, \dim_{\theta_1} \Lambda + \eta]$, and for each $\delta \in (0,1)$, let $K = K(\delta) \in \mathbb{N}$ be such that $n^{-K} \leq \delta < n^{-(K-1)}$. 
For some string $\ii$, let $\mathcal{B}_{\lfloor K/\theta\rfloor}^{K,\ii}$ denote the set of level-$\lfloor K/\theta\rfloor$ approximate squares within the level-$K$ approximate square $B_{K}(\ii)$. For each $B_{K}(\ii)$ it is more cost efficient (in terms of $s$-cost, up to irrelevant multiplicative constants depending only on $\Lambda$) to subdivide it into level-$\lfloor K/\theta\rfloor$ approximate squares if and only if
\begin{equation}\label{eq:50}
n^{-Ks} \geq \#\mathcal{B}_{\lfloor K/\theta\rfloor}^{K,\ii} \cdot n^{-sK/\theta}.
\end{equation}

To determine $\#\mathcal{B}_{\lfloor K/\theta\rfloor}^{K,\ii}$ for some $\theta < 1/\gamma$, we compare the sequences that define $B_{K}(\ii)$ and a level-$\lfloor K/\theta\rfloor$ approximate square within it:
\[
\begin{array}{c|c|c|ccc}
	i_1  \;\dotsb\;  i_{K} & \ih_{K+1}  \;\dotsb\;  \ih_{\gamma(K)} &   &&& \\
	\underbrace{j_1  \;\dotsb\;  j_{K}}_{\text{equal}} & \underbrace{j_{K+1} \;\dotsb\; j_{\gamma(K)}}_{\text{same column}} & \underbrace{j_{\gamma(K)+1} \;\dotsb\; j_{\lfloor K/\theta \rfloor }}_{\in[N] \text{ freely}} & \underbrace{\jh_{\lfloor K/\theta \rfloor+1} \;\dotsb\; \jh_{\gamma(K/\theta)}}_{\in[M] \text{ freely}}.
\end{array}
\]
Thus, $\#\mathcal{B}_{\lfloor K/\theta\rfloor}^{K,\ii} = N^{\lfloor K/\theta \rfloor - \gamma(K)}\cdot M^{\gamma(K/\theta)- \lfloor K/\theta \rfloor }\cdot \prod_{\ell=K+1}^{\gamma(K)}N_{\ih_{\ell}}$. Substituting this back into~\eqref{eq:50}, we get after algebraic manipulations that it is more cost-efficient to subdivide if and only if 
\begin{equation*}
s\geq \frac{\theta}{(1-\theta)}\frac{\gamma-1}{\log n} \Bigg( \frac{1}{\gamma(K)-K} \sum_{\ell=K+1}^{\gamma(K)} \log N_{\ih_{\ell}} \Bigg)  
+ \frac{\theta}{1-\theta}( 1/\theta-\gamma ) \frac{\log N}{\log n} + \frac{\gamma-1}{1-\theta}\frac{\log M}{\log n}. 
\end{equation*}
But since $s \leq \dim_{\theta_1} \Lambda + \eta$, if is more cost-efficient to subdivide then the following condition for the average must hold: 
\begin{equation*}
	\frac{1}{\gamma(K)-K} \sum_{\ell=K+1}^{\gamma(K)} \log N_{\ih_{\ell}} \leq \log (N/M) - \Big(\frac{1}{\theta}-1\Big) \frac{\log n}{\gamma-1} ( \dim_{\mathrm B} \Lambda - \dim_{\theta_1} \Lambda - \eta).
\end{equation*}

As $\theta\to 0$, the right-hand side tends to $-\infty$, so there exists $\theta_0 < \theta_1 /2$ small enough that for all $\theta \leq 2\theta_0$ it is more cost efficient not to subdivide any of the level-$K$ approximate squares, if using only scales $\delta$ and $\delta^{1/\theta}$. 
Now, since $\dim_{\theta} \Lambda$ is strictly increasing in $\theta$ by Corollary~\ref{cor:allprop}, there exists $\epsilon < \eta$ small enough that $\dim_{\theta_0} \Lambda + \epsilon < \dim_{2\theta_0} \Lambda$. Then by the definition of $\theta_0$ and since $\dim_{\theta_1} \Lambda + \eta < \dim_{\mathrm B} \Lambda$, there exists $\delta_1 < 1$ such that if $\{U_i\}$ is a cover of $\Lambda$ using just two scales $\delta$ and $\delta'$ with $\delta' \leq \delta^{1/(2\theta_0)} < \delta \leq \delta_1$, then for all $\theta \leq 2\theta_0$, 
\[ \sum_i |U_i|^{\dim_{\theta} \Lambda + \epsilon} \geq \sum_i |U_i|^{\dim_{\theta_1} \Lambda + \eta} \geq 1.\]
Since  $\dim_{\theta_0} \Lambda + \epsilon < \dim_{2\theta_0} \Lambda$, there exists $\delta_0 < \delta_1$ such that if $\{U_i\}$ is a cover using just two scales $\delta$, $\delta'$ with $\delta^{1/(2\theta_0)} < \delta' \leq \delta \leq \delta_0$, then for all $\theta \leq \theta_0$, $\sum_i |U_i|^{\dim_{\theta} \Lambda + \epsilon} \geq \sum_i |U_i|^{\dim_{\theta_0} \Lambda + \epsilon} \geq 1$, completing the proof. 
\end{proof}

\subsection{Proof of Theorems \ref{t:grid} and \ref{thm:multifractal}}\label{s:multifractalproof}

We use primes to denote the parameters of $\Lambda'$, and use notation from Section~\ref{subsec:multifractal}. 
In particular, in this section only, $I'$ will denote the rate function associated with $\Lambda'$, and not the derivative of $I$. 

\begin{lemma}\label{lem:multifractallemma}
Let $\Lambda$ and $\Lambda'$ be two Bedford--McMullen carpets with non-uniform fibres defined on grids of size $m \times n$ and $m' \times n'$ respectively. 
If either of~\ref{item1} or~\ref{itemnow2} from Theorem~\ref{thm:multifractal} hold, then $\frac{\log n}{\log n'} = \frac{\log m}{\log m'} \in \mathbb{Q}$. 
\end{lemma}

\begin{proof}
First assume that~\ref{item1} holds. 
Since the intermediate dimensions of $\Lambda$ and $\Lambda'$ have phase transitions at $\gamma^{-1}$ and $(\gamma')^{-1}$ respectively (see Corollary~\ref{cor:allprop}), we must have $\gamma=\gamma'$. To show that $\frac{\log n'}{\log n} \in \mathbb{Q}$, note that Theorem~\ref{thm:main} and the equality of the intermediate dimensions for $\theta \in (\gamma^{-1},1)$ tells us that for an open interval of $s$, 
\[ \frac{1}{\log n} I\left(\left(s-\frac{\log M}{\log m}\right)\log n\right) = \frac{1}{\log n'} I'\left(\left(s-\frac{\log M'}{\log m'}\right)\log n'\right). \]
Taking Legendre transforms of both sides and using scaling properties of Legendre transforms, using~\eqref{eq:22}, for all $\lambda$,  
\begin{equation}\label{eq:rational1}
 \frac{1}{\log n}\log \left( \frac{1}{M}\sum_{\ih=1}^{M_0}
 R_{\ih} N_{\ih}^{\lambda} \right) + \lambda \frac{\log M}{\log m}  =
  \frac{1}{\log n'}\log \left( \frac{1}{M'}\sum_{\jh=1}^{M_0'} R'_{\jh}  (N'_{\jh})^{\lambda} \right) + \lambda \frac{\log M'}{\log m'}.  
  \end{equation}
  Fix $K \in \mathbb{N}$ large enough that \[ N_1^{\frac{\log n'}{\log n}} (N_2/N_1)^K  <   N'_{M_0'} \cdot (n')^{\left( \frac{\log M'}{\log m'} - \frac{\log M}{\log m}\right)}.\] 
  Then exponentiating~\eqref{eq:rational1}, 
\begin{align*}
  &\frac{M^{\frac{\log n'}{\log n}}}{M'} \sum_{\jh=1}^{M_0'} R'_{\jh} \left( N'_{\jh} \cdot (n')^{\left( \frac{\log M'}{\log m'} - \frac{\log M}{\log m}\right)}\right)^{\lambda} =  \left( \sum_{\ih=1}^{M_0} R_{\ih} N_{\ih}^{\lambda} \right)^{\frac{\log n'}{\log n}}  \\
  &\phantom{--}= (R_1 N_1^\lambda)^{\frac{\log n'}{\log n}} \left( \sum_{k=0}^{K}  \frac{\frac{\log n'}{\log n} \left(\frac{\log n'}{\log n} - 1\right) \dotsb \left(\frac{\log n'}{\log n} - k + 1 \right) }{k!} \left( \sum_{\ih=2}^{M_0}  \frac{R_{\ih}}{R_1}\left(\frac{N_{\ih}}{N_1}\right)^{\lambda} \right)^{k}  \right)  \\*
  &\phantom{----}+  O\left(\left( N_1^{\frac{\log n'}{\log n}} \left(\frac{N_2}{N_1}\right)^{K+1}\right)^{\lambda}  \right)  
  \end{align*}
as $\lambda \to +\infty$, by the generalised binomial theorem. This means that the coefficient of $N_1^{\frac{\log n'}{\log n}} (N_2/N_1)^{K \lambda}$, which is a polynomial equation in $\frac{\log n'}{\log n}$ with rational coefficients, must be 0. So $\frac{\log n'}{\log n}$ is algebraic. But $n^{\frac{\log n'}{\log n}} = n'$, so by the Gelfond--Schneider theorem, $\frac{\log n'}{\log n} \in \mathbb{Q}$. 

Now we assume only~\ref{itemnow2}. We first show that $\frac{\log m'}{\log m} \in \mathbb{Q}$. Since the functions $\beta_{\nu}$ and $\beta_{\nu'}$ are equal, using~\eqref{eq:definebeta} with the change of variable $\lambda \coloneqq \gamma^{-1} + (1-\gamma^{-1}) \xi$, 
\[ \frac{ \frac{\lambda - \gamma^{-1}}{1-\gamma^{-1}} \log N - \log \sum_{\ih=1}^{M_0} R_{\ih} N_{\ih}^{\lambda}}{\log m} 
= \frac{ \frac{\lambda - \gamma^{-1}}{1-\gamma^{-1}} \log N' - \log \sum_{\jh=1}^{M'_0} 
{R_{\jh}N_{\jh}'}^{(\gamma')^{-1} + 
(1-(\gamma'))^{-1}
\frac{\lambda - \gamma^{-1}}{1-\gamma^{-1}} } }{\log m'}\] 
for all $\lambda$. A similar argument to the above using the generalised binomial theorem and Gelfond--Schneider theorem now gives that $\frac{\log m'}{\log m} \in \mathbb{Q}$. 
To show moreover that $\gamma = \gamma'$, note that these quantities are bi-Lipschitz invariants which depend only on the respective carpets, not the choice of iterated function system (see \cite[Theorem~3.3]{Fraser2018secondassouad}). So since $m$ and $m'$ are multiplicatively dependent, we can iterate the IFS to assume without loss of generality that $\Lambda$ is defined on an $m \times n$ grid, and $\Lambda'$ is defined on an $m \times n'$ grid, for the same $m$. Then by~\eqref{eq:definebeta}, 
\begin{equation}\label{eq:rational2} N^{-\xi} \sum_{\ih=1}^{M_0} R_{\ih} N_{\ih}^{\gamma^{-1} + (1-\gamma^{-1})\xi}  =  {N'}^{-\xi} \sum_{\jh=1}^{M'_0} R_{\jh}{N_{\jh}'}^{(\gamma')^{-1} + (1-(\gamma')^{-1})\xi} 
\end{equation}
for all $\xi$. So using a similar argument to the proof of \cite[Theorem~1.2]{RaoPreprintlipschitz}, equating exponential terms and coefficients gives $M_0=M'_0$. Also, $N_{\ih}^{\frac{\log m}{\log n}} = (N'_{\ih})^{\frac{\log m}{\log n'}}$, so 
\begin{equation}\label{eq:rational3} 
N_{\ih}' =  N_{\ih}^{\frac{\log n'}{\log n}} \qquad \mbox{for all} \quad \ih \in \{1,\dotsc,M_0\}. 
\end{equation}
Equating corresponding exponential bases in~\eqref{eq:rational2}, applying~\eqref{eq:rational3} and using that the carpet has non-uniform fibres (so not all $N_{\ih}$ are equal) shows that $n=n'$. %
In particular, $\frac{\log n'}{\log n} \in \mathbb{Q}$, as required. 
\end{proof}

\begin{proof}[Proof of Theorem~\ref{t:grid}]
\ref{i:gridonecarpet}
The equality $\log n / \log m = \log n' / \log m'$ follows from observing the phase transitions of the Assouad spectrum (see~\cite[Theorem~3.3]{Fraser2018secondassouad}) or intermediate dimensions. The fact that $\log n / \log n' \in \Q$ follows from Lemma~\ref{lem:multifractallemma}. 

\ref{i:gridtwocarpets}
The forward implication follows from~\ref{i:gridonecarpet}, and the backward implication holds by iterating the IFSs (recalling the discussion after the statement of Theorem~\ref{t:grid}). 
\end{proof}

We now prove Theorem~\ref{thm:multifractal}. 
Since the intermediate dimension and multifractal spectra obviously do not depend on the grid on which the carpet is defined, in light of Lemma~\ref{lem:multifractallemma} and Theorem~\ref{t:grid} it suffices to assume henceforth that both carpets are defined on the same $m \times n$ grid to begin with. 
We already mentioned that the equivalence of~\ref{itemnow2} and~\ref{itemnow5} was proved in~\cite[Theorem~1.2]{RaoPreprintlipschitz}. To complete the proof, we show that $\ref{item1}\Rightarrow\ref{itemnow3}\Rightarrow\ref{item4}\Rightarrow\ref{itemnow5}\Rightarrow\ref{itemlq}\Rightarrow\ref{itemnow5}\Rightarrow\ref{item4}\Rightarrow\ref{item1}$. Of these, the implication $\ref{item1}\Rightarrow\ref{itemnow3}$ is obvious. 

\begin{proof}[Proof of $\ref{itemnow3}\Rightarrow\ref{item4}$]
Assume $\dim_{\theta}\Lambda = \dim_{\theta}\Lambda'$ on the open interval $(a,b)\subset [\gamma^{-1},1]$. After rearranging the formula in Theorem~\ref{thm:main} for $\dim_{\theta}\Lambda$ in the case $L=1$, we obtain that
\begin{equation*}
\dim_{\theta}\Lambda=\dim_{\mathrm B}\Lambda - \Big( \frac{1}{\theta}-1 \Big) \frac{I\big(t_1(\dim_{\theta}\Lambda)\big)}{\log n}.
\end{equation*}
By Corollary~\ref{cor:allprop} part~\ref{itemi}, $\dim_{\theta}\Lambda$ and $\dim_{\theta}\Lambda'$ are real analytic on $(\gamma^{-1},1)$, hence $\dim_{\theta}\Lambda = \dim_{\theta}\Lambda'$ on the whole interval $[\gamma^{-1},1]$. In particular, $\dim_{\mathrm B}\Lambda=\dim_{\mathrm B}\Lambda'$, so 
\begin{align*}
I\big(t_1(\dim_{\theta}\Lambda)\big) &= I'\big(t'_1(\dim_{\theta}\Lambda')\big) \\
&= I'\Big(\Big(\dim_{\theta}\Lambda'-\frac{\log (M' M/M)}{\log m'}\Big)\log n'\Big) \\
&= I'\Big( \underbrace{\Big(\dim_{\theta}\Lambda-\frac{\log M}{\log m}\Big)\log n}_{=t_1(\dim_{\theta}\Lambda)} -\gamma \log(M'/M) \Big).
\end{align*}
Setting $t=t_1(\dim_{\theta}\Lambda)$, we see that $I(t)=I'(t-\gamma\log(M'/M))$ on the open interval $(t_1(\dim_{a}\Lambda),t_1(\dim_{b}\Lambda))$. Since the rate function is analytic,~\ref{item4} follows. 
\end{proof}

\begin{proof}[Proof of $\ref{item4}\Rightarrow\ref{itemnow5}$]
Assume $I(t) = I'(t-\gamma \log(M'/M))$ on an open interval of $t$. Without loss of generality we assume that $M'\geq M$. Using definition~\eqref{eq:22} of the rate function,
\begin{equation*}
I'(t-\gamma \log(M'/M))=\sup_{\lambda\in \R} \bigg\{\lambda t - \log\bigg( \frac{(M'/M)^{\gamma \lambda}}{M'} \sum_{\ih=1}^{M'_0} R'_{\ih} (N'_{\ih})^{\lambda}\bigg)\bigg\}.
\end{equation*} 
Since $I$ and $I'$ are convex functions, their Legendre transforms must agree on an open interval, implying that
\begin{equation}\label{eq:65}
\frac{1}{M} \sum_{\ih=1}^{M_0} R_{\ih} N_{\ih}^{\lambda} = \frac{1}{M'} \sum_{\ih=1}^{M'_0} R'_{\ih} \big((M'/M)^{\gamma }N'_{\ih}\big)^{\lambda}
\end{equation}
on an open interval of $\lambda$. 

From here the proof follows the idea of the proof of \cite[Theorem~1.2]{RaoPreprintlipschitz}. Taking the $k$-th derivative of both sides of~\eqref{eq:65} with respect to $\lambda$ gives
\begin{equation}\label{eq:66}
\frac{1}{M} \sum_{\ih=1}^{M_0} R_{\ih} N_{\ih}^{\lambda}\cdot (\log N_{\ih})^k = \frac{1}{M'} \sum_{\ih=1}^{M'_0} R'_{\ih} \big((M'/M)^{\gamma }N'_{\ih}\big)^{\lambda} \cdot \big(\log( (M'/M)^{\gamma }N'_{\ih})\big)^k.
\end{equation}
Recall that the $N_{\ih}$ and $N'_{\ih}$ are ordered in decreasing order. Since~\eqref{eq:66} holds for all $k$, the largest term on either side must be equal, so $N_{1}/N'_{1} = (M'/M)^{\gamma}$, and also its coefficient
\begin{equation*}
\frac{R_{1} N_{1}^{\lambda}}{M} = \frac{R'_{1} \big((M'/M)^{\gamma }N'_{1}\big)^{\lambda} }{M'} \;\Longleftrightarrow\; \frac{R'_{1}}{R_{1}} = \frac{M'}{M}\cdot \left(\frac{M}{M'}\right)^{\gamma\lambda} \cdot \left(\frac{N_1}{N'_1}\right)^{\lambda} =  \frac{M'}{M}.
\end{equation*}
After subtracting these terms from both sides of~\eqref{eq:66}, we repeat the argument for the next largest term and so on. If $M_0\neq M'_0$ then after $\min\{M_0,M'_0\}$ steps one side would be 0 and the other non-zero, a contradiction. Hence, we conclude that $M_0= M'_0$ and $N_{\ih}/N'_{\ih}=(R'_{\ih}/R_{\ih})^{\gamma}= (M'/M)^{\gamma}$ for all $\ih=1,\dotsc,M_0$.
\end{proof}

\begin{proof}[Proof of $\ref{itemnow5}\Rightarrow\ref{itemlq}$]
If~\ref{itemnow5} holds, then substituting $R_{\ih}' = R_{\ih}M'/M$ and $N_{\ih}' = N_{\ih}(M'/M)^{-\gamma}$ gives $N' = N(M'/M)^{1-\gamma}$. Substituting into~\eqref{e:lqformula} gives $T_{\nu}(q) = T_{\nu'}(q)$. 
\end{proof}

\begin{proof}[Proof of $\ref{itemlq}\Rightarrow\ref{itemnow5}$]
Equating the constant term and coefficient of $q$ from~\eqref{e:lqformula} gives that 
\[ \sum_{\ih=1}^{M_0'} R_{\ih}' (N_{\ih}')^q = \left(\frac{N'}{N}\right)^{\frac{\log n - q\log n}{\log(n/m)}} \sum_{\ih=1}^{M_0} R_{\ih} (N_{\ih})^q. \]
By the same differentiation argument from~\cite{RaoPreprintlipschitz} that was used in the proof of~$\ref{item4}\Rightarrow\ref{itemnow5}$, $M_0 = M_0'$ and $N_{\ih}/N'_{\ih}=(R'_{\ih}/R_{\ih})^{\gamma}= (N'/N)^{\frac{\log n}{\log (n/m)}}$ for all $\ih$. 
Now observing that 
\[ \frac{M'}{M} = \frac{\sum_{\ih=1}^{M_0} R_{\ih}'}{\sum_{\jh=1}^{M_0} R_{\jh}} = \left(\frac{N'}{N}\right)^{\frac{\log n}{\log (n/m)}}\]
 shows that~\ref{itemnow5} holds. 
\end{proof}

\begin{proof}[Proof of $\ref{itemnow5}\Rightarrow\ref{item4}$]
Assume that $M_0= M'_0$ and $N_{\ih}/N'_{\ih}=(R'_{\ih}/R_{\ih})^{\gamma}= (M'/M)^{\gamma}$ for all $\ih=1,\dotsc,M_0$. Then~\eqref{eq:65} holds for every $\lambda\in\mathbb{R}$. Since both sides of~\eqref{eq:66} are strictly positive for $k=2$, both sides of~\eqref{eq:65} are convex functions of $\lambda$. Hence, their Legendre transforms are also equal: 
\begin{equation*}
\sup_{\lambda\in \R} \bigg\{\lambda t - \log\bigg( \frac{1}{M} \sum_{\ih=1}^{M_0} R_{\ih} N_{\ih}^{\lambda}\bigg)\bigg\} = \sup_{\lambda\in \R} \bigg\{\lambda t - \log\bigg( \frac{(M'/M)^{\gamma \lambda}}{M'} \sum_{\ih=1}^{M'_0} R'_{\ih} (N'_{\ih})^{\lambda}\bigg)\bigg\}
\end{equation*}
for all $t$, which is precisely $I(t) = I'(t-\gamma \log(M'/M))$.
\end{proof}

\begin{proof}[Proof of $\ref{item4}\Rightarrow\ref{item1}$]
Assume $I(t) = I'(t-\gamma \log(M'/M))$ for every $t\in\mathbb{R}$. We claim that for every $s\in(\dim_{\mathrm H}\Lambda,\dim_{\mathrm B}\Lambda)$,
\begin{equation}\label{eq:67}
t'_{\ell}(s) = t_{\ell}(s) - \gamma \log\left(\frac{M'}{M}\right) \quad\text{ for every } \ell\in\N.
\end{equation}
The proof goes by induction on $\ell$. For $\ell=1$, $t_1(s)= \big( s-\log(M'M/M)\log m \big)\log n = t_1(s)-\gamma\log(M'/M)$. Assuming~\eqref{eq:67} for $\ell-1$, we prove for $\ell$:
\begin{align*}
t'_{\ell}(s) &= T'_s(t'_{\ell-1}(s)) \\*
&= \left( s-\frac{\log M'}{\log m} \right)\log n +\gamma I'\left(t_{\ell-1}(s) - \gamma \log\left(\frac{M'}{M}\right)\right) \\
&\stackrel{\ref{item4}}{=} \left( s-\frac{\log(M'M/M)}{\log m} \right)\log n +\gamma I\big(t_{\ell-1}(s)\big) \\
&= T_s(t_{\ell-1}(s)) - \gamma\log\left(\frac{M'}{M}\right), 
\end{align*} 
which completes the proof of~\eqref{eq:67} since $T_s(t_{\ell-1}(s))=t_{\ell}(s)$. 

From~\eqref{eq:67} and assumption~\ref{item4} it immediately follows that 
\begin{equation}\label{eq:68}
I(t_{\ell}(s))=I'(t'_{\ell}(s)). 
\end{equation}
Assumption~\ref{item4} also implies~\ref{itemnow5}, thus we know that
\begin{equation}\label{eq:69}
N'=\sum_{\ih=1}^{M'_0} R'_{\ih} N'_{\ih} \stackrel{\ref{itemnow5}}{=} \sum_{\ih=1}^{M_0} \frac{M'R_{\ih}}{M} \cdot N_{\ih} \left( \frac{M}{M'}\right)^{\gamma} = N\left( \frac{M}{M'}\right)^{\gamma-1}. 
\end{equation}
Writing $s'_{\theta}=\dim_{\theta}\Lambda'$, using~\eqref{eq:67},~\eqref{eq:68},~\eqref{eq:69} and algebraic manipulations, we obtain
\begin{align*}
0&= -s'_{\theta} \log n + \gamma^L\theta \log N' - (\gamma^L\theta - 1) t'_L(s'_{\theta}) + \gamma(1-\gamma^{L-1}\theta)(\log M' - I'(t'_L(s'_{\theta}))) \\
&= -s'_{\theta} \log n + \gamma^L\theta \log N - (\gamma^L\theta - 1) t_L(s'_{\theta}) + \gamma(1-\gamma^{L-1}\theta)(\log M - I(t_L(s'_{\theta}))).
\end{align*} 
By Theorem~\ref{thm:main}, $\dim_{\theta}\Lambda$ is the unique solution to this equation, so $\dim_{\theta}\Lambda=\dim_{\theta}\Lambda'$. This completes the proof of Theorem~\ref{thm:multifractal}. 
\end{proof}

\newgeometry{top=30mm,bottom=30mm,left=29mm,right=29mm}

\chapter*{References}
\addcontentsline{toc}{chapter}{References}

\printbibliography[heading=none]

\restoregeometry

\end{document}